\documentclass[opre]{informs3}

\allowdisplaybreaks

\RequirePackage{tgtermes}
\RequirePackage{newtxtext}
\RequirePackage{newtxmath}
\RequirePackage{endnotes}
\usepackage{changepage}

\OneAndAHalfSpacedXI 

\usepackage{amsmath,amsfonts,cancel}
\usepackage{pgfplots} 
\usepackage{adjustbox}
\usepackage[shortlabels]{enumitem}
\usepackage{algorithm}
\usepackage[noend]{algpseudocode}
\usepackage{mathtools}
\usepackage{dsfont}
\usepackage{boldline} 
\usepackage{tcolorbox}
\usepackage{booktabs}
\usepackage{xspace}
\usepackage{subcaption}
\usepackage{tensor}
\usepackage{natbib}

\algnewcommand{\Inputs}[1]{  
  \State \textbf{Inputs:}
  \Statex \hspace*{\algorithmicindent}\parbox[t]{.8\linewidth}{\raggedright #1}
}
\algnewcommand{\Initialize}[1]{
  \State \textbf{Initialize:}
  \Statex \hspace*{\algorithmicindent}\parbox[t]{.8\linewidth}{\raggedright #1}
}
\algrenewcommand\algorithmicloop{}  
\newcommand{\Begin}{\Loop}
\newcommand{\End}{\EndLoop}

\usepackage{placeins}
\usepackage[colorlinks=true,breaklinks=true,bookmarks=true,urlcolor=magenta,citecolor=magenta,linkcolor=magenta,bookmarksopen=false,draft=false]{hyperref}
\usepackage[capitalize]{cleveref}

\usepackage{thmtools}
\usepackage{thm-restate}

\usepackage{etoolbox}
\makeatletter
\@ifundefined{theoremstyle}{
  \usepackage{amsthm}
  \newtheorem{lemma}{Lemma}
  
  \newtheorem{definition}{Definition}

  \theoremstyle{definition}
  \newtheorem{remark}{Remark}
  \newtheorem{assumption}{Assumption}
}{}
\makeatother
\newenvironment{rproofof}[1]{ \ifdefined\opre \begin{proof}{\it Proof of #1.}
\else \begin{proof}[Proof of #1] \fi }{ \ifdefined\opre 
\Halmos \end{proof} \else \end{proof} \fi  }

\theoremstyle{TH}%

\usepackage[table]{xcolor}
\definecolor{royalblue}{RGB}{54,97,223}
\definecolor{nupurple}{RGB}{64,31,104}
\definecolor{kulblue}{RGB}{36,146,178}
\definecolor{indigoblue}{RGB}{110,141,237}
\definecolor{lightred}{RGB}{255,110,103}
\definecolor{darkorange}{RGB}{205,144,50}
\definecolor{darkcyan}{RGB}{57,167,208}

\usepackage{color-edits}
\addauthor{srs}{orange}
\addauthor{nha}{teal}
\addauthor{mu}{indigoblue}
\addauthor{or}{blue}

\newcommand{\Exp}[1]{\mathbb E \left[ #1 \right]} 

\renewcommand{\Pr}{\mathbb{P}}

\newcommand{\E}{\mathbb{E}}

\DeclarePairedDelimiter{\abs}{\lvert}{\rvert}

\renewcommand{\S}{\mathcal{S}}

\newcommand{\bfs}{\mathbf{s}}
\newcommand{\Hls}{H_{\textsf{LS}}}
\newcommand{\Hbl}{H_{\textsf{BL}}}
\newcommand{\bl}{\textsf{BL}}
\newcommand{\ls}{\textsf{LS}}
\newcommand{\lsl}{\textsf{LSL}}
\newcommand{\NSBLIC}{\textsf{\scalebox{0.9}{NSIC-BL}}\xspace}
\newcommand{\NSLSIC}{\textsf{\scalebox{0.9}{NSIC-LS}}\xspace}
\newcommand{\NSLSICL}{\textsf{\scalebox{0.9}{NSIC-LSL}}\xspace}
\newcommand{\ADSWITCH}{\textsf{\scalebox{0.9}{A}\scalebox{0.75}{D}\scalebox{0.9}{S}\scalebox{0.75}{WITCH}}\xspace}
\newcommand{\LAIS}{\textsf{\scalebox{0.9}{LAIS}}\xspace}
\newcommand{\SWUCRL}{\textsf{\scalebox{0.9}{SWUCRL2-CW}}\xspace}
\newcommand{\IOPEA}{\textsf{\scalebox{0.9}{IOPEA}}\xspace}
\newcommand{\OIOPEA}{\textsf{\scalebox{0.9}{Oracle IOPEA}}\xspace}
\newcommand{\SIOPEA}{\textsf{\scalebox{0.9}{Schedule IOPEA}}\xspace}
\newcommand{\MDPC}{\textsf{\scalebox{0.9}{MDP-C}}\xspace}
\newcommand{\OMDPC}{\textsf{\scalebox{0.9}{Oracle MDP-C}}\xspace}
\newcommand{\SMDPC}{\textsf{\scalebox{0.9}{Schedule MDP-C}}\xspace}

\newcommand{\calE}{\mathcal{E}}
\newcommand{\taugammaopt}{\tau_t^{*_\gamma}}
\newcommand{\Rb}{R_{\tau}}
\newcommand{\Tb}{\mathcal{T}}
\newcommand{\Cb}{6}
\newcommand{\A}{\mathcal{A}}
\newcommand{\Avk}[2]{\ensuremath{\mathcal{A}_{#1}^{#2}}}

\usepackage{multirow}

\crefname{assumption}{assumption}{assumptions}

\mathchardef\mhyphen="2D 
\ifdefined \argmin
\else \DeclareMathOperator*{\argmin}{arg\,min}
\fi
\ifdefined \argmax
\else \DeclareMathOperator*{\argmax}{arg\,max}
\fi

\DeclarePairedDelimiter{\ceil}{\lceil}{\rceil}

\let\originalleft\left
\let\originalright\right
\renewcommand{\left}{\mathopen{}\mathclose\bgroup\originalleft}
\renewcommand{\right}{\aftergroup\egroup\originalright}

\usepackage{graphicx}
\usepackage{csquotes}
\usepackage{graphics}
\usepackage{float}
\pgfplotsset{compat=1.18} 

\usepackage{tikz}
\usetikzlibrary{patterns}
\usetikzlibrary{positioning, angles, quotes, calc, arrows.meta} 
\usetikzlibrary{decorations.pathreplacing}
\usepgfplotslibrary{fillbetween}
\usepackage{subfiles} 

\usepackage{natbib}
 \bibpunct[, ]{(}{)}{,}{a}{}{,}%
 \def\bibfont{\small}%
\renewenvironment{proof}[1][Proof]{%
  \par\noindent{\itshape #1.} \ignorespaces
}{%
  \hfill\Halmos\par
}

\renewcommand{\paragraph}[1]{\subsubsection*{#1}}

\EquationsNumberedThrough    

\TheoremsNumberedThrough     
\ECRepeatTheorems  %

\MANUSCRIPTNO{} 

\begin{document}


\RUNAUTHOR{Amiri, Sinclair, and Udenio}

\RUNTITLE{Non-Stationary Inventory Control with Lead Times}

\TITLE{Non-Stationary Inventory Control with Lead Times}

\ARTICLEAUTHORS{%

\AUTHOR{Nele H. Amiri}
\AFF{Department of Decision Sciences and Information Management,
KU Leuven, \EMAIL{nelehelena.thomsen@kuleuven.be}}

\AUTHOR{Sean R. Sinclair}
\AFF{Department of Industrial Engineering and Management Sciences,
Northwestern University, \EMAIL{sean.sinclair@northwestern.edu}}

\AUTHOR{Maximiliano Udenio}
\AFF{Department of Decision Sciences and Information Management,
KU Leuven, \EMAIL{maxi.udenio@kuleuven.be}}
} 

\ABSTRACT{%
We study non-stationary single-item, periodic-review inventory control problems in which the demand distribution is unknown and may change over time. We analyze how demand non-stationarity affects learning performance across inventory models, including systems with demand backlogging or lost-sales, both with and without lead times. For each setting, we propose an adaptive online algorithm that optimizes over the class of base-stock policies and establish performance guarantees in terms of dynamic regret relative to the optimal base-stock policy at each time step. 
The algorithms leverage the convexity and one-sided feedback structure of inventory costs to enable counterfactual policy evaluation despite demand censoring. In backlogging systems and lost-sales models with zero lead time, our algorithms adapt to unknown demand changes while matching, up to logarithmic factors, the rates known for the corresponding stationary learning problems. In lost-sales systems with positive lead times, the combination of demand censoring and delayed replenishment restricts counterfactual policy evaluation and leads to weaker regret guarantees.
We complement the theoretical analysis with simulation results showing that our methods significantly outperform existing non-oracle benchmarks.
}%





\maketitle

\section{Introduction}  \label{sec:intro}
Data-driven inventory control has transformed large-scale production and supply chain systems by enabling firms to integrate ever-larger streams of data into their replenishment decisions. While inventory control has always relied on historical evidence, traditional approaches have been largely limited to periodic analysis of restricted datasets to parameterize static policies.
Leading retailers such as Amazon and Walmart now have the capability to process high-volume operational data to dynamically update replenishment decisions, enabling them to improve demand forecasting, reduce stockouts and overstock, and accelerate fulfillment~\citep{AmazonInventory, AmazonDelivery, Walmart}.

Despite this, much of the theoretical foundation for inventory control continues to assume that demand follows a fixed, stationary distribution~\citep{zhang2020closing,yuan2021marrying,agrawal2022learning,chen2024learning}, an assumption that is often violated in modern supply chains, where both endogenous and exogenous factors cause structural breaks in demand.
Such regime changes are difficult to anticipate, and even in hindsight it is often challenging to determine whether observed shifts are temporary or permanent.
Moreover, these structural breaks vary widely in their origin, magnitude, and duration. For instance, US online sales during the 2025 Thanksgiving holiday period reached \$44.2 billion, a 7.7\% year-over-year increase, with Black Friday alone experiencing a 10.4\% surge~\citep{Adobe, Mastercard}.  
Exogenous shocks can lead to similarly severe changes. Tightening export controls imposed by China in 2025, for instance, resulted in 40\% of European firms reporting delays in obtaining export licenses, prompting one in three to seek alternative sourcing options~\citep{EUChamberCommerceChina}.  These patterns illustrate the limitations of stationary demand models and highlight the need for adaptive inventory policies capable of responding to {\em unknown} repeated structural shifts in demand.

Under non-stationary demand, decision makers must continually adjust their replenishment policies while remaining robust to stochastic noise, and without knowing {\em if}, {\em when}, or {\em how} the underlying demand distribution has shifted. This challenge is exacerbated by the fact that the frequency and magnitude of changes are typically unknown, precluding any ability to tune learning or adaptation rates in advance. The particularities of inventory systems introduce additional difficulties. Observations may be censored in lost-sales settings, dynamically coupled across periods through inventory carry-over, and delayed due to lead times. With lead times, inventory available today reflects replenishment decisions made under past demand conditions, yet it is consumed by demand from the current regime. 

Much of the existing literature has focused on algorithms that perform optimally under stationary demand~\citep{huh2009asymptotic, huh2011adaptive, yuan2021marrying, agrawal2022learning, chen2024learning}, yet even infrequent deviations from stationarity can lead to substantial performance loss for such methods \citep{tunc2011cost}.
Recent work has started relaxing the stationarity assumption, but introduces other limitations. Some approaches assume that demand changes only infrequently~\citep{chen2021data}, or follows a cyclic pattern with known cycle length~\citep{gong2024bandits}. Others use forecast-based techniques, which separate prediction from control. These methods often degrade when historical data quickly becomes obsolete, and rarely provide strong performance guarantees \citep{theodorou2025forecast}. 
A separate line of work formulates non-stationary inventory control as a general reinforcement learning (RL) problem with finite state-action spaces, or assumes discrete demand~\citep{chen2021data, cheung2023nonstationary, mao2025model}. While broadly applicable, such formulations can scale poorly with lead times and do not exploit inventory-specific feedback structures.
These considerations motivate our study on the implications of non-stationary demand. Specifically, we address the following research questions.
\begin{itemize}
    \item {\em What is the performance cost of unknown, piecewise stationary demand in base-stock policy optimization, and when can adaptive algorithms match stationary learning rates without knowing the extent of non-stationarity in advance?}
    \item {\em How do demand censoring and lead times alter the feedback available for policy evaluation and change point detection?}
\end{itemize}

\subsection{Contributions}
We study single-item, periodic-review inventory control problems with continuous and time-varying demand. The latter is modeled as piecewise stationary with a non-stationarity measure~$S$ which denotes the unknown number of stationary segments over a time horizon~$T$.
Our analysis covers systems with demand backlogging and lost-sales, both with or without a deterministic lead time~$L$. We restrict attention to base-stock policies, which replenish inventory up to a target level after each review.
We measure performance by dynamic regret, defined as the cumulative difference between the stationary average cost of the employed base-stock policy and that of the optimal base-stock policy for the demand distribution in each time step.

\subsubsection*{Feedback Structures for Non-Stationary Base-Stock Learning.}
A central challenge in non-stationary inventory control is the tension between {\em policy evaluation} and {\em change point detection}. Inventory carry-over and lead times couple costs over time, thus accurately estimating the cost of a given replenishment policy requires long runs under a fixed policy. At the same time, effective detection of demand shifts requires frequent experimentation with alternative policies, since otherwise, changes in performance can be confounded with sampling noise.
We use counterfactual policy evaluation to reduce this tension. Building on information-sharing ideas from stationary inventory control and reinforcement learning~\citep{gong2024bandits,zhang2025information,chen2025managing}, we show that trajectories generated under one base-stock level can be reused to estimate the costs of other base-stock levels, with the scope of reuse depending on the inventory model (\cref{lemma:feedback_structure}). Under backlogging, this yields a full-feedback structure and naturally gives rise to a change detection and elimination algorithm. Under lost-sales, the feedback is one-sided, and with positive lead times it is valid only under more restricted policy sequences. These distinctions drive the different algorithmic designs and regret guarantees in the three settings.

\subsubsection*{Regret Lower Bound.} 
We begin by establishing an $\Omega (\sqrt{ST})$ dynamic regret lower bound for non-stationary inventory control with demand backlogging, or lost-sales with zero lead time, measured against the optimal base-stock policy at each time step (\cref{thm:lower_bound}). 
The bound holds even when the degree of non-stationarity~$S$ is known. Since this rate matches the minimax lower bounds for stationary settings (\cref{lemma:lower_bound_stationary} and \cite{zhang2020closing}), it shows that the need for adaptivity to switching demand does not increase the difficulty of learning in these settings when demand is effectively stationary.

\subsubsection*{Algorithms with Regret Guarantees.}
Building on these feedback structures, we develop online algorithms tailored to each inventory model. The algorithms optimize over the class of base-stock policies on the continuous interval~$[0,U]$, given an upper bound~$U$ on the optimal base-stock level. They employ a common structure of successive elimination combined with restarts triggered by detected changes in the cost function, allowing them to adapt to unknown and abrupt demand shifts without requiring prior knowledge of the number of stationary segments~$S$, except where explicitly stated. 
The backlogging algorithm leverages full feedback to detect changes using cost estimates for all base-stock levels. The lost-sales algorithms require additional mechanisms for change point detection to compensate for censored observations: targeted exploration when $L=0$, and a convexity-based method when $L>0$.
For the proposed algorithms, we establish regret guarantees in terms of dynamic regret, with polynomial dependence on the lead time and the size of the policy space~$[0,U]$. See \cref{table:regret_bounds} for a summary of our results.

\begin{table*}[!t]
\resizebox{\textwidth}{!}{
\begin{tabular}{c c c c c c}
\hline\up 
\textbf{Setting} & \textbf{Reference} & \textbf{Demand} & \multicolumn{2}{c}{\textbf{Non-Stat. Measure}} & \textbf{Dynamic Regret} \\ 
\hline\up 
\multirow{3}{*}{\shortstack{Backlogging \\ $L\geq 0$}}
& \multirow{2}{*}{\cite{gong2024bandits}}
& \multirow{2}{*}{\shortstack{continuous, \\ multi-product}}
& \multirow{2}{*}{$H$}
& \multirow{2}{*}{known}
& \multirow{2}{*}{$(L+1)UH^{5/2}T^{5/6}$} \\

& & & & & \\
                                                      & this paper & continuous & $S$ & unknown & $(L+1)\sigma S^{1/2}T^{1/2}$ \\
\specialrule{0.1pt}{\aboverulesep}{\belowrulesep}
\multirow{5}{*}{\begin{tabular}{c}\shortstack{Lost-sales \\ $L=0$}\end{tabular}} & \cite{chen2021data}$^\dagger$ & discrete & $S$ & unknown & $ST^{1/2}$ \\
                                    & \multirow{2}{*}{\cite{cheung2023nonstationary}$^\dagger$} & \multirow{2}{*}{discrete} & \multirow{2}{*}{$V$} & known & $K V^{1/3} T^{2/3}$ \\
                                    & & & & unknown & $K^{7/6}V^{1/4}T^{3/4}$ \\
                                    & \cite{gong2024bandits} & continuous & $H$ & known & $UH^{7/2}T^{1/2}$ \\
                                    & this paper & continuous & $S$ & unknown & $U^{3/2} S^{1/2} T^{1/2}$ \\
\specialrule{0.1pt}{\aboverulesep}{\belowrulesep}
\multirow{2}{*}{\shortstack{Lost-sales \\ $L>0$}} & \multirow{2}{*}{this paper} & \multirow{2}{*}{continuous} & \multirow{2}{*}{$S$} & known & $(L+1)^{2/3} U S^{1/3} T^{2/3}$ \\
& & & & unknown & $(L+1)^{2/3} U \min\{ S T^{2/3}, S^{1/2} T^{3/4} \}$\down \\ \hline
\end{tabular} 
}
\caption{Upper bounds on the dynamic regret (up to logarithmic factors) of non-stationary inventory control algorithms. Non-stationarity measures are number of switches~$S$, variation budget~$V$, and cycle length~$H$. The time horizon is~$T$, lead time~$L$, and $U$ and $K$ are inventory capacity (continuous demand) or the number of inventory levels (discrete demand), respectively. The variable $\sigma^2$ is the variance proxy for the sub-Gaussian demand distribution in the backlogging model. References marked with $^\dagger$ use a stronger regret benchmark. Our guarantees translate to this with the same order when $L=0$ or with knowledge of $S$ (\cref{app:regret_transformation}).}
\label{table:regret_bounds}
\end{table*}

\begin{itemize}
    \item {\em Demand Backlogging.}\; In this model, counterfactual feedback yields a full-information structure. The cost of any base-stock level can be estimated from data collected under any other level. This allows us to combine change point detection with successive elimination in a direct way. Our algorithm, \NSBLIC (\Cref{algo:backlog}), attains dynamic regret ${\tilde{\mathcal{O}}((L+1)\sqrt{ST})}$ without prior knowledge of $S$ (\cref{thm:regret_backlogging}).
    This rate matches the lower bound in $S$ and $T$ up to logarithmic factors (\cref{thm:lower_bound}), and in the stationary case when $S=1$, it recovers the corresponding $\sqrt{T}$ benchmark rate (\cref{lemma:lower_bound_stationary}). This highlights that our algorithm is robust to non-stationarity while remaining near-optimal when demand is stationary.
    \item {\em Lost-Sales, $L = 0$.}\; In the lost-sales model with instantaneous replenishment, demand censoring yields left-sided feedback where data from a higher base-stock level can be used to evaluate any lower level. 
    This asymmetry creates a difficulty for change detection, since previously eliminated high base-stock levels may become optimal after a demand shift. We design \NSLSIC (\Cref{algo:lost_sales}), which combines successive elimination with targeted sampling of the maximal base-stock level to monitor such changes.
    The algorithm achieves expected dynamic regret
    ${\tilde{\mathcal{O}}(U^{3/2}\sqrt{ST})}$ (\cref{thm:regret_LS_L_zero}).  This bound is again near-optimal in both $T$ and unknown $S$ (\cref{thm:lower_bound}), and it matches the rate for the stationary problem~\citep{zhang2020closing}, demonstrating robustness to non-stationarity at no cost to optimality when demand is stationary. This improves upon existing methods with discrete demand \citep{chen2021data}.
    
    \item {\em Lost-Sales, $L > 0$.}\;
    With positive lead times, the left-sided feedback becomes substantially more restrictive.  Counterfactual evaluation is valid only along non-decreasing sequences of base-stock levels. As a result, the targeted exploration strategy used when $L=0$ would repeatedly interrupt the sample paths needed for cost estimation. We design \NSLSICL (\Cref{algo:lost_sales_lead_time}), which instead uses convexity of the cost function together with left-sided feedback to detect relevant changes without explicitly sampling eliminated base-stock levels.
    When $S$ is known, \NSLSICL achieves regret ${\tilde{\mathcal{O}}((L+1)^{2/3} U S^{1/3} T^{2/3})}$, and ${\tilde{\mathcal{O}}((L+1)^{2/3} U \min\{ S T^{2/3}, S^{1/2} T^{3/4}\})}$ when $S$ is unknown (\cref{thm:regret_LS_L_positive}). These are, to the best of our knowledge, the first regret guarantees for non-stationary lost-sales systems with positive lead times.
\end{itemize}
Taken together, these results show how the feedback structure of the inventory model affects the regret guarantees for non-stationary base-stock learning. In the backlogging model and the zero-lead time lost-sales model, counterfactual feedback is sufficiently rich to obtain $\sqrt{ST}$-type regret rates, matching the corresponding lower bound up to logarithmic factors. In essence, there is {\em no inherent penalty for adaptivity to regime changes if demand is stationary}.  In the positive-lead time lost-sales model, censoring and delayed replenishment become more deeply intertwined.  Our algorithm addresses this by using convexity to replace explicit exploration of larger base-stock levels, but the resulting regret guarantees are weaker. Establishing matching lower bounds for this setting is an important open question.

\subsubsection*{Empirical Validation.} Finally, we demonstrate the strong practical performance of our algorithms via extensive computational experiments. Specifically, we empirically investigate how performance depends on (i) lead time $L$, (ii) (knowledge of) the number of changes $S$, (iii) demand censoring, and (iv) the use of counterfactual feedback.  We observe that our algorithms outperform strong benchmark methods, including baselines that have prior knowledge of the degree of non-stationarity or the change points.

\subsubsection*{Paper Organization.}
In the next section, we review the related literature. Our demand backlogging and lost-sales inventory control models are introduced in \cref{sec:model}. \cref{sec:problem} explores the difficulties of designing algorithms for these problems under non-stationary conditions and establishes the regret lower bound. Our algorithms and their theoretical regret guarantees are presented in \cref{sec:algorithm}. We demonstrate their numerical performance in the simulation experiments in \cref{sec:experiments}, and conclude in \cref{sec:conclusions}.
\section{Related Literature}  \label{sec:related_literature}

We discuss the lines of research most closely related to the study of non-stationary online learning and its applications in inventory control. For more background, we refer the reader to the books \cite{agarwal2019reinforcement, lattimore2020bandit} for multi-armed bandits, and \citet{sutton1998reinforcement, puterman2014markov, bertsekas2019reinforcement} for RL. Further details on inventory control models are in \citet{porteus1990stochastic, zipkin2000foundations}.

\paragraph{Non-Stationary Bandits.} 
Multi-armed bandits provide a canonical model for the exploration-exploitation trade-off in sequential decision-making. Most works assume a stationary environment, in which the objective is to compete with the best fixed arm. In non-stationary settings, however, the learner aims to minimize dynamic regret, comparing their performance against the best time-varying sequence of arms. This requires algorithms to carefully balance exploiting historical data with discarding information that has become outdated due to changes in the reward function.

One major line of research focuses on {\em switching bandits}, where the costs of $K$ arms may undergo abrupt changes, with at most $S$ piecewise stationary segments over the horizon~$T$ \citep{garivier2011upper}.  \citet{auer2002nonstochastic} show that any algorithm faces a dynamic regret of $\Omega(\sqrt{KST})$.  
Algorithms designed for switching bandits typically use either passive adaptation (through discounting or sliding windows)~\citep{kocsis2006discounted, garivier2011upper}, or active adaptation, combining stationary algorithms with explicit change point detection~\citep{hartland2007change,liu2018change,cao2019nearly,foussoul2023mnl,lindstaahl2023change,huang2025change}. Because the latter have shown better performance empirically~\citep{peng2024complexity}, the algorithms we propose follow an active-adaption paradigm.
While most algorithms require prior knowledge of $S$, an important exception is \citet{auer2019adaptively} who show that it is possible to attain the same regret without knowledge of the number of switches. The authors introduce \ADSWITCH, an arm-elimination algorithm that judiciously triggers exploration among all suboptimal arms to balance the regret incurred by detection delay with that incurred by exploitation.  \cite{besson2022efficient} obtain similar guarantees under a minimum separation condition on change points. We employ ideas from \ADSWITCH in the lost-sales setting when $L = 0$, where we additionally exploit the information feedback by exploring only the largest base-stock policy $U$.

A second major framework quantifies non-stationarity using a total variation budget $V$, which measures the cumulative change of the cost of any arm and thus captures drifting environments in which changes occur gradually over time~\citep{besbes2015non}.  \citet{besbes2014stochastic} show a minimax lower bound of $\Omega((KV)^{1/3}T^{2/3})$ and propose a near-optimal algorithm for the case of known $V$.  \citet{cheung2022hedging} introduce an algorithm designed for unknown variation budgets which, however, achieves sub-optimal regret $\tilde{\mathcal{O}}((KV)^{1/3} T^{2/3} + K^{1/4} T^{3/4})$.  Other variants of non-stationary bandits include rotting, restless, and Brownian bandits~\citep{levine2017rotting,whittle1988restless,seznec2020single,ortner2012regret,slivkins2008adapting}.

When extending beyond finitely many arms, infinite-armed bandits impose additional structure such as Lipschitz continuity or convexity. For example, \citet{nguyen2025non} establish regret lower bounds for Lipschitz continuous cost functions. In online convex optimization, \citet{besbes2015non} show that strongly convex losses with first-order feedback admit regret on the order of $\sqrt{VT}$, while \citet{wang2025adaptivity} prove that with only zeroth-order feedback or gradient feedback of convex functions the rate degrades to ${\Omega(V^{1/3}T^{2/3})}$.  For general convex functions under zeroth-order feedback, algorithms typically achieve regret on the order of $T^{3/4}$ \citep{flaxman2004online,zhao2021bandit}.  While these results apply to newsvendor problems with no inventory carry-over, they do not cleanly extend to the inventory control systems with dynamics we study.

\paragraph{Non-Stationary Reinforcement Learning.}
In reinforcement learning, the minimax regret lower bound for stationary, undiscounted reward Markov decision processes (MDPs) in the infinite horizon average cost setting is ${\Omega(\sqrt{D_{\max}\abs{\mathcal{S}}\abs{\mathcal{A}}T})}$, where $D_{\max}$ denotes the maximum diameter of the MDP, and $\mathcal{S}$ and $\mathcal{A}$ represent the state and action space, respectively \citep{auer2008near}.
To handle non-stationarity, techniques similar to those in the bandit literature have been adapted to RL \citep{auer2008near,gajane2018sliding,ortner2020variational}.  We note that these works typically measure regret using the expected cumulative cost incurred by the algorithm, whereas our main results use stationary average cost regret, see \cref{app:related_lit_discussion_avg_cost} for further discussion.
A notable contribution is the algorithm of \citet{cheung2023nonstationary}, similar to the bandit algorithm by \citet{cheung2022hedging}, which attains ${\tilde{\mathcal{O}}(D_{\max}\abs{\mathcal{S}}^{2/3} \abs{\mathcal{A}}^{1/2} V^{1/4} T^{3/4})}$ regret, where $V$ denotes the unknown variation budget in the reward and state transition function.
They further establish the lower bound ${\Omega(D_{\max}^{2/3}(\abs{\mathcal{S}}\abs{\mathcal{A}} V)^{1/3} T^{2/3})}$, showing that non-stationary MDPs are fundamentally harder than their stationary counterparts. 
In the episodic RL setting, \citet{mao2025model} follow similar ideas and obtain $\tilde{\mathcal{O}}((\abs{\mathcal{S}}\abs{\mathcal{A}}V)^{1/3} HT^{2/3} + (HT)^{3/4})$ regret, where $H$ is the episode length. \cite{wei2021non} introduce a black-box approach that is near-optimal in the episodic setting, but requires access to either $D_{\max}$ or the amount of variation when applied to the infinite horizon problem.

Beyond finite state and action spaces, \cite{domingues2021kernel} derive regret guarantees for non-stationary finite horizon RL with Lipschitz continuous cost functions of ${\tilde{\mathcal{O}}(H^2 \min\{ V^{1/3} (T/H)^{\frac{2d+2}{2d+3}}, V^{1/2} (T/H)^{\frac{2d+1}{2d+2}}\})}$, where $d$ is the covering dimension of the state-action space equipped with a metric, and the variation budget is known. This naturally covers inventory control problems, but the dimension $d$ is linear in the lead time, resulting in exponential in contrast to the polynomial dependence we establish.
Overall, existing approaches to non-stationary RL typically require either a known amount of non-stationarity, or finite state–action spaces. However, when applied to inventory control, the state space grows exponentially in the lead time.

\paragraph{Inventory Control with Non-Stationary Demand.}
Early work on inventory control under non-stationarity assumes that the demand process follows a specific parametric form, such as Markovian or autoregressive processes \citep{sethi1997optimality,graves1999single,treharne2002adaptive}.
A complementary stream adopts forecast-based frameworks, which serially apply time-series estimation techniques to inform subsequent replenishment decisions \citep{cao2019quantile,ren2024data}. These studies emphasize the tight interaction between demand forecasting methods and inventory control, showing how forecast bias, uncertainty, and model misspecification propagate into replenishment decisions under evolving demand patterns \citep{strijbosch2011interaction,goltsos2022inventory,babai2025fifty}. Recent contributions extend this using supervised and deep reinforcement learning methods to adapt policies directly from demand forecasts or other contextual information \citep{dehaybe2024deep,temizoz2025zero}. 
Our work contributes to the literature on fully data-driven, non-parametric methods with performance guarantees.

Inventory control models have been studied extensively under stationary demand \citep{huh2009asymptotic,huh2011adaptive,ban2020confidence,yuan2021marrying,chen2024learning,xie2026deepstock}.
In lost-sales models with either censored or uncensored demand, \citet{zhang2020closing} show that the static regret is lower bounded as $\Omega(\sqrt{T})$. Building on the stochastic convex optimization algorithm of \cite{agarwal2011stochastic}, \cite{agrawal2022learning} propose an online algorithm for lost-sales settings with a deterministic lead time~$L$ that achieves static regret ${\tilde{O}((L+1)U\sqrt{T})}$, where $U$ is an upper bound on the optimal base-stock level. Notably, unlike general RL-based approaches, this rate grows only linearly with the lead time.

Several works leverage the specific inventory feedback structures to improve learning efficiency. \citet{yuan2021marrying} utilize the one-sided feedback structure under censored demand to integrate stochastic gradient descent with bandit algorithms, while \citet{gong2020provably,gong2024bandits} apply these information structures to customize $Q$-learning for episodic and unknown cyclic demand patterns. \citet{chen2025managing} demonstrate that this information-sharing property remains valid even in systems with stochastic product returns. \citet{zhang2025information} provide a unifying framework for these concepts by formalizing information-ordered policies.
Our work builds on this line of research.  We use counterfactual policy evaluation and ordered feedback as algorithmic primitives, and focus on how these can be combined with change point detection to learn base-stock policies under unknown, piecewise stationary demand.

One recent step in this direction is provided by \citet{gong2024bandits}, who examine a non-discarding multi-product backlogging model with cyclic demand, positive lead time, and fixed joint-ordering cost. Their proposed algorithm obtains a regret of ${\tilde{\mathcal{O}}((L+1) U H^{5/2} T^{5/6})}$ where $H$ is the known cycle length. For lost-sales problems with zero lead time, they design an algorithm which achieves ${\tilde{\mathcal{O}}(H^{7/2} U T^{1/2})}$ regret. Their regret definition as well as the cyclic-demand model and multi-product structure are different from our setting, but the comparison highlights the role of inventory-specific feedback in non-stationary learning.

The lost-sales problem of gradually changing discrete demand with zero lead time and fixed order cost is studied by \cite{cheung2023nonstationary}. The authors propose a sliding window algorithm with confidence widening, which obtains a rate of ${\tilde{\mathcal{O}}(\abs{\mathcal{S}} V^{1/3} T^{2/3})}$ under a stronger regret benchmark where ${\abs{\mathcal{S}}}$ refers to the finite number of inventory levels, and the variation budget~$V$ is known.
In a finite horizon formulation with episode length $H$, \citet{mao2025model} establish a near-optimal dynamic regret of ${\tilde{\mathcal{O}}((\abs{\mathcal{S}}\abs{\mathcal{A}}V)^{1/3}HT^{2/3})}$ where $\mathcal{A}$ is the finite set of feasible policies and $V$ is known.
Finally, \citet{chen2021data} study inventory control under piecewise stationary discrete demand. 
While their algorithm nearly matches the ${\Omega(\sqrt{T})}$ lower bound, it relies on the assumption that the number of switches satisfies ${S=\mathcal{O}(\log(T))}$, and it does not extend to higher rates of change, continuous demand distributions, or models with lead time. When the frequency of switches is higher, it suffers regret linear in $S$.
Taken together, prior work has developed tools to leverage information feedback for stationary inventory optimization, and non-stationary learning without lead times.  Our contribution is to combine counterfactual policy evaluation with active change point detection for continuous demand base-stock learning under unknown piecewise stationary demand, and to show how the resulting guarantees differ across backlogging as well as lost-sales systems with and without lead time.
\section{Problem Formulation} \label{sec:model}

\paragraph{Technical Notation.} 
In what follows, we use ${[\cdot]^+ \coloneq \max \{ \cdot, 0 \}}$.  For any ${s, t \in \mathbb{N}^2}$ we set $[s,t] \coloneq \{s,\dots, t\}$ and ${[t] \coloneq [1,t]}$. We use $\Pr_F(\cdot)$ and $\E_F[\cdot]$ respectively to denote the probability of an event and the expectation of a random variable when the underlying randomness has a distribution function $F$. Lastly, we use $\mathcal{O}(\cdot)$ to denote rates omitting terms with no dependence on the number of time steps $T$.

\subsection{Model Setup}
We consider a non-stationary variant of classical single-item, periodic-review inventory control problems \citep{huh2009adaptive,huh2009asymptotic,zhang2020closing,agrawal2022learning}.
A retailer is faced with making ordering decisions $Q_t$ online over a period of steps ${t=1, \dots , T}$. At the beginning of each step $t$, the inventory manager observes the current inventory level $I_t$ before replenishment, as well as the $L$ previous unfulfilled orders in the pipeline, denoted ${Q_{t-L}, \dots , Q_{t-1}}$. Here, the integer ${L \in \mathbb{N}_0}$ is the deterministic lead time, or delay in the number of steps between placing an order and receiving it. We set the initial inventory vector ${(I_t, Q_{t-L}, \dots , Q_{t-1})}$ at $t=1$ to the zero vector.
At the beginning of each step, the decision maker picks an order quantity $Q_t$ to arrive at step $t+L$.
Subsequently, the order placed $L$ time steps earlier arrives. Next, an unobserved demand $D_t$ is generated independently from an unknown, potentially time-varying, distribution~$F_t$ with finite mean, and independent from the order decisions. Depending on whether we model backlogged systems or lost-sales, the next on-hand inventory level is given by
\begin{equation} \label{eq:on-hand_inventory_balance}
    I_{t+1} = \begin{cases}
              I_t + Q_{t-L} - D_t & \text{backlogging} \\
              [I_t + Q_{t-L} - D_t]^+ = I_t + Q_{t-L} - Y_t & \text{lost-sales},
              \end{cases}
\end{equation} 
where ${Y_t \coloneq \min\{I_t + Q_{t-L}, D_t\}}$ is the quantity sold in the lost-sales regime.
Under demand backlogging, we see that $I_{t+1}$ may be negative corresponding to unfulfilled orders. With lost-sales, unmet demand $[D_t - I_t - Q_{t-L}]^+$ is lost and censored, that is, only the sales $Y_t$ are observed.

Any understocked inventory ${[D_t - I_t - Q_{t-L}]^+}$ incurs a cost with a known non-negative parameter~$b$ and leftover inventory ${[I_t + Q_{t-L} - D_t]^+}$ incurs holding costs with known non-negative parameter~$h$. The immediate cost is the sum of holding and lost-sales cost. Due to demand censoring under lost-sales, the lost-sales cost is not observable.
Therefore, following \cite{agrawal2022learning} we define the pseudo cost incurred in time step $t$ via
\begin{align} \label{eq:cost_function}
    C_t &= h \cdot [I_t + Q_{t-L} - D_t]^+ + b \cdot [D_t - I_t - Q_{t-L}]^+ - b \cdot D_t \notag \\
        &= h \cdot (I_t + Q_{t-L} - Y_t) - b \cdot Y_t.
\end{align}
The excess inventory and sales in \cref{eq:cost_function} are, as opposed to the lost-sales in the immediate cost function, always observed. We note that the true cost equals the pseudo cost up to an additive $b \cdot D_t$ factor, and hence has no impact on algorithm performance or the average cost function.  We use this cost for theoretical analysis, but in the numerical results we present metrics measured on the true cost function.

\paragraph{Base-Stock Policies.}
The learning agent optimizes over time-varying base-stock policies, a well-known heuristic class which is asymptotically optimal under demand backlogging, and in lost-sales systems if the lead time is zero or the ratio of the lost-sales penalty parameter and the holding cost parameter goes to infinity \citep{karlin1958inventory,scarf1960optimality,huh2009asymptotic}.
At the beginning of each step $t$, the decision maker chooses a base-stock level~$\tau_t$. The order decision~$Q_t$ resulting from this brings the inventory position up to level~$\tau_t$, i.e., $Q_t$ is the non-negative difference between $\tau_t$ and the inventory on-hand plus on-order (less backlogs). Hence we have that $Q_t = [\tau_t - I_t - \sum_{i=1}^L Q_{t-i}]^+$ and the inventory update is given by
\vspace{-3pt}
\begin{equation}
     \label{eq:inventory_balance}
    (I_{t+1}, Q_{t-L+1}, \dots , Q_{t}) = \begin{cases}
              (I_t + Q_{t-L} - D_t, Q_{t-L+1}, \dots , Q_{t-1}, [\tau_t - I_t - \sum_{i=1}^L Q_{t-i}]^+) & \text{backlogging} \\
              (I_t + Q_{t-L} - Y_t, Q_{t-L+1}, \dots , Q_{t-1}, [\tau_t - I_t - \sum_{i=1}^L Q_{t-i}]^+) & \text{lost-sales}.
              \end{cases}
\end{equation}

We define the expected infinite horizon average cost for base-stock policy $\tau$ with demand sampled from distribution $F_t$ as
\vspace{-17pt}
\begin{equation} \label{eq:asymptotic_cost_function}
    \mu_t(\tau) = \mathbb{E}_{F_t} \Big[ \lim_{T \rightarrow \infty} \frac{1}{T} \sum_{t'=1}^T C_{t'}(\tau) \Bigm\vert D \sim F_t, (I_1, Q_{1-L}, \dots , Q_{0})=\mathbf{0} \Big],
\end{equation}
where we denote $C_{t'}(\tau)$ as the pseudo cost at time step~$t'$ when following base-stock policy~$\tau$.
\citet{agrawal2022learning} show that the average pseudo cost differs from the true cost by ${b \cdot \mathbb{E}_{D \sim F_t}[D]}$, a quantity independent of the policy.  Hence, in the following, we focus on the pseudo cost which we refer to as the (expected) cost function. We further use ${\tau_t^* = \argmin_{\tau \in [0, U]} \mu_t(\tau)}$ to denote the optimal base-stock level when the demand distribution is $F_t$, and make the following assumptions on the demand distributions ${(F_t)_{t \in [T]}}$.

\begin{assumption} \label{ass:main_assumption}
For all $t \in [T]$ we have that
\begin{enumerate}
    \item $\tau_t^* \in [0,U]$, \label{ass:bounded_action_space}
    \item $F_t(0) > 0$, \label{ass:positive_F}
    \item Under lost-sales with $L>0$, $\mathbb{E}_{F_t}[\inf\{t' \mid \sum_{i=1}^{t'}D_i \geq 1, D_i \sim F_t \ \forall i \in [t']\}] \leq \nu < \infty$, \label{ass:bounded_depletion_steps}
    \item Under demand backlogging, $F_t$ are sub-Gaussian with variance proxy $\sigma^2$. \label{ass:sub-Gaussian_demand} 
\end{enumerate} 
\end{assumption}

The first condition is common in the literature, and can be imposed, for instance, by assuming that $F_t$ is supported on $[0,U]$ for all $t$, or by adding appropriate mean and variance bounds on ${(F_t)_{t \in [T]}}$ \citep{agrawal2022learning}. The second condition is required for establishing that ${\mu_t(\tau)}$ is convex in $\tau$ for all ${\tau \in [0,U]}$~\citep{janakiraman2004lost}. The third condition states that the expected number of realizations from $F_t$ required for their sum to exceed one is upper bounded by a finite value~$\nu$. We use this in the analysis of the algorithm for the lost-sales setting with a positive lead time to ensure that we can switch from higher to lower base-stock levels in a constant number of time steps (see Definition 1 in \citet{agrawal2022learning} or Assumption 1b in \citep{gong2024bandits}). We emphasize, however, that the value of $\nu$ is not required to be known.
When demand is backlogged, we assume the demand distribution is sub-Gaussian and an upper bound on the parameter~$\sigma$ given. This condition ensures that the inventory backlog does not grow without limit, and is necessary to ensure that our estimates of the average cost function are concentrated. This is a mild condition, satisfied by bounded demands and by most light-tailed distributions common in inventory models and applications.

\paragraph{Objective.} 
We measure the performance of an online algorithm by its dynamic regret, the cumulative difference in the average cost of the base-stock policies employed versus the optimal base-stock policy in each time step for the specific demand distribution~$F_t$.
\begin{definition}[Dynamic Regret] \label{def:regret}
    The {\bf dynamic regret} of a sequence of base-stock policies given by base-stock levels ${(\tau_t)_{t \in [T]}}$ with respect to the optimal sequence of base-stock policies is defined as
    \begin{equation*}
        R(T) \coloneq \sum_{t=1}^T \Big( \mu_t(\tau_t) - \min_{\tau \in [0,U]} \mu_t(\tau)\Big).
    \end{equation*}
\end{definition}

Our regret guarantees measure in-class regret over the class of base-stock policies, rather than regret against a globally optimal policy. Base-stock policies are optimal in the backlogging setting and for lost-sales systems with zero lead time~\citep{karlin1958inventory,scarf1960optimality}. For lost-sales systems with positive lead times, they are known to have strong empirical performance and asymptotic optimality guarantees~\citep{huh2009asymptotic, bu2023asymptotic, yuan2024asymptotic}. Moreover, benchmarking against base-stock policies is standard in online inventory learning due to their analytical tractability~\citep{agrawal2022learning,zhang2020closing}.

\begin{remark}
The dynamic regret in \Cref{def:regret} is defined in terms of expected average costs $\mu_t(\tau_t)$ rather than expected random costs $\E_{F_t}[C_t(\tau_t)]$ realized along a particular sample path, as is common in non-stationary learning research~\citep{garivier2011upper,auer2019adaptively,besson2022efficient,huang2025change}. We use this benchmark to focus on how effectively our algorithms adapt to unknown changes in the demand distribution without knowledge of the degree of non-stationarity. Nevertheless, if the latter is known or when $L=0$, we obtain the same order of regret for the stronger benchmark. Details are provided in \cref{app:regret_transformation}.
\end{remark}
 
Due to demand non-stationarity, regret upper bounds depend on the variation in the sequence of demand distributions~$F_t$.  For instance, consider a setting where $F_t$ changes across each time step.  Then there is no hope of achieving sublinear regret since the algorithm is unable to learn $F_t$.  
We model variations in the demand distribution as piecewise stationary, defined via the total number of changes in $F_t$.

\begin{definition}[Non-Stationarity Measure]
    The {\bf total number of stationary time intervals} in the demand distribution over time horizon~$T$ is
    \begin{equation*}
        S \coloneq 1 + \sum_{t=2}^T \mathds{1}\{F_{t-1} \neq F_t\}.
    \end{equation*}
\end{definition}

Demand distributions, change points, and the total number of changes are unknown to the decision maker. The demand distributions can be considered as being selected by an oblivious adversary, while the demand itself is stochastic and follows their respective distributions.

This regime-switching model, as demonstrated in prior work such as \cite{chen2021data}, reflects the empirical reality that demand often remains stable over certain periods and changes abruptly due to, for instance, seasonal shifts, promotions, new product introductions, or external shocks~\citep{Adobe, Mastercard, EUChamberCommerceChina}. 
Other works consider total variation measures $V$ or constant known demand cycle lengths $H$~\citep{cheung2023nonstationary,keskin2023nonstationary,gong2024bandits,an2025nonstationary,mao2025model}. Since $S$ provides a natural upper bound on both $H$ and $V$, it is the most stringent metric and an interpretable proxy for the operational burden of non-stationarity.

\subsection{Structure of Inventory Control Problems} \label{sec:structure_inventory_control_problems}

In order to learn efficient replenishment policies in non-stationary environments, we must exploit the structures inherent to inventory control problems.
Under \Cref{ass:main_assumption} we establish that the system satisfies several properties regarding the cost functions $\mu_t(\tau)$ for all ${t \in [T]}$, the first of which is that the average cost function is Lipschitz continuous in the base-stock level $\tau$.

\begin{restatable}[Lipschitz Property]{lemma}{LipschitzLemma} \label{lemma:lipschitz}
    The function $\mu_t(\tau)$ is Lipschitz continuous in $\tau$ with factor $\max\{h, b\}$. 
\end{restatable}
The result for the lost-sales regime is derived in \cite[Lemma 8]{agrawal2022learning}. The proof of \Cref{lemma:lipschitz} for backlogging is stated in \Cref{app:proof_lipschitz_concentration_backlogging}.  We will see in the design of our algorithms that we leverage the Lipschitz property by discretizing the {\em continuous} set of base-stock levels for practical implementation.  The next lemma establishes that the average cost $\mu_t(\tau)$ is convex in the base-stock level $\tau$.

\begin{lemma}[Convexity] \label{lemma:convexity}
    For any demand distribution $F_t$ with ${F_t(0) > 0}$, the function $\mu_t(\tau)$ is convex in~$\tau$.
\end{lemma} 
In the backlogging model, convexity of the infinite horizon average cost follows directly from the piecewise linear form of the per-period cost in the base-stock level. Since each per-period cost is convex, the time average over i.i.d.\ demand converges to a convex limit. With lost-sales, \citet[Theorem 12]{janakiraman2004lost} show that the average cost function is convex under \Cref{ass:main_assumption}.\ref{ass:positive_F}.

Existing work leveraging convexity in non-stationary settings requires access to richer feedback beyond zeroth-order information~\citep{besbes2015non,lugosi2024hardness,wang2025adaptivity}. The average cost depends, however, on an implicit, non-smooth distribution of inventory that changes discontinuously with the base-stock level, making differentiation intractable. Importantly, inventory control problems have a unique advantage where the feedback structure contains side information that is richer than standard zeroth-order observations.  As formalized in the following lemma, the cost associated with a given base-stock policy can be counterfactually inferred from trajectories generated under certain other base-stock policies, depending on the inventory model.  This phenomenon has been leveraged for stationary RL settings by \citet{zhang2025information}, and we extend their results to a sequence of time-varying base-stock levels as required in our analysis.

\begin{restatable}[Counterfactual Policy Evaluation]{lemma}{LSFLemma} \label{lemma:feedback_structure} 
    Suppose $(\tau_{t'})_{t' \in [s,t]}$ is the sequence of base-stock levels implemented over a time interval $[s,t]$, and $\tau$ a fixed base-stock level.
    Let ${\left(I_{t'}, Q_{t' - L}, \dots , Q_{t' - 1}\right)_{t' \in [s,t]}}$ and ${(I^{\tau}_{t'}, Q^{\tau}_{t' - L}, \dots , Q^{\tau}_{t' - 1})_{t' \in [s, t]}}$ denote the sequences of inventory vectors collected from policies $(\tau_{t'})_{t' \in [s,t]}$ and $\tau$, respectively, under a shared sequence of demands ${(D_{t'})_{t' \in [s,t]}}$.
    Then the costs $(C_{t'}(\tau))_{t' \in [s,t]}$ incurred by policy $\tau$ can be counterfactually evaluated in each of the following cases
    \begin{enumerate}
        \item Under demand backlogging, for any $\tau \in [0,U]$, \label{item:feedback_structure_backlogging} 
        \item Under lost-sales if $L=0$, for any ${\tau \leq \tau_{t'} \ {\forall t' \in [s, t]}}$, if $\max\{I_s,\tau\} \geq I^{\tau}_s$, \label{item:feedback_structure_LS_L0}
        \item Under lost-sales if $L>0$, for any ${\tau \leq \tau_{t'} \ {\forall t' \in [s, t]}}$, if the sequence $(\tau_{t'})_{t' \in [s,t]}$ is non-decreasing and
        \vspace{-4pt}
        \[
            \textstyle I_{s} \geq I^{\tau}_{s}, \qquad Q_{s - i} \geq Q^{\tau}_{s - i} \ \forall i=1, \dots , L, \qquad \tau_{s} - \Big(I_{s} + \sum_{i=1}^L Q_{s-i}\Big) \geq \tau - \Big(I^{\tau}_{s} + \sum_{i=1}^L Q^{\tau}_{s - i}\Big) \geq 0.
        \] \label{item:feedback_structure_LS_L_positive}
    \end{enumerate}
\end{restatable} 
\vspace{-25pt}
We present the proof of \Cref{lemma:feedback_structure} for lost-sales in \Cref{app:proof_counterfactual_policy_evaluation}. In the backlog setting, \cref{eq:on-hand_inventory_balance} implies that demand realizations are always observed. Since the system's evolution under any base-stock policy is fully determined by the initial state and the demand sequence, this yields a full-feedback structure which enables counterfactual evaluation of any base-stock policy, even when the implemented policy varies over time.
In lost-sales systems, higher base-stock levels dominate lower ones in on-hand inventory, enabling left-sided feedback that allows the sales and costs of lower policies to be inferred from data collected under higher policies despite demand censoring~\citep{gong2020provably,yuan2021marrying,gong2024bandits,zhang2025information,chen2025managing}. 
With positive lead times, however, valid counterfactual evaluation requires a non-decreasing base-stock level, which is highly restrictive. Under left-sided feedback, a large set of policies can be evaluated by starting from a high base-stock level and gradually reducing it as uncertainty declines. Any reduction forces an inventory reset to satisfy the initial state conditions, interrupting sample paths and degrading the statistical accuracy of cost estimates. 
\section{Challenges in Inventory Optimization under Non-Stationary Demand} \label{sec:problem}

Counterfactual feedback renders successive elimination algorithms particularly effective under stationarity. In non-stationary settings, however, change detection requires monitoring the entire set of policies, which fundamentally conflicts with the elimination mechanism. This section characterizes this tension and explains why adapting elimination-based methods to non-stationary inventory control is non-trivial.

\subsubsection*{Stationary Algorithms.}
Algorithms for learning optimal base-stock policies under stationary demand rely on accurate cost estimation and elimination of suboptimal policies using confidence bounds.
Let the costs associated with a base-stock level ${\tau \in [0, U]}$ be observed over a contiguous time interval $[s,t-1]$, either by directly sampling $\tau$, or via the counterfactual feedback in \cref{lemma:feedback_structure}.
We estimate the expected asymptotic cost of $\tau$ based on these observations by
\vspace{-5pt}
\begin{equation} \label{eq:cost_estimates}
    \hat{\mu}(\tau, s, t) \coloneq \frac{1}{t - s} \sum_{t' = s}^{t-1} C_{t'} (\tau).
\end{equation}
\Cref{lemma:concentration} bounds the difference between the finite horizon average cost and the expected infinite horizon~cost.
\vspace{-20pt}
\begin{restatable}[Concentration of Expected Empirical Cost and Expected Asymptotic Cost]{lemma}{ConcentrationLemma} \label{lemma:concentration}
    Let~${[s, t-1]}$ be a stationary time interval and ${\tau \in [0,U]}$.  Suppose that the (counterfactual) inventory position at time~$s$ with which costs for $\tau$ via \cref{eq:cost_function} are computed satisfies ${I_s + \sum_{i=1}^L Q_{s-i} \leq \tau}$ and, if demand is backlogged, that ${t - s \geq L}$.  Then for any ${\delta > 0}$ with probability at least ${1 - \delta}$ we have that
    \vspace{-2pt}
    \begin{equation} \label{eq:concentration_bound}
        \abs*{\hat{\mu}(\tau, s, t) - \mu_s(\tau)} \leq b_{s,t},
    \end{equation}
    where $b_{s,t}$ denotes the confidence radius of the cost estimates obtained from observations in $[s, t-1]$.
    In particular,
    \begin{align*}
        b_{s,t} = \begin{cases}        
            \Hbl \sqrt{\frac{2\log(4(L+1) / \delta)}{t-s}}
            & \text{with } \Hbl = 2\sqrt{2} \sigma \sqrt{(L+1)(Lh^2 + (h+b)^2(4L+5))} \text{ under backlogging} \\[6pt]
            \Hls \sqrt{\frac{2 \log(2 / \delta)}{t-s}}
            & \text{with } \Hls = 72(L+3)U \max\{h,b\} \text{ under lost-sales.}
        \end{cases}
    \end{align*}
\end{restatable}
\noindent The proof for the lost-sales model is presented in \citep[Lemma 4]{agrawal2022learning}, and the backlog case is proved in \Cref{app:proof_lipschitz_concentration_backlogging}.
We note that, by Lemma \ref{lemma:concentration}, obtaining high-confidence cost estimates, and thereby sublinear regret, requires sufficiently many {\em consecutive} observations under a fixed policy. 

Stationary successive elimination algorithms \citep{agrawal2022learning,zhang2025information} maintain an {\em active set} $\A_k$ of candidate base-stock levels that, with high probability, contains the optimum, and iteratively discard suboptimal levels across epochs. In epoch~$k$, the {\em most informative} policy in the active set ${\tau_k = \sup \mathcal{A}_k}$ is executed to estimate the costs of all base-stock levels in $\A_k$ (\cref{lemma:feedback_structure}). After collecting enough samples, the algorithm updates the active set to form $\A_{k+1} \subseteq \A_k$ by retaining only those base-stock levels whose estimated costs are within the confidence radius of the empirical optimum implied by the concentration bounds in \cref{lemma:concentration}, and the process repeats in the next epoch.

\subsubsection*{Extension to Non-Stationary Algorithms.}
Non-stationarity introduces a competing objective beyond learning optimal base-stock levels within stationary segments. To achieve regret sublinear in both time and the number of changes, the algorithm must detect distribution shifts quickly and discard outdated data, while avoiding overly frequent reinitializations.  Since the confidence bounds $b_{s,t}$ scale as $1 / \sqrt{t-s}$ (\cref{lemma:concentration}), frequent resets shorten sample paths and therefore degrade estimation accuracy.

A standard approach combines a stationary learning algorithm with change point detection and restarting \citep{auer2019adaptively,besson2022efficient,huang2025change}. After each round of data collection, the algorithm checks both whether suboptimal policies can be eliminated and whether a change has occurred. When a change is detected, all prior observations are discarded and the algorithm restarts. 
This requires monitoring the performance of all policies that could potentially become optimal, forcing occasional exploration of every policy, including those previously identified as suboptimal \citep{auer2019adaptively}. In inventory control, this would translate to periodically sampling base-stock levels outside the current active set $\A_k$.
Under demand backlogging, the feedback is full-information, and changes can be detected without any explicit exploration, since costs of all base-stock policies can be evaluated counterfactually.

Under lost-sales, however, counterfactual feedback is asymmetric, so explicit {\em forced exploration} of policies outside $\A_k$ is necessary. Executing the most informative policy ${\tau_k = \sup \A_k}$ reveals costs only for lower base-stock levels, hence changes affecting levels larger than $\tau_k$ remain undetectable unless those policies are executed.  Sampling the maximal level~$U$ exposes such changes and underlies our approach for lost-sales with zero lead time. In the presence of both demand censoring and lead times, this approach breaks down. Counterfactual evaluation is only valid along non-decreasing sequences of base-stock levels (\cref{lemma:feedback_structure}), thus switching between $\tau_k$ and $U$ interrupts the learning process. This requirement for sufficiently long runs of consecutive samples of non-decreasing base-stock levels creates a tension between estimation accuracy and exploration. As a result, without exploiting additional problem structure, any algorithm must either sacrifice fast detection, leading to delayed adaptation, or suffer from slow convergence to optimal policies.

Before introducing our algorithms, we establish a minimax lower bound for non-stationary inventory control. The lower bound applies to both demand backlogging with lead times and lost-sales systems with zero lead time. The proof (\cref{app:lower_bound}) builds on the stationary lost-sales lower bound construction of \citet{zhang2020closing}, adapted to a non-stationary environment.

\begin{restatable}{theorem}{LowerBound} \label{thm:lower_bound}
    The expected regret of any learning algorithm for the demand backlogging or the lost-sales model with zero lead time and ${T \geq 5}$ is lower bounded as ${\mathbb{E}[R(T)] = \Omega (\sqrt{ST})}$, even when $S$ is known.
\end{restatable}
When $S=1$, this recovers the standard $\Omega(\sqrt{T})$ stationary lower bound (\cref{lemma:lower_bound_stationary} and \cite{zhang2020closing}). For larger $S$, the $\sqrt{S}$ factor captures the unavoidable cost of adapting to switching demand. The bound matches the minimax rate for the non-stationary multi-period newsvendor problem \citep[Theorem~3]{chen2025learning}, indicating that, in the zero-lead time setting, inventory carry-over does not increase the fundamental difficulty of learning under non-stationarity.

Related $\sqrt{T}$-type lower bounds have been established for episodic lost-sales models with cyclic demand \citep[Theorem~2]{gong2024bandits} and for settings with discrete switching demand \citep[Theorem~1]{chen2021data}. Our lower bound differs in that it is stated for the base-stock regret benchmark under piecewise stationary demand and covers both the backlogging and zero-lead time lost-sales models in our framework. To the best of our knowledge, analogous minimax lower bounds for non-stationary lost-sales inventory systems with positive lead times have not yet been established.
\section{Algorithms for Non-Stationary Inventory Control} \label{sec:algorithm}

Our algorithms for non-stationary inventory control integrate four key components: (i) discretization of the policy space, (ii) an active-set elimination rule for sampling base-stock policies, (iii) a cost-estimation step for updating empirical performance of sets of policies using the feedback structure, and (iv) change point detection for triggering restarts.
Together, these elements form a restarting-based approach that successively eliminates suboptimal base-stock levels and adaptively tracks the current optimal level.
Before specializing the algorithms to the inventory models considered, we provide a high-level overview of each component.

\paragraph{High-Level Algorithm Structure.}
Our algorithms proceed in a hierarchical time structure consisting of {\em episodes $v$} and {\em epochs $k$}.
Episodes correspond to intervals between detected change points in the demand distribution. Thus, each episode represents a period during which the environment is presumed stationary.  
At the start of a new episode, all historical data collected thus far is discarded, and learning restarts from a full candidate set of possibly optimal base-stock levels. Cost estimation and decision-making is thus performed solely based on the data collected in the current episode $v$, the first time step of which we denote as $t_v$.
Within each episode, the algorithm advances through a sequence of epochs ${k = 1,2,\ldots}$, each defined by the base-stock level $\tau_v^k$ currently being implemented. We denote the initial time step of epoch~$k$ of episode~$v$ by~$\alpha_{v,k}$, starting at ${\alpha_{v,1}=t_v}$.

\paragraph{Discretizing the Space of Base-Stock Levels.} Since the set of base-stock levels ${\tau \in [0,U]}$ is continuous, we first discretize the policy space into a $\gamma$-Net $\A_\gamma \coloneq \left\{0, \gamma, 2\gamma, \ldots, \lfloor U/\gamma\rfloor \gamma, U\right\}$. This serves as the candidate set of base-stock levels across all of the algorithms.~\footnote{In backlog systems and under lost-sales with zero lead time, this discretization is only for {\em computational} rather than {\em theoretical} purposes. However, with lost-sales and positive lead times, the discretization is needed to bound the number of epochs per episode.}  By \cref{lemma:lipschitz} the average cost function $\mu_t(\tau)$ is Lipschitz continuous, hence this discretization incurs at most ${\mathcal{O}(T \gamma)}$ regret. We further note that the algorithm only needs to detect changes larger than $\gamma$ in magnitude.

\paragraph{Active Set Elimination.} In each episode~$v$ and epoch~$k$ the algorithm maintains an {\em active set} $\Avk{v}{k}$ of near-optimal base-stock levels.  At the start of a new episode $v$ this is initialized to be the entire discrete set of base-stock levels ${\Avk{v}{1} = \A_\gamma}$.  However, across episodes the algorithm gradually refines this set by eliminating base-stock levels that appear statistically suboptimal.  At a high level, all base-stock levels ${\tau \in \Avk{v}{k}}$ are removed from the current active set for which the estimated cost is significantly larger than the empirically lowest cost.  The algorithm's {\em selection rule} is then to play the {\em most informative base-stock level} $\tau_v^k$ throughout epoch~$k$ of episode~$v$, i.e.,
\vspace{-5pt}
\begin{equation}
\label{eq:selection_rule}
    \tau_v^k = \sup \Avk{v}{k}.
\end{equation}
Note that, at the start of a new episode when ${\Avk{v}{1} = \A_\gamma}$, the selected base-stock level is set to be ${\tau_v^1 = U}$.

\paragraph{Updating Cost Estimates.} Choosing the largest plausibly optimal base-stock level allows the algorithm to fully exploit the feedback structure established in \cref{lemma:feedback_structure}. 
Specifically, from sampling the series of base-stock levels ${(\tau_v^k)_{k \geq 1}}$ in episode~$v$, the cost estimates $\hat{\mu}(\tau, s, t)$ can be evaluated
\begin{enumerate}
    \item Under backlogging, for all ${\tau \in \mathcal{A}_{\gamma}}$ and $s \geq t_v$, i.e., since the start of the current episode $v$,
    \item Under lost-sales if $L=0$, for all ${\tau \leq \tau_v^k}$ and $s \geq t_v$, i.e., since the start of the current episode $v$,
    \item Under lost-sales if $L>0$, for all ${\tau \leq \tau_v^k}$ and ${s \geq \bar{\alpha}_{v,k} \coloneq \inf \{t \geq \alpha_{v,k} \mid I_t + \sum_{i=0}^L Q_{t-i} \leq \tau_v^k\}}$, i.e., since the first time step after the most recent reduction of the base-stock level in $\alpha_{v,k}$ at which the total inventory pipeline falls below $\tau_v^k$.
\end{enumerate}

To determine these counterfactual costs, we define an inventory vector ${(I_{t}^{\tau}, Q_{t-L}^{\tau}, \dots , Q_{t-1}^{\tau})}$ for each base-stock level ${\tau \in \mathcal{A}_{\gamma}}$. When observing the demand~$D_t$ or the counterfactual sales ${Y^{\tau}_t \coloneq \min\{I^{\tau}_t + Q^{\tau}_{t-L}, D_t\}}$ under policy $\tau$, the corresponding vector is updated according to \cref{eq:inventory_balance}, and the cost evaluated using \cref{eq:cost_function}.

\paragraph{Change Point Detection.} While narrowing down the set of potentially optimal base-stock levels $\Avk{v}{k}$, the algorithm monitors for shifts in the underlying demand distribution.
At each step, it checks whether the expected cost of any base-stock level has changed relative to previous subintervals within the current episode.
Specifically, if the cost estimates across two time windows differ by more than a multiple of the confidence radii, this is interpreted as evidence of a change point. Consequently, the current episode is terminated and a new one initialized, all past data is discarded and learning restarts from the complete set of policies. 

With this high-level description in hand, we turn to the specific algorithms developed for the different inventory settings, detailing how each instantiates these components under its respective feedback structure and system dynamics.

\subsection{Demand Backlogging} \label{sec:backlogging}
In the setting where demand is fully backlogged, all unmet demand is carried over to subsequent periods, hence no sales are lost and the full realized demand is observed in every time step.
This full-feedback property greatly simplifies learning since every base-stock level can be evaluated directly from the observed costs.
We now describe how each component of the general framework is instantiated under this setting.  See \cref{algo:backlog} for the full algorithm pseudo code.

\paragraph{Active Set Elimination.}
At time step~$t$ in epoch~$k$ all base-stock levels ${\tau \in \Avk{v}{k}}$ are removed from the current active set for which there exists a time step ${s \in [t_v, t]}$ such that the estimated cost ${\hat{\mu}(\tau, s, t)}$ is significantly larger than the empirically lowest cost. More specifically, the active set in time step~$t$ in epoch~$k+1$ of episode~$v$ is defined as
\begin{align} \label{eq:elimination_condition_bl}
    \Avk{v}{k+1} &= \Avk{v}{k} \setminus \big\{ \tau \in \Avk{v}{k} \mid \exists s \in [t_v,t] \text{ with } t - s \geq L: \hat{\mu}(\tau, s, t) - \underset{\tau' \in \A_{\gamma}}{\min} \hat{\mu}(\tau', s, t) > 4 b_{s, t} \big\}.
\end{align}
Note that a new epoch is triggered whenever this condition is satisfied and a base-stock level is eliminated.
Based on this, the algorithm selects the base-stock level ${\tau_v^{k+1} = \sup \Avk{v}{k+1}}$ to play across each time step in the new epoch.~\footnote{The algorithm does not necessarily need to play the {\em most informative} policy, and can instead select ${\tau_v^{k+1} = \argmin_{\tau \in \Avk{v}{k+1}} \hat{\mu}(\tau, t_v, \alpha_{v,k+1})}$ since under demand backlogging all base-stock policies have the same informational value. We keep the selection rule in \cref{eq:selection_rule} for expository purposes.}

\paragraph{Updating Cost Estimates.} At the start, the counterfactual inventory states ${(I_1^\tau, Q_{1-L}^\tau, \ldots, Q^{\tau}_0)}$ for all base-stock levels ${\tau \in \A_\gamma}$ are initialized to the zero vector. This guarantees that the conditions in \cref{lemma:concentration} are satisfied, thereby enabling concentration of the empirical cost estimates $\hat{\mu}(\tau, s, t)$.
Once initialized, the counterfactual states evolve deterministically according to the base-stock update rule in \cref{eq:inventory_balance}, driven by the realized demand sequence. Because the inventory position under any fixed base-stock policy never exceeds its corresponding base-stock level, the conditions of \cref{lemma:concentration} continue to hold for all subsequent time steps.

\paragraph{Change Point Detection.}
At each step, our algorithm checks whether the expected cost of any base-stock level has shifted relative to earlier intervals in the current episode.
We identify a change in $\mu_t(\tau)$ if for any ${\tau \in \mathcal{A}_{\gamma}}$ and any pair of time intervals $[s_1, s_2]$ and $[s, t]$ with ${t_v \leq s_1 < s}$ and ${s_2 - s_1 \geq L}$ as well as ${t-s \geq L}$ the confidence intervals around the cost estimates are separated. Specifically, if
\begin{equation} \label{eq:change_condition_bl}
    \abs*{\hat{\mu}(\tau, s_1, s_2) - \hat{\mu}(\tau, s, t)} > b_{s_1, s_2} + b_{s,t}
\end{equation}
we conclude that the cost function must have changed. All information collected is discarded and the elimination procedure restarts with the next episode.
Indeed, if there was no change within time interval $[t_v, t]$, then inequality~\eqref{eq:change_condition_bl} and the concentration inequality~\eqref{eq:concentration_bound} are mutually exclusive.~\footnote{Upon detecting a change, it suffices to discard historical observations only for base-stock levels that exhibit statistically significant cost changes, and only prior to step $s$ in line~\ref{algo_line:change_detection_bl} of \cref{algo:backlog}, rather than the change detection time ${t_{v+1}}$.}
To characterize the performance of \NSBLIC, we state the following regret guarantee, with the proof of \Cref{thm:regret_backlogging} provided in \Cref{app:regret_backlogging}.

\begin{algorithm}[!t]
\caption{\NSBLIC: Non-Stationary Base-Stock Policy Optimization (Backlogging)}
 \label{algo:backlog}
\begin{algorithmic}[1]
\State \textbf{Inputs:} Time horizon $T$, discretization parameter $\gamma$, cost parameters $h$ and $b$, lead time $L$, upper bound \Statex \hspace{\algorithmicindent} 
on optimal base-stock levels $U$, upper bound on sub-Gaussian parameter $\sigma$, small probability $\delta$
\State \textbf{Initialize:} Discretize set of base-stock levels $\mathcal{A}_{\gamma}=\{ 0, \gamma, 2\gamma, \dots , \lfloor U/\gamma \rfloor \gamma, U\}$,
\Statex \hspace{\algorithmicindent} Set counterfactual states $(I^{\tau}_1, Q^{\tau}_{1-L}, \dots , Q^{\tau}_{0}) \gets (0, \dots , 0)\ \forall \tau \in \mathcal{A}_{\gamma},\, \text{episode count } v \gets 0, t \gets 0$
\Begin {\bf Start a new episode:}
    \State Episode count $v \gets v+1$, episode start time $t_v \gets t$, epoch count $k \gets 0$, active set $\Avk{v}{1} \gets \A_\gamma$  \label{algo_line:reset_Avk}
    \Begin {\bf Start a new epoch:}
        \State Epoch count $k \gets k +1$ and $\tau_v^k \gets \sup \Avk{v}{k}$
        \While{$t < T$}
            \State $t \gets t + 1$
            \State Play base-stock policy $\tau_v^k$ by ordering $Q_t = \left[\tau_v^k - I_t - \sum_{i=1}^L Q_{t-i}\right]^+$
            \State Observe demand $D_t$ and costs $C_t(\tau)$ for all $\tau \in \mathcal{A}_{\gamma}$ \Comment{Updating Cost Estimates}
            \State Update $(I^{\tau}_t, Q^{\tau}_{t-L}, \dots , Q^{\tau}_{t-1}) \ \forall \tau \in \mathcal{A}_{\gamma}$ by \cref{eq:inventory_balance}
            \If{$t \geq t_v + L$} 
                \If{$\exists \tau \in \mathcal{A}_{\gamma}$ and $s,s_1,s_2$ with $t_v \leq s_1 < s$ s.t. $s_2 - s_1 \geq L$, $t-s \geq L$, and \cref{eq:change_condition_bl} holds} 
                \label{algo_line:change_detection_bl}
                    \State \textbf{Start a new episode} \Comment{Change Point Detection}
                \EndIf
            
                \If{$\exists \tau \in \Avk{v}{k} \text{ and } s \in [t_v,t] \text{ with } t - s \geq L \text{ s.t. } \hat{\mu}(\tau, s, t) - \underset{\tau' \in \A_{\gamma}}{\min} \hat{\mu}(\tau', s, t) > 4 b_{s, t}$} 
                    \State Update $\Avk{v}{k}$ by \cref{eq:elimination_condition_bl} \Comment{Active Set Elimination}
                \State \textbf{Start a new epoch}
                \EndIf
            \EndIf
        \EndWhile
    \End
\End
\end{algorithmic}
\end{algorithm}


\begin{restatable}{theorem}{BackLogRegret} \label{thm:regret_backlogging}
For any $\delta \in (0,1)$, the dynamic regret of \NSBLIC (\Cref{algo:backlog}) is upper bounded with probability at least ${1 - \delta}$ via
\vspace{-8pt}
\begin{equation*}
    R(T) \leq 24 \Hbl \sqrt{ST \log (4(L+1)T^2U/\delta \gamma)} + \left(T\gamma + SLU\right) \max\{h,b\},
\end{equation*}
where ${\Hbl = 2\sqrt{2} \sigma \sqrt{(L+1)(Lh^2 + (h+b)^2(4L+5))}}$. As a result, by setting ${\gamma = \mathcal{O}(T^{-1/2})}$ we have ${R(T) = \tilde{\mathcal{O}}(\sigma (L+1)\sqrt{ST})}$.
\end{restatable}

Notably, \NSBLIC achieves the regret guarantee in \cref{thm:regret_backlogging} without knowing the amount of non-stationarity~$S$. The upper bound aligns, up to logarithmic factors, with the $\Omega(\sqrt{ST})$ lower bound (\cref{thm:lower_bound}), and this rate cannot be improved even if $S$ was known in advance. In particular, when ${S = 1}$, corresponding to the stationary setting, our bound reduces to $\mathcal{O}(\sqrt{T})$, recovering the lower bound for stationary backlogging systems established in \cref{lemma:lower_bound_stationary}.  Thus, the use of an adaptive algorithm in inventory systems with full backlogging incurs no loss in performance when demand is stationary, even in the presence of deterministic lead times. Since base-stock policies are optimal in this setting \citep{karlin1958inventory,scarf1960optimality}, the algorithm effectively learns the optimal policy.
Finally, note that \NSBLIC does not leverage the convexity of the cost function and 
its regret exhibits no asymptotic dependence on the size of the policy space~$U$.
 

\subsection{Lost-Sales with Zero Lead Time}
Next, we consider the lost-sales model with zero lead time. Since demand is censored, the system does not benefit from the full feedback structure available under demand backlogging. Nonetheless, the feedback remains significantly richer than the observation of only the realized cost, due to left-sided information.

\paragraph{Active Set Elimination.}
The active set in time step $t$ in epoch~$k+1$ of episode~$v$ is defined as
\vspace{-2pt}
\begin{equation} \label{eq:elimination_condition_ls}
    \Avk{v}{k+1} = \Avk{v}{k} \setminus \big\{ \tau \in \Avk{v}{k} \mid \exists s \in [t_v,t] : \hat{\mu}(\tau, s, t) - \underset{\tau' \in \A_{\gamma}, \tau' \leq \tau_v^k}{\min} \hat{\mu}(\tau', s, t) > \Cb b_{s, t} \big\}.
\end{equation}
Note that this is the same strategy (up to a constant) as the active set elimination rule in the demand backlog setting (\cref{eq:elimination_condition_bl}).
Again, the algorithm selects the base-stock level ${\tau_v^{k+1} = \sup \Avk{v}{k+1}}$ to play across epoch~$k+1$ (\cref{eq:selection_rule}). 
For the base-stock levels ${\tau > \tau_v^{k+1}}$ that are eliminated at time~$t$ based on the realization of \cref{eq:elimination_condition_ls} by a previous time step ${s \in [t_v, t]}$ in episode~$v$, the estimated mean and the estimated suboptimality gap that led to this eviction are recorded
\vspace{-5pt}
\[
    \tilde{\mu}_v(\tau) \coloneq \hat{\mu}(\tau, s, t) \quad \text{and}\quad \tilde{\Delta}_v(\tau) \coloneq \hat{\mu}(\tau, s, t) - \underset{\tau' \in \A_{\gamma}, \tau' \leq \tau_v^k}{\min} \hat{\mu}(\tau', s, t).
\]
The gathered information about eliminated base-stock levels is used to monitor for changes among these discarded policies. 

The core difficulty in the lost-sales setting with left-sided information is change detection, which requires collecting sufficient subsequent samples at base-stock levels that have already been removed from the active set.
Obtaining these observations may incur high regret each time these policies are selected.
This challenge is illustrated in \Cref{fig:convex_functions}, which shows the graphs of two cost functions, before and after a change at time~$t+1$, with $\tau_t^{*}$ and $\tau_{t+1}^{*}$ denoting their respective optima. 
The average cost of all policies ${\tau \leq \tau_v^k}$ changes only slightly.  Hence, by sampling $\tau_v^k$ and approaching the old optimum $\tau_t^{*}$, it becomes increasingly difficult to detect the shift in the optima from ${\tau_t^* \leq \tau_v^k}$ to ${\tau_{t+1}^* > \tau_v^k}$ receiving only feedback for policies ${\tau \leq \tau_v^k}$.

\begin{figure}[t]
    \centering
    \input{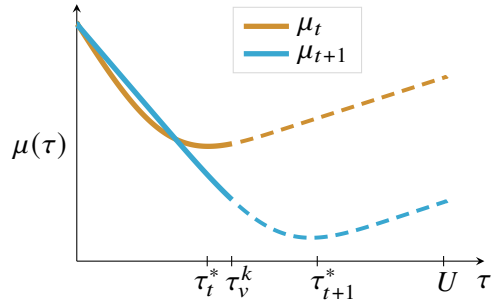}
    \caption{Illustration of the graphs of two expected cost functions, $\mu_t$ and $\mu_{t+1}$, that are difficult to distinguish under the feedback structure of the lost-sales model. Under left-sided feedback, executing policy~$\tau_v^k$ yields cost realizations only for policies $\tau \leq \tau_v^k$ (solid lines), but provides no information on any $\tau > \tau_v^k$ (dashed lines). Therefore, policies greater than $\tau_v^k$ need to be sampled to identify the change in the optimal base-stock level from $\tau_t^* \leq \tau_v^k$ to $\tau_{t+1}^* > \tau_v^k$.}
\label{fig:convex_functions}
\end{figure}
The example highlights the necessity of occasionally sampling base-stock levels greater than $\tau_v^k$ to avoid missing changes in the optimal base-stock level under censored feedback. \NSLSIC (\Cref{algo:lost_sales}) employs a forced exploration schedule inspired by the \ADSWITCH algorithm \citep{auer2019adaptively}. It leverages side information to selectively target exploration, rather than sampling suboptimal base-stock levels uniformly. In particular, the algorithm selects only the most informative policy~$U$, which provides feedback that can be used to infer costs across the entire policy set. 
The exploration schedule introduces a count $N$ for {\em sampling obligations} for base-stock level $U$. At each time step in episode~$v$, with probability ${p_i = \varepsilon_i \sqrt{v/UT \log(2T^2U/\delta\gamma)}}$, the algorithm increments $N$ by ${n_i = \lceil 2 \log(2T^2U/\delta\gamma)/\varepsilon_i^2 \rceil}$, which is sufficient to detect a potential change of size at least ${\varepsilon_i=2^{-i}}$ for ${1 \leq i \leq \log_2(1/\gamma)}$. Whenever $N$ is positive, base-stock level~$U$ is sampled and count $N$ decremented by one. This mechanism ensures the algorithm continues to gather sufficient subsequent full-feedback samples. See \citet{auer2019adaptively} for more details on how the exploration rate is selected.

\paragraph{Updating Cost Estimates.}
At the start at $t=1$, we initialize the counterfactual inventory $I_1^{\tau}$ for each ${\tau \in \A_{\gamma}}$ to zero.
In the following time steps, the stock levels evolve according to the base-stock update \Cref{eq:inventory_balance}.
We have two sets of base-stock policies we need to maintain estimates for, policies ${\tau \leq \tau_v^k}$ and policies ${\tau > \tau_v^k}$. For the first, we update the counterfactual inventory levels and observe costs in each time step, regardless of which level ${\tau_t \geq \tau_v^k}$ is selected. Base-stock policies greater than $\tau_v^k$ are updated only when ${\tau_t = U}$ because only in these time steps the counterfactual sales~$Y^{\tau}_t$ for ${\tau > \tau_v^k}$ are obtained.

For fixed base-stock policies with instantaneous replenishment the condition $I_t^\tau \leq \tau$ is always maintained for ${\tau \leq \tau_t}$.
This ensures we can invoke \cref{lemma:feedback_structure} and determine counterfactual costs via \cref{eq:cost_function} for evaluating all policies ${\tau \leq \tau_t}$. 
Moreover, we can apply \cref{lemma:concentration} and obtain concentration of the empirical cost estimates $\hat{\mu}(\tau,s,t)$ for stationary intervals $[s,t]$ (for ${\tau > \tau_v^k}$ with $[s,t]$ being a subset of a sampling block for policy~$U$).

\paragraph{Change Point Detection.}
Similar to \cref{eq:change_condition_bl} we identify changes in the costs of base-stock levels ${\tau \leq \tau_v^k}$ in episode~$v$ by testing if for a pair of time intervals $[s_1, s_2]$ and $[s, t]$ with ${t_v \leq s_1 < s}$ it holds that
\vspace{-8pt}
\begin{equation} \label{eq:change_condition_ls}
    \left|\hat{\mu}(\tau, s_1, s_2) - \hat{\mu}(\tau, s, t)\right| > b_{s_1, s_2} + b_{s,t}.
\end{equation}
Changes in the costs of previously discarded base-stock levels $\tau >\tau_v^k$ are identified if that level might have become optimal through the condition
\vspace{-12pt}
\begin{equation} \label{eq:change_condition_ls_bad}
    \left|\hat{\mu}(\tau, s, t) - \tilde{\mu}_v(\tau)\right| > \tilde{\Delta}_v(\tau)/4 + b_{s,t}
\end{equation}
for any ${t_v \leq s < t}$ such that ${\tau_{t'} = U}$ for all ${t' \in [s,t]}$.

\begin{algorithm}[!t]
\caption{\NSLSIC: Non-Stationary Base-Stock Policy Optimization (Lost-Sales, Lead Time $L=0$)}
 \label{algo:lost_sales}
\begin{algorithmic}[1]
        \State \textbf{Inputs:} Time horizon $T$, discretization parameter $\gamma$, cost parameters $h$ and $b$, 
        \Statex \hspace{\algorithmicindent} upper bound on optimal base-stock levels $U$, small probability $\delta$
        \State \textbf{Initialize:} Discretize set of base-stock levels $\mathcal{A}_{\gamma}=\{ 0, \gamma, 2\gamma, \dots , \lfloor U/\gamma \rfloor \gamma, U\}$,
        \Statex \hspace{\algorithmicindent} Set counterfactual inventory levels $I^{\tau}_1 \gets 0$ $\forall \tau \in \mathcal{A}_{\gamma},\, \text{episode count } v \gets 0,\, t \gets 0, N \gets 0$ \label{algo_line:initialize_counterfactual_inventory_ls}
        
        \Begin \textbf{Start a new episode:}
        \State Episode count $v \gets v+1$, episode start time $t_v \gets t$, epoch count $k \gets 0$, active set $\Avk{v}{1} \gets \A_{\gamma}$
        \Begin \textbf{Start a new epoch:}
        \State Epoch count $k \gets k + 1$ and $\tau_v^k \gets \sup \Avk{v}{k}$
        
        \While{$t < T$}
            \State $t \gets t + 1$
            \For{all $i \geq 1$ with $\varepsilon_i = 2^{-i} \geq \max\{\gamma, \tilde{\Delta}_v(U)/ 16\Hls\}$ with prob. $2^{-i}\sqrt{\frac{v}{UT\log(2T^2U/\delta\gamma)}}$}
                \State $N \gets N + \lceil 2^{2i+1} \log(2T^2U/\delta\gamma) \rceil$  \Comment{Sampling Obligations for Policy $U$}
            \EndFor
            \If{$N \geq 1$}
                $\tau_t \gets U$ and $N \gets N - 1$
            \Else
                $\ \tau_t \gets \tau_{v}^{k}$
            \EndIf
        
            \State Play base-stock policy $\tau_t$ by ordering $Q_t = \left[\tau_t - I_t\right]^+$
            \State Observe sales $Y^{\tau}_t$ and costs $C_t(\tau)$ for all $\tau \leq \tau_t$ \Comment{Updating Cost Estimates}
            \State Update $I^{\tau}_t \ \forall \tau \leq \tau_t$ by \cref{eq:inventory_balance}

            \If{$\exists \tau \leq \tau_v^k$ and $s,s_1,s_2$ with $t_v \leq s_1 < s$ s.t. \cref{eq:change_condition_ls} holds \Comment{Change Point Detection} \\
            \qquad \qquad \qquad or $\exists \tau > \tau_v^k$ and $t_v \leq s \leq t$ s.t. \cref{eq:change_condition_ls_bad} holds}
            \textbf{Start a new episode} 
            \EndIf
            \If{$\exists \tau \in \Avk{v}{k}$ and $s \in [t_v,t] \text{ s.t. } \hat{\mu}(\tau, s, t) - \min_{\tau' \in \A_{\gamma}, \tau' \leq \tau_v^k} \hat{\mu}(\tau', s, t) > \Cb b_{s, t}$}
                \State Update $\Avk{v}{k}$ by \cref{eq:elimination_condition_ls} \Comment{Active Set Elimination}
                \State Record $\tilde{\mu}_{v}(\tau) \text{ and } \tilde{\Delta}_{v}(\tau) \ \forall \tau \in \Avk{v}{k}\setminus \Avk{v}{k+1}$
            \If{$\tau_v^{k+1} \neq \tau_v^k$} \textbf{Start a new epoch} \EndIf
        \EndIf
        \EndWhile
        \End
    \End
\end{algorithmic}
\end{algorithm}


\cref{thm:regret_LS_L_zero} quantifies the regret performance of the \NSLSIC algorithm.

\begin{restatable}{theorem}{LSzeroLRegret} \label{thm:regret_LS_L_zero}
The expected dynamic regret of \NSLSIC (\Cref{algo:lost_sales}) is upper bounded via
\vspace{-5pt}
\begin{align*}
    \E \left[R(T)\right] < \left(1 - \delta\right) \Bigg(&\left(\frac{2456}{15}+\bigg(\frac{64}{3} + 384\max\{h,b\}\bigg)\sqrt{U} + \frac{256 \log_2(1/\gamma)}{7\sqrt{U}} \right)\Hls \sqrt{ST \log(2T^2U/\delta\gamma)} \\
    &\quad + \frac{64}{3} \Hls S + T\gamma \max\{h,b\} \Bigg) + \delta TU\max\{h,b\},
\end{align*}
where $\Hls = 216U \max\{h,b\}$ and $\delta \in (0,1)$. As a result, by setting ${\gamma = \mathcal{O}\left( T^{-1/2}\right)}$ and $\delta = \Theta(T^{-2})$ we have ${\E \left[R(T)\right] = \tilde{\mathcal{O}}\left( U^{3/2}\sqrt{ST}\right)}$.
\end{restatable}
In contrast to \cref{thm:regret_backlogging}, the guarantee in here is for expected regret since we need to bound the expected number of sampling obligations for base-stock level $U$.
The proof of \Cref{thm:regret_LS_L_zero} is given in \Cref{app:regret_LS_L0}. 

With an optimal discretization, the regret bound matches, up to logarithmic factors, the ${\Omega(\sqrt{ST})}$ minimax lower bound, and this rate cannot be improved even when $S$ is known (\cref{thm:lower_bound}). In particular, when ${S=1}$, it recovers the $\Omega(\sqrt{T})$ lower bound on the static regret for zero-lead time lost-sales inventory systems with stationary demand \citep[Proposition~1]{zhang2020closing}. 
Since base-stock policies are optimal in this setting \citep{scarf1960optimality}, the algorithm learns the optimal policy.
These comparisons establish that \NSLSIC is order-optimal under both stationary and non-stationary demand. 
Moreover, in view of the $\Omega(\sqrt{ST})$ bound established for switching multi-armed bandits \citep{auer2002nonstochastic}, this result demonstrates that the deliberate exploitation of left-sided feedback offsets the complexity introduced by continuous and non-stationary demand, and inventory carry-over across time steps.
Unlike existing approaches \citep{chen2021data}, \NSLSIC requires no restrictions on the degree of non-stationarity, or on the discreteness of the demand space.
Finally, we find that learning with censored feedback requires a substantially more sophisticated method than the uncensored setting, yet the same order of regret remains achievable. In light of \Cref{thm:lower_bound,thm:regret_backlogging}, we conclude that in the absence of lead time, non-stationarity does not impede learnability, even when combined with demand censoring and an unknown frequency of change.


\subsection{Lost-Sales with Positive Lead Time}
Lastly, we extend the preceding analysis to positive lead times. In the lost-sales model with instantaneous replenishment we were able to occasionally switch between policy $\tau_v^k$ and exploratory policy~$U$ to obtain full-feedback samples required for change detection. With a positive lead time, however, we can only infer counterfactual costs as long as the policy played is non-decreasing (\Cref{lemma:feedback_structure}). This implies two constraints, on the update frequency of base-stock level $\tau_v^k$, and the capacity to occasionally explore levels larger than~$\tau_v^k$, since both involve downward switches of the policy pursued.

In fact, with lead time the conditions for left-sided feedback do not allow any {\em explicit} exploration, even when sampling only a single eliminated base-stock level occasionally. 
In \NSLSIC, the required frequency of switches between base-stock levels~$\tau_v^k$ and $U$ is of order~$\sqrt{T}$ (see bound on $R_{2.1}(T)$ in \cref{app:regret_LS_L0}). To illustrate how this affects the regret if the sample path of cost observations is interrupted after every decrease in policy, let an episode be of length~$T$ and consist of $\sqrt{T}$ epochs of length $\sqrt{T}$ each. Then, the regret solely of time steps when $\tau_v^k$ is sampled is ${\sqrt{T}\sum_{t=1}^{\sqrt{T}} b_{1,t} = \tilde{\mathcal{O}}(1) \cdot \sqrt{T} \sum_{t=1}^{\sqrt{T}} \sqrt{1/t}=\tilde{\mathcal{O}}(T^{3/4})}$. This shows that in this setting forced exploration for change detection is incompatible with maintaining algorithmic efficiency.

To overcome this, we leverage side information in combination with convexity. 
Crucially, we only need to change our change point detection strategy for changes after which the new optimal base-stock level is larger than $\tau_v^k$. All other changes can be detected quickly by comparing cost observations of both old and new optimum, which are gathered while implementing base-stock policy~$\tau_v^k$.
Due to convexity, it is possible to obtain the missing {\em right-sided} information about whether a change in the cost function has caused the optimal base-stock level to be larger than $\tau_v^k$, without ever sampling a level greater than $\tau_v^k$. 
Specifically, let $\tau^*_t$ and $\tau^*_{t+1}$ be optima of cost functions $\mu_t$ and $\mu_{t+1}$ before and after a change, respectively, with ${\tau^*_t < \tau_v^k < \tau^*_{t+1}}$. We leverage the fact that ${\mu_t(\tau_v^k)=\mu_{t+1}(\tau_v^k)=\mu_{t+1}(\tau^*_{t+1})}$, hence $\tau_v^k$ is optimal after the change and no reset necessary, or there exists some ${\tau \leq \tau_v^k}$ such that ${\mu_t(\tau) \neq \mu_{t+1}(\tau)}$. In particular, as depicted in \cref{fig:convex_functions_3}, the graphs cannot coincide at both points $\tau_t^*$ and $\tau_v^k$.
This leads to a useful simplification for detecting upward shifts in the optimal base-stock level. In order to determine whether ${\tau_t^* < \tau_v^k < \tau_{t+1}^*}$, it suffices to perform a change point check using only cost observations at base-stock levels ${\tau \leq \tau_v^k}$. These cost observations are directly available from left-sided feedback. A statistically significant difference between ${\mu_t(\tau_t^*)}$ and ${\mu_{t+1}(\tau_t^*)}$, or between ${\mu_t(\tau_v^k)}$ and ${\mu_{t+1}(\tau_v^k)}$, is then enough to certify that a change has occurred.

However, as illustrated in \cref{fig:convex_functions}, change detection may become difficult if the difference in costs is only marginal for all ${\tau \leq \tau_v^k}$. For such cases, we need to ensure that old and new function values are sufficiently different at least at base-stock level~$\tau_v^k$. This, in turn, requires that $\tau_v^k$ is not chosen too close to the optimal base-stock level~$\tau_t^*$, since otherwise high-confidence detection of changes in the underlying cost function may become impossible. Because we cannot rely on strong convexity of the cost function to guarantee such separation, we need to enforce a minimum suboptimality gap of order $\gamma$ at~$\tau_v^k$.

\begin{figure}[!t]
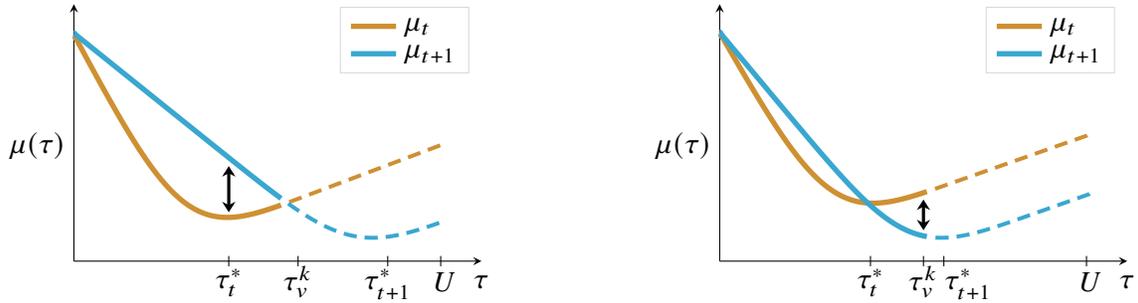

    \centering
    \begin{subfigure}{0.48\textwidth}
        \centering
        \input{figures/convex_functions_3}
        \caption{If $\mu_t(\tau_v^k) = \mu_{t+1}(\tau_v^k)$, then $\mu_t(\tau_{t}^*) \neq \mu_{t+1}(\tau_t^*)$.}
    \end{subfigure}
    \hfill
    \begin{subfigure}{0.48\textwidth}
        \centering
        \input{figures/convex_functions_4}
        \caption{If $\mu_t(\tau_t^*) = \mu_{t+1}(\tau_t^*)$, then $\mu_t(\tau_v^k) \neq \mu_{t+1}(\tau_v^k)$.}
        \label{fig:convex_functions_3_2}
    \end{subfigure}
    \caption{The graphs of two convex functions $\mu_t$ and $\mu_{t+1}$, before and after a change. Their minima shift from $\tau_t^*$ to $\tau_{t+1}^*$, with ${\tau_t^* < \tau_v^k < \tau_{t+1}^*}$. By convexity, the function values cannot coincide at both points $\tau_t^*$ and $\tau_v^k$. We can exploit this property because the costs of both policies are observed (solid lines), as opposed to policies ${\tau > \tau_v^k}$ (dashed lines).}
    \label{fig:convex_functions_3}
\end{figure}

\paragraph{Active Set Elimination.}
To implement these insights into the algorithm, we introduce an additional condition on the elimination rule for the active set. Let ${\bar{\alpha}_{v,k} \coloneq \inf \{t \geq \alpha_{v,k} \mid I_t + \sum_{i=0}^L Q_{t-i} \leq \tau_v^k\}}$ be the first time step after the start of epoch~$k$ in episode~$v$ in which the sum of inventory on-hand and in transit is at most $\tau_v^k$. The active set is updated according to 
\begin{equation} \label{eq:active_set_elimination_condition_lsl}
    \Avk{v}{k+1} = 
            \Avk{v}{k} \setminus \Bigg\{ \tau \in \Avk{v}{k} \Biggm\vert
        \begin{array}{l} \exists s \in [\bar{\alpha}_{v,k},t] : \hat{\mu}(\tau, s, t) - \min_{\tau' \in \A_{\gamma}, \tau' \leq \tau_v^k} \hat{\mu}(\tau', s, t) > 4 b_{s, t}, \\[5pt]
        \hat{\mu}(\tau - \gamma, \bar{\alpha}_{v,k}, t) - \min_{\tau' \in \A_{\gamma}, \tau' \leq \tau_v^k} \hat{\mu}(\tau', \bar{\alpha}_{v,k}, t) > 2b_{\bar{\alpha}_{v,k}, t} + \max \{ h, b\} \gamma
        \end{array}
        \Bigg\}.
\end{equation}
We refer to the first inequality in \cref{eq:active_set_elimination_condition_lsl} as the {\em elimination condition}
\vspace{-3pt}
\begin{equation} \label{eq:elimintation_condition_lsl}
    \hat{\mu}(\tau, s, t) - \min_{\tau' \in \A_{\gamma}, \tau' \leq \tau_v^k} \hat{\mu}(\tau', s, t) > 4 b_{s, t} \text{ for some } s \in [\bar{\alpha}_{v,k},t].
\end{equation}
The second condition in \cref{eq:active_set_elimination_condition_lsl} guarantees a minimum estimated suboptimality gap of the base-stock level in $\A_{\gamma}$ that is the next-smaller level relative to the one being eliminated. In particular, when applied to the smallest eliminated level that exceeds all levels in $\Avk{v}{k+1}$, it enforces that the largest base-stock remaining in the active set ${\tau = \sup \Avk{v}{k+1}}$ satisfies the {\em separation condition}
\vspace{-3pt}
\begin{align} \label{eq:separation_condition}
    \hat{\mu}(\tau, \bar{\alpha}_{v,k}, t) - \min_{\tau' \in \A_{\gamma}, \tau' \leq \tau_v^k} \hat{\mu}(\tau', \bar{\alpha}_{v,k}, t) > 2b_{\bar{\alpha}_{v,k}, t} + \max \{ h, b\} \gamma.
\end{align}

The base-stock policy implemented throughout epoch~$k+1$ is, following \cref{eq:selection_rule}, the largest base-stock level in the active set ${\tau_v^{k+1} = \sup \Avk{v}{k+1}}$.


\paragraph{Updating Cost Estimates.}
By \cref{lemma:feedback_structure} we can counterfactually obtain cost estimates ${\hat{\mu}(\tau,s,t)}$ for all policies ${\tau \leq \tau_v^k}$ over intervals ${[s,t]}$ within an epoch~$k$, provided that ${I_s + \sum_{i=1}^L Q_{s-i} \leq \tau_v^k}$ and the counterfactual states at time~$s$ satisfy the conditions of \cref{lemma:feedback_structure}.
This is satisfied in the first time step at $t=1$ when the inventory vector is the zero vector and all counterfactual vectors are initialized as the zero vectors as well. In the following, the inventory states for all policies ${\tau \leq \tau_v^k}$ evolve according to the base-stock update \Cref{eq:inventory_balance}, with costs determined by \cref{eq:cost_function} (see lines~\ref{algo_line:observe_costs_L} and \ref{algo_line:update_states_L} of \cref{algo:lost_sales_lead_time}).

At the start of each epoch~${k>1}$ within an episode~$v$, inventory is reduced from a total of $\tau_v^{k-1}$ to the new target base-stock level $\tau_v^{k}$. The data collected in these intermediate steps ${\alpha_{v,k}, \dots , \bar{\alpha}_{v,k}-1}$ is not used for estimation. From time $\bar{\alpha}_{v,k}$ the condition ${I_{\bar{\alpha}_{v,k}} + \sum_{i=1}^L Q_{\bar{\alpha}_{v,k}-i} \leq \tau_v^k}$ is satisfied. In this time step the counterfactual inventory states are reset in line~\ref{algo_line:reset} of \cref{algo:lost_sales_lead_time} to state vectors that satisfy the conditions for left-sided feedback stated in \cref{lemma:feedback_structure}.
For a given factual inventory vector ${(I_t, Q_{t-L}, \dots , Q_{t})}$ with ${\tau_v^k \geq I_t + \sum_{i=1}^L Q_{t-i}}$ and any ${\tau \leq \tau_v^k}$, these conditions are, for instance, satisfied by the counterfactual inventory vector defined by
\begin{equation}  \label{eq:counterfactual_inventory_vector_reset}
    I^{\tau}_t = \min \{ \tau, I_t\}, \quad
    Q^{\tau}_{t-L} = \min \big\{ Q_{t-L}, \tau - I^{\tau}_t \big\}, \quad
    \textstyle Q^{\tau}_{t-i} = \min \Big\{ Q_{t-i}, \tau - \Big(I^{\tau}_t + \sum_{j=i+1}^{L} Q^{\tau}_{t-j} \Big) \Big\} , i=L-1,\dots , 1.
\end{equation}

For all subsequent time steps in the epoch ${\bar{\alpha}_{v,k}+1, \dots , \alpha_{v,k+1}-1}$ base-stock level $\tau_v^k$ is constant, and the conditions in \cref{lemma:feedback_structure} remain to be satisfied. 
The same condition ${\tau \geq I^{\tau}_{\bar{\alpha}_{v,k}} + \sum_{i=1}^L Q^{\tau}_{\bar{\alpha}_{v,k}-i}}$ ensures we obtain concentration of the empirical cost estimates $\hat{\mu}(\tau,s,t)$ for all policies ${\tau \leq \tau_v^k}$ and stationary intervals ${[s, t] \subseteq [\bar{\alpha}_{v,k}, \alpha_{v,k+1}-1]}$ (\cref{lemma:concentration}).
Lastly, note that the regret bound for \Cref{algo:lost_sales_lead_time} does not require a minimum separation between change points to ensure stationary segments do not end before the inventory process reaches $\bar{\alpha}_{v,k}$. Our analysis controls the cumulative cost incurred during such steps by upper bounding the total number of episodes, the amount by which the inventory level is reduced within an episode, the number of steps required to deplete one unit of inventory via \Cref{ass:main_assumption}.\ref{ass:bounded_depletion_steps}, and the cost per time step via Lipschitz continuity of the cost function.

\paragraph{Change Point Detection.}
We identify a change at time~$t$ in an epoch~$k$ of episode~$v$ by testing 
\begin{equation} \label{eq:change_condition_lsl}
    \left|\hat{\mu}(\tau, s_1, s_2) - \hat{\mu}(\tau, s, t)\right| > b_{s_1, s_2} + b_{s,t}
\end{equation}
for all base-stock levels ${\tau \leq \tau_v^k}$ and all combinations of time steps $s_1, s_2$ and $s$ with ${\bar{\alpha}_{v,k} \leq s_1 < s < t}$ and ${[s_1, s_2]}$ being a subinterval of any epoch at most $k$ in episode~$v$.~\footnote{As mentioned by \cite{auer2019adaptively}, the time complexity for checking change point condition \eqref{eq:change_condition_bl}, \eqref{eq:change_condition_ls} or \eqref{eq:change_condition_lsl} in step~$t$ is ${\mathcal{O} \left(t^3/\gamma \right)}$ and can be reduced to ${\mathcal{O}\left(\log (T)^2/\gamma \right)}$ per time step by only checking intervals of a certain minimal length. We further emphasize that these computations can be done in parallel.}

\begin{algorithm}[!t]
\caption{\NSLSICL: Non-Stationary Base-Stock Policy Optimization (Lost-Sales, Lead Time $L > 0$)}
 \label{algo:lost_sales_lead_time}
\begin{algorithmic}[1]
        \State \textbf{Inputs:} Time horizon $T$, discretization parameter $\gamma$, cost parameters $h$ and $b$, lead time $L$, 
        \Statex \hspace{\algorithmicindent} upper bound on optimal base-stock levels $U$, small probability $\delta$
        \State \textbf{Initialize:} Discretize set of base-stock levels $\mathcal{A}_{\gamma}=\{ 0, \gamma, 2\gamma, \dots , \lfloor U/\gamma \rfloor \gamma, U\}$,
        \Statex \hspace{\algorithmicindent} Set counterfactual states $(I^{\tau}_1, Q^{\tau}_{1-L}, \dots , Q^{\tau}_{0}) \gets (0, \dots , 0)$ $\forall \tau \in \mathcal{A}_{\gamma},\, v \gets 0,\, t \gets 0$
        
        \Begin \textbf{Start a new episode:}
        \State Episode count $v \gets v+1$, episode start time $t_v \gets t$, epoch count $k \gets 0$, active set $\Avk{v}{1} \gets \A_{\gamma}$
        \Begin \textbf{Start a new epoch:}
        \State Epoch count $k \gets k + 1$ and $\tau_v^k \gets \sup \Avk{v}{k}$
        
        \While{$t < T$}
            \State $t \gets t + 1$
            \State Play base-stock policy $\tau_v^k$ by ordering $Q_t = \left[\tau_v^k - I_t - \sum_{i=1}^L Q_{t-i}\right]^+$
            \State Observe sales $Y^{\tau}_t$ and costs $C_t(\tau)$ for all $\tau \leq \tau_v^k$ \Comment{Updating Cost Estimates} \phantomsection \label{algo_line:observe_costs_L}
            \State Update $(I^{\tau}_t, Q^{\tau}_{t-L}, \dots , Q^{\tau}_{t-1}) \ \forall \tau \leq \tau_v^k$ by \cref{eq:inventory_balance} \phantomsection \label{algo_line:update_states_L}

            \If{$t = \bar{\alpha}_{v,k}$}
                Reset $(I^{\tau}_t, Q^{\tau}_{t-L}, \dots , Q^{\tau}_{t-1}) \ \forall \tau \in \mathcal{A}_{\gamma}$ to \cref{eq:counterfactual_inventory_vector_reset} \phantomsection \label{algo_line:reset} 
            \EndIf
                   
            \If{$\exists \tau \leq \tau_v^k$ and $s,s_1,s_2$ with $\bar{\alpha}_{v,k} \leq s_1 < s < t$ and $[s_1,s_2] \subseteq [\bar{\alpha}_{v,k'}, \alpha_{v,k'+1} -1]$ \\
                \qquad \qquad \qquad for a $k'\leq k$ s.t. \cref{eq:change_condition_lsl} holds} \textbf{Start a new episode} \Comment{Change Point Detection}
            \EndIf      
            
            \If{$\exists \tau \in \Avk{v}{k} \text{ and } s \in [\bar{\alpha}_{v,k},t] \text{ s.t. } \hat{\mu}(\tau, s, t) - \min_{\tau' \in \A_{\gamma}, \tau' \leq \tau_v^k} \hat{\mu}(\tau', s, t) > 4 b_{s, t}$ and \\
                \qquad \qquad \qquad $\hat{\mu}(\tau - \gamma, \bar{\alpha}_{v,k},t) - \min_{\tau' \in \A_{\gamma}, \tau' \leq \tau_v^k} \hat{\mu}(\tau', \bar{\alpha}_{v,k},t) > 2b_{\bar{\alpha}_{v,k},t} +\max \{h,b\}\gamma$}
                    \State Update $\Avk{v}{k}$ by \cref{eq:active_set_elimination_condition_lsl} \Comment{Active Set Elimination}
                    \If{$\tau_v^{k+1} \neq \tau_v^k$} \textbf{Start a new epoch} \EndIf
            \EndIf
        \EndWhile
        \End
    \End
\end{algorithmic}
\end{algorithm}
%

The following theorem establishes the regret bound for our algorithm in the non-stationary lost-sales setting with positive lead time, thereby extending existing theory to this previously uncharacterized regime.
\begin{restatable}{theorem}{LSpositiveLRegret} \label{thm:regret_LS_L_positive}
The expected dynamic regret of \NSLSICL (\Cref{algo:lost_sales_lead_time}) is upper bounded via
\begin{align*}
    \E \left[R(T)\right] < \textstyle \left(1 - \delta\right) &\Bigg(24 \Hls {\textstyle \sqrt{ST \log\left(\frac{2T^2U}{\delta\gamma}\right)\left(\frac{U}{\gamma} + 2\right)}} + \frac{16\Hls U}{\gamma}\min\bigg\{\sqrt{ST {\textstyle\log\left(\frac{2T^2U}{\delta\gamma}\right)}}, \frac{8\Hls S {\textstyle\log\left(\frac{2T^2U}{\delta\gamma}\right)}}{\max \{ h, b\} \gamma}\bigg\} \\
    & \textstyle \qquad + \left(3T\gamma + LSU\lceil U/\gamma \rceil + SU^2 \nu \right) \max\{h,b\}\Bigg) + \delta T U \max\{h,b\},
\end{align*}
where ${\Hls = 72(L+3)U \max\{h,b\}}$ and ${\delta \in (0,1)}$. Thus, setting ${\delta = \Theta\left(T^{-2}\right)}$ and ${\gamma = \Theta(U (L+1)^{2/3} T^{-1/4})}$ yields $\E \left[R(T)\right] = \tilde{\mathcal{O}}\left(U(L+1)^{2/3} S^{1/2}T^{3/4}\right)$ whereas choosing ${\gamma = \Theta(U (L+1)^{2/3} T^{-1/3})}$ gives $\E \left[R(T)\right] = \tilde{\mathcal{O}}\left(U(L+1)^{2/3} ST^{2/3}\right)$. Therefore, it holds that $\E \left[R(T)\right] = \tilde{\mathcal{O}}\left(U(L+1)^{2/3} \min\{S^{1/2}T^{3/4}, ST^{2/3}\}\right)$.
Moreover, if $S$ is known, we obtain an upper bound of ${\tilde{\mathcal{O}} \left(U (L+1)^{2/3} S^{1/3}T^{2/3} \right)}$ by setting ${\gamma = \Theta\left(U(L+1)^{2/3}S^{1/3}T^{-1/3}\right)}$.
\end{restatable}
The guarantee in \Cref{thm:regret_LS_L_positive} is for expected regret since we need to bound the expected number of steps needed to reduce inventory via \Cref{ass:main_assumption}.\ref{ass:bounded_depletion_steps}.
The proof of \Cref{thm:regret_LS_L_positive} is given in \Cref{app:regret_LS_Lpositive}.

It remains an open question whether the regret rates stated in \cref{thm:regret_LS_L_positive} are optimal with respect to $S$ and $T$. Nonetheless, applying the ${\Omega(V^{1/3}T^{2/3})}$ lower bound for drifting multi-armed bandits from \citet{besbes2014stochastic} with a variation budget ${V = \Theta(S)}$ suggests near-optimality for the case when $S$ is known. Moreover, recent minimax lower bounds for drifting linear bandits \citep[Theorem~1]{cheung2022hedging} and MDPs \citep[Theorems~1 \& 2]{zhou2020nonstationary}, for non-stationary Lipschitz bandits \citep[Theorem~1]{nguyen2025non}, and the lower bound for undiscounted non-stationary MDPs by \citet[Proposition 1]{mao2025model} exhibit this same scaling.
Furthermore, if the $T^{2/3}$-type bound cannot be improved, then the ${\tilde{\mathcal{O}}\left( \min\{S^{1/2}T^{3/4}, ST^{2/3}\}\right)}$ regret when $S$ is unknown is likewise near-optimal with respect to $T$.
Note further that, although not central to our results, when tuning the discretization parameter optimally, the regret scales even sublinearly in the lead time, further improving over the linear rate achieved by \citet{agrawal2022learning}. 

Comparing these results with \cref{thm:regret_LS_L_zero} and the ${\Omega(\sqrt{T})}$ lower bound for stationary demand settings \cite[Proposition 1]{zhang2020closing} together with the $\tilde{\mathcal{O}}((L+1)\sqrt{T})$ upper bound obtained by stationary algorithms \citep{agrawal2022learning,zhang2025information}, we observe that demand censoring leads to a learning loss under non-stationarity only when coupled with lead time. Moreover, only in this setting, knowledge of the degree of non-stationarity confers an informational advantage for our algorithm.


\section{Experimental Study}  \label{sec:experiments}
We evaluate the proposed algorithms through a series of computational experiments.
The implementation is available at \href{https://github.com/NeleAmiri/nonstationary-inventory.git}{https://github.com/NeleAmiri/nonstationary-inventory.git}.

\subsection{Benchmark Algorithms and Performance Metrics}

We evaluate our proposed algorithms, \NSBLIC, \NSLSIC, and \NSLSICL (without tuning the discretization parameter based on $S$), relative to the following algorithms from the literature.
The first two benchmark algorithms are designed for non-stationary inventory control in lost-sales settings with discrete demand. Accordingly, demand realizations are projected to the nearest integer.
\begin{itemize}
    \item \LAIS: Learning Algorithm For Inventory Control With Shifting Demand \citep{chen2021data}, which is tailored to problems with zero lead time and switching demand with $S = \mathcal{O}(\log(T))$.
    \item \SWUCRL: Sliding Window Upper Confidence Bound For Reinforcement Learning With Confidence-Widening \citep{cheung2023nonstationary}.  This algorithm is designed for gradually changing demand constrained by a variation budget~$V$.  While this algorithm can in principle accommodate lead times, the state space grows exponentially in $L$, making it computationally infeasible even for moderate values. We therefore report its results only for the case $L = 0$.
\end{itemize}
Furthermore, we consider stationary algorithms equipped with a predetermined restarting schedule. After every $T/S$ time steps, the learning procedure is restarted and all past data are discarded. This grants the algorithms access to the value of $S$, but not to the locations of the change points.
\begin{itemize}
    \item \SMDPC: Learning Algorithm For MDP With Convex Cost Function (\MDPC) \citep{agrawal2022learning} with restarting schedule based on clairvoyant knowledge of~$S$.
    \item \SIOPEA: Information-Ordered Epoch Based Policy Elimination Algorithm (\IOPEA) \citep{zhang2025information} with restarting schedule based on clairvoyant knowledge of~$S$.
\end{itemize}
Lastly, we consider oracle variants of the previous two algorithms that obtain information not only of the frequency of changes, but also of the exact change points. These methods serve as reference baselines that isolate the impact of imperfect change point information.
\begin{itemize}
    \item \OMDPC: \MDPC restarting precisely at change points.
    \item \OIOPEA: \IOPEA restarting precisely at change points.
\end{itemize}
We do not benchmark against Meta-HQL or Mimic-QL \citep{gong2024bandits}, as these algorithms are designed for cyclic demand. Their mechanisms are not directly comparable to those tailored to piecewise stationary demand, and any observed performance differences would primarily reflect a mismatch in modeling assumptions.

We compare the algorithms based on dynamic regret~$R(T)$, as well as relative regret at the final time step
\[
    R^{\text{rel}}(T) \coloneq \frac{R(T)}{\sum_{t=1}^T \mu_t(\tau_t^*)} \cdot 100 \%.
\]
All performance measures are computed with the true expected costs rather than the pseudo cost. For notational convenience, however, we express them here via the pseudo cost~$\mu$.

\subsection{Experimental Setup}
We consider the following families of demand distributions.
Within each family, the parameters defining the demand distribution after a change point are independently sampled from the ranges specified below.
\begin{itemize}
    \item \textit{Normal:} Demand is $[D]^+$ where $D \sim \mathcal{N}(\mu, \sigma^2)$. The mean $\mu$ is sampled uniformly from $[1,100]$, and the standard deviation is fixed at $\sigma=20$.
    \item \textit{Uniform:} Demand is drawn from a continuous uniform distribution $\text{Uniform}(a, a+\lambda)$ with $a$ and $\lambda$ sampled independently and uniformly from $[1, 100]$ and $[0, 50]$, respectively.~\footnote{This distribution does not satisfy the condition $F_t(0)>0$ used in our analysis. Nonetheless, we include it in the numerical experiments to assess empirical robustness to violations of the assumption.}
\end{itemize}
In addition, we conducted simulations using Poisson and exponential distributions. The results for the latter are qualitatively similar to those obtained under the normal and uniform distributions and therefore omitted.

Following prior work, change points are drawn uniformly and independently from the time horizon $[T]$~\citep{chen2021data}.
Across all problem instances, the time horizon is fixed at $T=10^4$, with cost parameters set to $h=1$ and $b=49$ and lead times $L \in \{0,2,5\}$. 
For our algorithms as well as the \MDPC and \IOPEA variants, we set the upper bound on the policy space to ${U = 1.2 \cdot \max_{t \in [T]} \tau_t^*}$, and use $\lceil U \rceil$ discrete, integer-valued base-stock levels or inventory levels and order quantities in \LAIS and \SWUCRL, respectively. While this choice of $U$ uses oracle knowledge of the demand distributions, in practice $U$ can instead be chosen sufficiently large using coarse prior information about the scale of demand.
We run $500$ replications for each experiment, and use Monte-Carlo simulation over a horizon of $5 \cdot 10^3$ time steps to estimate the optimal base-stock policies for each stationary time interval.

\subsection{Results} \label{sec:results_experiments}
We investigate the empirical performance of our algorithms and the impact of (knowledge of) the amount of non-stationarity $S$, lead time $L$, demand censoring, as well as the use of counterfactual feedback.

\paragraph{Comparison with Non-Stationary Benchmark Algorithms.}
\cref{table:simulation_results_relative_regret_normal} reports the relative regret of all algorithms of interest, here for demand sampled from the truncated normal distribution. Results for the uniform distribution are similar and hence omitted. 
For the lost-sales setting with zero lead time, we observe that our algorithm \NSLSIC consistently outperforms \LAIS and \SWUCRL by a wide margin, even when ${S = \mathcal{O}(\log(T))}$, and in particular under stationary conditions. 
The higher regret of \SWUCRL is expected, as it is an RL algorithm designed for gradually drifting environments. In our regime with abrupt demand switches, the overhead of learning transition dynamics and maintaining conservative confidence bounds leads to slower adaptation compared to bandit-style approaches.

In general, our algorithms consistently outperform the \MDPC benchmarks, including both the Oracle and Schedule variants, particularly in the backlogging setting. While the \IOPEA benchmarks achieve lower regret overall due to access to both side and oracle information, our algorithms remain comparatively close to them especially for large values of $S$.

\begin{table*}[!t]
\scriptsize
\centering

\setlength{\tabcolsep}{3pt}
\renewcommand{\arraystretch}{0.8}

\begin{subtable}{0.55\textwidth}
\centering
\begin{tabular}{l r r r r r r}
\hline\up 
\textbf{$S$} & $1$ & $\mathcal{O}(1)$ & $\log(T)$ & $T^{1/3}$ & $T^{1/2}$ & $T^{2/3}$ \\ \hline\up 
\OIOPEA & 3.78 & 6.40 & 9.82 & 17.65 & 56.41 & 140.24 \\
\OMDPC & 49.51 & 103.76 & 164.29 & 281.43 & 792.07 & 1719.46 \\
\specialrule{0.1pt}{\aboverulesep}{\belowrulesep}
\SIOPEA & 3.79 & 136.68 & 146.72 & 154.95 & 165.42 & 167.94 \\
\SMDPC & 49.29 & 155.78 & 209.02 & 344.03 & 883.39 & 2010.87 \\
\specialrule{0.1pt}{\aboverulesep}{\belowrulesep}
\NSBLIC / \NSLSIC & 6.09 & 98.02 & 110.14 & 117.52 & 127.04 & 139.63 \\
\LAIS & - & - & - & - & - & - \\
\SWUCRL & - & - & - & - & - & -\down \\ \hline
\end{tabular}
\caption{Backlog, $L=0$.}
\end{subtable}
\hfill
\begin{subtable}{0.42\textwidth}
\centering
\begin{tabular}{r r r r r r}
\hline\up 
$1$ & $\mathcal{O}(1)$ & $\log(T)$ & $T^{1/3}$ & $T^{1/2}$ & $T^{2/3}$ \\ \hline\up 
3.19 & 6.50 & 10.17 & 18.29 & 58.01 & 140.72 \\
48.93 & 105.48 & 162.65 & 283.06 & 796.50 & 1719.50 \\
\specialrule{0.1pt}{\aboverulesep}{\belowrulesep}
3.18 & 131.30 & 143.62 & 156.57 & 164.71 & 167.35 \\
49.94 & 159.61 & 206.59 & 350.52 & 899.78 & 2016.29 \\
\specialrule{0.1pt}{\aboverulesep}{\belowrulesep}
6.74 & 97.65 & 113.17 & 123.21 & 133.00 & 142.35 \\
119.04 & 137.79 & 153.19 & 196.54 & 376.18 & 556.63 \\
1364.36 & 1477.18 & 1197.38 & 1358.95 & 1270.08 & 1263.52\down \\ \hline 
\end{tabular}
\caption{Lost-sales, $L=0$.}
\end{subtable}

\vspace{1em}

\begin{subtable}{0.55\textwidth}
\centering
\begin{tabular}{l r r r r r r}
\hline\up 
\OIOPEA & 4.96 & 15.41 & 20.12 & 30.45 & 86.39 & 208.66 \\
\OMDPC & 51.92 & 173.70 & 282.11 & 504.29 & 1421.99 & 2985.15 \\
\specialrule{0.1pt}{\aboverulesep}{\belowrulesep}
\SIOPEA & 4.79 & 223.29 & 230.77 & 248.90 & 245.09 & 248.47 \\
\SMDPC & 51.57 & 320.97 & 408.49 & 647.67 & 1586.79 & 3482.29 \\
\specialrule{0.1pt}{\aboverulesep}{\belowrulesep}
\NSBLIC / \NSLSICL & 5.96 & 161.03 & 182.17 & 199.51 & 217.99 & 223.52\down \\ \hline 
\end{tabular}
\caption{Backlog, $L=2$.}
\end{subtable}
\hfill
\begin{subtable}{0.42\textwidth}
\centering
\begin{tabular}{r r r r r r}
\hline\up 
4.46 & 11.97 & 19.39 & 33.75 & 98.23 & 225.41 \\
29.63 & 94.12 & 139.77 & 224.55 & 538.39 & 1059.99 \\
\specialrule{0.1pt}{\aboverulesep}{\belowrulesep}
4.34 & 190.40 & 218.87 & 246.62 & 265.69 & 267.88 \\
31.05 & 146.06 & 188.26 & 280.91 & 598.04 & 1216.25 \\
\specialrule{0.1pt}{\aboverulesep}{\belowrulesep}
8.46 & 70.76 & 96.60 & 172.89 & 195.23 & 208.54\down \\ \hline 
\end{tabular}
\caption{Lost-sales, $L=2$.}
\end{subtable}

\vspace{1em}

\begin{subtable}{0.55\textwidth}
\centering
\begin{tabular}{l r r r r r r}
\hline\up 
\OIOPEA & 5.51 & 40.43 & 47.28 & 60.67 & 135.26 & 241.71 \\
\OMDPC & 78.30 & 279.60 & 451.45 & 801.56 & 2245.44 & 3680.65 \\
\specialrule{0.1pt}{\aboverulesep}{\belowrulesep}
\SIOPEA & 5.59 & 416.81 & 379.40 & 354.06 & 354.04 & 285.62 \\
\SMDPC & 79.07 & 601.30 & 713.32 & 1050.85 & 2498.82 & 4256.72 \\
\specialrule{0.1pt}{\aboverulesep}{\belowrulesep}
\NSBLIC / \NSLSICL & 6.23 & 177.35 & 258.96 & 246.83 & 270.20 & 266.79\down \\ \hline 
\end{tabular}
\caption{Backlog, $L=5$.}
\end{subtable}
\hfill
\begin{subtable}{0.42\textwidth}
\centering
\begin{tabular}{r r r r r r}
\hline\up 
5.42 & 39.09 & 59.30 & 92.11 & 200.56 & 363.89 \\
27.60 & 95.75 & 141.71 & 217.04 & 486.89 & 846.52 \\
\specialrule{0.1pt}{\aboverulesep}{\belowrulesep}
5.52 & 253.86 & 343.75 & 382.51 & 413.36 & 414.36 \\
28.07 & 173.14 & 201.06 & 285.52 & 536.85 & 960.93 \\
\specialrule{0.1pt}{\aboverulesep}{\belowrulesep}
9.70 & 78.94 & 117.53 & 188.55 & 251.48 & 264.71\down \\ \hline 
\end{tabular}
\caption{Lost-sales, $L=5$.}
\end{subtable}

\caption{Relative regret $R^{\text{rel}}(T)$ (in \%) in the backlogging (left) and the lost-sales (right) setting, for different values of $L$ and $S$. Demand is drawn from the truncated normal distribution.}
\label{table:simulation_results_relative_regret_normal}
\end{table*}

\paragraph{Impact of $S$.}
Next, we investigate how the empirical performance of our algorithms is impacted by the frequency of change. 
\cref{fig:simulations_regret_vs_S} compares the scaling of the dynamic regret with respect to $S$ across the different lead times and under the backlogging and lost-sales demand model. 
Across all settings, the regret exhibits sub-power-law growth in $S$, indicating diminishing marginal sensitivity to additional changes. When $S$ is large, the environment effectively approaches a continuously drifting regime, in which the marginal cost of an additional change decreases. This behavior suggests that all three algorithms are robust not only to abrupt demand shifts but also adapt naturally to more gradual, continuous forms of non-stationarity, consistent with variation-budget adaptivity.

Finally, while our theoretical guarantees predict less favorable dependence on $S$ and larger regret for fixed $T$ under lost-sales with lead time, this effect is not observed empirically. In our experiments, performance under backlogging and lost-sales settings is nearly indistinguishable.  This lies in stark contrast to \MDPC variants that perform significantly worse under backlogging than lost-sales if lead time is positive (\cref{table:simulation_results_relative_regret_normal}).

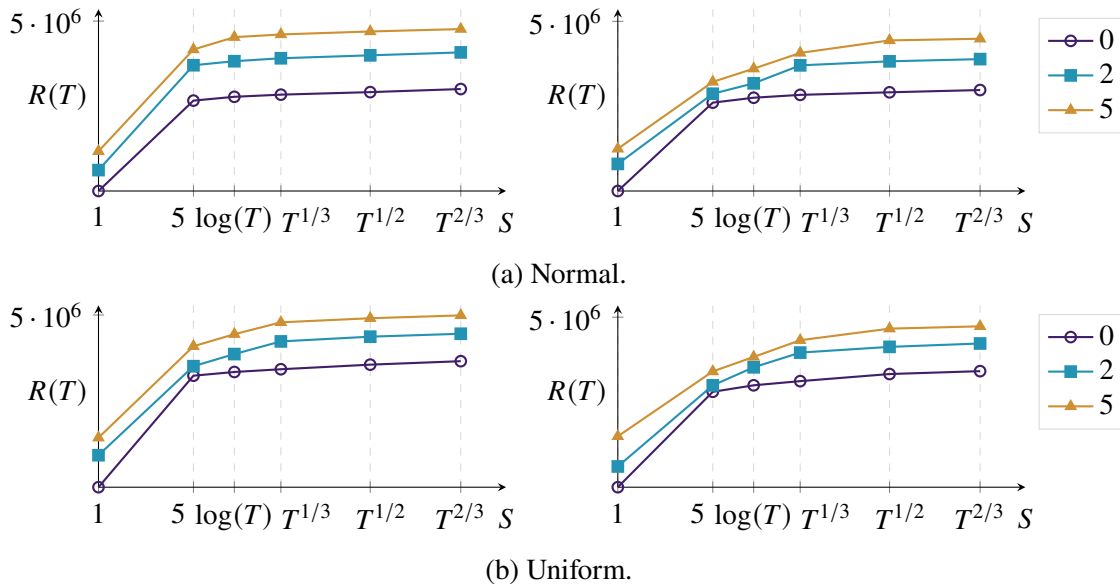
\begin{figure}[!thb]
    \centering
    \begin{subfigure}{0.9\textwidth}
        \centering

\begin{tikzpicture}

\begin{axis}[
    name=plot_normal,
    width=6.5cm,
    height=3.7cm,
    xmajorgrids=true, 
    grid style={dashed, gray!30},
    xmode=log,
    ymode=log,
    ymax=7.2e6,
    ytick={5e6},
    yticklabels={$5 \cdot 10^6$},
    axis x line=bottom,
    axis y line=middle,
    xmax=1000,
    enlargelimits=false,
    xtick={1,5,10,22,100,464},
    xticklabels={
                {$1$},
                {\hspace{-4mm}$5$},
                {\hspace{-0mm}$\log(T)$},
                {\hspace{7mm}$T^{1/3}$},
                {\hspace{3mm}$T^{1/2}$},
                {$T^{2/3}$}
                },
    xlabel={$S$},
    ylabel={$R(T)$},
    xlabel style={
    at={(axis description cs:1,-0.06)},
    anchor=north
    },
    ylabel style={
    at={(axis description cs:-0.1,0.65)},
    anchor=north
    },
    clip=false
]

\addplot[thick, mark=o, color=nupurple] 
coordinates {
    (1, 26634.19)
    (5, 431401.69)
    (10, 485030.31)
    (22, 518219.44)
    (100, 560226.61)
    (464, 615913.70)
};

\addplot[thick, mark=square*, color=kulblue] 
coordinates {
    (1, 50738.84)
    (5, 1282553.17)
    (10, 1451829.26)
    (22, 1590315.40)
    (100, 1740099.34)
    (464, 1904491.68)
};

\addplot[thick, mark=triangle*, color=darkorange] 
coordinates {
    (1, 90585.03)
    (5, 2083121.36)
    (10, 3045527.14)
    (22, 3310691.24)
    (100, 3629397.87)
    (464, 3907818.183)
};
\end{axis}

\hspace{0.5cm}
\begin{axis}[
    name=plot_normal_LS,
    at={(plot_normal.right of south east)},
    xshift=0.75cm,
    width=6.5cm,
    height=3.7cm,
    xmajorgrids=true, 
    grid style={dashed, gray!30},
    xmode=log,
    ymode=log,
    ymax=7.2e6,
    ytick={5e6},
    yticklabels={$5 \cdot 10^6$},
    axis x line=bottom,
    axis y line=middle,
    xmax=1000,
    enlargelimits=false,
    xtick={1,5,10,22,100,464},
    xticklabels={
                {$1$},
                {\hspace{-4mm}$5$},
                {\hspace{-0mm}$\log(T)$},
                {\hspace{7mm}$T^{1/3}$},
                {\hspace{3mm}$T^{1/2}$},
                {$T^{2/3}$}
                },    
    xlabel={$S$},
    ylabel={$R(T)$},
    xlabel style={
    at={(axis description cs:1,-0.06)},
    anchor=north
    },
    ylabel style={
    at={(axis description cs:-0.1,0.65)},
    anchor=north
    },
    clip=false
]

\addplot[thick, mark=o, color=nupurple] 
coordinates {
    (1, 29811.11)
    (5, 429249.98)
    (10, 497989.17)
    (22, 542792.40)
    (100, 586529.88)
    (464, 627968.01)
};

\addplot[thick, mark=square*, color=kulblue] 
coordinates {
    (1, 67533.12)
    (5, 561262.19)
    (10, 769127.60)
    (22, 1323109.54)
    (100, 1494553.89)
    (464, 1597085.75)
};

\addplot[thick, mark=triangle*, color=darkorange] 
coordinates {
    (1, 106564.18)
    (5, 803340.69)
    (10, 1197800.43)
    (22, 1925142.86)
    (100, 2809526.07)
    (464, 2954686.35)
};
\end{axis}

\begin{axis}[
  hide axis,
  name=dummy_axis,
  xmin=0, xmax=1,
  ymin=0, ymax=1,
  legend columns=1,
  legend cell align=left,
  at={(plot_normal_LS.north east)},
  anchor=north east,                 
  xshift=1.5cm,                     
  yshift=-0.0cm,  
  legend style={
    draw=gray!30,
  }
]
\addlegendimage{thick, mark=o, color=nupurple}
\addlegendentry{0}

\addlegendimage{thick, mark=square*, color=kulblue}
\addlegendentry{2}

\addlegendimage{thick, mark=triangle*, color=darkorange}
\addlegendentry{5}
\end{axis}

\end{tikzpicture}
        \caption{Normal.}
        \label{fig:simulations_regret_vs_S_normal}
    \end{subfigure}
    \begin{subfigure}{0.9\textwidth}
        \centering

\begin{tikzpicture}

\begin{axis}[
    name=plot_uniform,
    width=6.5cm,
    height=3.7cm,
    xmajorgrids=true, 
    grid style={dashed, gray!30},
    xmode=log,
    ymode=log,
    ymax=7.2e6,
    ytick={5e6},
    yticklabels={$5 \cdot 10^6$},
    axis x line=bottom,
    axis y line=middle,
    xmax=1000,
    enlargelimits=false,
    xtick={1,5,10,22,100,464},
    xticklabels={
                {$1$},
                {\hspace{-4mm}$5$},
                {\hspace{-0mm}$\log(T)$},
                {\hspace{7mm}$T^{1/3}$},
                {\hspace{3mm}$T^{1/2}$},
                {$T^{2/3}$}
                },    
    xlabel={$S$},
    ylabel={$R(T)$},
    xlabel style={
    at={(axis description cs:1,-0.06)},
    anchor=north
    },
    ylabel style={
    at={(axis description cs:-0.1,0.65)},
    anchor=north
    },
    clip=false
]

\addplot[thick, mark=o, color=nupurple] 
coordinates {
    (1, 7151.95)
    (5, 497065.80)
    (10, 572589.85)
    (22, 635360.55)
    (100, 756809.46)
    (464, 863202.01)
};

\addplot[thick, mark=square*, color=kulblue] 
coordinates {
    (1, 24146.70)
    (5, 713272.52)
    (10, 1131786.28)
    (22, 1831645.54)
    (100, 2192970.95)
    (464, 2450610.19)
};

\addplot[thick, mark=triangle*, color=darkorange] 
coordinates {
    (1, 46958.83)
    (5, 1523254.51)
    (10, 2410266.59)
    (22, 3784808.11)
    (100, 4423824.656)
    (464, 4929485.674)
};
\end{axis}

\hspace{0.5cm}
\begin{axis}[
    name=plot_uniform_LS,
    at={(plot_uniform.right of south east)},
    xshift=0.75cm,
    width=6.5cm,
    height=3.7cm,
    xmajorgrids=true, 
    grid style={dashed, gray!30},
    xmode=log,
    ymode=log,
    ymax=7.2e6,
    ytick={5e6},
    yticklabels={$5 \cdot 10^6$},
    axis x line=bottom,
    axis y line=middle,
    xmax=1000,
    enlargelimits=false,
    xtick={1,5,10,22,100,464},
    xticklabels={
                {$1$},
                {\hspace{-4mm}$5$},
                {\hspace{-0mm}$\log(T)$},
                {\hspace{7mm}$T^{1/3}$},
                {\hspace{3mm}$T^{1/2}$},
                {$T^{2/3}$}
                },
    xlabel={$S$},
    ylabel={$R(T)$}, 
    xlabel style={
    at={(axis description cs:1,-0.06)},
    anchor=north
    },
    ylabel style={
    at={(axis description cs:-0.1,0.65)},
    anchor=north
    },
    clip=false
]

\addplot[thick, mark=o, color=nupurple] 
coordinates {
    (1, 25837.07)
    (5, 495881.06)
    (10, 606571.17)
    (22, 690054.71)
    (100, 859860.73)
    (464, 940366.68)
};

\addplot[thick, mark=square*, color=kulblue] 
coordinates {
    (1, 49105.25)
    (5, 605369.95)
    (10, 1058977.13)
    (22, 1670659.95)
    (100, 1994451.45)
    (464, 2214828.11)
};

\addplot[thick, mark=triangle*, color=darkorange] 
coordinates {
    (1, 124887.29)
    (5, 932498.12)
    (10, 1459254.87)
    (22, 2450997.94)
    (100, 3507540.38)
    (464, 3769577.238)
};
\end{axis}

\begin{axis}[
  hide axis,
  name=dummy_axis,
  xmin=0, xmax=1,
  ymin=0, ymax=1,
  legend columns=1,
  legend cell align=left,
  at={(plot_uniform_LS.north east)},
  anchor=north east,                 
  xshift=1.5cm,                     
  yshift=-0.0cm,  
  legend style={
    draw=gray!30,
  }
]
\addlegendimage{thick, mark=o, color=nupurple}
\addlegendentry{0}

\addlegendimage{thick, mark=square*, color=kulblue}
\addlegendentry{2}

\addlegendimage{thick, mark=triangle*, color=darkorange}
\addlegendentry{5}
\end{axis}

\end{tikzpicture}
        \caption{Uniform.}
        \label{fig:simulations_regret_vs_S_uniform}
    \end{subfigure}
    \caption{Dynamic regret $R(T)$ vs. $S$ of \NSBLIC under backlogging (left), and \NSLSIC or \NSLSICL under lost-sales (right), shown on a log-log scale and for different values of lead time~$L$.  In \cref{fig:simulations_regret_vs_S_normal} demand is drawn from the truncated normal distribution and in \cref{fig:simulations_regret_vs_S_uniform} from the uniform distribution.}
    \label{fig:simulations_regret_vs_S}
\end{figure}

\paragraph{Impact of Lead Time.} We examine the performance of the proposed algorithms as $L$ grows large. 
\cref{fig:simulations_regret_vs_L} illustrates the dynamic regret against the lead time $L$ for the different values of $S$. 
As is to be expected from the theoretical performance guarantees, we observe that the regret attains the same order of magnitude under backlogging and lost-sales, and grows approximately linear in $L$. This is in contrast to discrete algorithms which, when applied to problems with positive lead times, attain regret (and computation) that scales exponentially with respect to $L$ \citep{cheung2023nonstationary,mao2025model}.

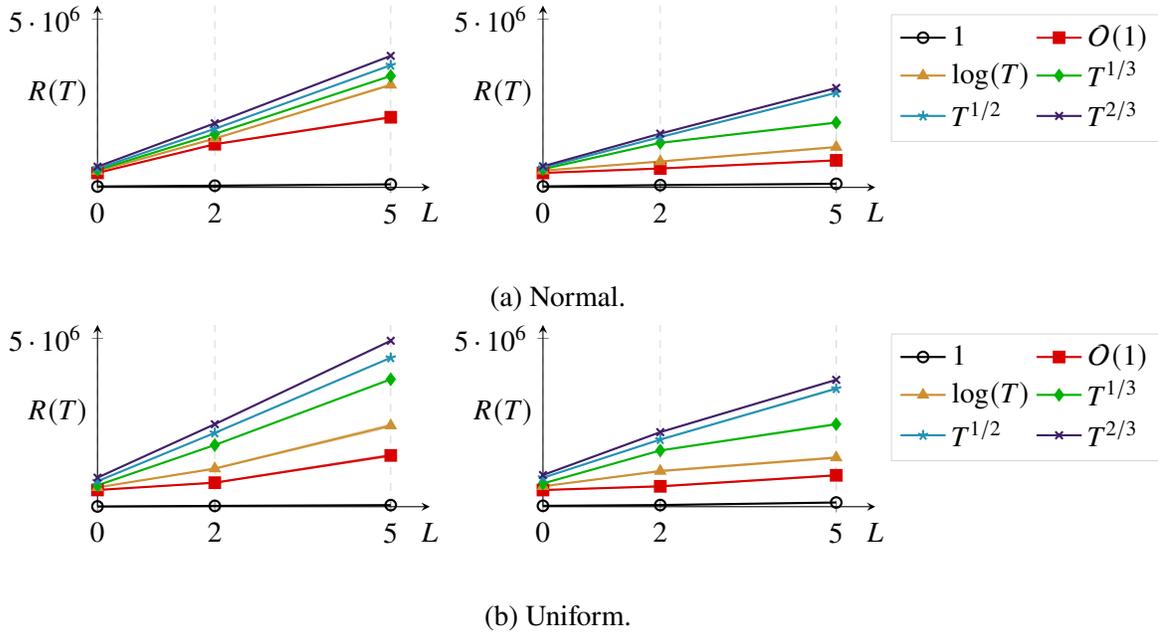
\begin{figure}[!thb]
    \centering
    \begin{minipage}{\textwidth}
    \centering
    \begin{subfigure}{0.9\textwidth}
        \begin{tikzpicture}

\begin{axis}[
  name=plot_normal,
  axis x line=bottom,
  axis y line=middle,
  width=5.5cm,
  height=3.7cm,
  xmajorgrids=true, 
  grid style={dashed, gray!30},
  xmin=0,
  xmax=0.85,
  ymin=0,
  ymax=5.4e6,
  ytick={5e6},
  yticklabels={$5 \cdot 10^6$},
  scaled y ticks=false,
  enlargelimits=false,
  xtick={0,0.3,0.75},
  xticklabels={$0$,$2$,$5$},
  xlabel={$L$},
  ylabel={$R(T)$},
  xlabel style={
  at={(axis description cs:1,-0.03)},
  anchor=north
  },
  ylabel style={
  at={(axis description cs:-0.12,0.65)},
  anchor=north
  },
  clip=false
]

\addplot[name path=lower1, draw=none] 
coordinates {(0,21260.05) (0.3,41277.19) (0.75,73451.70)};
\addplot[name path=upper1, draw=none] 
coordinates {(0,32008.32) (0.3,60200.49) (0.75,107718.35)};
\addplot[fill=black!50, fill opacity=0.3, forget plot] 
fill between[of=lower1 and upper1];

\addplot[thick, mark size=2pt, color=black, mark=o, solid, forget plot]
coordinates {(0,26634.19) (0.3,50738.84) (0.75,90585.03)};

\addplot[name path=lower2, draw=none] 
coordinates {(0,423923.7984) (0.3,1254995.732) (0.75,2037283.266)};
\addplot[name path=upper2, draw=none] 
coordinates {(0,438879.5841) (0.3,1310110.603) (0.75,2128959.451)};
\addplot[fill=red!85!black, fill opacity=0.3, forget plot] 
fill between[of=lower2 and upper2];

\addplot[thick, mark size=2pt, color=red!85!black, mark=square*, solid, forget plot]
coordinates {(0,431401.69) (0.3,1282553.17) (0.75,2083121.36)};

\addplot[name path=lower3, draw=none] 
coordinates {(0,479939.4709) (0.3,1431961.074) (0.75,3002699.652)};
\addplot[name path=upper3, draw=none] 
coordinates {(0,490121.1548) (0.3,1471697.445) (0.75,3088354.622)};
\addplot[fill=darkorange, fill opacity=0.3, forget plot] 
fill between[of=lower3 and upper3];

\addplot[thick, mark size=2pt, color=darkorange, mark=triangle*, solid, forget plot]
coordinates {(0,485030.31) (0.3,1451829.26) (0.75,3045527.14)};

\addplot[name path=lower4, draw=none] 
coordinates {(0,514877.8233) (0.3,1577234.334) (0.75,3281158.989)};
\addplot[name path=upper4, draw=none] 
coordinates {(0,521561.0592) (0.3,1603396.465) (0.75,3340223.482)};
\addplot[fill=green!50, fill opacity=0.3, forget plot] 
fill between[of=lower4 and upper4];

\addplot[thick, mark size=2pt, color=green!70!black, mark=diamond*, solid, forget plot]
coordinates {(0,518219.44) (0.3,1590315.40) (0.75,3310691.24)};

\addplot[name path=lower5, draw=none] 
coordinates {(0,558173.2906) (0.3,1733200.398) (0.75,3612659.028)};
\addplot[name path=upper5, draw=none] 
coordinates {(0,562279.9209) (0.3,1746998.278) (0.75,3646136.702)};
\addplot[fill=kulblue, fill opacity=0.3, forget plot] 
fill between[of=lower5 and upper5];

\addplot[thick, mark size=2pt, color=kulblue, mark=star, forget plot]
coordinates {(0,560226.61) (0.3,1740099.34) (0.75,3629397.87)};

\addplot[name path=lower6, draw=none] 
coordinates {(0,613907.8274) (0.3,1898276.763) (0.75,3896096.439)};
\addplot[name path=upper6, draw=none] 
coordinates {(0,617919.5809) (0.3,1910706.597) (0.75,3919539.927)};
\addplot[fill=nupurple, fill opacity=0.3, forget plot] 
fill between[of=lower6 and upper6];

\addplot[thick, mark size=2pt, color=nupurple, mark=x, forget plot]
coordinates {(0,615913.70) (0.3,1904491.68) (0.75,3907818.183)};

\end{axis}

\begin{axis}[
  name=plot_normal_LS,
  at={(plot_normal.right of south east)},
  xshift=1.25cm,
  axis x line=bottom,
  axis y line=middle,
  width=5.5cm,
  height=3.7cm,
  xmajorgrids=true, 
  grid style={dashed, gray!30},
  xmin=0,
  xmax=0.85,
  ymin=0,
  ymax=5.4e6,
  ytick={5e6},
  yticklabels={$5 \cdot 10^6$},
  scaled y ticks=false,
  enlargelimits=false,
  xtick={0,0.3,0.75},
  xticklabels={$0$,$2$,$5$},
  ylabel={$R(T)$},
  xlabel={$L$},
  xlabel style={
  at={(axis description cs:1,-0.03)},
  anchor=north
  },
  ylabel style={
  at={(axis description cs:-0.12,0.65)},
  anchor=north
  },
  clip=false
]

\addplot[name path=lower1, draw=none] 
coordinates {(0,26758.17) (0.3,58326.12) (0.75,89141.425)};
\addplot[name path=upper1, draw=none] 
coordinates {(0,32864.05) (0.3,76740.11) (0.75,123986.93)};
\addplot[fill=black!50, fill opacity=0.3, forget plot] 
fill between[of=lower1 and upper1];

\addplot[thick, mark size=2pt, color=black, mark=o, solid, forget plot]
coordinates {(0,29811.11) (0.3,67533.12) (0.75,106564.18)};

\addplot[name path=lower2, draw=none] 
coordinates {(0,418806.0217) (0.3,539338.0388) (0.75,779041.7952)};
\addplot[name path=upper2, draw=none] 
coordinates {(0,439693.9473) (0.3,583186.3314) (0.75,827639.5924)};
\addplot[fill=red!85!black, fill opacity=0.3, forget plot] 
fill between[of=lower2 and upper2];

\addplot[thick, mark size=2pt, color=red!85!black, mark=square*, solid, forget plot]
coordinates {(0,429249.98) (0.3,561262.19) (0.75,803340.69)};

\addplot[name path=lower3, draw=none] 
coordinates {(0,489165.9847) (0.3,743331.33) (0.75,1164397.316)};
\addplot[name path=upper3, draw=none] 
coordinates {(0,506812.3543) (0.3,794923.87) (0.75,1231203.55)};
\addplot[fill=darkorange, fill opacity=0.3, forget plot] 
fill between[of=lower3 and upper3];

\addplot[thick, mark size=2pt, color=darkorange, mark=triangle*, solid, forget plot]
coordinates {(0,497989.17) (0.3,769127.60) (0.75,1197800.43)};

\addplot[name path=lower4, draw=none] 
coordinates {(0,535575.8615) (0.3,1309247.078) (0.75,1879049.52)};
\addplot[name path=upper4, draw=none] 
coordinates {(0,550008.9373) (0.3,1336971.996) (0.75,1971236.207)};
\addplot[fill=green!50, fill opacity=0.3, forget plot] 
fill between[of=lower4 and upper4];

\addplot[thick, mark size=2pt, color=green!70!black, mark=diamond*, solid, forget plot]
coordinates {(0,542792.40) (0.3,1323109.54) (0.75,1925142.86)};

\addplot[name path=lower5, draw=none] 
coordinates {(0,581453.4071) (0.3,1483140.728) (0.75,2785486.262)};
\addplot[name path=upper5, draw=none] 
coordinates {(0,591606.3541) (0.3,1505967.06) (0.75,2833565.885)};
\addplot[fill=kulblue, fill opacity=0.3, forget plot] 
fill between[of=lower5 and upper5];

\addplot[thick, mark size=2pt, color=kulblue, mark=star, forget plot]
coordinates {(0,586529.88) (0.3,1494553.89) (0.75,2809526.07)};

\addplot[name path=lower6, draw=none] 
coordinates {(0,624228.3899) (0.3,1590466.368) (0.75,2942426.196)};
\addplot[name path=upper6, draw=none] 
coordinates {(0,631707.6344) (0.3,1603705.13) (0.75,2966946.504)};
\addplot[fill=nupurple, fill opacity=0.3, forget plot] 
fill between[of=lower6 and upper6];

\addplot[thick, mark size=2pt, color=nupurple, mark=x, forget plot]
coordinates {(0,627968.01) (0.3,1597085.75) (0.75,2954686.35)};

\end{axis}

\begin{axis}[
  hide axis,
  name=dummy_axis,
  xmin=0, xmax=1,
  ymin=0, ymax=1,
  legend columns=2,
  legend cell align=left,
  at={(plot_normal_LS.north east)}, 
  anchor=north east,                  
  xshift=5.2cm,                    
  yshift=-0.0cm,    
  legend style={
    draw=gray!30,
  }
]

\addlegendimage{thick, mark=o, color=black}
\addlegendentry{1}

\addlegendimage{thick, mark=square*, color=red!85!black}
\addlegendentry{$\mathcal{O}(1)$}

\addlegendimage{thick, mark=triangle*, color=darkorange}
\addlegendentry{$\log(T)$}

\addlegendimage{thick, mark=diamond*, color=green!70!black}
\addlegendentry{$T^{1/3}$}

\addlegendimage{thick, mark=star, color=kulblue}
\addlegendentry{$T^{1/2}$}

\addlegendimage{thick, mark=x, color=nupurple}
\addlegendentry{$T^{2/3}$}
\end{axis}

\end{tikzpicture}
        \caption{Normal.}
        \label{fig:simulations_regret_vs_L_normal}
    \end{subfigure}
    \begin{subfigure}{0.9\textwidth}
        \begin{tikzpicture}

\begin{axis}[
  name=plot_uniform,
  axis x line=bottom,
  axis y line=middle,
  width=5.5cm,
  height=3.7cm,
  xmajorgrids=true, 
  grid style={dashed, gray!30},
  xmin=0,
  xmax=0.85,
  ymin=0,
  ymax=5.4e6,
  ytick={5e6},
  yticklabels={$5 \cdot 10^6$},
  scaled y ticks=false,
  enlargelimits=false,
  xtick={0,0.3,0.75},
  xticklabels={$0$,$2$,$5$},
  xlabel={$L$},
  ylabel={$R(T)$},
  xlabel style={
  at={(axis description cs:1,-0.03)},
  anchor=north
  },
  ylabel style={
  at={(axis description cs:-0.12,0.65)},
  anchor=north
  },
  clip=false,
]

\addplot[name path=lower1, draw=none] 
coordinates {(0,7046.62) (0.3,22150.41) (0.75,45061.49)};
\addplot[name path=upper1, draw=none] 
coordinates {(0,7257.27) (0.3,26143.0) (0.75,48856.16)};
\addplot[fill=black!50, fill opacity=0.3, forget plot] 
fill between[of=lower1 and upper1];

\addplot[thick, mark size=2pt, color=black, mark=o, solid, forget plot]
coordinates {(0,7151.95) (0.3,24146.70) (0.75,46958.83)};

\addplot[name path=lower2, draw=none] 
coordinates {(0,487848.7276) (0.3,686163.4108) (0.75,1470193.678)};
\addplot[name path=upper2, draw=none] 
coordinates {(0,506282.8642) (0.3,740381.6384) (0.75,1576315.331)};
\addplot[fill=red!85!black, fill opacity=0.3, forget plot] 
fill between[of=lower2 and upper2];

\addplot[thick, mark size=2pt, color=red!85!black, mark=square*, solid, forget plot]
coordinates {(0,497065.80) (0.3,713272.52) (0.75,1523254.51)};

\addplot[name path=lower3, draw=none] 
coordinates {(0,565891.1529) (0.3,1098623.675) (0.75,2339642.172)};
\addplot[name path=upper3, draw=none] 
coordinates {(0,579288.5497) (0.3,1164948.878) (0.75,2480891.001)};
\addplot[fill=darkorange, fill opacity=0.3, forget plot] 
fill between[of=lower3 and upper3];

\addplot[thick, mark size=2pt, color=darkorange, mark=triangle*, solid, forget plot]
coordinates {(0,572589.85) (0.3,1131786.28) (0.75,2410266.59)};

\addplot[name path=lower4, draw=none] 
coordinates {(0,630320.9351) (0.3,1815422.604) (0.75,3748579.404)};
\addplot[name path=upper4, draw=none] 
coordinates {(0,640400.1635) (0.3,1847868.485) (0.75,3821036.814)};
\addplot[fill=green!50, fill opacity=0.3, forget plot] 
fill between[of=lower4 and upper4];

\addplot[thick, mark size=2pt, color=green!70!black, mark=diamond*, solid, forget plot]
coordinates {(0,635360.55) (0.3,1831645.54) (0.75,3784808.109)};

\addplot[name path=lower5, draw=none] 
coordinates {(0,752111.2868) (0.3,2178736.791) (0.75,4399370.369)};
\addplot[name path=upper5, draw=none] 
coordinates {(0,761507.6303) (0.3,2207205.117) (0.75,4448278.943)};
\addplot[fill=kulblue, fill opacity=0.3, forget plot] 
fill between[of=lower5 and upper5];

\addplot[thick, mark size=2pt, color=kulblue, mark=star, forget plot]
coordinates {(0,756809.46) (0.3,2192970.95) (0.75,4423824.656)};

\addplot[name path=lower6, draw=none] 
coordinates {(0,859468.9024) (0.3,2440977.728) (0.75,4912228.029)};
\addplot[name path=upper6, draw=none] 
coordinates {(0,866935.1118) (0.3,2460242.645) (0.75,4946743.319)};
\addplot[fill=nupurple, fill opacity=0.3, forget plot] 
fill between[of=lower6 and upper6];

\addplot[thick, mark size=2pt, color=nupurple, mark=x, forget plot]
coordinates {(0,863202.01) (0.3,2450610.19) (0.75,4929485.674)};

\end{axis}

\begin{axis}[
  name=plot_uniform_LS,
  at={(plot_uniform.right of south east)},
  xshift=1.25cm,
  axis x line=bottom,
  axis y line=middle,
  width=5.5cm,
  height=3.7cm,
  xmajorgrids=true, 
  grid style={dashed, gray!30},
  xmin=0,
  xmax=0.85,
  ymin=0,
  ymax=5.4e6,
  ytick={5e6},
  yticklabels={$5 \cdot 10^6$},
  scaled y ticks=false,
  enlargelimits=false,
  xtick={0,0.3,0.75},
  xticklabels={$0$,$2$,$5$},
  xlabel={$L$},
  ylabel={$R(T)$},
  xlabel style={
  at={(axis description cs:1,-0.03)},
  anchor=north
  },
  ylabel style={
  at={(axis description cs:-0.12,0.65)},
  anchor=north
  },
  clip=false
]

\addplot[name path=lower1, draw=none] 
coordinates {(0,23299.05) (0.3,42232.41) (0.75,121397.0)};
\addplot[name path=upper1, draw=none] 
coordinates {(0,28375.1) (0.3,55978.10) (0.75,128377.57)};
\addplot[fill=black!50, fill opacity=0.3, forget plot] 
fill between[of=lower1 and upper1];

\addplot[thick, mark size=2pt, color=black, mark=o, solid, forget plot]
coordinates {(0,25837.07) (0.3,49105.25) (0.75,124887.29)};

\addplot[name path=lower2, draw=none] 
coordinates {(0,482990.8399) (0.3,575552.3935) (0.75,900226.4042)};
\addplot[name path=upper2, draw=none] 
coordinates {(0,508771.2725) (0.3,635187.4994) (0.75,964769.8312)};
\addplot[fill=red!85!black, fill opacity=0.3, forget plot] 
fill between[of=lower2 and upper2];

\addplot[thick, mark size=2pt, color=red!85!black, mark=square*, solid, forget plot]
coordinates {(0,495881.06) (0.3,605369.95) (0.75,932498.12)};

\addplot[name path=lower3, draw=none] 
coordinates {(0,594155.9296) (0.3,1020983.284) (0.75,1416984.672)};
\addplot[name path=upper3, draw=none] 
coordinates {(0,618986.4092) (0.3,1096970.977) (0.75,1501525.06)};
\addplot[fill=darkorange, fill opacity=0.3, forget plot] 
fill between[of=lower3 and upper3];

\addplot[thick, mark size=2pt, color=darkorange, mark=triangle*, solid, forget plot]
coordinates {(0,606571.17) (0.3,1058977.13) (0.75,1459254.87)};

\addplot[name path=lower4, draw=none] 
coordinates {(0,680226.8752) (0.3,1646940.564) (0.75,2394536.443)};
\addplot[name path=upper4, draw=none] 
coordinates {(0,699882.551) (0.3,1694379.327) (0.75,2507459.439)};
\addplot[fill=green!50, fill opacity=0.3, forget plot] 
fill between[of=lower4 and upper4];

\addplot[thick, mark size=2pt, color=green!70!black, mark=diamond*, solid, forget plot]
coordinates {(0,690054.71) (0.3,1670659.95) (0.75,2450997.94)};

\addplot[name path=lower5, draw=none] 
coordinates {(0,851767.011) (0.3,1973153.309) (0.75,3472207.61)};
\addplot[name path=upper5, draw=none] 
coordinates {(0,867954.4415) (0.3,2015749.595) (0.75,3542873.14)};
\addplot[fill=kulblue, fill opacity=0.3, forget plot] 
fill between[of=lower5 and upper5];

\addplot[thick, mark size=2pt, color=kulblue, mark=star, forget plot]
coordinates {(0,859860.73) (0.3,1994451.45) (0.75,3507540.38)};

\addplot[name path=lower6, draw=none] 
coordinates {(0,934383.762) (0.3,2192644.187) (0.75,3756545.864)};
\addplot[name path=upper6, draw=none] 
coordinates {(0,946349.595) (0.3,2237012.022) (0.75,3782608.605)};
\addplot[fill=nupurple, fill opacity=0.3, forget plot] 
fill between[of=lower6 and upper6];

\addplot[thick, mark size=2pt, color=nupurple, mark=x, forget plot]
coordinates {(0,940366.68) (0.3,2214828.11) (0.75,3769577.23)};

\end{axis}

\begin{axis}[
  hide axis,
  name=dummy_axis,
  xmin=0, xmax=1,
  ymin=0, ymax=1,
  legend columns=2,
  legend cell align=left,
  at={(plot_uniform_LS.north east)},
  anchor=north east,                  
  xshift=5.2cm,                     
  yshift=-0.0cm,                     
  legend style={
    draw=gray!30,        
  }
]

\addlegendimage{thick, mark=o, color=black}
\addlegendentry{1}

\addlegendimage{thick, mark=square*, color=red!85!black}
\addlegendentry{$\mathcal{O}(1)$}

\addlegendimage{thick, mark=triangle*, color=darkorange}
\addlegendentry{$\log(T)$}

\addlegendimage{thick, mark=diamond*, color=green!70!black}
\addlegendentry{$T^{1/3}$}

\addlegendimage{thick, mark=star, color=kulblue}
\addlegendentry{$T^{1/2}$}

\addlegendimage{thick, mark=x, color=nupurple}
\addlegendentry{$T^{2/3}$}
\end{axis}

\end{tikzpicture}
        \caption{Uniform.}
        \label{fig:simulations_regret_vs_L_uniform}
    \end{subfigure}
    \end{minipage}
    
    \caption{Dynamic regret $R(T)$ vs. lead time~$L$ of \NSBLIC under backlogging (left), and \NSLSIC or \NSLSICL under lost-sales (right), shown for different values of~$S$. In \cref{fig:simulations_regret_vs_L_normal} demand is drawn from the truncated normal distribution and in \cref{fig:simulations_regret_vs_L_uniform} from the uniform distribution. }
    \label{fig:simulations_regret_vs_L}
\end{figure}

\paragraph{Impact of the Use of Counterfactual Feedback and Unknown Frequency of Change.}
The variants of \MDPC and \IOPEA are included as comparison methods designed to isolate the role of counterfactual information and prior knowledge of non-stationarity, in contrast to the active change detection used by our algorithms.  We note that \IOPEA is essentially identical to our algorithms when ignoring the change detection components, while \MDPC relies solely on the convexity and smoothness of the cost function.

The results in \cref{table:simulation_results_relative_regret_normal} show a substantial performance gap of our algorithms relative to both \SMDPC and even \OMDPC.  This highlights the substantial value of leveraging counterfactual feedback.  Moreover, our algorithms often nearly match the performance of the strongest benchmark, \OIOPEA.  This is particularly notable given that \OIOPEA restarts exactly at the true change points, incurring neither detection delay nor unnecessary resets, whereas our algorithms have no prior knowledge on the level of non-stationarity $S$ or the location of the change points.  We note that our methods may occasionally restart unnecessarily, leading to a modest performance gap when $S = 1$. This gap becomes more pronounced as $S$ increases and the informational advantage of known change points grows.  Nevertheless, across all regimes, our algorithms consistently outperform scheduled-restart methods and \OMDPC, supporting our theoretical finding that near-optimal performance of \NSBLIC and \NSLSIC does not rely on prior knowledge of $S$, and even \NSLSICL surpasses algorithms with such knowledge. Overall, our active change detection mechanisms prove more beneficial than passive adaptation via predetermined restarts, even when side information is incorporated.
\section{Conclusions}  \label{sec:conclusions}
We studied non-stationary inventory control and showed that the difficulty of learning is fundamentally governed by the information structure of the system.  Non-stationarity creates a tension between change detection and cost estimation, which is exacerbated by delayed and censored feedback under lead times and lost-sales. By exploiting counterfactual information inherent to inventory dynamics, we resolve this tension and design algorithms that respond to unknown and repeated demand shifts.  For backlogging and lost-sales models without lead times, our algorithms achieve near-optimal dynamic regret and incur no additional cost when demand is stationary. In contrast, for lost-sales with positive lead times, our results indicate a genuine loss of information, making non-stationary base-stock policy optimization fundamentally harder.

From a practical perspective, our results show that effective adaptation to structural demand changes is possible without prior knowledge of demand variability. This holds even in highly dynamic environments with lost-sales and lead times, where our algorithms reduce operational complexity and support reliable decision-making in modern supply chains.
An important direction for further research is to determine whether the performance guarantees for lost-sales systems with strictly positive lead times are tight, as well as extending our results to alternative notions of non-stationarity (e.g., gradually drifting demand) and to settings with stochastic lead times. Finally, an interesting open question is whether it is possible to obtain regret bounds for positive lead time settings under the transient expected cost regret that retain the same dependence on $T$ while achieving a sublinear dependence on $S$ without requiring knowledge of the latter.

\clearpage

\bibliographystyle{informs2014} 
\bibliography{bibliography} 

@misc{Adobe,
  author = {{Adobe}},
  year    = {2025},
  month   = {12},
  day     = {02},
  title   = {Adobe: Cyber Monday Hits Record \$14.25 Billion in Online Spending with Over \$1 Billion Driven by Buy Now Pay Later},
  url     = {https://news.adobe.com/news/2025/12/adobe-cyber-monday-hits-record}, 
  urldate = {2025-12-09},
  note = {Accessed July 30, 2026}
}

@article{agarwal2011stochastic,
  title={Stochastic convex optimization with bandit feedback},
  author={Agarwal, Alekh and Foster, Dean P and Hsu, Daniel J and Kakade, Sham M and Rakhlin, Alexander},
  journal={Advances in Neural Information Processing Systems},
  volume={24},
  year={2011}
}

@article{agarwal2019reinforcement,
  title={Reinforcement learning: Theory and algorithms},
  author={Agarwal, Alekh and Jiang, Nan and Kakade, Sham M and Sun, Wen},
  journal={CS Dept., UW Seattle, Seattle, WA, USA, Tech. Rep},
  volume={32},
  pages={96},
  year={2019}
}

@article{agrawal2022learning,
  title={Learning in Structured MDPs with Convex Cost Functions: Improved Regret Bounds for Inventory Management},
  author={Agrawal, Shipra and Jia, Randy},
  journal={Operations Research},
  volume={70},
  number={3},
  pages={1646--1664},
  year={2022},
  publisher={INFORMS}
}

@misc{AmazonDelivery,
  author = {{Amazon.com}, {Inc.}},
  year    = {2025},
  month   = {06},
  day     = {11},
  title   = {Amazon announces 3 AI-powered innovations to get packages to customers faster},
  url     = {https://www.aboutamazon.com/news/operations/amazon-ai-innovations-delivery-forecasting-robotics}, 
  urldate = {2025-12-09},
  note = {Accessed July 30, 2026}
}

@misc{AmazonInventory,
  author = {{Amazon.com}, {Inc.}},
  year    = {2021},
  month   = {10},
  day     = {01},
  title   = {The evolution of Amazon’s inventory planning system},
  journal = {Amazon Science},
  url     = {https://www.amazon.science/latest-news/the-evolution-of-amazons-inventory-planning-system}, 
  urldate = {2025-12-09},
  note = {Accessed July 30, 2026}
}

@article{an2025nonstationary,
  title={The nonstationary newsvendor with (and without) predictions},
  author={An, Lin and Li, Andrew A and Moseley, Benjamin and Ravi, R},
  journal={Manufacturing \& Service Operations Management},
  year={2025},
  publisher={INFORMS}
}

@article{auer2002nonstochastic,
  title={The nonstochastic multiarmed bandit problem},
  author={Auer, Peter and Cesa-Bianchi, Nicolo and Freund, Yoav and Schapire, Robert E},
  journal={SIAM journal on computing},
  volume={32},
  number={1},
  pages={48--77},
  year={2002},
  publisher={SIAM}
}

@article{auer2008near,
  title={Near-optimal regret bounds for reinforcement learning},
  author={Auer, Peter and Jaksch, Thomas and Ortner, Ronald},
  journal={Advances in neural information processing systems},
  volume={21},
  year={2008}
}

@inproceedings{auer2019adaptively,
  title={Adaptively tracking the best bandit arm with an unknown number of distribution changes},
  author={Auer, Peter and Gajane, Pratik and Ortner, Ronald},
  booktitle={Conference on Learning Theory},
  pages={138--158},
  year={2019},
  organization={PMLR}
}

@article{babai2025fifty,
  title={Fifty Years of Inventory Research from a Forecasting Perspective},
  author={Babai, M Zied and Syntetos, Aris A and Teunter, Ruud H},
  journal={European Journal of Operational Research},
  year={2025},
  publisher={Elsevier}
}

@book{bertsekas2019reinforcement,
  title={Reinforcement learning and optimal control},
  author={Bertsekas, Dimitri},
  volume={1},
  year={2019},
  publisher={Athena Scientific}
}

@article{besbes2014stochastic,
  title={Stochastic multi-armed-bandit problem with non-stationary rewards},
  author={Besbes, Omar and Gur, Yonatan and Zeevi, Assaf},
  journal={Advances in neural information processing systems},
  volume={27},
  year={2014}
}

@article{besbes2015non,
  title={Non-stationary stochastic optimization},
  author={Besbes, Omar and Gur, Yonatan and Zeevi, Assaf},
  journal={Operations research},
  volume={63},
  number={5},
  pages={1227--1244},
  year={2015},
  publisher={INFORMS}
}

@article{besson2022efficient,
  title={Efficient change-point detection for tackling piecewise-stationary bandits},
  author={Besson, Lilian and Kaufmann, Emilie and Maillard, Odalric-Ambrym and Seznec, Julien},
  journal={Journal of Machine Learning Research},
  volume={23},
  number={77},
  pages={1--40},
  year={2022}
}

@article{bu2023asymptotic,
  title={Asymptotic optimality of base-stock policies for perishable inventory systems},
  author={Bu, Jinzhi and Gong, Xiting and Chao, Xiuli},
  journal={Management Science},
  volume={69},
  number={2},
  pages={846--864},
  year={2023},
  publisher={INFORMS}
}

@inproceedings{cao2019nearly,
  title={Nearly optimal adaptive procedure with change detection for piecewise-stationary bandit},
  author={Cao, Yang and Wen, Zheng and Kveton, Branislav and Xie, Yao},
  booktitle={The 22nd International Conference on Artificial Intelligence and Statistics},
  pages={418--427},
  year={2019},
  organization={PMLR}
}

@article{cao2019quantile,
  title={Quantile forecasting and data-driven inventory management under nonstationary demand},
  author={Cao, Ying and Shen, Zuo-Jun Max},
  journal={Operations Research Letters},
  volume={47},
  number={6},
  pages={465--472},
  year={2019},
  publisher={Elsevier}
}

@article{chen2021data,
  title={Data-driven inventory control with shifting demand},
  author={Chen, Boxiao},
  journal={Production and Operations Management},
  volume={30},
  number={5},
  pages={1365--1385},
  year={2021},
  publisher={SAGE Publications Sage CA: Los Angeles, CA}
}

@article{chen2024learning,
  title={Learning to order for inventory systems with lost sales and uncertain supplies},
  author={Chen, Boxiao and Jiang, Jiashuo and Zhang, Jiawei and Zhou, Zhengyuan},
  journal={Management Science},
  volume={70},
  number={12},
  pages={8631--8646},
  year={2024},
  publisher={INFORMS}
}

@article{chen2025learning,
  title={Learning When to Restart: Nonstationary Newsvendor from Uncensored to Censored Demand},
  author={Chen, Xin and Lyu, Jiameng and Yuan, Shilin and Zhou, Yuan},
  journal={arXiv preprint arXiv:2509.18709},
  year={2025}
}

@article{chen2025managing,
  title={Managing Lost-Sale Inventory Systems under Unknown Demand and Return Distributions},
  author={Chen, Boxiao and Jasin, Stefanus and Luo, Qi and Zhang, Mengxiao},
  journal={Available at SSRN 5624311},
  year={2025}
}

@article{cheung2022hedging,
  title={Hedging the drift: Learning to optimize under nonstationarity},
  author={Cheung, Wang Chi and Simchi-Levi, David and Zhu, Ruihao},
  journal={Management Science},
  volume={68},
  number={3},
  pages={1696--1713},
  year={2022},
  publisher={INFORMS}
}

@article{cheung2023nonstationary,
  title={Nonstationary reinforcement learning: The blessing of (more) optimism},
  author={Cheung, Wang Chi and Simchi-Levi, David and Zhu, Ruihao},
  journal={Management Science},
  volume={69},
  number={10},
  pages={5722--5739},
  year={2023},
  publisher={INFORMS}
}

@article{dehaybe2024deep,
  title={Deep reinforcement learning for inventory optimization with non-stationary uncertain demand},
  author={Dehaybe, Henri and Catanzaro, Daniele and Chevalier, Philippe},
  journal={European Journal of Operational Research},
  volume={314},
  number={2},
  pages={433--445},
  year={2024},
  publisher={Elsevier}
}

@inproceedings{domingues2021kernel,
  title={A kernel-based approach to non-stationary reinforcement learning in metric spaces},
  author={Domingues, Omar Darwiche and M{\'e}nard, Pierre and Pirotta, Matteo and Kaufmann, Emilie and Valko, Michal},
  booktitle={International Conference on Artificial Intelligence and Statistics},
  pages={3538--3546},
  year={2021},
  organization={PMLR}
}

@misc{duchi2023lecture,
  title={Lecture notes on statistics and information theory},
  author={Duchi, John},
  url={https://www.web.stanford.edu/class/stats311/lecture-notes.pdf},
  year={2023}
}

@article{EUChamberCommerceChina,
 author  = {{European Union Chamber of Commerce in China}},
 year    = {2025},
 month   = {12},
 day     = {01},
 title   = {European Chamber survey finds one in three looking to divert sourcing away from China to mitigate impact of export controls},
 journal = {European Union Chamber of Commerce in China},
 url     = {https://www.europeanchamber.com.cn/en/press-releases/3757},
 urldate = {2025-12-09},
 note = {Accessed July 30, 2026}
}

@article{foussoul2023mnl,
  title={Mnl-bandit in non-stationary environments},
  author={Foussoul, Ayoub and Goyal, Vineet and Gupta, Varun},
  journal={arXiv preprint arXiv:2303.02504},
  year={2023}
}

@article{flaxman2004online,
  title={Online convex optimization in the bandit setting: gradient descent without a gradient},
  author={Flaxman, Abraham D and Kalai, Adam Tauman and McMahan, H Brendan},
  journal={arXiv preprint cs/0408007},
  year={2004}
}

@article{gajane2018sliding,
  title={A sliding-window algorithm for markov decision processes with arbitrarily changing rewards and transitions},
  author={Gajane, Pratik and Ortner, Ronald and Auer, Peter},
  journal={arXiv preprint arXiv:1805.10066},
  year={2018}
}

@inproceedings{garivier2011upper,
  title={On upper-confidence bound policies for switching bandit problems},
  author={Garivier, Aur{\'e}lien and Moulines, Eric},
  booktitle={International conference on algorithmic learning theory},
  pages={174--188},
  year={2011},
  organization={Springer}
}

@article{goltsos2022inventory,
  title={Inventory--forecasting: Mind the gap},
  author={Goltsos, Thanos E and Syntetos, Aris A and Glock, Christoph H and Ioannou, George},
  journal={European journal of operational research},
  volume={299},
  number={2},
  pages={397--419},
  year={2022},
  publisher={Elsevier}
}

@article{gong2020provably,
  title={Provably More Efficient Q-Learning in the One-Sided-Feedback/Full-Feedback Settings},
  author={Gong, Xiao-Yue and Simchi-Levi, David},
  journal={arXiv preprint arXiv:2007.00080},
  year={2020}
}

@article{gong2024bandits,
  title={Bandits atop reinforcement learning: Tackling online inventory models with cyclic demands},
  author={Gong, Xiao-Yue and Simchi-Levi, David},
  journal={Management Science},
  volume={70},
  number={9},
  pages={6139--6157},
  year={2024},
  publisher={INFORMS}
}

@article{graves1999single,
  title={A single-item inventory model for a nonstationary demand process},
  author={Graves, Stephen C},
  journal={Manufacturing \& Service Operations Management},
  volume={1},
  number={1},
  pages={50--61},
  year={1999},
  publisher={INFORMS}
}

@inproceedings{hartland2007change,
  title={Change point detection and meta-bandits for online learning in dynamic environments},
  author={Hartland, C{\'e}dric and Baskiotis, Nicolas and Gelly, Sylvain and Sebag, Michele and Teytaud, Olivier},
  booktitle={CAp 2007: 9{\`e} Conf{\'e}rence francophone sur l'apprentissage automatique},
  pages={237--250},
  year={2007}
}

@article{huang2025change,
  title={Change Detection-Based Procedures for Piecewise Stationary MABs: A Modular Approach},
  author={Huang, Yu-Han and Gerogiannis, Argyrios and Bose, Subhonmesh and Veeravalli, Venugopal V},
  journal={arXiv preprint arXiv:2501.01291},
  year={2025}
}

@article{huh2009adaptive,
  title={An adaptive algorithm for finding the optimal base-stock policy in lost sales inventory systems with censored demand},
  author={Huh, Woonghee Tim and Janakiraman, Ganesh and Muckstadt, John A and Rusmevichientong, Paat},
  journal={Mathematics of Operations Research},
  volume={34},
  number={2},
  pages={397--416},
  year={2009},
  publisher={INFORMS}
}

@article{huh2009asymptotic,
  title={Asymptotic optimality of order-up-to policies in lost sales inventory systems},
  author={Huh, Woonghee Tim and Janakiraman, Ganesh and Muckstadt, John A and Rusmevichientong, Paat},
  journal={Management Science},
  volume={55},
  number={3},
  pages={404--420},
  year={2009},
  publisher={INFORMS}
}

@article{huh2011adaptive,
  title={Adaptive data-driven inventory control with censored demand based on Kaplan-Meier estimator},
  author={Huh, Woonghee Tim and Levi, Retsef and Rusmevichientong, Paat and Orlin, James B},
  journal={Operations Research},
  volume={59},
  number={4},
  pages={929--941},
  year={2011},
  publisher={INFORMS}
}

@article{janakiraman2004lost,
  title={Lost-sales problems with stochastic lead times: Convexity results for base-stock policies},
  author={Janakiraman, Ganesh and Roundy, Robin O},
  journal={Operations Research},
  volume={52},
  number={5},
  pages={795--803},
  year={2004},
  publisher={INFORMS}
}

@article{karlin1958inventory,
  title={Inventory models of the Arrow-Harris-Marschak type with time lag},
  author={Karlin, Samuel},
  journal={Studies in the mathematical theory of inventory and production},
  year={1958},
  publisher={Stanford University Press}
}

@article{keskin2023nonstationary,
  title={The nonstationary newsvendor: Data-driven nonparametric learning},
  author={Keskin, N Bora and Min, Xu and Song, Jing-Sheng Jeannette},
  journal={Available at SSRN 3866171},
  year={2023}
}

@inproceedings{kocsis2006discounted,
  title={Discounted ucb},
  author={Kocsis, Levente and Szepesv{\'a}ri, Csaba},
  booktitle={2nd PASCAL Challenges Workshop},
  volume={2},
  pages={51--134},
  year={2006}
}

@book{lattimore2020bandit,
  title={Bandit algorithms},
  author={Lattimore, Tor and Szepesv{\'a}ri, Csaba},
  year={2020},
  publisher={Cambridge University Press}
}

@article{levine2017rotting,
  title={Rotting bandits},
  author={Levine, Nir and Crammer, Koby and Mannor, Shie},
  journal={Advances in neural information processing systems},
  volume={30},
  year={2017}
}

@article{lindstaahl2023change,
  title={Change point detection with adaptive measurement schedules for network performance verification},
  author={Lindst{\aa}hl, Simon and Proutiere, Alexandre and Johnsson, Andreas},
  journal={Proceedings of the ACM on Measurement and Analysis of Computing Systems},
  volume={7},
  number={3},
  pages={1--30},
  year={2023},
  publisher={ACM New York, NY, USA}
}

@inproceedings{liu2018change,
  title={A change-detection based framework for piecewise-stationary multi-armed bandit problem},
  author={Liu, Fang and Lee, Joohyun and Shroff, Ness},
  booktitle={Proceedings of the AAAI Conference on Artificial Intelligence},
  volume={32},
  number={1},
  year={2018}
}

@article{lugosi2024hardness,
  title={On the hardness of learning from censored and nonstationary demand},
  author={Lugosi, G{\'a}bor and Markakis, Mihalis G and Neu, Gergely},
  journal={INFORMS Journal on Optimization},
  volume={6},
  number={2},
  pages={63--83},
  year={2024},
  publisher={INFORMS}
}

@article{mao2025model,
  title={Model-free nonstationary reinforcement learning: Near-optimal regret and applications in multiagent reinforcement learning and inventory control},
  author={Mao, Weichao and Zhang, Kaiqing and Zhu, Ruihao and Simchi-Levi, David and Ba{\c{s}}ar, Tamer},
  journal={Management Science},
  volume={71},
  number={2},
  pages={1564--1580},
  year={2025},
  publisher={INFORMS}
}

@misc{Mastercard,
  author  = {Mastercard},
  year    = {2025},
  month   = {11},
  day     = {29},
  title   = {Mastercard SpendingPulse: {U.S.} Black Friday retail sales up +4.1\% year over year as holiday momentum builds},
  url     = {https://www.mastercard.com/us/en/news-and-trends/press/2025/november/mastercard-spendingpulse--us-black-friday-retail-sales-up--4-1--.html},
  urldate = {2025-12-09},
  note = {Accessed July 30, 2026}
}

@article{nguyen2025non,
  title={Non-Stationary Lipschitz Bandits},
  author={Nguyen, Nicolas and Gaucher, Solenne and Vernade, Claire},
  journal={arXiv preprint arXiv:2505.18871},
  year={2025}
}

@inproceedings{ortner2012regret,
  title={Regret bounds for restless markov bandits},
  author={Ortner, Ronald and Ryabko, Daniil and Auer, Peter and Munos, R{\'e}mi},
  booktitle={International conference on algorithmic learning theory},
  pages={214--228},
  year={2012},
  organization={Springer}
}

@inproceedings{ortner2020variational,
  title={Variational regret bounds for reinforcement learning},
  author={Ortner, Ronald and Gajane, Pratik and Auer, Peter},
  booktitle={Uncertainty in Artificial Intelligence},
  pages={81--90},
  year={2020},
  organization={PMLR}
}

@inproceedings{peng2024complexity,
  title={The complexity of non-stationary reinforcement learning},
  author={Peng, Binghui and Papadimitriou, Christos},
  booktitle={International Conference on Algorithmic Learning Theory},
  pages={972--996},
  year={2024},
  organization={PMLR}
}

@article{porteus1990stochastic,
  title={Stochastic inventory theory},
  author={Porteus, Evan L},
  journal={Handbooks in operations research and management science},
  volume={2},
  pages={605--652},
  year={1990},
  publisher={Elsevier}
}

@book{puterman2014markov,
  title={Markov decision processes: discrete stochastic dynamic programming},
  author={Puterman, Martin L},
  year={2014},
  publisher={John Wiley \& Sons}
}

@article{ren2024data,
  title={Data-driven inventory policy: Learning from sequentially observed non-stationary data},
  author={Ren, Ke and Bidkhori, Hoda and Shen, Zuo-Jun Max},
  journal={Omega},
  volume={123},
  pages={102942},
  year={2024},
  publisher={Elsevier}
}

@article{rivasplata2012subgaussian,
  title={Subgaussian random variables: An expository note},
  author={Rivasplata, Omar},
  journal={Internet publication, PDF},
  volume={5},
  year={2012}
}

@article{sethi1997optimality,
  title={Optimality of (s, S) policies in inventory models with Markovian demand},
  author={Sethi, Suresh P and Cheng, Feng},
  journal={Operations Research},
  volume={45},
  number={6},
  pages={931--939},
  year={1997},
  publisher={INFORMS}
}

@inproceedings{seznec2020single,
  title={A single algorithm for both restless and rested rotting bandits},
  author={Seznec, Julien and Menard, Pierre and Lazaric, Alessandro and Valko, Michal},
  booktitle={International Conference on Artificial Intelligence and Statistics},
  pages={3784--3794},
  year={2020},
  organization={PMLR}
}

@article{scarf1960optimality,
  title={The optimality of (S, s) policies in the dynamic inventory problem},
  author={Scarf, Herbert},
  year={1960},
  publisher={Stanford University Press Stanford, CA}
}

@inproceedings{slivkins2008adapting,
  title={Adapting to a Changing Environment: the Brownian Restless Bandits.},
  author={Slivkins, Aleksandrs and Upfal, Eli},
  booktitle={COLT},
  pages={343--354},
  year={2008}
}

@article{strijbosch2011interaction,
  title={On the interaction between forecasting and stock control: The case of non-stationary demand},
  author={Strijbosch, Leo WG and Syntetos, Aris A and Boylan, John E and Janssen, Elleke},
  journal={International Journal of Production Economics},
  volume={133},
  number={1},
  pages={470--480},
  year={2011},
  publisher={Elsevier}
}

@book{sutton1998reinforcement,
  title={Reinforcement learning: An introduction},
  author={Sutton, Richard S and Barto, Andrew G and others},
  volume={1},
  year={1998},
  publisher={MIT press Cambridge}
}

@article{temizoz2025zero,
  title={Zero-shot generalization in inventory management: Train, then estimate and decide},
  author={Temiz{\"o}z, Tarkan and Imdahl, Christina and Dijkman, Remco and Lamghari-Idrissi, Douniel and Van Jaarsveld, Willem},
  journal={European Journal of Operational Research},
  year={2025},
  publisher={Elsevier}
}

@article{theodorou2025forecast,
  title={Forecast accuracy and inventory performance: Insights on their relationship from the M5 competition data},
  author={Theodorou, Evangelos and Spiliotis, Evangelos and Assimakopoulos, Vassilios},
  journal={European Journal of Operational Research},
  volume={322},
  number={2},
  pages={414--426},
  year={2025},
  publisher={Elsevier}
}

@article{ban2020confidence,
  title={Confidence intervals for data-driven inventory policies with demand censoring},
  author={Ban, Gah-Yi},
  journal={Operations Research},
  volume={68},
  number={2},
  pages={309--326},
  year={2020},
  publisher={INFORMS}
}

@article{treharne2002adaptive,
  title={Adaptive inventory control for nonstationary demand and partial information},
  author={Treharne, James T and Sox, Charles R},
  journal={Management Science},
  volume={48},
  number={5},
  pages={607--624},
  year={2002},
  publisher={INFORMS}
}

@book{tsybakov2009nonparametric,
  title={Introduction to Nonparametric Estimation},
  author={Tsybakov, Alexandre B},
  year={2009},
  publisher={Springer}
}

@article{tunc2011cost,
  title={The cost of using stationary inventory policies when demand is non-stationary},
  author={Tunc, Huseyin and Kilic, Onur A and Tarim, S Armagan and Eksioglu, Burak},
  journal={Omega},
  volume={39},
  number={4},
  pages={410--415},
  year={2011},
  publisher={Elsevier}
}

@misc{Walmart,
  author = {{Walmart Inc.}},
  year    = {2025},
  month   = {7},
  day     = {17},
  title   = {Walmart’s U.S. Supply Chain Playbook Goes Global — and It’s Reinventing Retail at Scale},
  url     = {https://corporate.walmart.com/news/2025/07/17/walmarts-us-supply-chain-playbook-goes-global-and-its-reinventing-retail-at-scale}, 
  urldate = {2025-12-09},
  note = {Accessed July 30, 2026}
}

@article{wang2025adaptivity,
  title={On adaptivity in nonstationary stochastic optimization with bandit feedback},
  author={Wang, Yining},
  journal={Operations Research},
  volume={73},
  number={2},
  pages={819--828},
  year={2025},
  publisher={INFORMS}
}

@inproceedings{wei2021non,
  title={Non-stationary reinforcement learning without prior knowledge: An optimal black-box approach},
  author={Wei, Chen-Yu and Luo, Haipeng},
  booktitle={Conference on learning theory},
  pages={4300--4354},
  year={2021},
  organization={PMLR}
}

@article{whittle1988restless,
  title={Restless bandits: Activity allocation in a changing world},
  author={Whittle, Peter},
  journal={Journal of applied probability},
  volume={25},
  number={A},
  pages={287--298},
  year={1988},
  publisher={Cambridge University Press}
}

@article{xie2026deepstock,
  title={Deepstock: Reinforcement learning with policy regularizations for inventory management},
  author={Xie, Yaqi and Hao, Xinru and Liu, Jiaxi and Ma, Will and Xin, Linwei and Cao, Lei and Zhang, Yidong},
  journal={arXiv preprint arXiv:2603.19621},
  year={2026}
}

@article{yuan2021marrying,
  title={Marrying stochastic gradient descent with bandits: Learning algorithms for inventory systems with fixed costs},
  author={Yuan, Hao and Luo, Qi and Shi, Cong},
  journal={Management Science},
  volume={67},
  number={10},
  pages={6089--6115},
  year={2021},
  publisher={Informs}
}

@article{yuan2024asymptotic,
  title={Asymptotic optimality of base-stock policies for lost-sales inventory systems with stochastic lead times},
  author={Yuan, Shilin and Lyu, Jiameng and Xie, Jinxing and Zhou, Yuan},
  journal={Operations Research Letters},
  volume={57},
  pages={107196},
  year={2024},
  publisher={Elsevier}
}

@article{zhang2020closing,
  title={Closing the gap: A learning algorithm for lost-sales inventory systems with lead times},
  author={Zhang, Huanan and Chao, Xiuli and Shi, Cong},
  journal={Management Science},
  volume={66},
  number={5},
  pages={1962--1980},
  year={2020},
  publisher={INFORMS}
}

@misc{zhang2025information,
      title={Reinforcement Learning in MDPs with Information-Ordered Policies}, 
      author={Zhongjun Zhang and Shipra Agrawal and Ilan Lobel and Sean R. Sinclair and Christina Lee Yu},
      year={2025},
      eprint={2508.03904},
      archivePrefix={arXiv},
      primaryClass={stat.ML},
      url={https://arxiv.org/abs/2508.03904}, 
}

@article{zhao2021bandit,
  title={Bandit convex optimization in non-stationary environments},
  author={Zhao, Peng and Wang, Guanghui and Zhang, Lijun and Zhou, Zhi-Hua},
  journal={Journal of Machine Learning Research},
  volume={22},
  number={125},
  pages={1--45},
  year={2021}
}

@article{zhou2020nonstationary,
  title={Nonstationary reinforcement learning with linear function approximation},
  author={Zhou, Huozhi and Chen, Jinglin and Varshney, Lav R and Jagmohan, Ashish},
  journal={arXiv preprint arXiv:2010.04244},
  year={2020}
}

@article{zipkin2000foundations,
  title={Foundations of inventory management},
  author={Zipkin, Paul Herbert},
  journal={McGraw-Hill},
  year={2000}
}


%
%
%

\newpage

\AtBeginEnvironment{APPENDICES}{%
  \renewcommand{\theHsection}{appendix.\Alph{section}}%
  \renewcommand{\theHsubsection}{\theHsection.\arabic{subsection}}%
  \renewcommand{\theHsubsubsection}{\theHsubsection.\arabic{subsubsection}}%
  \renewcommand{\theHfigure}{\theHsection.\arabic{figure}}%
  \renewcommand{\theHtable}{\theHsection.\arabic{table}}%
  \renewcommand{\theHequation}{\theHsection.\arabic{equation}}%
  \renewcommand{\theHtheorem}{\theHsection.\arabic{theorem}}%
}
\makeatother

\crefalias{section}{appendix}
\begin{APPENDICES}
\OneAndAHalfSpacedXI 
\section{\texorpdfstring{Proofs of Lemmas in \Cref{sec:model,sec:problem}}{Proofs of Lemmas in Section 3 and 4}} 

\subsection{\texorpdfstring{Proof of \Cref{lemma:lipschitz} (Lipschitz Property) and \Cref{lemma:concentration} (Concentration Inequality) for the Backlog Setting}{Proof of Concentration Inequality for Backlog}} \label[appendix]{app:proof_lipschitz_concentration_backlogging}

We provide the proofs of the Lipschitz property and the concentration inequality of the cost function under demand backlogging. 
All results in this subsection hold under stationary conditions, that is, ${F \coloneq F_t = F_{t'} \  \forall t, t' \in [T]}$.
We begin by introducing
\[
    \textstyle \mathcal{S}^{\tau} \coloneq \big\{\mathbf{s} \coloneq (s(0), \dots , s(L)) \in \mathbb{R}^{L+1} \bigm\vert \sum_{i=0}^{L} s(i) = \tau\big\}
\] 
as the set of vectors of dimension $L+1$ whose elements sum to ${\tau \in [0, U]}$.
On this state space, we define a Markov reward process (MRP) formulation of our inventory control system.

\begin{definition}[Markov Reward Process in the Inventory System with Backlogging] \label{def:MRP_backlogging}
    Consider the {\bf stationary backlogging inventory system} with lead time ${L \geq 0}$ under base-stock policies as a Markov reward process. The stochastic process with starting state ${\mathbf{s}_{1} = (s_1(0), \ldots, s_1(L)) \in \S^\tau}$ is given by the state transitions for $t \geq 2$:
    \begin{equation} \label{eq:definition_s}
        \textstyle \mathbf{s}_{t} \coloneq \left(s_{t-1}(0) - D_{t-1} + s_{t-1}(1), s_{t-1}(2), \dots , s_{t-1}(L), \tau - \sum_{i=0}^{L-1} s_{t}(i)\right)
    \end{equation}
    governed by demand realizations $D_{t-1} \sim F$. The {\bf immediate cost} is
        $C_t(\tau) \coloneq h \cdot (s_t(0) - Y_t) - b \cdot Y_t$,
    where ${Y_t = \min\{s_t(0), D_t\}}$ for $t \geq 1$.
\end{definition}

We note that this representation of the backlog system is different (but equivalent) to the model defined in \cref{eq:inventory_balance}.
The states in this MRP correspond to the system immediately after replenishment and the placement of a new order, but before demand is realized
    $\mathbf{s}_t = \left(I_t + Q_{t-L}, Q_{t-L+1}, \dots , Q_t\right)$.
This choice of notation is convenient for the subsequent proofs, since in this formulation the total amount of the inventory position coincides with the base-stock level.

Unlike the lost-sales setting where the impact of heavy-tailed demand is naturally truncated, the backlogging model allows unmet demand to accumulate over time. As a result, the state space of the system is unbounded, since large deviations in demand can propagate into arbitrarily high backlogs. To bound backlog costs with high probability, we need control over demand fluctuations. Therefore, \Cref{ass:main_assumption}.\ref{ass:sub-Gaussian_demand} states that $F$ is sub-Gaussian with variance proxy $\sigma^2$, that is, 
    $\mathbb{E}_{F}[e^{\lambda(D-\mathbb{E}_{F}[D])}]\leq e^{\frac{\lambda^2\sigma^2}{2}}$ for all $\lambda \in \mathbb{R}$.
We show that this condition guarantees exponential concentration of demand, which in turn allows us to derive high-probability bounds on costs. 

We start off with the following result, that states that in our MRP the amount ordered is always equal to the most recently observed demand.  We leverage this to establish that the inventory vector converges to its stationary state after at most $L$ time steps. Hence, starting from any state ${\mathbf{s}_1 \in \S^\tau}$ there is a simple representation for its average cost function.

\begin{lemma} \label{lemma:stationary_distribution_backlogging}
    Consider the stochastic process in \cref{def:MRP_backlogging} with a fixed starting state $\bfs_1 \in \S^\tau$.
    Then, for all $t \geq 2$ and demand realizations $D_{t-1}$, the state $\mathbf{s}_t \in \S^\tau$ and the order quantity is $Q_t = s_t(L) = D_{t-1}$.  Moreover, the stochastic process of states and costs is $(L+1)$-dependent, that is, states or costs separated by more than $L+1$ time steps are independent.
    For $t > L$, the states follow the stationary distribution
    \begin{equation} \label{eq:definition_stationary_s}
        \textstyle \bfs_t \stackrel{d}{=} \left(\tau - \sum_{i=1}^L D_{t-i}, D_{t-L}, \ldots, D_{t-1}\right),
    \end{equation}
    where $D_{t-L}, \ldots, D_{t-1} \sim F$. Furthermore, the average cost satisfies
    \begin{equation} \label{eq:stationary_distribution_cost_function_backlogging}
        \textstyle \mu(\tau \mid \bfs_1) \coloneqq \mathbb{E}_{F} \left[ \lim_{T \rightarrow \infty} \frac{1}{T} \sum_{t=1}^T C_{t}(\tau) \Bigm\vert D \sim F, \bfs_1 \right] = \E_F[C(\tau)],
    \end{equation}
    where ${C_t(\tau) = h \cdot (s_t(0) - Y_t) - b \cdot Y_t}$ with ${Y_t = \min\{s_t(0), D_t\}}$ and $\bfs_t$ is distributed as in \cref{eq:definition_stationary_s}.
\end{lemma}
\begin{proof}
    \noindent {\bf States and order quantities.}\quad We show that $\bfs_t \in \S^\tau$ and $Q_t = s_t(L) = D_{t-1}$ for $t \geq 2$ by induction. The base case $\bfs_1 \in \S^\tau$ is true by assumption. For the inductive step we have that
    \vspace{-7pt}
    \[
        \textstyle Q_t = s_t(L) = \tau - \sum_{i=0}^{L-1} s_t(i) = \tau - \sum_{i=0}^{L} s_{t-1}(i) + D_{t-1} = D_{t-1},
        \vspace{-5pt}
    \]
    where in the last step we used the induction hypothesis that ${\bfs_{t-1} \in \S^\tau}$. To see that ${\bfs_t \in \S^\tau}$ as well, note that
    \vspace{-5pt}
    \[
        \textstyle \sum_{i=0}^{L} s_t(i) = \sum_{i=0}^{L} s_{t-1}(i) - D_{t-1} + s_t(L) = \tau - D_{t-1} + D_{t-1} = \tau
        \vspace{-5pt}
    \] 
    since $\bfs_{t-1} \in \S^\tau$ and $s_t(L) = D_{t-1}$.
    
    \medskip
    \noindent {\bf Stationary distribution of the stochastic process and $(L+1)$-dependence.}\quad Next we characterize the stationary distribution of the stochastic process.  By the definition of the process in \cref{eq:definition_s}, the linearity of the pipeline shift, and the fact that $s_t(L) = D_{t-1}$ for all $t \geq 2$, we see that after $L$ time steps we have ${\bfs_{L+1} = \left(\tau - \sum_{i=1}^L D_{L+1-i}, D_{1}, \ldots, D_{L} \right)}$, where each of the $D \sim F$. This is clearly independent of the starting state $\bfs_1$. Moreover, for any ${t > L}$ we have ${\bfs_t \stackrel{d}{=} \left(\tau - \sum_{i=1}^L D_{t-i}, D_{t-L}, \ldots, D_{t-1}\right)}$.
    
    This shows that $\bfs_t$ follows the stationary distribution of the stochastic process. Moreover, by inspection one can observe that when the system is in steady state, the distribution of states and costs is time-invariant as determined entirely by the most recent $L+1$ demand realizations. Consequently, the transition kernel does not depend on $t$ and the MRP is $(L+1)$-dependent.
    
    \medskip
    \noindent {\bf Average cost.}\quad 
    Using the above definition of the stationary distribution of the stochastic process, we have
    \vspace{-5pt}
    \begin{align*}
        \mu(\tau \mid \bfs) &\coloneqq \textstyle \mathbb{E}_{F} \left[ \lim_{T \rightarrow \infty} \frac{1}{T} \sum_{t=1}^T C_{t}(\tau) \Big| D \sim F, \mathbf{s}_1 = \mathbf{s} \right] \\
        & = \textstyle \mathbb{E}_{F} \left[ \lim_{T \rightarrow \infty} \frac{1}{T} \left(\sum_{t=1}^{L} C_t(\tau) + \sum_{t=L+1}^T C_t(\tau) \right)\right]
        = \mathbb{E}_{F}[C_{L+1}(\tau)],
    \vspace{-5pt}
    \end{align*}
    where we use the fact that the limit of the first $L$ time steps goes to zero, and the remaining time steps are i.i.d.\ under stationarity and equal in distribution to $C_{L+1}(\tau)$.
\end{proof}

\Cref{lemma:stationary_distribution_backlogging} applies under the condition that the initial inventory position equals the base-stock level. While the condition may not hold for the implemented policy after a reduction in the base-stock level, it remains satisfied for the counterfactual inventory vectors, as each is defined with a fixed base-stock level.

Next, we show that the average cost function is Lipschitz continuous with respect to the base-stock level. Since \Cref{lemma:stationary_distribution_backlogging} states that the average cost is independent of the starting state as long as $\bfs_1 \in \S^\tau$, we abuse notation slightly and omit its dependence.

\LipschitzLemma*
 
\begin{proof}
By definition of the average cost and \cref{lemma:stationary_distribution_backlogging}, under a base-stock policy with parameter $\tau$, the inventory on-hand before demand is ${s_t(0) = \tau - \sum_{i=1}^L D_{t-i}}$ and the immediate cost is
\vspace{-5pt}
\[
    C_t(\tau) = h \cdot (s_t(0) - \min\{ s_t(0), D_t\}) - b \cdot \min\{s_t(0), D_t\}.
    \vspace{-5pt}
\]
Similarly, under base-stock level $\tau'$ we have $s_t'(0) = \tau'- \sum_{i=1}^L D_{t-i}$.
Using the identity $\min\{s_t(0), D_t\} = \frac{s_t(0) + D_t - \abs*{s_t(0)-D_t}}{2}$, the cost function can be written as
\vspace{-5pt}
\begin{align} \label{eq:cost_in_lipschitz_proof}
    C_t(\tau) &= h \cdot s_t(0) - (h + b) \cdot \frac{s_t(0) + D_t -\abs*{s_t(0)-D_t}}{2} \notag \\
    & = \frac{h+b}{2} \abs*{s_t(0)-D_t} + \frac{h-b}{2} s_t(0) - \frac{h+b}{2} D_t.
\end{align}

\noindent Using $\abs*{\abs*{s_t(0) - D_t} - \abs*{s_t'(0) - D_t}} \leq \abs*{s_t(0) - s_t'(0)}$ and $s_t(0) - s_t'(0) = \tau - \tau'$, it follows 
that
\begin{align*}
    \abs*{C_t(\tau) - C_t(\tau')} &= \frac{h + b}{2} \abs*{\left(\abs*{s_t(0) - D_t} - \abs*{s_t'(0) - D_t}\right)} + \frac{\abs*{h-b}}{2} \abs*{\left(s_t(0) - D_t \right) - \left(s_t'(0) - D_t\right)} \\
    & \leq \frac{h + b}{2} \abs*{s_t(0) - s_t'(0)} + \frac{\abs*{h-b}}{2} \abs*{s_t(0) - s_t'(0)} \\ 
    & = \left(\frac{h + b}{2} + \frac{\abs*{h - b}}{2}\right) \cdot \abs*{\tau - \tau'} \\
    & \leq \max\{h, b\} \cdot \abs*{\tau - \tau'}.
\end{align*}
Taking expectation over the demand $D_t$ yields the result.
%
\end{proof}

Next, we establish \cref{lemma:concentration} for the demand backlog setting.  We begin by leveraging the fact that demand is sub-Gaussian to show that the induced cost function is similarly sub-Gaussian.

\begin{lemma} \label{lemma:costs_sub-Gaussian}
    Given that the demand $F$ is $\sigma^2$-sub-Gaussian, for any ${\tau \in [0,U]}$ the centered costs ${C_t(\tau) - \mathbb{E}_F[C(\tau)]}$ are sub-Gaussian with variance proxy ${\sigma_C^2 \coloneq 2\sigma^2(Lh^2 + (h+b)^2(4L+5))}$.
\end{lemma}
\begin{proof}
By \cref{lemma:stationary_distribution_backlogging}, the system attains the steady state in \cref{eq:definition_stationary_s} for ${t > L}$. We analyze the costs during the transient phase, in which the system has not yet reached stationarity, and in the steady state separately by decomposing the costs into two terms. 

\noindent {\bf First $L$ time steps.}\quad
Given initial state $\mathbf{s}_1$, the inventory on-hand or backlog at time ${1 \leq t \leq L}$ is given by
    $s_t(0) = \sum_{i=0}^{t-1} s_1(i) - \sum_{i=1}^{t-1} D_i$,
where for $t = 1$ we implicitly assume the second sum is zero, and the costs by
\vspace{-2pt}
\begin{align*}
    C_t(\tau) &= h \cdot (s_t(0) - \min\{s_t(0), D_t\}) - b \cdot \min\{s_t(0), D_t\} \\
    &= h s_t(0) - (h+b) \min\{s_t(0), D_t\} \\
    &= \textstyle h \left(\sum_{i=0}^{t-1} s_1(i) - \sum_{i=1}^{t-1} D_i\right) - (h+b) \min\big\{\sum_{i=0}^{t-1} s_1(i) - \sum_{i=1}^{t-1} D_i, D_t\big\}.
\end{align*}
Using the identity 
\vspace{-5pt}
\begin{align*}
    \textstyle \min\big\{\sum_{i=0}^{t-1} s_1(i) - \sum_{i=1}^{t-1} D_i, D_t\big\} = D_t - \left[D_t - \sum_{i=0}^{t-1} s_1(i) + \sum_{i=1}^{t-1} D_i\right]^+
    = D_t - \left[\sum_{i=1}^{t} D_i - \sum_{i=0}^{t-1} s_1(i)\right]^+
\end{align*}
in the previous expression for $C_t(\tau)$ gives
\vspace{-6pt}
\[
    \textstyle C_t(\tau) = h \sum_{i=0}^{t-1} s_1(i) - h\sum_{i=1}^{t-1} D_i - (h+b) D_t + (h+b) \left[\sum_{i=1}^{t} D_i - \sum_{i=0}^{t-1} s_1(i)\right]^+.
    \vspace{-5pt}
\]
We decompose the centered cost function as ${C_t(\tau) - \mathbb{E}_F[C_t(\tau)] = \bar{C}^1_t + \bar{C}^2_t}$ with
\vspace{-6pt}
\begin{align*}
    \bar{C}^1_t &\coloneq \textstyle -h\sum_{i=1}^{t-1} \left(D_{i} - \mathbb{E}_F[D_{i}]\right) - (h+b) \left(D_t - \mathbb{E}_F[D_t]\right), \\
    \bar{C}^2_t &\coloneq \textstyle (h+b)\left(\left[\sum_{i=1}^t D_{i} - \sum_{i=0}^{t-1} s_1(i) \right]^+ - \mathbb{E}_F \left[\left[\sum_{i=1}^t D_{i} - \sum_{i=0}^{t-1} s_1(i) \right]^+\right]\right).
\end{align*}
The terms $\bar{C}^1_t$ are centered linear combinations of independent sub-Gaussian random variables $(D_{i})_{i=1, \dots , t}$, each with variance proxy $\sigma^2$, and are therefore sub-Gaussian with variance proxy
    ${\bar{\sigma}_1^2 = \sigma^2 (\max\{L-1, 0\}h^2 +(h+b)^2)}$.
The non-linear terms $\bar{C}^2_t$ are scalar multiples of the centered positive part of ${\sum_{i=1}^t D_{i} - \sum_{i=0}^{t-1} s_1(i)}$.
By assumption, the demands are $\sigma^2$-sub-Gaussian, hence the sums $\sum_{i=1}^t D_{i}$ are sub-Gaussian with variance proxy $t\sigma^2$. Moreover, it is easy to see that the function ${\phi(x) \coloneq (h+b)[x - \sum_{i=0}^{t-1} s_1(i)]^+}$ is $(h+b)$-Lipschitz. By standard symmetrization arguments (see \citep[Section 4]{duchi2023lecture} for details), we have that $\bar{C}^2_t$ are sub-Gaussian with parameter at most 
    ${\bar{\sigma}_2^2 = 4(h+b)^2L\sigma^2}$.

\medskip
\noindent {\bf Time steps $t \geq L+1$.}\quad
By \cref{lemma:stationary_distribution_backlogging}, the inventory from time step ${t = L+1}$ can be expressed as ${s_t(0) = \tau - \sum_{i=1}^L D_{t-i}}$. 
Therefore, we can rewrite the stationary cost function as
\vspace{-7pt}
\begin{align*}
    C_t(\tau) &= h \cdot s_t(0) - (h+b) \cdot \min\{s_t(0), D_t\} \\
    &= \textstyle h \tau - h \sum_{i=1}^L D_{t-i} - (h+b) D_t + (h+b) [D_t - s_t(0)]^+ \\
    &= \textstyle h \tau - h \sum_{i=1}^L D_{t-i} - (h+b) D_t + (h+b) \left[\sum_{i=0}^L D_{t-i} - \tau \right]^+.
\end{align*}
Following the same steps as in the case above, we decompose the centered cost function as ${C_t(\tau) - \mathbb{E}_F[C(\tau)] = C^1 + C^2}$ with
\vspace{-7pt}
\begin{align*}
    C^1 &\coloneq \textstyle -h\sum_{i=1}^L \left(D_{t-i} - \mathbb{E}_F[D_{t-i}]\right) - (h+b) \left(D_t - \mathbb{E}_F[D_t]\right), \\
    C^2 &\coloneq \textstyle (h+b)\left(\left[\sum_{i=0}^L D_{t-i} - \tau \right]^+ - \mathbb{E}_F \left[\left[\sum_{i=0}^L D_{t-i} - \tau\right]^+\right]\right)
\end{align*}
and see that $C^1$ is sub-Gaussian with variance proxy $\sigma_1^2 = \sigma^2(Lh^2 + (h+b)^2)$.

Next, we control the non-linear term $C^2$. Let $D^{(L)} \coloneq \sum_{i=0}^L D_{t-i}$ and $\phi(x) \coloneq (h+b)[x-\tau]^+$. The function~$\phi$ is $(h+b)$-Lipschitz, and $D^{(L)}$ is sub-Gaussian with variance proxy $(L+1)\sigma^2$. By the same argument as above, we have that $C^2$ is sub-Gaussian with parameter
$\sigma_2^2 = 4(h+b)^2 (L+1)\sigma^2$.

Finally, we consider the total centered costs ${C_t(\tau) - \E_F[C_t(\tau)]}$. In both cases when ${t \leq L}$ and ${t > L}$ we decompose the cost into two terms ($(\bar{C}_t^1, \bar{C}_t^2)$ and $(C_1, C_2)$, respectively).  The two components share the same demand realizations in both cases and are thus dependent. However, using \citep[Theorem~2.7]{rivasplata2012subgaussian} their sums are ${(\bar{\sigma}_1 + \bar{\sigma}_2)^2}$ and ${(\sigma_1 + \sigma_2)^2}$-sub-Gaussian respectively. Noting that ${\bar{\sigma}_1^2 \leq \sigma_1^2}$ and ${\bar{\sigma}_2^2 \leq \sigma_2^2}$, we have that for all $t$ the centered costs ${C_t(\tau) - \E_F[C(\tau)]}$ are sub-Gaussian with parameter
\[
    (\sigma_1 + \sigma_2)^2 \leq 2(\sigma_1^2 + \sigma_2^2) = 2\sigma^2(Lh^2 + (h+b)^2(4L+5)) =: \sigma_C^2.
    \vspace{-0.9cm}
\]
\end{proof}

\ConcentrationLemma*

\begin{proof}
    First, note that the result for the lost-sales inventory model is shown directly in \citep[Lemma 4]{agrawal2022learning}. Therefore, we focus on proving the result for the backlog model.
    Let ${\tau \in [0,U]}$ be a fixed base-stock level and $[s,t-1]$ a stationary time interval of length ${N \coloneq t-s}$.     
    Without loss of generality, we write $s=1$ and $t-1=N$ for notational convenience. We aim to control the deviation of the empirical average cost $\hat{\mu}(\tau, s, t) \coloneq \frac{1}{N}\sum_{t'=1}^N C_{t'}(\tau)$ from its stationary mean ${\mu(\tau) = \mathbb{E}_F[C(\tau)]}$.

    We start off by noting the equivalence between the states in the inventory process defined in \cref{eq:inventory_balance} to the MRP in \cref{eq:definition_s}.  Indeed, since the starting state satisfies ${I_s + \sum_{i=1}^L Q_{s-i} \leq \tau}$, under base-stock policy $\tau$, we have that ${Q_s = \tau - (I_s + \sum_{i=1}^L Q_{s-i})}$ and that ${I_s + \sum_{i=1}^L Q_{s-i} + Q_s = \tau}$.  Thus, we note that the inventory process is equivalent to the MRP with a starting state ${\bfs_1 = (I_s + Q_{s-L}, Q_{s-L+1}, \ldots, Q_s) \in \S^\tau}$ and demand distribution $F_s$.
    By \cref{lemma:stationary_distribution_backlogging}, the system will reach the steady state given in \cref{eq:definition_stationary_s} in at most $L$ steps. We handle the first $L$ and the remaining $N-L$ time steps separately. 

    \medskip
    \noindent {\bf First $L$ time steps.}\quad 
    For the case ${L \geq 1}$, we start by deriving a high-probability bound for the cumulative costs in the first $L$ steps.
    By \Cref{lemma:costs_sub-Gaussian}, the centered costs are $\sigma^2_C$-sub-Gaussian with ${\sigma_C^2 = 2\sigma^2(Lh^2 + (h+b)^2(4L+5))}$. We have for any ${t' \geq 1}$ and ${\varepsilon \geq 0}$ that
    \[
        \Pr_{F}\left(\abs*{C_{t'}(\tau) - \mathbb{E}_F[C(\tau)]} \geq \varepsilon\right) \leq \textstyle 2 \exp \left( -\varepsilon^2/\sigma_C^2\right).
    \]
    Setting $\varepsilon = \sigma_C \sqrt{ \log(4L/\delta)}$ yields $\Pr_{F}\left(\abs*{C_{t'}(\tau) - \mathbb{E}_F[C(\tau)]} \geq \sigma_C \sqrt{\log(4L/\delta)}\right) \leq \textstyle \frac{\delta}{2L}$.
    With a union bound over the first $L+1$ steps, we get
        $\textstyle \Pr_{F}\left(\abs*{\sum_{t'=1}^{L+1} \left(C_{t'}(\tau) - \mathbb{E}_F[C(\tau)]\right)} \geq L\sigma_C \sqrt{\log(4L/\delta)}\right) \leq \frac{\delta}{2}$,
    or equivalently
    \begin{equation} \label{eq:bound_first_L_steps}
        \textstyle \Pr_{F}\left(\frac{1}{N} \abs*{\sum_{t'=1}^{L+1} \left(C_{t'}(\tau) - \mathbb{E}_F[C(\tau)]\right)} \geq \frac{L}{N} \sigma_C \sqrt{\log(4L/\delta)}\right) \leq \frac{\delta}{2}.
    \end{equation}

    \medskip
    \noindent {\bf Time steps $L+1$ to $N$.}\quad
    Next, we derive the concentration bound for the stationary time steps ${t' \in [L+1, N]}$ if $N>L$.
    Let the sum of centered costs be denoted as
        $\textstyle S_N \coloneq \sum_{t'=L+1}^N \left(C_{t'}(\tau) - \mathbb{E}_F[C(\tau)]\right)$.
    We decompose the sum into the sum of $L+1$ different components, where each sub-component is the sum of independent random variables. Indeed, for ${j \in \{0,1,\dots , L\}}$, we define residue classes of the time steps in $[L+1,N]$ by partitioning them into blocks with spacing $L+1$, that is,
        $G_j \coloneq \{t' \in [L+1,N]:t' \equiv j \pmod{L+1}\}$.
    Accordingly, we decompose the sum
        $\textstyle S_N = \sum_{j=0}^{L}S_{N_j}$ with $S_{N_j} \coloneq \sum_{t' \in G_j} \left(C_{t'}(\tau) - \mathbb{E}_F[C(\tau)]\right)$.
    By \cref{lemma:stationary_distribution_backlogging}, the costs incurred from time step ${t'=L}$ are ${(L+1)}$-dependent and thus for each ${j \in \{0,1, \dots , L\}}$, the partial sum $S_{N_j}$ is a sum of at most $N_j \coloneq \abs*{G_j} \leq \lceil (N-L)/(L+1) \rceil$ mutually independent sub-Gaussian variables.
    By \Cref{lemma:costs_sub-Gaussian}, the centered costs $C_{t'}(\tau) - \mathbb{E}_F[C(\tau)]$ are $\sigma_C^2$-sub-Gaussian, that is, 
    \[
        \textstyle \mathbb{E}_F\left[\exp \left(\varepsilon \left(C_{t'}(\tau) - \mathbb{E}_F[C(\tau)]\right)\right)\right] \leq \exp \left(\varepsilon^2 \sigma_C^2 /2\right) \quad \forall \varepsilon \in \mathbb{R}.
    \]
    It follows from the independence of the summands of the $S_{N_j}$ that the latter are sub-Gaussian. It holds that
    \begin{equation} \label{eq:sub-Gaussion_block_concentration}
        \textstyle \mathbb{E}_F\left[\exp \left(\varepsilon S_{N_j}\right)\right]
        \leq \exp \left(\varepsilon^2 \sigma_C^2 N_j /2\right) \quad \forall \varepsilon \in \mathbb{R}.
    \end{equation}
    
    Next, we bound the moment generating function of the total sum. We apply Hölder's inequality for expectations with equal conjugate exponents $p=L+1$ to the non-negative random variables $\exp (\lambda S_{N_j})$. This gives the following bound for the total sum $S_N$,
    \begin{align*}
        \textstyle \mathbb{E}_F \left[\exp \left( \lambda S_{N}\right)\right] = \mathbb{E}_F \left[\exp \left( \lambda \sum_{j=0}^L S_{N_j}\right)\right] 
        = \mathbb{E}_F \left[ \prod_{j=0}^L \exp \left( \lambda S_{N_j}\right)\right] 
        \leq \prod_{j=0}^L\left( \mathbb{E}_F \left[ \exp \left( p\lambda S_{N_j}\right)\right]\right)^{1/p} \quad \forall \lambda \in \mathbb{R}.
    \end{align*}
    Using \cref{eq:sub-Gaussion_block_concentration} with ${\varepsilon = p\lambda}$, and ${p=L+1}$ as well as ${\sum_{j=0}^L N_j = N-L}$, we obtain
    \begin{align*}
        \textstyle \prod_{j=0}^L\left( \mathbb{E}_F \left[ \exp \left( p\lambda S_{N_j}\right)\right]\right)^{1/p} \leq \textstyle \prod_{j=0}^L\left( \exp \Big( \frac{p^2 \lambda^2 \sigma_C^2 N_j}{2}\Big)\right)^{1/p}
        = \textstyle \exp \Big( \frac{p \lambda^2 \sigma_C^2 \sum_{j=0}^L N_j}{2} \Big)
        = \textstyle \exp \Big( \frac{(L+1) \lambda^2 \sigma_C^2 (N-L)}{2} \Big).
    \end{align*}
    With the standard Chernoff bound argument, we have for any $u>0$ and $\lambda>0$ that
    \[
        \Pr_{F}(S_N \geq u) \leq \exp (-\lambda u) \cdot \mathbb{E}_F [\exp (\lambda S_N)] \leq \textstyle \exp \Big(-\lambda u + \frac{(L+1) \lambda^2 \sigma_C^2 (N-L)}{2}\Big).
    \]
    Setting ${\lambda = u/(L+1)\sigma_C^2 (N-L)}$, this gives $\Pr_{F}(S_N \geq u) \leq \exp\Big( -\frac{u^2}{2(L+1) \sigma_C^2 (N-L)} \Big)$.
    By symmetry and with ${\varepsilon = \sigma_C \sqrt{\frac{2(L+1)(N-L)\log(4/\delta)}{N^2}}}$, it follows the bound on the absolute difference between the empirical and the expected cost that
    \begin{equation}
        \textstyle \Pr_{F}\left( \frac{1}{N} \abs*{\sum_{t'=L+1}^N \left( C_{t'}(\tau) - \mathbb{E}_F[C(\tau)] \right)} \geq \varepsilon \right)
        = \Pr_{F}\left(\frac{1}{N}\abs*{S_N} \geq \varepsilon \right) \leq \frac{\delta}{2}.
    \end{equation}

    \medskip
    \noindent {\bf Combining both sets of time steps.}\quad
    We combine the above bound with that for the first $L$ time steps in \cref{eq:bound_first_L_steps}. By a union bound, we obtain with probability at least $1 - \delta$ that
    \[
        \textstyle \frac{1}{N} \abs*{\sum_{t'=1}^N \left( C_{t'}(\tau) - \mathbb{E}_F[C(\tau)]\right)} \leq \textstyle \sigma_C \left(\frac{L}{N} \sqrt{\log(4L/\delta)} + \sqrt{\frac{2(L+1)\log(4/\delta)}{N}}\right).
    \]
    Finally, under the stated assumption ${L \leq N}$ it holds that $\frac{L}{N} \leq \sqrt{\frac{L+1}{N}}$, and the above bound implies that $\abs{\hat{\mu}(\tau, s,t) - \mu_s(\tau)} \leq b_{s,t}$, where $b_{s,t} = \Hbl \sqrt{\frac{2\log(4(L+1)/\delta)}{N}}$ with $\Hbl = 2\sqrt{L+1} \sigma_C =  2\sqrt{2} \sigma \sqrt{(L+1)(Lh^2 + (h+b)^2(4L+5))}$.
\end{proof}

\subsection{\texorpdfstring{Proof of \Cref{lemma:feedback_structure} (Counterfactual Policy Evaluation)}{Proof of Counterfactual Policy Evaluation}} \label[appendix]{app:proof_counterfactual_policy_evaluation}

We establish the left-sided feedback in the lost-sales regime under a time-varying base-stock level as stated in \cref{lemma:feedback_structure}.
The lemma claims that we can evaluate the cost estimates $\hat{\mu}(\tau, s, t)$ of any base-stock policy with level $\tau$ when a sequence of base-stock levels $(\tau_{t'})_{t' \in [s,t]}$ with $\tau \leq \tau_{t'}$ for all steps $t'$ in a contiguous interval $[s,t]$ is implemented. If lead time is positive, this holds only for non-decreasing sequences $(\tau_{t'})_{t' \in [s,t]}$.

\LSFLemma*
\begin{proof}
    As mentioned in the discussion of \cref{lemma:feedback_structure} in \cref{sec:model}, the result for the backlogging setting follows directly from the fact that demand is always observed and the cost function is fully specified given the demand realizations.
    To show the result for the lost-sales setting, we first prove that, under the stated conditions on the starting inventory vectors, the inventory level after replenishment under $\tau_{t'}$ always exceeds that of the policy parameterized by the lower level $\tau$.
    More precisely, it will be proved by induction over the time step $t'$ that for all ${t' \in [s, t]}$ it holds that $I_{t'} + Q_{t'-L} \geq I^{\tau}_{t'} + Q^{\tau}_{t'-L}$.
    After proving that this condition holds, we show that it is sufficient for full counterfactual inference.
    We begin by considering the setting with lead time $L=0$. \\
    \noindent \textbf{Base case} $t'=s$: \
    By assumption we have $\max\{I_s,\tau\} \geq I^{\tau}_s$ and $\tau_s \geq \tau$. Since
    \[
         I_{s} + Q_{s} = I_s + [\tau_s - I_s]^+ = \max\{I_s, \tau_s\} \quad \text{ and } \quad I^{\tau}_{s} + Q^{\tau}_{s} = I^{\tau}_s + [\tau - I^{\tau}_s]^+ = \max\{I^{\tau}_s, \tau\}
    \]
    and as the maximum is monotone non-decreasing, we get $I_{s} + Q_{s} =\max\{I_s, \tau_s\} \geq \max\{I^{\tau}_s, \tau\} = I^{\tau}_{s} + Q^{\tau}_{s}$.

    \noindent \textbf{Inductive step}: \
    Suppose the induction hypothesis $I_{t'} + Q_{t'} \geq I^{\tau}_{t'} + Q^{\tau}_{t'}$ holds. This implies for the inventory level in the next time step $I_{t'+1} = [I_{t'} + Q_{t'} - D_{t'}]^+ \geq [I^{\tau}_{t'} + Q^{\tau}_{t'} - D_{t'}]^+ = I^{\tau}_{t'+1}$.
    By the assumption $\tau_{t'+1} \geq \tau$, it follows that
        $I_{t'+1} + [\tau_{t'+1} - I_{t'+1}]^+ = \max\{I_{t'+1}, \tau_{t'+1}\} \geq \max\{I^{\tau}_{t'+1}, \tau\} = I^{\tau}_{t'+1} + [\tau - I^{\tau}_{t'+1}]^+$.
    Thus, the stated result holds for all $t' \in [s,t]$.
    
    \medskip
    Next, we show the result for the case $L>0$.
    We show by induction over $t'$ that, if the starting inventory vectors satisfy the conditions specified in \Cref{lemma:feedback_structure}.\ref{item:feedback_structure_LS_L_positive}, then they do so for all ${t' \in [s,t]}$ in which $\tau_{t'}$ does not decrease. The stated result follows immediately from inequalities ${I_{t'} \geq I^{\tau}_{t'}}$ and ${Q_{t'-L} \geq Q^{\tau}_{t'-L}}$.
    
    \noindent \textbf{Base case}: \
    For $t'=s$, this is trivially satisfied by assumption. 
    
    \noindent \textbf{Inductive step}: \
    Suppose the induction hypothesis 
    \vspace{-5pt}
    \[
        I_{t'} \geq I^{\tau}_{t'} \qquad
        Q_{t'-i} \geq Q^{\tau}_{t'-i} \ \forall i=1, \dots , L, \qquad
        \tau_{t'} - \left(I_{t'} + {\textstyle \sum_{i=1}^{L} Q_{t'-i}}\right) \geq \tau - \left(I^{\tau}_{t'} + {\textstyle \sum_{i=1}^{L} Q^{\tau}_{t'-i}}\right) \geq 0
    \]
    holds. Recall that the inventory balance equations are given by
    \vspace{-5pt}
    \begin{align*}
        \left(I_{t'+1}, Q_{t'-L+1}, \dots , Q_{t'}\right) &= \Big([I_{t'} + Q_{t'-L} - D_{t'}]^+, Q_{t'-L+1}, \dots , Q_{t'-1}, [\tau_{t'} - (I_{t'} + {\textstyle \sum_{i=1}^L Q_{t'-i}})]^+\Big), \\
        \left(I_{t'+1}^{\tau}, Q_{t'-L+1}^{\tau}, \dots , Q_{t'}^{\tau}\right) &= \Big([I_{t'}^{\tau} + Q_{t'-L}^{\tau} - D_{t'}]^+, Q_{t'-L+1}^{\tau}, \dots , Q^{\tau}_{t'-1}, [\tau - (I^{\tau}_{t'} + {\textstyle \sum_{i=1}^L Q^{\tau}_{t'-i}})]^+\Big).
    \end{align*} 
    Under the induction hypothesis, we have for the first and the last elements
    \vspace{-5pt}
    \begin{align*}
        I_{t'+1} = [I_{t'} + Q_{t'-L} - D_{t'}]^+ &\geq [I_{t'}^{\tau} + Q_{t'-L}^{\tau} - D_{t'}]^+ = I_{t'+1}^{\tau}, \\
        \textstyle Q_{t'} = [\tau_{t'} - (I_{t'} + \sum_{i=1}^L Q_{t'-i})]^+ &\geq \textstyle [\tau - (I^{\tau}_{t'} + \sum_{i=1}^L Q^{\tau}_{t'-i})]^+ = Q_{t'}^{\tau}
    \end{align*}
    and for $i=1, \dots , L-1$ that $Q_{t'-i} \geq Q^{\tau}_{t'-i}$.
    What remains to be shown is
    \[
        \textstyle \tau_{t'+1} - \left(I_{t'+1} + \sum_{i=0}^{L-1} Q_{t'-i}\right) \geq \tau - \left(I^{\tau}_{t'+1} + \sum_{i=0}^{L-1} Q^{\tau}_{t'-i}\right) \geq 0.
    \]
    The second inequality is evident from the induction hypothesis and the inventory dynamics under base-stock level $\tau$.
    To show the first inequality, recall that the sequence of base-stock levels implemented is assumed to be non-decreasing. Therefore, it holds that
    \vspace{-5pt}
    \begin{align*}
        \textstyle \tau_{t'+1} - \left(I_{t'+1} + \sum_{i=0}^{L-1} Q_{t'-i}\right) &\geq \textstyle \tau_{t'} - \left(I_{t'+1} + \sum_{i=0}^{L-1} Q_{t'-i}\right) \\
        & = \textstyle \tau_{t'} - \left([I_{t'} + Q_{t'-L} - D_{t'}]^+ + \sum_{i=1}^{L-1} Q_{t'-i} + [\tau_{t'} - (I_{t'} + \sum_{i=1}^L Q_{t'-i})]^+\right) \\
        & = \textstyle \tau_{t'} - \left([I_{t'} + Q_{t'-L} - D_{t'}]^+ + \sum_{i=1}^{L-1} Q_{t'-i} + \tau_{t'} - (I_{t'} + \sum_{i=1}^L Q_{t'-i})\right) \\
        & = I_{t'} + Q_{t'-L} - [I_{t'} + Q_{t'-L} - D_{t'}]^+,
    \end{align*}
    where the second to last equality follows from the induction hypothesis. Accordingly, under the policy with fixed level $\tau$, it holds the equality
        $\textstyle \tau - \left(I^{\tau}_{t'+1} + \sum_{i=0}^{L-1} Q^{\tau}_{t'-i}\right) = I^{\tau}_{t'} + Q^{\tau}_{t'-L} - \left[I^{\tau}_{t'} + Q^{\tau}_{t'-L} - D_{t'}\right]^+$.
    Therefore, the inequality to be shown reduces to 
    \vspace{-5pt}
    \[
        I_{t'} + Q_{t'-L} - [I_{t'} + Q_{t'-L} - D_{t'}]^+ \geq I^{\tau}_{t'} + Q^{\tau}_{t'-L} - \left[I^{\tau}_{t'} + Q^{\tau}_{t'-L} - D_{t'}\right]^+,
    \]
    which is equivalent to $\min\{I_{t'} + Q_{t'-L}, D_{t'}\} \geq \min\{I_{t'}^{\tau} + Q^{\tau}_{t'-L}, D_{t'}\}$ 
    and satisfied by the induction hypothesis.

    \noindent {\bf Sufficiency of Larger On-Hand Inventory for Counterfactual Inference.}\quad 
    It follows from the above that the on-hand inventory after replenishment under policy $\tau_{t'}$ is always at least as large as that under policy $\tau$. Specifically, if $I_{t'} + Q_{t'-L} \geq I^{\tau}_{t'} + Q^{\tau}_{t'-L}$, then the counterfactual sales under $\tau$, 
    \vspace{-5pt}
    \[
        Y^{\tau}_{t'} \coloneq \min \{ I^{\tau}_{t'} + Q^{\tau}_{t'-L}, D_{t'}\} = \min\{ I^{\tau}_{t'} + Q^{\tau}_{t'-L}, Y_{t'}\},
    \]
    and consequently the leftover inventory 
    \vspace{-5pt}
    \[
        I^{\tau}_{t'+1} = I^{\tau}_{t'} + Q^{\tau}_{t'-L} - Y^{\tau}_{t'} = [I^{\tau}_{t'} + Q^{\tau}_{t'-L} - Y_{t'}]^+
    \]
    can be evaluated directly using the sales $Y_{t'}$ observed by the implementation of $\tau_{t'}$.
    As a result, for known cost parameters $h$ and $b$, the cost function $C_{t'}(\tau) = h \cdot I^{\tau}_{t'+1} - b \cdot Y^{\tau}_{t'}$ for any ${\tau \leq \tau_{t'}}$ is measurable with respect to the data generated under $\tau_{t'}$. 
\end{proof}

\section{\texorpdfstring{Proof of \cref{thm:lower_bound}}{Proof of the Lower Bound}} \label[appendix]{app:lower_bound}

Before establishing the lower bound for the non-stationary problem, we first derive a lower bound under stationary conditions. The proof closely follows that of \citet[Proposition 1]{zhang2020closing}, who establish the lower bound for the stationary demand lost-sales model. For completeness, we slightly generalize the result to apply for both backlogging and lost-sales models.

\begin{lemma}[Static Regret Lower Bound Under Stationary Demand] \label{lemma:lower_bound_stationary}
Consider the inventory control problem with stationary demand and zero lead time, under either demand backlogging or lost-sales.  For any learning algorithm and time horizon $T \geq 5$, the expected regret is lower bounded by $\Exp{R(T)} = \Omega(\sqrt{T})$.
\end{lemma}
\begin{proof}
    Consider the following problem instance.  Demand is either fully backlogged or lost, the cost parameters satisfy $h = b = 1$, and the time horizon is $T \geq 5$.  In each period, demand $D$ is drawn i.i.d.\ from one of the following two continuous distributions
    \[
    \resizebox{\linewidth}{!}{$
    \begin{aligned}
        F^a(x) &=
        \begin{cases}
        \left(\frac{1}{8} + \frac{1}{4\sqrt{T}}\right)x, & 0 \le x < 4 \\
        \frac{1}{2} + \frac{1}{\sqrt{T}}, & 4 \le x < 400 \\
        \left(\frac{1}{8} - \frac{1}{4\sqrt{T}}\right)(x-400) + \frac{1}{2} + \frac{1}{\sqrt{T}}, & 400 \le x < 404 \\
        1, & x \ge 404,
        \end{cases}
        \quad
        F^b(x) =
        \begin{cases}
        \left(\frac{1}{8} - \frac{1}{4\sqrt{T}}\right)x, & 0 \le x < 4 \\
        \frac{1}{2} - \frac{1}{\sqrt{T}}, & 4 \le x < 400 \\
        \left(\frac{1}{8} + \frac{1}{4\sqrt{T}}\right)(x-400) + \frac{1}{2} - \frac{1}{\sqrt{T}}, & 400 \le x < 404 \\
        1, & x \ge 404.
        \end{cases}
    \end{aligned}
    $}
    \]
    The distributions $F^a$ and $F^b$ are constructed such that their optimal base-stock levels are far apart.  Consequently, any learning algorithm must correctly identify which of the two distributions governs demand in order to select a near-optimal policy.  However, the two distributions differ only by perturbations of order ${1 / \sqrt{T}}$. As a result, distinguishing between $F^a$ and $F^b$ requires on the order of $T$ samples.  Moreover, when the algorithm selects a base-stock level corresponding to the wrong region, it incurs an expected per-period regret of order ${1 / \sqrt{T}}$.  Since the algorithm must make such mistakes with constant probability at each time step, the cumulative regret is lower bounded by $\Omega(\sqrt{T})$.
       
    We first show that under $F^a$, whenever the base-stock level selected $\tau_t$ is larger than $200$, the instantaneous regret is bounded from below as $\mu(\tau_t) - \mu(\tau^*_a) \geq \frac{h+b}{\sqrt{T}} \left(\tau_t - \tau^*_a\right)$, where $\tau^*_a = \argmin_{\tau \in [0,U]} \mu(\tau)$ denotes the optimal base-stock level when $D \sim F^a$. 
    Similarly, under $F^b$ when ${\tau_t \leq 200}$, the cost is ${\mu(\tau_t) - \mu(\tau^*_b) \geq \frac{h+b}{\sqrt{T}} \left(\tau_t - \tau^*_b\right)}$ where ${\tau^*_b = \argmin_{\tau \in [0,U]} \mu(\tau)}$ is the optimal base-stock level when $D \sim F^b$.  

    We start by assuming that the demand follows distribution $F^a$. With zero lead time and i.i.d.\ demand, the expected asymptotic pseudo cost in {\em both} the lost-sales and demand backlog model can be expressed as
    \vspace{-4pt}
    \begin{align*}
        \mu(\tau) &\coloneq \textstyle \E_{F^a} \left[\lim_{T \rightarrow \infty} \frac{1}{T} \sum_{t=1}^T C_t(\tau) \right]
                  = \E_{F^a} \left[C_t(\tau) \right]  \\
                  &= h \cdot \E_{F^a}[[\tau - D]^+] + b \cdot \E_{F^a}[[D - \tau]^+] - b \cdot \E_{F^a}[D].
        \vspace{-5pt}
    \end{align*}
    The derivative with respect to the base-stock level is given by 
    \vspace{-5pt}
    \[
        \mu'(\tau) = h \cdot \Pr_{F^a}(D \leq \tau) - b \cdot \Pr_{F^a}(D > \tau) = (h + b)F^a (\tau) - b.
        \vspace{-5pt}
    \]
    Note that the optimal base-stock level in both the lost-sales and the backlogging model satisfies ${\mu'(\tau^*_a)=(h+b)F^a(\tau^*_a) - b=0}$.  Hence when $b = h = 1$ this gives ${F^a(\tau) = \frac{b}{h+b} = \frac{1}{2}}$. 
    Since $F^a(x)$ increases on $[2,4]$ linearly from $\frac{1}{4} + \frac{1}{2 \sqrt{T}} < \frac{1}{2}$ to $\frac{1}{2} + \frac{1}{\sqrt{T}}$, the optimal base-stock level under $F^a$ satisfies $\tau^*_a \in (2,4)$. 
    
    Using the definition of $F^a$, for $\tau \in [4,400)$ we have that 
        ${\textstyle \mu'(\tau) = (h+b)\left(\frac{1}{2} + \frac{1}{\sqrt{T}}\right) - b > 0}$.
    With the convexity of the cost function (\cref{lemma:convexity}), we obtain for any $\tau_t \geq 4$ that
    \vspace{-4pt}
    \begin{align*}
        \mu(\tau_t) - \mu(\tau^*_a) &= \textstyle \int_{\tau^*_a}^{\tau_t} \mu'(u) du = \int_{\tau^*_a}^{4} \mu'(u) du + \int_{4}^{\tau_t} \mu'(u) du \\
        &\geq \textstyle \int_{4}^{\tau_t} \left( (h+b)\left(\frac{1}{2} + \frac{1}{\sqrt{T}}\right) - b\right) du = \left( (h+b)\left(\frac{1}{2} + \frac{1}{\sqrt{T}}\right) - b\right) \left( \tau_t - 4\right) \\
        &= \textstyle \frac{h+b}{\sqrt{T}} \left( \tau_t - 4\right).
    \end{align*}
    Consequently, if $\tau_t > 200$, then inventory is overstocked by more than $196$ units, and the cost per time step is at least $196 (h+b)/\sqrt{T}$.
    Analogously, under the cumulative demand distribution $F^b$, the equation $F^b(\tau^*_b) = \frac{1}{2}$ has the solution $\tau^*_b \in (400, 402)$. For $\tau \in [4,400)$, the derivative of the cost function is ${\mu'(\tau) = (h+b)\left(\frac{1}{2} - \frac{1}{\sqrt{T}}\right) - b < 0}$ and 
    \vspace{-4pt}
    \begin{align*}
        \mu(\tau_t) - \mu(\tau^*_b) &=\textstyle  -\int_{\tau_t}^{\tau^*_b} \mu'(u) du = -\int_{\tau_t}^{400} \mu'(u) du - \int_{400}^{\tau^*_b} \mu'(u) du \\
        &\geq \textstyle -\int_{\tau_t}^{400} \left( (h+b)\left(\frac{1}{2} - \frac{1}{\sqrt{T}}\right) - b\right) du = -\left( (h+b)\left(\frac{1}{2} - \frac{1}{\sqrt{T}}\right) - b\right) \left( 400 - \tau_t\right) \\
        &= \textstyle \frac{h+b}{\sqrt{T}} \left( 400 - \tau_t\right).
    \end{align*}
    Therefore, whenever $\tau_t \leq 200$, the inventory is understocked by more than $200$ units, and the instantaneous cost is at least $200 (h+b)/\sqrt{T}$. 

    Taken together, the (static) expected regret is lower bounded as
    \vspace{-5pt}
    \begin{align} \label{eq:regret_lower_bound_Pa_Pb}
        \Exp{R(T)} &\geq \textstyle 196 \frac{h+b}{\sqrt{T}} \max \big\{ \sum_{t=1}^T \Pr_{F^a}(\tau_t > 200), \sum_{t=1}^T \Pr_{F^b}(\tau_t \leq 200) \big\} \notag \\
        &\geq \textstyle 196 \frac{h+b}{2\sqrt{T}} \sum_{t=1}^T \max \big\{ \Pr_{F^a}(\tau_t > 200), \Pr_{F^b}(\tau_t \leq 200) \big\}.
    \end{align} 
    Using Tsybakov's Theorem \citep[Theorem~2.2]{tsybakov2009nonparametric}, it holds that
    \[
        \max \Big\{ \Pr_{F^a}(\tau_t > 200), \Pr_{F^b}(\tau_t \leq 200) \Big\} \geq \frac{1}{4} \exp\left(-\text{KL}_{t-1}(\Pr_{F^a},\Pr_{F^b}) \right),
    \]
    where $\text{KL}_t(\Pr_{F^a},\Pr_{F^b})=\E_{F^a}\left[ \log \left( \frac{\Pr_{F^a}(D_1, \dots , D_t)}{\Pr_{F^b}(D_1, \dots , D_t)}\right)\right]$ is the Kullback-Leibler divergence, and we have that $\text{KL}_t(\Pr_{F^a},\Pr_{F^b}) \leq 7t/T$ \citep{zhang2020closing}.
    It follows that
    \[
        \max \Big\{ \Pr_{F^a}(\tau_t > 200), \Pr_{F^b}(\tau_t \leq 200) \Big\} \geq \frac{1}{4} \exp\left(-7(t-1)/T\right) \geq \frac{1}{4} \exp(-7),
    \]
    and finally with \cref{eq:regret_lower_bound_Pa_Pb} that
        $\Exp{R(T)} \geq 196 \frac{h+b}{2\sqrt{T}} \sum_{t=1}^{T}\frac{1}{4} \exp(-7) = \frac{49}{2} (h+b) \exp(-7) \sqrt{T} = \Omega (\sqrt{T})$.
\end{proof}

\LowerBound* 
\begin{proof}
    The proof proceeds by constructing an instance such that any algorithm must incur ${\mathcal{O}(\sqrt{ST})}$ regret.
    We assume, without loss of generality, that the algorithm is given exact knowledge of change points. Moreover, we allow any inventory held at a change point to be cleared instantaneously at zero cost.  These assumptions strictly simplify the problem, and therefore any lower bound established applies to the original problem, in which change points are unknown and inventory can only be depleted through stochastic demand.

    Let the change points be placed uniformly over the time horizon~$T$, resulting in $S$ stationary segments each of length approximately~${T/S}$.  Within each segment, demand is i.i.d.\ and drawn from a continuous distribution. By \cref{lemma:lower_bound_stationary}, the worst-case expected regret incurred over a single stationary segment of length~${T/S}$ is lower bounded by $\Omega(\sqrt{T/S})$.  Summing over the $S$ segments yields $\Exp{R(T)} = \sum_{s=1}^S \Omega(\sqrt{T/S}) = \Omega(\sqrt{ST})$.
\end{proof}
\section{Regret Analysis} \label[appendix]{app:regret_analysis}

In this section we provide the proofs of the regret upper bounds for our algorithms. It is organized as follows.
\begin{itemize}
    \item \cref{app:common_regret_analysis}: Additional notation, general regret decomposition, and auxiliary results
    \item \cref{app:regret_backlogging}: Proof of \cref{thm:regret_backlogging} for \NSBLIC (\cref{algo:backlog})
    \item \cref{app:regret_LS_L0}: Proof of \cref{thm:regret_LS_L_zero} for \NSLSIC (\cref{algo:lost_sales})
    \item \cref{app:regret_LS_Lpositive}: Proof of \cref{thm:regret_LS_L_positive} for \NSLSICL (\cref{algo:lost_sales_lead_time})
\end{itemize}

\subsection{Notation and Auxiliary Results}
\label[appendix]{app:common_regret_analysis}

We start off by providing a unified regret analysis for all of the algorithms that we develop. The regret naturally decomposes into two components, the discretization error, which accumulates since the algorithms only employ base-stock levels from the finite grid $\A_\gamma$, as well as the estimation error, which arises because the algorithm selects base-stock levels based on finite-sample estimates rather than the true expected costs.

\begin{lemma}[General Regret Decomposition] \label{lemma:regret_decomposition}
    For \NSBLIC, \NSLSIC, and \NSLSICL (\cref{algo:backlog,algo:lost_sales,algo:lost_sales_lead_time}), the dynamic regret can be decomposed as $R(T) = R_{\gamma}(T) + R_{\tau}(T)$, where 
    \begin{align*}
        R_{\gamma}(T) \coloneq \sum_{t \in [T]} \left( \mu_t(\taugammaopt) - \mu_t(\tau_t^*) \right) \ \text{ and } \
        R_{\tau}(T) \coloneq \sum_{ t \in [T]} \left( \mu_t(\tau_t) - \mu_t(\taugammaopt) \right).
    \end{align*}
    Moreover, the regret due to discretization satisfies 
    \(
        R_{\gamma}(T) \leq T \gamma \max\{h,b\}. 
    \)
\end{lemma}
\begin{proof}
One readily verifies, following the definitions of each component, that the terms $R_{\gamma}(T)$ and $\Rb(T)$ together constitute the total regret ${R(T) = \sum_{t \in [T]} \left( \mu_t(\tau_t) - \mu_t(\tau_t^*) \right)}$.
By the Lipschitz continuity of the cost function (\Cref{lemma:lipschitz}), the discretization error accumulated over the time horizon is upper bounded as ${R_{\gamma}(T) = \sum_{t \in [T]} ( \mu_t(\taugammaopt) - \mu_t(\tau_t^*)) \leq T \gamma \max\{ h, b \}}$.
\end{proof}

To analyze the estimation error, we partition the horizon into intervals defined by true and detected change points. Within each interval, the demand distribution is fixed, hence our concentration inequality (\cref{eq:concentration_bound}) can be employed since costs are stationary.
Moreover, within each interval the algorithm operates without restarts. This allows us to bound the suboptimality gap of policy $\tau_v^k$ via the selection criteria (\cref{eq:elimination_condition_bl,eq:elimination_condition_ls,eq:active_set_elimination_condition_lsl}) and use the fact that there has not been a change point detected, i.e., \cref{eq:change_condition_bl,eq:change_condition_ls,eq:change_condition_ls_bad,eq:change_condition_lsl} are not satisfied. Towards this, we define some additional notation for the rest of the analysis as follows.

\begin{figure}[!t]
    \centering
    \begin{tikzpicture}[
rewnode/.style={}, 
obsnode/.style={rectangle, draw=black!20, fill=blue!5, thick, minimum size=5mm},
actionnode/.style={rectangle, draw=black!20, fill=black!5, thick, minimum size=5mm},
smallnode/.style={text width=0.3cm},
mediumnode/.style={text width=0.7cm},
widenode/.style={text width=2.5cm},
]

    \node[smallnode] at (0, 0) (t_v)    {$t_v$};
    \node[mediumnode]          (t_v1)   [right=65mm of t_v] {$t_{v+1}$};
    
    \node[smallnode]           (tau_v0)  [above=1mm of t_v] {$\beta_{v,1}$};
    \node[smallnode]           (tau_v)   [above right=1mm and 37mm of t_v] {$\beta_{v,2}$};
    \node[mediumnode]          (tau_v1)  [above right=1mm and 48mm of t_v] {$\beta_{v,3}$};
    \node[widenode]            (tau_v2)  [above right=1mm and -10mm of t_v1] {$\beta_{v,4}=\beta_{v+1,1}$};

    \node[mediumnode]           (alpha_0)  [below right=0.5mm and -5.3mm of t_v] {$\alpha_{v,1}$};
    \node[mediumnode]           (alpha_1)  [below right=0.5mm and 4mm of t_v] {$\alpha_{v,2}$};
    \node[mediumnode]           (alpha_2)  [below right=0.5mm and 22mm of t_v] {$\alpha_{v,3}$};
    \node[mediumnode]           (alpha_3)  [below right=0.5mm and 44mm of t_v] {$\alpha_{v,4}$};
    \node[widenode]             (alpha_4)  [below right=0.5mm and -10mm of t_v1] {$\alpha_{v,5}=\alpha_{v+1,1}$};


    \draw[thick , ->] ($(t_v.north)+(0,0.1)$) -- ($(t_v1.north)+(0.3,0.1)$); 
    \draw (t_v.north) ++ (0,0.2) -- ++ (0,-0.2);
    \draw (t_v1.north) ++ (0,0.2) -- ++ (0,-0.2);
    \draw (tau_v.south) ++ (0,0.1) -- ++ (0,-0.2);
    \draw (tau_v1.south) ++ (0,0.1) -- ++ (0,-0.2);

    \draw[dashed] (alpha_1.north) ++ (0,0.1) -- ++ (0,0.6);
    \draw[dashed] (alpha_2.north) ++ (0,0.1) -- ++ (0,0.6);
    \draw[dashed] (alpha_3.north) ++ (0,0.1) -- ++ (0,0.6);

    \draw [decorate,decoration={brace,amplitude=5pt,raise=4ex}] (0.15,0.4) -- (4.1,0.4) node[midway,yshift=3em]{$I_{v,1}$};
    \draw [decorate,decoration={brace,amplitude=5pt,raise=4ex}] (4.4,0.4) -- (5.4,0.4) node[midway,yshift=3em]{$I_{v,2}$};
    \draw [decorate,decoration={brace,amplitude=5pt,raise=4ex}] (5.7,0.4) -- (7.1,0.4) node[midway,yshift=3em]{$I_{v,3}$};

    

    \fill[pattern=north east lines, pattern color=gray!70] 
    ($(alpha_1.north)+(0.03,0.43)$) rectangle ++(0.3,0.2);

    \fill[pattern=north east lines, pattern color=gray!70] 
    ($(alpha_2.north)+(0.03,0.43)$) rectangle ++(0.4,0.2);

    \fill[pattern=north east lines, pattern color=gray!70] 
    ($(alpha_3.north)+(0.03,0.43)$) rectangle ++(0.25,0.2);
    \node[text=black] at (-1.4,0.08) {\textit{episode}};
    \node[text=black] at (-1.4,-0.5) {\textit{epochs}};
    
\end{tikzpicture}
    \caption{An episode $[t_v , t_{v+1} - 1]$ containing $S_v=3$ stationary segments and consisting of four epochs. The change at $\beta_{v,3}$ follows an undetected change at $\beta_{v,2}$ and causes the beginning of the next episode at $t_{v+1}$. The shaded areas indicate time steps ${t \in [\alpha_{v,k}, \bar{\alpha}_{v,k} - 1]}$ when the sum of inventory on-hand and in transit is reduced from $\tau_v^{k-1}$ to $\tau_v^k$ by suspending orders.}
\label{fig:episode}
\end{figure}
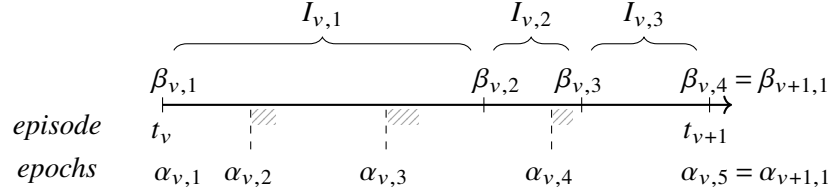

We use $\hat{S}$ to denote the total number of episodes, which is equivalent to one plus the total number of detected changes. 
Moreover, let the actual number of stationary segments within an episode~$v$ be denoted by $S_v$ and the $i$'th change point within episode~$v$ by $\beta_{v,i}$ with 
    ${t_v = \beta_{v,1} \leq \beta_{v,2} < \dots < \beta_{v, S_v} \leq \beta_{v,S_v + 1} = t_{v+1}}$
and ${t_{\hat{S}+1} \coloneq T}$.
We use ${I_{v,i} \coloneq [\beta_{v,i}, \beta_{v, i+1} - 1]}$ with $v=1,\dots , \hat{S}$ and $i=1, \dots , S_v$ to denote the stationary time intervals without detected or actual changes. 
See \cref{fig:episode} for an illustration of the variables introduced by means of the example of one episode.

Next we establish that for each algorithm the construction of the estimates for the mean cost function $\hat{\mu}(\tau,s,t)$ are accurate for $\mu_s(\tau)$ with high probability.
\begin{lemma} \label{lemma:probability_good_event}
    For any $\delta \in (0,1)$ let $\delta' = \delta \gamma /T^2U$. Consider estimates $\hat{\mu}(\tau,s,t)$ constructed from the cost observations collected while executing the sequence of policies $(\tau_{t'})_{t'\in [s,t]}$ via \cref{algo:backlog,algo:lost_sales,algo:lost_sales_lead_time} together with confidence terms 
    $b_{s,t} = H_{\bl} \sqrt{\frac{2\log(4(L+1)/\delta')}{t-s}}$ or $b_{s,t} = H_{\ls} \sqrt{\frac{2\log(2/\delta')}{t-s}}$, respectively.
    Define the events
    \vspace{-5pt}
    \begin{align*}
        \calE_{\bl} & = \left\{\abs*{\hat{\mu}(\tau,s,t) - \mu_s(\tau)} \leq b_{s,t} \ \forall \tau \in \mathcal{A}_{\gamma} \text{ and time steps } s, t \text{ with } t -s \geq L, \text{ and } F_{t'} = F_s \ \forall t' \in [s,t-1] \right\}, \\[4pt]
        \calE_{\ls} &=
        \left\{
        \abs*{\hat{\mu}(\tau,s,t) - \mu_s(\tau)} \le b_{s,t}
        \;\middle|\;
        \begin{array}{l}
        \forall \tau \leq \tau_v^k \text{ and time steps }
        s,t \text{ with } F_{t'} = F_s\ \forall t' \in [s,t-1], \\[2pt]
        \forall \tau \in \mathcal{A}_{\gamma}
        \text{ and time steps } s,t \text{ with } F_{t'} = F_s \text{ and } \tau_{t'} = U\ \forall t' \in [s,t-1]
        \end{array}
        \right\}, \\[6pt]
        \calE_{\lsl} &= \left\{\abs*{\hat{\mu}(\tau,s,t) - \mu_s(\tau)} \leq b_{s,t} \;\middle|\;
        \begin{array}{l}
        \forall \tau \leq \tau_v^k \text{ and time steps } s, t \text{ with } [s,t] \subseteq [\bar{\alpha}_{v,k},\alpha_{v,k+1}-1] \\[2pt] \text{ for some episode } v \text{ and epoch } k \text{ and } F_{t'} = F_s \ \forall t' \in [s,t-1] 
        \end{array}
        \right\}.   
    \end{align*}
    For any $\calE \in  \{\calE_{\bl}, \calE_{\ls}, \calE_{\lsl}\}$ we have that $\Pr(\calE) \geq 1 - \delta$.
\end{lemma}
In $\calE_{\bl}$ we require stationary time intervals for which $t-s \geq L$ in order to satisfy the requirements of \cref{lemma:concentration}.  We are able to derive estimates for all possible base-stock levels ${\tau \in \A_\gamma}$ using the feedback structure (\cref{lemma:feedback_structure}). In $\calE_{\ls}$ this is also the case when playing the base-stock policy $\tau_{t'} = U$ due to the left-sided feedback.  Moreover, by selecting the {\em largest} base-stock level via the selection rule (\cref{eq:selection_rule}) the concentration holds for all base-stock levels $\tau \leq \tau_v^k$.
Lastly, in $\calE_{\lsl}$ besides stationarity we require the inventory position to not exceed the implemented base-stock level $\tau_v^k$ to allow for left-sided feedback as per \cref{lemma:feedback_structure}. With the construction of $\bar{\alpha}_{v,k}$ as the first time step in epoch~$k$ in which this condition is met, we obtain concentration for all base-stock levels ${\tau \leq \tau_v^k}$ over stationary subintervals of ${[\bar{\alpha}_{v,k}, \alpha_{v,k+1}-1]}$.

\begin{proof}
    For each event, we provide a separate proof depending on the algorithm employed.
    
    \medskip
    \noindent {\bf \NSBLIC Algorithm.}\quad
    This algorithm implements base-stock policy $\tau_v^k$ in epoch~$k$ of episode~$v$. By \cref{lemma:feedback_structure}, the observations collected under policy $\tau_v^k$ allow us to determine the costs ${(C_{t'}(\tau))_{t' \in [s,t]}}$ and hence the estimates $\hat{\mu}(\tau, s,t)$ for any ${\tau \in \mathcal{A}_{\gamma}}$ and any ${[s,t] \subseteq [T]}$. 
    Notably, this also holds for the time steps right after the base-stock level is decreased at the beginning of the next epoch $k+1$. In these steps, the algorithm suspends orders over a sequence of time steps in order to bring the inventory position from $\tau_v^{k}$ to $\tau_v^{k+1}$. Despite the fact that no orders are being placed to reduce inventory, demands and therefore the counterfactual costs remain to be observed.
    
    Using \cref{lemma:concentration}, the concentration inequalities ${\abs*{\hat{\mu}(\tau,s,t) - \mu_s(\tau)} \leq b_{s,t}}$ hold with probability $1-\delta'$ for any stationary interval $[s,t]$ with $t-s \geq L$. The condition on the initial inventory position stated in \cref{lemma:concentration} is always satisfied because the costs used to construct the estimators $\hat{\mu}(\tau,s,t)$ are computed via \cref{eq:cost_function} from the counterfactual states $(I^{\tau}_{t'}, Q^{\tau}_{t'-L}, \dots , Q^{\tau}_{t'-1})$, which always satisfy $I^{\tau}_s + \sum_{i=0}^L Q^{\tau}_{s-i} \leq \tau$. This follows from their initialization and the fact that the base-stock level according to which each of the states is updated is fixed.
    Thus we have via a union bound over the $T$ possible values of $s$ and $t$ and the $U / \gamma$ different values of $\tau$ in $\A_\gamma$ we have that $\Pr(\calE_{\bl}) \geq 1 - \delta$ with $\delta = \delta' T^2 U/\gamma$.
    
    \medskip
    \noindent {\bf \NSLSIC Algorithm.}\quad
    This algorithm alternates between playing base-stock policy $\tau_v^k$ with sampling blocks of policy $U$.
    We start off by considering an interval $[s,t]$ for which the algorithm plays base-stock policy $\tau_t = U$.  Here we obtain estimates $\hat{\mu}(\tau, s, t)$ over the stationary interval $[s,t]$ by using the left-sided feedback from \cref{lemma:feedback_structure} so long as $I_{s}^\tau \leq \tau$ for all $\tau \in \mathcal{A}_{\gamma}$.  
    However, this condition is satisfied at all time steps~$t'$ when playing base-stock policy $U$.  Indeed, initially at $t' = 0$ the counterfactual on-hand inventory $I_{t'}^\tau = 0$ by line~\ref{algo_line:initialize_counterfactual_inventory_ls} of \cref{algo:lost_sales}.  Moreover, the condition that $I_{t'}^\tau \leq \tau$ then holds by definition of the inventory dynamics in \cref{eq:inventory_balance}. Then applying \cref{lemma:concentration} we obtain that the concentration inequality $|\hat{\mu}(\tau, s,t) - \mu_s(\tau)| \leq b_{s,t}$ holds with probability $1 - \delta'$.

    Next we consider the time steps where the algorithm implements the base-stock policy $\tau_v^k$.  Again we obtain estimates $\hat{\mu}(\tau, s, t)$ for all $\tau \leq \tau_v^k$ by using the left-sided feedback so long as $I_{s}^\tau \leq \tau$.  This condition is satisfied by the same reasoning as above. 
    Again, applying \cref{lemma:concentration} we obtain that the concentration inequality $|\hat{\mu}(\tau, s,t) - \mu_s(\tau)| \leq b_{s,t}$ holds with probability $1 - \delta'$.
    Then we have via a union bound over the $T$ possible values of $s$ and $t$ and the $U/\gamma$ different values of $\tau$ in $\A_\gamma$ that $\Pr(\calE_{\ls}) \geq 1 - \delta$ with $\delta = \delta' T^2 U/\gamma$.    

    \medskip
    \noindent {\bf \NSLSICL Algorithm.}\quad
    In this algorithm, the only base-stock level implemented is $\tau_v^k$ and that is decreased over the course of each episode~$v$. By \cref{lemma:feedback_structure}, playing $\tau_v^k$ allows us to observe costs for all ${\tau \in \tau_v^k}$ over intervals $[s,t]$, in which (i) the sequence of base-stock policies $\tau_v^k$ is non-decreasing, and (ii) the condition on factual and counterfactual inventory vectors stated in \cref{lemma:feedback_structure}.\ref{item:feedback_structure_LS_L_positive} is satisfied. Both conditions hold for all intervals $[s,t] \subseteq [\bar{\alpha}_{v,k}, \alpha_{v,k+1} - 1]$. 
    In step $\bar{\alpha}_{v,k}$, the first time step after which the algorithm has suspended orders until the inventory position is reduced from a total of $\tau_v^{k-1}$ to at most the new target base-stock level $\tau_v^{k}$, the counterfactual inventory states are reset according to \cref{eq:counterfactual_inventory_vector_reset} in line~\ref{algo_line:reset} of \cref{algo:lost_sales_lead_time}. From $\bar{\alpha}_{v,k}$ on, until the end of the epoch, $\tau_v^k$ is constant and the condition in \cref{lemma:feedback_structure}.\ref{item:feedback_structure_LS_L_positive} on the starting inventory states is satisfied for all subsequent time steps in the same epoch (see proof of \cref{lemma:feedback_structure} in \cref{app:proof_counterfactual_policy_evaluation}).
    
    Therefore, in epoch~$k$ of episode~$v$ we can evaluate ${\hat{\mu}(\tau,s,t)}$ for all ${\tau \leq \tau_v^k}$ and all ${[s,t] \subseteq [\bar{\alpha}_{v,k}, \alpha_{v,k+1} - 1]}$.
    By \cref{lemma:concentration}, the concentration inequality ${\abs*{\hat{\mu}(\tau, s,t) - \mu_s(\tau)} \leq b_{s,t}}$ for stationary subintervals of ${[s,t] \subseteq [\bar{\alpha}_{v,k}, \alpha_{v,k+1} - 1]}$ holds with probability ${1 - \delta'}$. With a union bound over the possible values of $s,t$ and $\tau$, we obtain with ${\delta = \delta' T^2 U/\gamma}$ that ${\Pr(\calE_{\lsl}) \geq 1 - \delta}$.
\end{proof}

For the rest of our analysis we will condition on the events $\calE_{\bl}, \calE_{\ls}$ or $\calE_{\lsl}$ holding. For ease of notation we will omit the specific algorithm and just denote the good-event as $\calE$, where the specific algorithm will be clear from the context.

We next show in \cref{lemma:false_alarms} that the algorithms only trigger restarts if there has been a change since the last restart. This allows us to bound the number of episodes with respect to $S$ (the actual total number of stationary segments). 
The proof of \cref{lemma:false_alarms} for \NSLSIC is adapted from \citep[Lemma 7]{auer2019adaptively}. 
\begin{lemma}[No False Positive Change Detections] \label{lemma:false_alarms}
    Under event $\mathcal{E}$, the number of episodes $\hat{S}$ of \NSBLIC, \NSLSIC and \NSLSICL (\cref{algo:backlog,algo:lost_sales,algo:lost_sales_lead_time}) is upper bounded by the number of stationary segments $S$. Furthermore, the total number of intervals satisfies ${\sum_v \sum_i I_{v,i} \leq 2S-1}$.
\end{lemma}
\begin{proof}
    We first show by contradiction that in each episode but the last there is at least one change. Assume that the algorithm has detected a change in time step $t > t_v$ of episode $v$, but there has not been a change in the demand distribution in time interval $[t_v, t]$.

\noindent {\bf \NSBLIC Algorithm.}\quad
    Under this algorithm, there must be a base-stock level ${\tau \in \A_\gamma}$ and time steps $s_1, s_2, s$ with ${t_v \leq s_1 < s_2 < t}$ and ${t_v \leq s < t}$, as well as ${t-s\geq L}$ and ${s_2 - s_1 \geq L}$ such that the {\em change point condition} (\cref{eq:change_condition_bl}) is satisfied, that is, 
        \begin{equation} \label{eq:backlogging_change_detection_contradiction}
            \abs*{\hat{\mu}(\tau, s_1, s_2) - \hat{\mu}(\tau, s, t)} > b_{s_1, s_2} + b_{s,t}.
        \end{equation}
    Under event $\mathcal{E}$, the concentration inequalities hold, that is,
    \begin{align*}
        \abs*{\hat{\mu}(\tau, s_1, s_2) - \mu_{s_1}(\tau)} \leq b_{s_1,s_2} \quad \text{and}\quad
        \abs*{\hat{\mu}(\tau, s, t) - \mu_s(\tau)} \leq b_{s,t}.
    \end{align*}
    However, since there was no change, it holds that $\mu_{s_1}(\tau) = \mu_s(\tau)$ and by the triangle inequality, we have
        $\abs*{\hat{\mu}(\tau, s_1, s_2) - \hat{\mu}(\tau, s, t)} \leq b_{s_1, s_2} + b_{s,t}$,
    which contradicts with \cref{eq:backlogging_change_detection_contradiction}.

    \medskip
    \noindent {\bf \NSLSIC Algorithm.}\quad
    There are two possible {\em change point conditions} that could trigger a restart in \NSLSIC. In the first case, there is a base-stock level $\tau \leq \tau_v^k$, which meets the first {\em change point condition} (\cref{eq:change_condition_ls}), i.e., there exist time steps $s_1, s_2, s$ such that ${t_v \leq s_1 < s_2 < t}$, ${t_v \leq s < t}$ and ${\abs*{\hat{\mu}(\tau, s_1, s_2) - \hat{\mu}(\tau, s, t)} > b_{s_1, s_2} + b_{s,t}}$.
    Following the same argument as above for \NSBLIC, under the good-event $\mathcal{E}$ this contradicts with the concentration inequalities (\cref{lemma:concentration}). 
    
    In the other case, the termination of episode~$v$ was triggered by the detection of a change in the mean cost of a base-stock level ${\tau > \tau_v^k}$ that meets the second {\em change point condition} (\cref{eq:change_condition_ls_bad}).
    Then there must exist a time step $s < t$ such that ${\tau_{t'} = U}$ for all ${t' \in [s,t]}$ and $\abs*{\hat{\mu}(\tau, s, t) - \tilde{\mu}_v(\tau)} > \tilde{\Delta}_v(\tau)/4 + b_{s,t}$.
    Here we note that $\tilde{\mu}_v(\tau) = \hat{\mu}(\tau, s', t')$ as well as ${\tilde{\Delta}_v(\tau) > \Cb b_{s',t'}}$ where $[s',t'] \subseteq [t_v, t]$ is the time interval in episode~$v$ based on which $\tau$ was eliminated in an earlier epoch by the {\em active set elimination condition} of the algorithm (\cref{eq:elimination_condition_ls}).
    Therefore, we have that
    \vspace{-7pt}
    \[
        \abs*{\hat{\mu}(\tau, s, t) - \tilde{\mu}_v(\tau)} = \abs*{\hat{\mu}(\tau, s, t) - \hat{\mu}(\tau, s', t')} \geq b_{s,t} + \frac{3}{2} b_{s', t'} > b_{s,t} + b_{s', t'},
        \vspace{-7pt}
    \]
    which again contradicts with the concentration guarantees (\cref{lemma:concentration}) since we assume there is no change in the current episode, so $\mu_{s_1}(\tau) = \mu_{t}(\tau)$ for all $s_1 \in [t_v,t]$.    

\medskip
\noindent {\bf \NSLSICL Algorithm.}\quad
    In this algorithm, the change detection must have been caused by a base-stock policy $\tau \leq \tau_v^k$ and time steps $s_1, s_2, s$ with ${\bar{\alpha}_{v,k} \leq s < t}$ as well as ${t_v \leq s_1 < s}$ and $[s_1,s_2]$ being a subinterval of any interval ${[\bar{\alpha}_{v,k'}, \alpha_{v,k'}-1]}$ of a previous epoch~${k'\leq k}$ in episode~$v$ and $\abs*{\hat{\mu}(\tau, s_1, s_2) - \hat{\mu}(\tau, s, t)} > b_{s_1, s_2} + b_{s,t}$, leading to a contradiction, conditioned on $\calE$, analogous to the ones above.

\medskip
    To conclude the proof, recall that there are $S-1$ changes in total, which directly implies that ${\hat{S} \leq S}$. Moreover, it follows that any episode~$v$ consists of ${S_v \geq 2}$ intervals~$I_{v,i}$, and the total number of intervals~$I_{v,i}$ is upper bounded by $\sum_{v=1}^{\hat{S}} S_v = {\hat{S} + S -1 \leq 2S - 1}$.
\end{proof}

\subsection{\texorpdfstring{Proof of \cref{thm:regret_backlogging}}{Proof of Theorem 1}} \label[appendix]{app:regret_backlogging}
\BackLogRegret*

\begin{proof}
    We show the result by conditioning on the good-event $\calE$, which holds with probability at least $1 - \delta$ via \cref{lemma:probability_good_event}.
    Applying \cref{lemma:regret_decomposition}, the analysis reduces to bounding $\Rb(T)$, the cumulative difference between $\mu(\tau_v^k)$ and $\mu(\taugammaopt)$ over all time steps in $[T]$.
    To derive this bound, we use the fact that by the construction of the active set in \NSBLIC, the difference between the cost estimates of $\tau_v^k$ to that of $\taugammaopt$ can be bounded by the confidence terms.

    Let $t \in I_{v,i}$ for some interval ${I_{v,i}=[\beta_{v,i}, \beta_{v, i+1}-1]}$ in episode $v$. Since our concentration inequalities require that the number of observations is at least $L$, we analyze the first $L$ time steps of any episode separately.
    Indeed, consider the first interval in episode $v$ with $i=1$ and hence $\beta_{v,i}=t_v$. By Lipschitzness of $\mu_t(\cdot)$ (\cref{lemma:lipschitz}), in each time step $t \in [t_v, t_v + L - 1]$ a regret of at most $U \max \{ h, b \}$ is incurred. Since there are at most $S$ episodes via \cref{lemma:false_alarms}, this implies
    \vspace{-5pt}
    \[
        \textstyle \sum_{v=1}^{\hat{S}} \sum_{t \in [t_v, t_v + L - 1]} \left( \mu_t(\tau_t) - \mu_t(\taugammaopt) \right) \leq SLU\max\{h,b\}.
    \]
    
    To analyze the regret in all other time steps $t$, let $k$ be the epoch of time step $t$. We partition the regret per time step as
    \vspace{-5pt}
    \[
        \mu_t(\tau_v^k) - \mu_t(\taugammaopt) = \left(\mu_t(\tau_v^k) - \hat{\mu}(\tau^k_v, \beta_{v,i}, t)\right) + \left(\hat{\mu}(\tau_v^k, \beta_{v,i}, t) - \hat{\mu}(\taugammaopt, \beta_{v,i}, t)\right) + \left(\hat{\mu}(\taugammaopt, \beta_{v,i}, t) - \mu_t(\taugammaopt)\right).
    \]
    By \Cref{lemma:concentration,lemma:probability_good_event} and since each interval $I_{v,i}$ exhibits stationarity, conditioned on event $\mathcal{E}$ we can bound the estimation errors as
        $\abs*{\hat{\mu}(\tau^k_v, \beta_{v,i}, t) - \mu_t(\tau_v^k)} \leq b_{\beta_{v,i},t} \text{ and }
        \abs*{\hat{\mu}(\taugammaopt, \beta_{v,i}, t) - \mu_t(\taugammaopt)} \leq b_{\beta_{v,i},t}$.
    
    Base-stock policy $\tau_v^k$ is updated in time steps ${t \geq t_v + L}$ according to the active set elimination rule in \cref{eq:elimination_condition_bl}. Therefore, we have that $\hat{\mu}(\tau_v^k, s, t) - \underset{\tau \in \mathcal{A}_{\gamma}}{\min} \hat{\mu}(\tau, s, t) \leq 4 b_{s, t}$ for any $s \in [t_v,t-1]$, and in particular for $s=\beta_{v,i}$. Together with ${\underset{\tau \in \mathcal{A}_{\gamma}}{\min} \hat{\mu}(\tau, \beta_{v,i}, t) \leq \hat{\mu}(\taugammaopt, \beta_{v,i}, t)}$ it follows that ${\hat{\mu}(\tau_v^k, \beta_{v,i}, t) - \hat{\mu}(\taugammaopt, \beta_{v,i}, t) \leq 4 b_{\beta_{v,i}, t}}$ and hence that $\mu_t(\tau_v^k) - \mu_t(\taugammaopt) \leq 6 b_{\beta_{v,i},t}$.
        
    Summing over the time steps in $I_{v,i}$, by Jensen's inequality we get the following upper bound for the regret per interval. By \cref{lemma:probability_good_event} we have with probability at least $1-\delta$ that
    \vspace{-5pt}
    \begin{align*}
        \sum_{t \in I_{v,i}} \left( \mu_{t}(\tau_t) - \mu_{t}(\taugammaopt) \right) &\leq 6 \Hbl \sqrt{2 \log \left(4(L+1)T^2U/\delta\gamma\right)} \sum_{t=1}^{\abs*{I_{v,i}}} \sqrt{\frac{1}{t}} \notag \\
        & \leq 6 \Hbl \sqrt{2 \log  \left(4(L+1)T^2U/\delta\gamma\right)} \left(2\sqrt{\abs*{I_{v,i}}}-1\right) \notag \\
        & \leq 12 \Hbl \sqrt{2 \log \left(4(L+1)T^2U/\delta\gamma\right) \abs*{I_{v,i}}}.
    \end{align*}
    Since there are at most $2S-1$ intervals $I_{v,i}$ of length at most $T$ it holds that
    \vspace{-5pt}
    \begin{equation*}
        \Rb(T) = \sum_{v=1}^{\hat{S}} \sum_{i=1}^{S_v} \sum_{t \in I_{v,i}} \left( \mu_t(\tau_t) - \mu_t(\taugammaopt) \right) \leq 24 \Hbl \sqrt{ST \log  \left(4(L+1)T^2U/\delta\gamma\right)} + SLU\max\{h,b\}.
    \end{equation*}

    Finally, from the preceding analysis together with \cref{lemma:regret_decomposition}, we conclude that under event $\mathcal{E}$ the dynamic regret of \NSBLIC (\cref{algo:backlog}) is given by
    \[
        R(T) = R_{\gamma}(T) + \Rb(T) \leq 24 \Hbl \sqrt{ST \log  \left(4(L+1)T^2U/\delta\gamma\right)} + \left(T\gamma + SLU\right) \max\{h,b\}.
    \]
    Setting $\gamma = \Theta\left(T^{-1/2}\right)$ and with $\Hbl = 2\sqrt{2}\sigma \sqrt{(L+1)\left(Lh^2 + (h+b)^2(4L+5)\right)}$ this yields
    \[
        R(T) \leq 24 \Hbl \sqrt{ST \log  \left(4(L+1)T^{5/2}U/\delta\right)} + \left(\sqrt{T} + SLU\right) \max\{h,b\} = \tilde{\mathcal{O}}\left(\sigma (L+1) \sqrt{ST} \right).
    \]
\end{proof}
\subsection{\texorpdfstring{Proof of \Cref{thm:regret_LS_L_zero}}{Regret Analysis Lost-Sales with L Zero}} \label[appendix]{app:regret_LS_L0}

\LSzeroLRegret*

\cref{thm:regret_LS_L_zero} states an upper bound on the {\em expected} dynamic regret. This is because in our algorithm, the regret depends on the random number of times base-stock level $U$ is sampled. The sampling obligations for $U$ are not determined deterministically by the algorithm’s history, but triggered with a certain probability. As a consequence, the frequency of such obligations grows only in expectation, and cannot be upper bounded with high probability.
Thus, the final performance guarantee is stated in expectation, however, we still analyze regret conditional on the event $\calE$.

Before presenting the proof we start by providing a regret decomposition complementary to that in \cref{lemma:regret_decomposition}. Specifically, we break down $\Rb(T)$ into components defined by ${\tau_t^g \coloneq \argmin_{\tau \in \A_{\gamma}, \tau \leq \tau_v^k} \mu_t(\tau)}$, which denotes the best base-stock level in $\A_{\gamma}$ that is no greater than $\tau_v^k$ in time step~$t$ within epoch~$k$ of episode~$v$. In the subsequent analyses of the individual components we exploit the fact that either ${\taugammaopt \leq \tau_v^k}$ and hence the costs of ${\tau_t^g = \taugammaopt}$ are always observed, or ${\taugammaopt > \tau_v^k}$ and by convexity $\tau_t^g$ coincides with $\tau_v^k$.

\begin{lemma} \label{lemma:regret_decomposition_LS_L0}
    For \NSLSIC (\cref{algo:lost_sales}), the dynamic regret can be decomposed as
    \vspace{-5pt}
    \begin{align*}
        R(T) &= R_{\gamma}(T) + R_1(T) + R_2(T) + R_3(T), \text{ where} \\[10pt]
        R_1(T) &= \sum_{\mathclap{\substack{t \in [T] \\ \tau_t = \tau_v^k, \taugammaopt < \tau^k_v}}} \ \left( \mu_t(\tau_t) - \mu_t\left(\tau^g_t\right) \right), \quad 
        R_2(T) = \sum_{\substack{t \in [T] \\ \tau_t = U}} \left( \mu_t(\tau_t) - \mu_t\left(\tau^g_t\right) \right), \quad
        R_3(T) = \sum_{\substack{t \in [T] \\ \tau_{v}^{k} < \taugammaopt}} \left( \mu_t\left(\tau^g_t\right) - \mu_t(\taugammaopt) \right)
    \end{align*}
    and $\tau_t^g \coloneq \underset{\tau \in \A_{\gamma}, \tau \leq \tau_v^k}{\argmin} \mu_t(\tau)$ for any time step~$t$ in an epoch~$k$ of episode~$v$.
\end{lemma}

\begin{proof}
    By applying \cref{lemma:regret_decomposition}, the regret can be decomposed as $R(T) = R_{\gamma}(T) + \Rb(T)$.
    To analyze $R_{\tau}(T)$, recall that at any given time the algorithm selects either base-stock level $U$ or level $\tau_v^k$. Therefore, we write        
    \begin{align*}
        \Rb(T) &= \sum_{t \in [T]} \left(\mu_t(\tau_t) - \mu_t(\taugammaopt)\right) \\
        &= \sum_{t \in [T]} \left( \mu_t(\tau_t) - \mu_t\left(\tau^g_t\right) \right) + \sum_{t \in [T]} \left( \mu_t\left(\tau^g_t\right) - \mu_t(\taugammaopt) \right) \\
        &= \ \underbrace{\sum_{\mathclap{t \in [T]: \tau_t = \tau_v^k}} \ \left( \mu_t(\tau_t) - \mu_t\left(\tau^g_t\right) \right)}_{R_1(T)} + \ \underbrace{\sum_{\mathclap{t \in [T]: \tau_t = U}} \ \left( \mu_t(\tau_t) - \mu_t\left(\tau^g_t\right) \right)}_{R_2(T)} + \underbrace{\sum_{t \in [T]} \left( \mu_t\left(\tau^g_t\right) - \mu_t(\taugammaopt) \right)}_{R_3(T)}.
    \end{align*}
    We next show that $R_1(T)$ and $R_3(T)$ can be expressed as claimed in the statement of the lemma.

    \noindent\textbf{Term $R_1(T)$.}\quad
    Due to convexity (\cref{lemma:convexity}), we have that $\tau_{v}^{k} = \tau_t^g$ if and only if $\tau_v^k \leq \taugammaopt$. Thus, when $\tau_v^k$ is played, regret with respect to $\tau_t^g$ is only incurred when $\taugammaopt < \tau_v^k$.  Hence
    ${R_{1}(T) \coloneq \sum_{\substack{t \in [T] \\ \tau_t = \tau_v^k, \taugammaopt < \tau^k_v}} \left( \mu_t(\tau_t) - \mu_t\left(\tau^g_t\right) \right)}$.

    

    \noindent\textbf{Term $R_3(T)$.}\quad
    Now note that the difference between $\mu_t\left(\tau^g_t\right)$ and $\mu_t(\taugammaopt)$ is positive whenever $\tau_v^k < \taugammaopt$, since otherwise $\taugammaopt = \tau_t^g$. 
    Hence we can rewrite this as ${R_{3}(T) \coloneq \sum_{t \in [T]: \tau_{v}^{k} < \taugammaopt} \left( \mu_t\left(\tau^g_t\right) - \mu_t(\taugammaopt) \right)}$.
\end{proof}

In the following, we provide a further decomposition and upper bounds for the terms $R_1(T), R_2(T)$ and $R_3(T)$ in the regret decomposition (\cref{lemma:regret_decomposition_LS_L0}). The analysis is inspired by that in \citet{auer2019adaptively}.
We analyze the decomposition in terms of $\Exp{\cdot \mid \calE}$, however, we just write this as $\cdot$ for expositional purposes, leaving the expectation unwritten.

\begin{rproofof}{\Cref{thm:regret_LS_L_zero}}
    \noindent\textbf{Bound of $R_1(T)$.}\quad
    The term $R_1(T)$ represents the regret with respect to $\tau_t^g$, the best base-stock level in $\A_{\gamma}$ less than or equal to $\tau_v^k$, in time steps when ${\tau_t = \tau_v^k > \taugammaopt}$.
    Let ${t \in I_{v,i}}$ for some stationary interval ${I_{v,i}=[\beta_{v,i}, \beta_{v, i+1}-1]}$ in episode~$v$ with ${\tau_t = \tau_{v}^{k}}$ for some epoch~$k$.
    Similar to the analysis of $R_1(T)$ for \NSBLIC in \cref{app:regret_backlogging}, we bound this term using the fact that cost observations of ${\tau_t^g = \taugammaopt}$ are collected in each time step since ${\tau_t^g \leq \tau_v^k}$ for all epochs~$k$ and episodes~$v$. 
    
    By the definition of $\tau_v^k$ in the policy selection and elimination rules (\cref{eq:selection_rule,eq:elimination_condition_ls}), we have that $\hat{\mu}(\tau_v^k,\beta_{v,i},t) - \min_{\tau \in \A_{\gamma}, \tau \leq \tau_v^k}\, \hat{\mu}(\tau,\beta_{v,i},t) \leq \Cb b_{\beta_{v,i},t}$. 
    Using the fact that 
    $\min_{\tau \in \A_{\gamma}, \tau \leq \tau_v^k}\, \hat{\mu}(\tau, \beta_{v,i}, t) \leq \hat{\mu}(\tau^g_t, \beta_{v,i}, t)$, this implies
        ${\hat{\mu}(\tau_v^k, \beta_{v,i}, t) - \hat{\mu}(\tau^g_t, \beta_{v,i}, t) \leq \Cb b_{\beta_{v,i}, t}}$.
    Since $[\beta_{v,i}, t]$ is stationary and $\tau_t^g \leq \tau_v^k$, from the good-event $\calE$ (\cref{lemma:probability_good_event}) it follows that
    \begin{equation*}
        \mu_{t}(\tau_v^k) - \mu_{t}(\tau^g_t) \leq \hat{\mu}(\tau_v^k, \beta_{v,i}, t) - \hat{\mu}(\tau_t^g, \beta_{v,i}, t) + 2b_{\beta_{v,i}, t} \leq 8 b_{\beta_{v,i}, t}.
    \end{equation*}
    Summing over the time steps in interval $I_{v,i}$, we obtain
    \begin{align*}
        \sum_{\substack{t \in I_{v,i}: \tau_t = \tau_v^k, \\ \taugammaopt < \tau^k_v}} \left( \mu_{t}(\tau_v^k) - \mu_{t}(\tau^g_t) \right) & \leq 8 \Hls \sqrt{2 \log(2T^2U/\delta\gamma)} \sum_{t=1}^{\abs*{I_{v,i}}} \sqrt{\frac{1}{t}}         \leq 16 \Hls \sqrt{2 \log(2T^2U/\delta\gamma)|I_{v,i}|}.
    \end{align*}
    Since there are at most $2S-1$ intervals $I_{v,i}$ of length at most $T$ (\cref{lemma:false_alarms}), it follows that
    \begin{equation*}
        R_{1}(T) = \sum_{v=1}^{\hat{S}} \sum_{i=1}^{S_v}\sum_{\substack{t \in I_{v,i}: \tau_t = \tau_v^k, \\ \taugammaopt < \tau^k_v}} \left( \mu_t(\tau_t) - \mu_t(\tau^g_t) \right) \leq 32 \Hls \sqrt{ST \log(2T^2U/\delta\gamma)}.
    \end{equation*}

\noindent\textbf{Bound of $R_2(T)$.}\quad
    This term measures the cumulative difference between the cost of base-stock level $U$ and that of level $\tau_t^g$ over time steps with $\tau_t = U$. 
    Note that we only need to consider time steps in which $U$ has already been eliminated from the active set.  Otherwise, we would have $\tau_v^k = U$ and the corresponding regret contribution is already captured by $R_1(T)$.
    
    We further subdivide the regret per step with respect to $\tau^g_t$ based on the following three disjoint sets of time steps $R_2(T) = R_{2.1}(T) + R_{2.2}(T) + R_{2.3}(T)$ with 
    \begin{align*}
    R_{2.1}(T) &\coloneq \sum_{t \in \Tb_{2.1}} \bigl(\mu_t(\tau_t) - \mu_t (\tau_t^g)\bigr), 
    & \Tb_{2.1} &= \{t \in [T] : \tau_t = U,\ \mu_t(U) - \mu_t(\tau^g_t) \leq 4 \tilde{\Delta}_v(U)\}, \\
    R_{2.2}(T) &\coloneq \sum_{t \in \Tb_{2.2}} \bigl(\mu_t(\tau_t) - \mu_t (\tau_t^g)\bigr), 
    & \Tb_{2.2} &= \left\{ \begin{array}{l}
                         t \in [T] : \tau_t = U, \ \mu_t(U) - \mu_t(\tau^g_t) > 4 \tilde{\Delta}_v(U), \\ 
                        \qquad\quad \mu_t(U) - \tilde{\mu}_v(U) > \tfrac{\mu_t(U) - \mu_t(\tau_t^g)}{2}
                        \end{array}\right\}, \\[8pt]
    R_{2.3}(T) &\coloneq \sum_{t \in \Tb_{2.3}} \bigl(\mu_t(\tau_t) - \mu_t (\tau_t^g)\bigr), 
    & \Tb_{2.3} &= \left\{ \begin{array}{l}
                         t \in [T]: \tau_t = U, \ \mu_t(U) - \mu_t(\tau^g_t) > 4 \tilde{\Delta}_v(U), \\
                        \qquad\quad \mu_t(U) - \tilde{\mu}_v(U) \leq \tfrac{\mu_t(U) - \mu_t(\tau_t^g)}{2}
                        \end{array}\right\}.
    \end{align*}
    The first term corresponds to time steps where policy $U$ has a low per-time step regret relative to the frequency of sampling obligations for that base-stock level. The other terms correspond to sets of time steps in which the difference in costs between base-stock level $U$ and level $\tau_t^g$ is large in comparison to the estimated suboptimality gap of $U$. In $\Tb_{2.2}$ this is because the cost for $U$ has increased compared to the time when that base-stock level was evicted, and in $\Tb_{2.3}$ the cost of level $U$ has changed little, hence the expected cost of $\tau_t^g$ must have decreased significantly.

\noindent\textbf{Bound of $R_{2.1}(T)$.}\quad
    To analyze this term, we use the fact that the instantaneous regret of base-stock level~$U$ with respect to $\tau_t^g$ is small relative to the estimated suboptimality gap at the time when $U$ was removed from the active set. Hence, by the construction of the sampling obligations for $U$, the number of times that base-stock policy is selected is small compared to the regret incurred per time step.
    
    We analyze this term separately over its contribution across different episodes.
    Towards this, for an episode~$v$, let ${t \in [t_v, t_{v+1}-1]}$ be in $\Tb_{2.1}$, i.e., ${\tau_t=U}$ and ${\mu_t(U) - \mu_t(\tau_t^{g}) \leq 4\tilde{\Delta}_v(U)}$.
    Recall that for each ${i \geq 1}$ with ${\varepsilon_i = 2^{-i} \geq \max\{\gamma, \tilde{\Delta}_v(U)/16\Hls\}}$, with probability ${p_i = \varepsilon_i\sqrt{v/UT \log(2T^2U/\delta\gamma)}}$ \NSLSIC adds a sampling obligation for base-stock policy $U$ of length ${n_i = \ceil*{2\log(2T^2U/\delta\gamma)/\varepsilon_i^2}}$.
    Using the fact that ${\varepsilon_i = 2^{-i} \leq 1/2}$, the expected number of times policy $U$ is selected in episode~$v$ for a specific $\varepsilon_i$ is therefore
    \begin{align*}
        (t_{v+1} - t_v) p_i n_i &=
        (t_{v+1} - t_v) \varepsilon_i \sqrt{\frac{v}{UT \log(2T^2U/\delta\gamma)}} \ceil*{\frac{2 \log(2T^2U/\delta\gamma)}{\varepsilon_i^2}} \\
        &\leq (t_{v+1} - t_v) \sqrt{\frac{v}{UT \log(2T^2U/\delta\gamma)}} \left(\frac{2 \log(2T^2U/\delta\gamma)}{\varepsilon_i} + \frac{1}{2}\right).
    \end{align*}
    Since ${\varepsilon_i \leq 1/2}$, and with a small value of $\delta$, we have that ${\log(2T^2U/\delta\gamma) \geq 1/4}$ and consequently ${2 \log(2T^2U/\delta\gamma)/\varepsilon_i + 1/2 \leq 3 \log(2T^2U/\delta\gamma)/\varepsilon_i}$. Moreover, since ${\varepsilon_i \geq \tilde{\Delta}_v(U)/16\Hls}$ and ${t \in \Tb_{2.1}}$, it holds that ${\varepsilon_i \geq \left( \mu_t(U) - \mu_t(\tau_t^g)\right) / 64\Hls}$.
    Therefore, summing over all possible values of $i$ we can bound the expected plays of $U$ in episode~$v$ by
    \begin{align*}
        &\sum_{\mathclap{\substack{i \geq 1: 2^{-i} = \varepsilon_i \\ \varepsilon_i \geq \left(\mu_t(U) - \mu_t\left(\tau_t^g\right)\right)}}} \ \ (t_{v+1} - t_v) \sqrt{\frac{v}{UT \log(2T^2U/\delta\gamma)}} \left(\frac{2 \log(2T^2U/\delta\gamma)}{\varepsilon_i} + \frac{1}{2} \right) \\
        &\leq (t_{v+1} - t_v) \sqrt{\frac{v}{UT \log(2T^2U/\delta\gamma)}} \quad \ \ \sum_{\mathclap{\substack{i \geq 1: 2^{-i} = \varepsilon_i \\ \varepsilon_i \geq \left(\mu_t(U) - \mu_t\left(\tau_t^g\right)\right)}}} \ \ \frac{3 \log(2T^2U/\delta\gamma)}{\varepsilon_i} \\
         & \leq (t_{v+1} - t_v) \sqrt{\frac{v}{UT \log(2T^2U/\delta\gamma)}} \frac{384\Hls \log(2T^2U/\delta\gamma)}{\mu_t(U) - \mu_t(\tau_t^{g})}.
    \end{align*}
    Note that the difference in expected costs per time step between any two base-stock levels is at most ${U\max\{h,b\}}$ (\cref{lemma:lipschitz}). 
    Using the fact that ${v \leq \hat{S} \leq S}$ (\cref{lemma:false_alarms}) we have that
    \begin{align*}
        R_{2.1}(T) = \sum_{t \in \Tb_{2.1}} \left( \mu_t(\tau_t) - \mu_t(\tau^g_t) \right) &\leq \textstyle \sum_{v=1}^{\hat{S}} (t_{v+1} - t_v) \sqrt{\frac{v}{UT \log(2T^2U/\delta\gamma)}} 384 \Hls \log(2T^2U/\delta\gamma) U \max \{ h, b\} \\
        &\leq \textstyle 384\Hls\sqrt{SU\log(2T^2U/\delta\gamma)/T} \max \{ h, b\} \sum_{v=1}^{\hat{S}} (t_{v+1} - t_v) \\
        &\leq 384 \Hls \sqrt{STU\log(2T^2U/\delta\gamma)}\max\{h,b\}.
    \end{align*}

\noindent\textbf{Bound of $R_{2.2}(T)$.}\quad
    To analyze this term we use the fact that the cost for $U$ has significantly increased compared to when it was evicted.  Hence, the base-stock level can only be selected a small number of time steps before the change is detected and a new episode starts.
    We again analyze this term separately across each episode.
    For an episode~$v$, let ${t \in [t_v, t_{v+1}-1]}$ be in $\Tb_{2.2}$ so that ${\tau_t=U}$, ${\mu_t(U) - \mu_t(\tau_t^{g}) > 4\tilde{\Delta}_v(U)}$ and ${\mu_t(U) - \tilde{\mu}_v(U) > \left(\mu_t(U) - \mu_t(\tau_t^g) \right)/2}$.
    
    Let $t_{v,j}$ denote the first time step of the $j$'th sampling block for $U$ in episode $v$ and ${t \geq t_{v,j}}$. Since the next episode has not been triggered at time $t$, {\em change point condition}~\eqref{eq:change_condition_ls_bad} is not satisfied, and so
    \begin{equation}
    \label{eq:2.2_step}
            \hat{\mu}(U, t_{v,j}, t) - \tilde{\mu}_v(U) \leq \tilde{\Delta}_v(U)/4 + b_{t_{v,j}, t}.
    \end{equation}

    Together with the definition of $t \in \Tb_{2.2}$ and the good-event $\calE$ (\cref{lemma:probability_good_event}), we have that
    \begin{align*}
        \mu_t(U) - \mu_t(\tau_t^g) &< 2\left( \mu_t(U) - \tilde{\mu}_v(U) \right) \quad \text{ since $t \in \Tb_{2.2}$}  \\
        &\leq 2 \left( \hat{\mu}(U, t_{v,j}, t) + b_{t_{v,j}, t} - \tilde{\mu}_v(U) \right)  \quad \text{by definition of $\calE$} \\
        &\leq 2 \left( \tilde{\Delta}_v(U)/4 + 2b_{t_{v,j}, t} \right) \quad \text{ by \cref{eq:2.2_step}}\\
        &< \left( \mu_t(U) - \mu_t(\tau_t^g) \right) / 8 + 4b_{t_{v,j}, t} \quad \text{ since $t \in \Tb_{2.2}$}.
    \end{align*}
    By rearranging this inequality we have ${\mu_t(U) - \mu_t(\tau_t^g) < \frac{32}{7} b_{t_{v,j},t}}$.

    In each time step, for a given ${1 \leq i \leq \log_2(1/\gamma)}$ the algorithm triggers a sampling block for $U$ of length ${n_i = \ceil*{2\log(2T^2U/\delta\gamma)/\varepsilon_i^2}}$ with probability ${p_i = \varepsilon_i \sqrt{v/UT \log(2T^2U/\delta\gamma)}}$.
    The expected regret incurred over this block is
    \begin{align*}
        p_i \sum_{t=t_{v,j}+1}^{t_{v,j} + n_i}(\mu_t(U) - \mu_t(\tau_t^g)) &< 
        p_i \frac{32}{7}\sum_{t=t_{v,j}+1}^{t_{v,j} + n_i} b_{t_{v,j},t} \\
        &= p_i \frac{32}{7} \Hls\sqrt{2\log(2T^2U/\delta\gamma)} \sum_{t=1}^{n_i} \sqrt{\frac{1}{t}} \\
        &\leq p_i \frac{64}{7} \Hls\sqrt{2\log(2T^2U/\delta\gamma)n_i} \\
        &\leq \frac{64}{7} \Hls\sqrt{2\log(2T^2U/\delta\gamma)} \varepsilon_i \sqrt{v/UT \log(2T^2U/\delta\gamma)} \sqrt{\ceil*{2\log(2T^2U/\delta\gamma)/\varepsilon_i^2}} \\
        &\leq \frac{128}{7} \Hls \sqrt{v(\log(2T^2U/\delta\gamma)+1)/UT}.
    \end{align*}
    Summing up this term over the $\log_2(1/\gamma)$ different values of $i$ and the number of time steps of the episode ${(t_{v+1} - t_v)}$, and using the fact that ${1 + \log(\cdot) \leq 4 \log(\cdot)}$, we obtain an upper bound for the expected regret of episode~$v$ via
    \[
        \frac{256}{7} \Hls \sqrt{\frac{v\log(2T^2U/\delta\gamma)}{UT}} \log_2(1/\gamma)(t_{v+1} - t_v).
    \]
    The expected dynamic regret accumulated over up to $S$ episodes, each of length at most $T$, is
    \begin{align*}
        R_{2.2}(T) = \sum_{t \in \Tb_{2.2}} \left( \mu_t(\tau_t) - \mu_t(\tau^g_t) \right) &< \sum_{v=1}^{\hat{S}} \frac{256}{7} \Hls \sqrt{\frac{v\log(2T^2U/\delta\gamma)}{UT}} \log_2\left(1/\gamma\right) \left(t_{v+1} - t_v\right) \\
        &\leq  \frac{256}{7} \Hls \sqrt{\frac{S\log(2T^2U/\delta\gamma)}{UT}} \log_2\left(1/\gamma\right) \sum_{v=1}^{\hat{S}} \left(t_{v+1} - t_v\right) \\
        &\leq \frac{256}{7} \Hls \sqrt{\frac{ST\log(2T^2U/\delta\gamma)}{U}} \log_2\left(1/\gamma\right).
    \end{align*}

\noindent\textbf{Bound of $R_{2.3}(T)$.}\quad
    To analyze this term we use the fact that the expected cost of $\tau_t^g$ has significantly decreased, causing a large regret each time base-stock level $U$ is selected.  However, we show that the change is detected quickly since the costs of $\tau_t^g$ are observed in each time step in $\Tb_{2.3}$. We again analyze the regret components separately across episodes.
    For an episode~$v$, let ${t \in I_{v,i} = [\beta_{v,i}, \beta_{v,i+1} - 1]}$ be in $\Tb_{2.3}$, i.e., ${\tau_t=U}$, ${\mu_t(U) - \mu_t(\tau_t^g) > 4\tilde{\Delta}_v(U)}$ and ${\mu_t(U) - \tilde{\mu}_v(U) \leq \left(\mu_t(U) - \mu_t(\tau_t^g) \right)/2}$.

    Recall that it suffices to consider time steps in which $U$ has already been eliminated from the active set.  Hence, let ${[s', t']}$ be the interval and $\tau'$ the base-stock level through which level $U$ is eliminated in an epoch~$k$ of episode~$v$, i.e., ${\tau' = \argmin_{\tau \in \A_{\gamma}, \tau \leq \tau_v^k} \hat{\mu}(\tau, s', t')}$.
    Using the fact that ${\hat{\mu}(\tau^g_t, s', t') \geq \hat{\mu}(\tau', s',t')}$ together with the identity ${\tilde{\Delta}_v(U) = \tilde{\mu}_v(U) - \hat{\mu}(\tau', s', t')}$ and the event $\calE$ (\cref{lemma:probability_good_event}) we have
    \begin{align*}
        \hat{\mu}(\tau_t^g, s', t') - \hat{\mu}(\tau_t^g, \beta_{v,i}, t) &\geq \hat{\mu}(\tau', s',t') -\mu_t(\tau^g_t) - b_{\beta_{v,i}, t} \\
        &= \tilde{\mu}_v(U) - \mu_t(\tau^g_t) - \tilde{\Delta}_v(U) - b_{\beta_{v,i}, t}.
    \end{align*}
    However, by definition of $\Tb_{2.3}$ we have that
    \[
    \tilde{\mu}_v(U) - \mu_t(\tau_t^g) - \tilde{\Delta}_v(U) \geq \frac{\mu_t(\tau_t^g) - \mu_t(U)}{2} - \tilde{\Delta}_v(U) \geq \frac{\mu_t(\tau_t^g) - \mu_t(U)}{2} - \frac{\mu_t(\tau_t^g) - \mu_t(U)}{4} = \frac{\mu_t(\tau_t^g) - \mu_t(U)}{4}
    \]
    and it holds that
    \begin{equation}
    \label{eq:2.3_step_1}
    \hat{\mu}(\tau_t^g, s', t') - \hat{\mu}(\tau_t^g, \beta_{v,i}, t) \geq \left(\mu_t(U) - \mu_t(\tau_t^g) \right)/4 - b_{\beta_{v,i}, t}.
    \end{equation}

    As base-stock level $U$ has been eliminated from the active set based on the observations in $[s',t']$, we have ${\tilde{\Delta}_v(U) > \Cb b_{s',t'}}$. 
    Moreover, since there has not been a change point detected at ${t \in I_{v,i}}$, applying the {\em change detection condition} \eqref{eq:change_condition_ls} to $\tau_t^g$, it follows that
    \begin{align}
         \hat{\mu}(\tau_t^g, s', t') - \hat{\mu}(\tau_t^g, \beta_{v,i}, t) &\leq b_{s',t'} + b_{\beta_{v,i}, t} \nonumber
         < \tilde{\Delta}_v(U)/\Cb + b_{\beta_{v,i}, t} \nonumber \\
         & < \left( \mu_t(U) - \mu_t(\tau_t^g) \right) /24 + b_{\beta_{v,i}, t} \quad \text{ since $t \in \Tb_{2.3}$}. \label{eq:2.3_step_2}
    \end{align}
    Using both \cref{eq:2.3_step_1} and \cref{eq:2.3_step_2} we have
    \[
    \left(\mu_t(U) - \mu_t(\tau_t^g) \right)/4 - b_{\beta_{v,i}, t} \leq \hat{\mu}(\tau_t^g, s', t') - \hat{\mu}(\tau_t^g, \beta_{v,i}, t) < \left( \mu_t(U) - \mu_t(\tau_t^g) \right) /24 + b_{\beta_{v,i}, t}.
    \]
    Rearranging the inequality and solving for ${\mu_t(U) - \mu_t(\tau_t^g)}$ gives $\mu_t(U) - \mu_t(\tau_t^g) < \frac{48}{5} b_{\beta_{v,i},t}$.
    The sum over all time steps in $I_{v,i}$ gives
    \[
        \sum_{t \in I_{v,i} \cap \Tb_{2.3}} \left(\mu_t(U) - \mu_t(\tau_t^g)\right) < \frac{48}{5} \Hls\sqrt{2\log(2T^2U/\delta\gamma)} \sum_{t=1}^{\abs*{I_{v,i} \cap \Tb_{2.3}}} \sqrt{\frac{1}{t}} \leq \frac{96}{5} \Hls\sqrt{2\log(2T^2U/\delta\gamma)\abs*{I_{v,i}}}
    \]
    and finally
    \[
        R_{2.3}(T) = \sum_{t \in \Tb_{2.3}} \left( \mu_t(\tau_t) - \mu_t(\tau^g_t) \right) < \sum_{v=1}^{\hat{S}} \sum_{i=1}^{S_v} \frac{96}{5} \Hls\sqrt{2\log(2T^2U/\delta\gamma)\abs*{I_{v,i}}} \leq \frac{192}{5} \Hls \sqrt{ST\log(2T^2U/\delta\gamma)}.
    \]

\noindent\textbf{Bound of $R_3(T)$.}\quad
    This regret contribution measures the difference in expected costs of the optimal base-stock policy $\taugammaopt$ with respect to $\tau_t^g$, accumulated over time steps in which $\taugammaopt$ is larger than $\tau_v^k$. Since ${\sup \Avk{v}{k} = \tau_v^k < \taugammaopt}$, this implies that $\taugammaopt$ has been eliminated from the active set and hence the difference ${\mu_t(\tau_t^g) - \mu_t(\taugammaopt)}$ is non-zero over these time steps. 
    Again, we only consider time steps, in which $U$ has been eliminated from the active set since otherwise we would have ${\taugammaopt \leq \tau_v^k = U}$.
    Note that by convexity, whenever ${\taugammaopt \geq \tau_v^k}$ it holds that ${\tau_v^k = \tau_t^g}$.
    
    To analyze this source of regret, we further distinguish between the case that the mean costs of previously near-optimal policies have significantly increased, as well as the case that the mean costs of previously costly policies have significantly decreased. We write $R_3(T) = R_{3.1}(T) + R_{3.2}(T)$ with
    \begin{align*}
        R_{3.1}(T) &\coloneq \sum_{t \in \Tb_{3.1}} \left( \mu_t(\tau_t^g) - \mu_t(\taugammaopt) \right), \quad \Tb_{3.1} \coloneq \{ t \in [T]: \tau_{v}^{k} < \taugammaopt, \tilde{\mu}_v(\taugammaopt) - \mu_t(\taugammaopt) \leq \tilde{\Delta}_v (\taugammaopt)/2 \}, \\
        R_{3.2}(T) &\coloneq \sum_{t \in \Tb_{3.2}} \left( \mu_t(\tau_t^g) - \mu_t(\taugammaopt) \right), \quad \Tb_{3.2} \coloneq \{t \in [T]: \tau_{v}^{k} < \taugammaopt, \tilde{\mu}_v(\taugammaopt) - \mu_t(\taugammaopt) > \tilde{\Delta}_v (\taugammaopt)/2 \}.
    \end{align*}
    
    Before deriving the upper bounds for these two terms we include a brief digression to give some intuition on the subdivision into the sets $\Tb_{3.1}$ and $\Tb_{3.2}$.
    In both sets of time steps the optimal base-stock level is larger than the highest level in the active set, hence the optimal level has increased since the time of its elimination.
    In the first case, the cost of the new optimal base-stock level has decreased by a significantly smaller amount than its estimated suboptimality gap at the time this level was removed from the active set. That means, the cost of the previously optimal base-stock policy must have increased considerably to make $\taugammaopt$ the new optimum. This can be detected quickly by playing $\tau_v^k$ as the cost realizations of the previously optimal base-stock policy are observed in each time step.
    
    In the second case, the cost of the new optimal base-stock policy has decreased significantly. Here, this does not necessarily imply a substantial change in the costs of policies ${\tau \leq \tau_v^k}$. Hence, in order for this change to be identified, it may be necessary to sample base-stock levels that are greater than $\tau_v^k$, i.e., level~$U$.
    The condition of minimum reduction in costs of base-stock policy~$\taugammaopt$ ensures we do not need too many samples of policy~$U$ in relation to the potentially large regret per time step to detect the change fast.

\noindent\textbf{Bound of $R_{3.1}(T)$.}\quad
    In these time steps, it holds that ${\tilde{\mu}_v(\taugammaopt) - \mu_t(\taugammaopt) \leq \tilde{\Delta}_v(\taugammaopt)/2}$.
    Consider an episode~$v$ and time step~$t \in I_{v,i}$ for a time interval without change ${I_{v,i}=[\beta_{v,i}, \beta_{v,i+1}-1]}$. Let $k$ be the epoch of time step~$t$. Since ${\tau_{v}^{k} < \taugammaopt}$, there is a time interval $[s',t']$ in some earlier epoch $k'<k$ of the same episode as well as a base-stock policy ${\tau' = \argmin_{\tau \in \A_{\gamma}, \tau \leq \tau_v^{k'}} \hat{\mu}(\tau, s', t')}$ which have caused $\taugammaopt$, the optimal base-stock policy in $I_{v,i}$, to be identified as suboptimal at time~$t'$ via the {\em active set elimination condition}~(\cref{eq:elimination_condition_ls}). Therefore, we have
    ${\hat{\mu}(\taugammaopt, s', t') - \hat{\mu}(\tau', s', t') > \Cb b_{s',t'}}$.
    Moreover, it holds that 
    \begin{equation} \label{eq:tau_hat_R_3_1_ls}
        \hat{\mu}(\tau',s',t')=\tilde{\mu}_v(\taugammaopt) - \tilde{\Delta}_v(\taugammaopt) \leq \mu_t(\taugammaopt) - \frac{\tilde{\Delta}_v(\taugammaopt)}{2}
         \leq \mu_t(\taugammaopt) - 3b_{s',t'},
    \end{equation}
    where the first inequality follows from the definition of $\tilde{\Delta}_v(\taugammaopt) = \tilde{\mu}_v(\taugammaopt) - \hat{\mu}(\tau', s', t')$, the second is due to the construction of the set $\Tb_{3.1}$, and the third due to the fact that ${\tilde{\Delta}_v(\taugammaopt) > \Cb b_{s',t'}}$ by the {\em active set elimination condition}~(\cref{eq:elimination_condition_ls}).
    %
    %
    In the following, we separate into whether or not ${\tau' > \tau_v^k}$, and derive for both cases that ${\mu_t(\tau_t^g) - \mu_t(\taugammaopt) \leq 2 b_{\beta_{v,i},t}}$.

    \noindent\textbf{Case ${\tau' > \tau_{v}^{k}}$.}\quad
    In this case we know that base-stock policy $\tau'$ is eliminated from the active set in some epoch~$k''$ with ${k' < k'' < k}$ (recall that ${\tau_v^k = \sup \Avk{v}{k}}$ and ${\Avk{v}{k} \subset \Avk{v}{k'}}$). 
    Hence, there was a time interval ${[s'',t'']}$ with ${t' < t'' < t}$ and a base-stock level ${\tau'' < \tau'}$ with ${\tau'' = \argmin_{\tau \in \A_{\gamma}, \tau \leq \tau_v^k} \hat{\mu}(\tau, s'', t'')}$ that has caused the elimination of policy $\tau'$ via \cref{eq:elimination_condition_ls}. This yields
    \begin{equation} \label{eq:6b_R3_1_ls}
        \hat{\mu}(\tau',s'',t'') - \hat{\mu}(\tau'', s'',t'') > \Cb b_{s'',t''}.
    \end{equation}
    As no change has been detected in $t \in I_{v,i}$, following the {\em change point condition}~(\cref{eq:change_condition_ls}) it must hold that
    \[
    \hat{\mu}(\tau', s'', t'') - \hat{\mu}(\tau', s', t') \leq b_{s', t'} + b_{s'', t''} \hfill \text{ and } \hfill \hat{\mu}(\tau'', s'', t'') - \hat{\mu}(\tau'', s', t') \leq b_{s', t'} + b_{s'', t''}.
    \]
    Thus we have
         ${\hat{\mu}(\tau',s'',t'') - \hat{\mu}(\tau'', s'',t'') \leq \hat{\mu}(\tau', s', t') - \hat{\mu}(\tau'', s', t') + 2b_{s', t'} + 2b_{s'',t''} 
         \leq 2b_{s', t'} + 2b_{s'',t''}}$,
    where we have used that ${\hat{\mu}(\tau', s', t') \leq \hat{\mu}(\tau'', s', t')}$ since ${\tau' = \argmin_{\tau \in \A_{\gamma}, \tau \leq \tau_v^{k'}}\, \hat{\mu}(\tau, s', t')}$. Consequently, with \cref{eq:6b_R3_1_ls} we have that ${6b_{s'',t''} < 2b_{s',t'} + 2b_{s'',t''}}$ and hence 
    \begin{equation} \label{eq:2b_b_R_3_1_ls}
        2b_{s'',t''} < b_{s',t'}.
    \end{equation}
    Furthermore, we know that when $\tau'$ was evicted based on the observations in $[s'',t'']$ in epoch $k''$, base-stock level ${\tau^g_t = \tau_v^k}$ was not evicted since it remained in the active set until epoch ${k > k''}$. Thus the cost estimate of $\tau_v^k$ over ${[s'',t'']}$ must be smaller than that of $\tau'$, i.e.,
    \begin{equation} \label{eq:tau_g_t_tau_R_3_1_ls}
        \hat{\mu}(\tau^g_t, s'',t'') \leq \hat{\mu}(\tau', s'',t'').
    \end{equation}
        
    Together with the good-event $\calE$ (\cref{lemma:probability_good_event}) and {\em change point condition}~\eqref{eq:change_condition_ls} this gives
    \begin{align*}
        \mu_t(\tau^g_t) &\leq \hat{\mu}(\tau_t^g, \beta_{v,i}, t) + b_{\beta_{v,i}, t} \quad \text{ by definition of $\calE$} \notag \\
        & \leq \hat{\mu}(\tau_t^g, s'',t'') + 2b_{\beta_{v,i}, t} + b_{s'',t''} \quad \text{ by change point condition (\cref{eq:change_condition_ls})}  \notag\\
        & \leq \hat{\mu}(\tau', s'',t'') + 2b_{\beta_{v,i}, t} + b_{s'',t''} \quad \text{ by \cref{eq:tau_g_t_tau_R_3_1_ls}}  \notag\\
        & \leq \hat{\mu}(\tau', s',t') + 2b_{\beta_{v,i}, t} + 2b_{s'',t''} + b_{s',t'} \quad \text{ by change point condition (\cref{eq:change_condition_ls})} \notag \\
        & < \hat{\mu}(\tau', s',t') + 2b_{\beta_{v,i}, t} + 2 b_{s', t'}  \quad \text{ by \cref{eq:2b_b_R_3_1_ls}}.
    \end{align*}
    Using inequality \eqref{eq:tau_hat_R_3_1_ls}, we get that the regret per time step is bounded as $\mu_t(\tau^g_t) - \mu_t(\taugammaopt) < 2b_{\beta_{v,i},t}$.
    
    \noindent\textbf{Case ${\tau' \leq \tau_{v}^{k}}$.}\quad
    By convexity we have that ${\mu_t(\tau^g_t) \leq \mu_t(\tau')}$ since ${\tau_v^k = \tau_t^g < \taugammaopt}$. Again by the good-event $\calE$ (\cref{lemma:probability_good_event}) and {\em change point condition}~\eqref{eq:change_condition_ls}, since the change has not been detected at time $t$, it holds that
    \begin{equation*}
        \mu_t(\tau^g_t) \leq \mu_t(\tau') \leq \hat{\mu}(\tau', \beta_{v,i}, t) + b_{\beta_{v,i},t} \leq \hat{\mu}(\tau', s', t') + 2b_{\beta_{v,i},t} + b_{s',t'}.
    \end{equation*}
    With inequality \eqref{eq:tau_hat_R_3_1_ls} we obtain the same bound as in the first case, $\mu_t(\tau^g_t) - \mu_t(\taugammaopt) \leq 2b_{\beta_{v,i},t}$.
    
\noindent\textbf{Final Bound of $R_{3.1}(T)$.}\quad
    In either of the above cases, summing over all time steps in $I_{v,i}$ yields
    \begin{align*}
        \textstyle \sum_{t \in I_{v,i} \cap \Tb_{3.1}} \left( \mu_t(\tau_t^g) - \mu_t(\taugammaopt) \right) &\leq \textstyle 2\Hls\sqrt{2 \log(2T^2U/\delta\gamma)} \sum_{t=1}^{\abs*{I_{v,i} \cap \Tb_{3.1}}}\sqrt{\frac{1}{t}} \\
        & \leq 4 \Hls \sqrt{2 \log(2T^2U/\delta\gamma) |I_{v,i}|}.
    \end{align*}
    Accumulated over up to $2S - 1$ intervals, we get for this regret contribution
    \begin{align*}
        R_{3.1}(T) = \sum_{t \in \Tb_{3.1}} \left( \mu_t(\tau_t^g) - \mu_t(\taugammaopt) \right) = \sum_{v=1}^{\hat{S}} \sum_{i=1}^{S_v} \sum_{t \in I_{v,i} \cap \Tb_{3.1}} \left( \mu_t(\tau_t^g) - \mu_t(\taugammaopt) \right) \leq 8 \Hls \sqrt{ST\log(2T^2U/\delta\gamma)}.
    \end{align*}

\noindent\textbf{Bound of $R_{3.2}(T)$.}\quad
    In these time steps, it holds that ${\tau_v^k < \taugammaopt}$ and ${\tilde{\mu}_v(\taugammaopt) - \mu_t(\taugammaopt) \geq \tilde{\Delta}_v(\taugammaopt)/2}$. The cost of the new optimal base-stock level has decreased significantly. We show that base-stock policy $U$ is selected (hence full counterfactual feedback collected) with sufficient frequency to detect the change quickly. 
    Let ${t \in I_{v,i}}$ for some interval ${I_{v,i} = [\beta_{v,i}, \beta_{v,i+1} - 1]}$, and let $[s',t']$ be the interval in the same episode~$v$ through which $\taugammaopt$ has been identified as suboptimal according to \cref{eq:elimination_condition_ls}. As there are no change detections within $I_{v,i}$, at any time step the instantaneous regret of base-stock policy $\tau^g_t$ with respect to the optimal base-stock policy is bounded via
    \begin{align*}
        \mu_t(\tau^g_t) - \mu_t(\taugammaopt) & \leq \tilde{\mu}_v(\taugammaopt) - \mu_t(\taugammaopt) + \hat{\mu}(\tau^g_t, \beta_{v,i}, t) - \tilde{\mu}_v(\taugammaopt) + b_{\beta_{v,i}, t} \quad \text{by definition of $\calE$ (\cref{lemma:probability_good_event})} \\
        & \leq \tilde{\mu}_v(\taugammaopt) - \mu_t(\taugammaopt) + \hat{\mu}(\tau^g_t, s', t') - \hat{\mu}(\taugammaopt, s',t') + b_{s', t'} + 2b_{\beta_{v,i}, t} \quad \text{by change cond. \eqref{eq:change_condition_ls}}.
    \end{align*}
    Since ${\tau_v^k < \taugammaopt}$ and by convexity (\cref{lemma:convexity}), it holds that ${\tau_t^g = \tau_v^k}$. Moreover, since $\tau_v^k$ remained in the active set while $\taugammaopt$ has been evicted, we have ${\hat{\mu}(\tau_t^g, s',t') = \hat{\mu}(\tau_v^k, s',t') \leq \hat{\mu}(\taugammaopt, s',t')}$ and
    \begin{align} \label{eq:regret_taug_taugammaopt}
        \mu_t(\tau^g_t) - \mu_t(\taugammaopt) & \leq \tilde{\mu}_v(\taugammaopt) - \mu_t(\taugammaopt) + b_{s', t'} + 2b_{\beta_{v,i}, t} \notag \\
        & < \tilde{\mu}_v(\taugammaopt) - \mu_t(\taugammaopt) + \tilde{\Delta}_v(\taugammaopt)/6 + 2b_{\beta_{v,i}, t}  \quad \text{ with $\tilde{\Delta}_v(\taugammaopt) > \Cb b_{s',t'}$} \notag \\
        & \leq \frac{4}{3}\left( \tilde{\mu}_v(\taugammaopt) - \mu_t(\taugammaopt) \right) + 2b_{\beta_{v,i}, t}  \quad \text{ by definition of $\Tb_{3.2}$}.
    \end{align}
    
    Let $\varepsilon$ be the largest ${\varepsilon_i=2^{-i} \geq \gamma}$ that satisfies ${\varepsilon_i \leq \left( \tilde{\mu}_v(\taugammaopt) - \mu_t(\taugammaopt) \right)/4\Hls}$, and let ${n=\lceil 2\log(2T^2U/\delta\gamma)/\varepsilon^2\rceil}$.
    Then ${\tilde{\mu}_v(\taugammaopt) - \mu_t(\taugammaopt) \leq 8 \Hls \varepsilon}$ and with \cref{eq:regret_taug_taugammaopt} it follows that
    \begin{equation} \label{eq:Rt_Bv2}
        \mu_t(\tau^g_t) - \mu_t(\taugammaopt) \leq \frac{32}{3}\Hls \varepsilon + 2b_{\beta_{v,i}, t}.
    \end{equation}
    
    Given there are at most $2S-1$ intervals $I_{v,i}$ with a total length of $T$ (\cref{lemma:false_alarms}), the contribution of the second term on the right hand side of inequality~\eqref{eq:Rt_Bv2} is
    \begin{equation*}
        \sum_{v=1}^{\hat{S}} \sum_{i=1}^{S_v} \sum_{t \in I_{v,i} \cap \Tb_{3.2}} 2b_{\beta_{v,i}, t} = \sum_{v=1}^{\hat{S}} \sum_{i=1}^{S_v} 2\Hls \sqrt{2\log(2T^2U/\delta\gamma)} \sum_{t = 1}^{\abs*{I_{v,i} \cap \Tb_{3.2}}} \sqrt{\frac{1}{t}} \leq 8\Hls\sqrt{ST \log(2T^2U/\delta\gamma)}.
    \end{equation*}
    Therefore, is suffices to bound the number of steps in $I_{v,i}$, and the regret per time step by the first term on the right hand size of inequality~\eqref{eq:Rt_Bv2} conditioned on the size of $\varepsilon$ through which the change can be detected.
    
    The instantaneous regret accumulates over the time steps until either the next change occurs or the current change is detected, whichever comes first. The change might not be detected sufficiently fast through only the cost observations of base-stock level ${\tau \leq \tau_v^k}$. By the construction of the set $\Tb_{3.2}$, however, the change in cost of $\taugammaopt$ is significant compared to the time of its eviction from the active set. Since ${\tau_v^k < \taugammaopt}$, the change is therefore detected when a sampling obligation for $U$ of sufficient length is added and {\em change point condition}~\eqref{eq:change_condition_ls_bad} is satisfied by any $\tau \in \A_{\gamma}$. In the following, we consider possible change magnitudes $\varepsilon$ and interval lengths $\abs{I_{v,i}}$, in relation to the length $n$ of the sampling obligation for base-stock level $U$ necessary to detect the change of size $\varepsilon$.    
    
    If ${\varepsilon \leq \sqrt{2 \log(2T^2U/\delta\gamma)/\abs{I_{v,i}}}}$, then with \cref{eq:Rt_Bv2} the regret is bounded as 
    \begin{equation*}
        \sum_{v=1}^{\hat{S}} \sum_{i=1}^{S_v} \frac{32}{3} \Hls \varepsilon \abs{I_{v,i}} \leq \frac{32}{3} \Hls \sum_{v=1}^{\hat{S}} \sum_{i=1}^{S_v} \sqrt{2\log(2T^2U/\delta\gamma)\abs{I_{v,i}}} = \frac{64}{3} \Hls \sqrt{ST \log(2T^2U/\delta\gamma)}.
    \end{equation*} 
    Likewise, if $\varepsilon > \sqrt{2 \log(2T^2U/\delta\gamma)/\abs{I_{v,i}}}$ and $\abs{I_{v,i}} < 2n$, then the regret contribution is 
    \begin{align} \label{eq:regret_short_intervals}
        \sum_{v=1}^{\hat{S}} \sum_{i=1}^{S_v} \frac{32}{3} \Hls \varepsilon \abs{I_{v,i}} &< \sum_{v=1}^{\hat{S}} \sum_{i=1}^{S_v} \frac{64}{3} \Hls \varepsilon n \quad \text{ using $\abs{I_{v,i}} < 2n$} \notag \\ 
        &\leq \textstyle \frac{64}{3} \Hls \sum_{v=1}^{\hat{S}} \sum_{i=1}^{S_v} \varepsilon \left( 2\log\left(\frac{2T^2U}{\delta\gamma}\right) / \varepsilon^2 + 1 \right) \quad \text{using $n=\lceil 2\log\left(\frac{2T^2U}{\delta\gamma}\right)/\varepsilon^2\rceil$} \notag \\
        &< \textstyle \frac{64}{3} \Hls \sum_{v=1}^{\hat{S}} \sum_{i=1}^{S_v} \left( \sqrt{2 \log\left(\frac{2T^2U}{\delta\gamma}\right)\abs{I_{v,i}}} + \varepsilon \right) \quad \text{since $\varepsilon > \sqrt{2 \log\left(\frac{2T^2U}{\delta\gamma}\right)/\abs{I_{v,i}}}$} \notag \\
        &\leq \textstyle \frac{128}{3} \Hls\sqrt{ST \log\left(\frac{2T^2U}{\delta\gamma}\right)} + \frac{64}{3} \Hls S \ \ \ \text{with Jensen's inequality and $\varepsilon = 2^{-i} < \frac{1}{2}$}.
    \end{align}
    
    Therefore, in the following we consider intervals of length $\abs{I_{v,i}} \geq 2 n$. We show that for these intervals, the change can be detected before the next change happens as the time difference between the two changes is larger than the number of samples of base-stock level $U$ required to detect the previous change.
    Let $t_{v,j}$ denote the first time step of the $j$'th sampling block for $U$ in episode $v$. After completing a sampling block of length $n$, with $\sqrt{\frac{2\log(2T^2U/\delta\gamma)}{n}} \leq \varepsilon \leq \tilde{\mu}_v(\taugammaopt) - \mu_t(\taugammaopt)/8\Hls$ we have that
    \begin{align*}
        \tilde{\mu}_v(\taugammaopt) - \hat{\mu}(\taugammaopt, t_{v,j}, t) &> \tilde{\mu}_v(\taugammaopt) - \mu_t(\taugammaopt) - \Hls\sqrt{\frac{2\log(2T^2U/\delta\gamma)}{n}}  \quad \text{by definition of $\calE$ (\cref{lemma:probability_good_event})} \\
        & \geq \tilde{\mu}_v(\taugammaopt) - \mu_t(\taugammaopt) + \Hls\sqrt{\frac{2\log(2T^2U/\delta\gamma)}{n}} - 2\Hls\varepsilon \ \text{with $\sqrt{\frac{2\log(2T^2U/\delta\gamma)}{n}} \leq \varepsilon$}\\
        & \geq \frac{3}{4} \left(\tilde{\mu}_v(\taugammaopt) - \mu_t(\taugammaopt) \right) + \Hls\sqrt{\frac{2\log(2T^2U/\delta\gamma)}{n}} \ \text{with $\varepsilon \leq \tilde{\mu}_v(\taugammaopt) - \mu_t(\taugammaopt)/8\Hls$} \\
        & \geq \tilde{\Delta}_v(\taugammaopt)/4 + \Hls\sqrt{\frac{2\log(2T^2U/\delta\gamma)}{n}} \quad \text{by definition of $\Tb_{3.2}$} 
    \end{align*}
    and hence the {\em change point condition}~\eqref{eq:change_condition_ls_bad} is satisfied. If the sampling block is not completed, this is captured by the above case $\abs{I_{v,i}} < 2n$.
        
    The remainder of the proof closely follows the arguments presented in \citet[B.2]{auer2019adaptively}.
    Let $\xi_l$ for $l=2, \dots , S$ denote the actual change points and $\xi_1 = 1$ as well as $\xi_{S+1}=T+1$. Then, $I_{v,i}$ is a subinterval of an interval without change $[\xi_l, \xi_{l+1}-1]$ defined by change detections that may occur between change points $\xi_l$ and $\xi_{l+1}$. Since we assume $\abs{I_{v,i}} \geq 2n$, it follows $\xi_{l+1} - \beta_{v,i} \geq 2n$.
    As defined in \NSLSICL, the sampling obligation that detects the change of size at least $\varepsilon$ is added with probability $p_v \varepsilon$ with ${p_v=\sqrt{v/UT \log(2T^2U/\delta\gamma)}}$ and, once added, triggers a restart after ${n=\lceil 2\log(2T^2U/\delta\gamma)/\varepsilon^2\rceil}$ steps.
    The regret contribution ${\bar{R}_{v,i} \coloneq \sum_{t \in I_{v,i} \cap \Tb_{3.2}} \left( \mu_t(\tau_t^g) - \mu_t(\taugammaopt) \right)}$ over interval $I_{v,i}$ with ${\abs{I_{v,i}} \geq 2n}$ is therefore using \cref{eq:Rt_Bv2} bounded as
    \begin{align*}
        \bar{R}_{v,i} & \leq \frac{32}{3} \Hls\varepsilon \; \left( \sum_{j=1}^{\xi_{l+1}-\beta_{v,i}-n} \left(1-p_v \varepsilon \right)^j + n \right) \\
        & \leq \frac{32}{3}\Hls\varepsilon \; \left( \frac{1 - \left( 1 - p_v \varepsilon \right)^{\xi_{l+1}-\beta_{v,i}-n}}{p_v \varepsilon} + n \right) \\
        & = \frac{32\Hls}{3p_v} \; \left( 1 - \left( 1 - p_v \varepsilon \right)^{\xi_{l+1}-\beta_{v,i}-n} \right) + \frac{32}{3}\Hls n\varepsilon \\
        & = \frac{32\Hls}{3 p_v} \; \left( 1 - \left( 1 - p_v \varepsilon \right)^{\xi_{l+1}-\beta_{v,i}-n} \right) + \frac{128}{3} \Hls \log(2T^2U/\delta\gamma)/\varepsilon.
    \end{align*}
    
    Let $q_{v,i}$ denote the probability that $I_{v,i}$ does not end with the detection of a change but an actual change, that is, ${\beta_{v,i+1} = \xi_{l+1}}$. Recall that we need at least $n$ samples to detect the change of size  $\varepsilon$. Hence, if no sampling obligation of at least this length is added for at least ${\abs{I_{v,i}} - n}$ steps, the change will not be detected and interval $I_{v,i}$ ends with an actual change at $\xi_{l+1}$. Since with probability ${p_v \varepsilon}$ the required sampling obligation is added, we have that $q_{v,i} \leq \left( 1 - p_v \varepsilon \right)^{\xi_{l+1}-\beta_{v,i}-n}$.
    
    Let $R_{v,i}$ denote the total future regret contribution of steps ${t \in \Tb_{3.2}}$ in intervals $I_{v,i}$ with ${\abs{I_{v,i}} \geq 2n}$ and ${I_{v,i} \subseteq [\xi_l, \xi_{l+1} - 1]}$, starting from $\beta_{v,i}$. We show by backward induction that this quantity is bounded as
    \begin{equation*}
        R_{v,i} \leq \sum_{v'=v}^{\hat{S}} \frac{32\Hls}{3 p_{v'}} + \frac{128}{3} \Hls \sqrt{\log(2T^2U/\delta\gamma) (2S - l - v)(T+1-\beta_{v,i})}.
    \end{equation*}
    This is trivially satisfied for ${\beta_{v,i} = T+1}$.
    The future expected regret from intervals that start at ${\beta_{v,i} \leq T}$ is the sum of the regret in the current interval and the future regret from the next interval. By induction, this is 
    \begin{align*}
        R_{v,i} & \leq \bar{R}_{v,i} + q_{v,i} R_{v,i+1} + \left( 1 - q_{v,i} \right) R_{v+1,1} \\
                & \leq \frac{32\Hls}{3 p_v} \left( 1 - q_{v,i} \right) + q_{v,i} \sum_{v'=v}^{\hat{S}} \frac{32\Hls}{3 p_{v'}} + (1 - q_{v,i}) \sum_{v'=v+1}^{\hat{S}} \frac{32\Hls}{3 p_{v'}} \\
                &\quad+ \frac{128}{3}\Hls q_{v,i}\sqrt{\log(2T^2U/\delta\gamma)(\xi_{l+1} - \beta_{v,i})} + \frac{128}{3}\Hls (1 - q_{v,i})\sqrt{\log(2T^2U/\delta\gamma)(\beta_{v+1,1} - \beta_{v,i})} \\
                &\quad + \frac{128}{3}\Hls q_{v,i}\sqrt{\log(2T^2U/\delta\gamma)(2S - l - v - 1)(T + 1 - \xi_{l+1})} \\
                &\quad + \frac{128}{3}\Hls (1-q_{v,i})\sqrt{\log(2T^2U/\delta\gamma)(2S - l - v - 1)(T + 1 - \beta_{v+1,1})} \\
                & = \frac{32\Hls}{3 p_v} \left( 1 - q_{v,i} \right) + \frac{32 \Hls}{3 p_v} q_{v,i} + \sum_{v'=v+1}^{\hat{S}} \frac{32\Hls}{3 p_{v'}} + \frac{128 \Hls}{3} q_{v,i} \sqrt{\log\left(\frac{2T^2U}{\delta\gamma}\right)(2S - l - v)(T + 1 - \beta_{v,i})} \\
                &\quad + \frac{128 \Hls}{3} \left(1 - q_{v,i}\right) \sqrt{\log(2T^2U/\delta\gamma)(2S - l - v)(T + 1 - \beta_{v,i})}\\
                & = \sum_{v'=v}^{\hat{S}} \frac{32\Hls}{3 p_{v'}} + \frac{128\Hls}{3}\sqrt{\log(2T^2U/\delta\gamma)(2S - l - v)(T + 1 - \beta_{v,i})},
    \end{align*}
    where in the second to last step we have used that 
    \begin{align*}
        \sqrt{\xi_{l+1} - \beta_{v,i}} + \sqrt{(2S - l - v - 1)(T + 1 - \xi_{l+1})} &\leq \sqrt{(2S - l - v)(T + 1 - \beta_{v,i})} \ \text{ and } \\
        \sqrt{\beta_{v+1,1} - \beta_{v,i}} + \sqrt{(2S - l - v - 1)(T + 1 - \beta_{v+1,1})} &\leq \sqrt{(2S - l - v)(T + 1 - \beta_{v,i})}.
    \end{align*}
    With the following bound on total regret accumulated over all long intervals 
    \[
        R_{1,1} \leq \sum_{v=1}^{\hat{S}} \frac{32\Hls}{3}\sqrt{\frac{UT \log(2T^2U/\delta\gamma)}{v}} + \frac{128\Hls}{3} \sqrt{2ST\log(2T^2U/\delta\gamma)} \leq \frac{256 + 64\sqrt{U}}{3} \Hls\sqrt{ST\log\left(\frac{2T^2U}{\delta\gamma}\right)}
    \]
    and the bound on the regret incurred over short intervals in \cref{eq:regret_short_intervals}, it follows for the total regret contribution
    \[
        R_{3.2}(T) = \sum_{t \in \Tb_{3.2}} \left(\mu_t(\tau_t^g) - \mu_t(\taugammaopt) \right)
        \leq \frac{256 + 64\sqrt{U}}{3} \Hls\sqrt{ST\log(2T^2U/\delta\gamma)} + \frac{64}{3} \Hls S.
    \]

\noindent\textbf{Final Regret Bound.}\quad
    We conclude that, conditioned on the event $\calE$, we obtain the following bounds for the regret components defined in \cref{lemma:regret_decomposition_LS_L0}
    \begin{align*}
        R_1(T) &\leq 32 \Hls \sqrt{ST \log(2T^2U/\delta\gamma)}, \\
        R_2(T) &< \left(\frac{192}{5} + 384\sqrt{U}\max\{h,b\} + \frac{256 \log_2(1/\gamma)}{7\sqrt{U}} \right) \Hls \sqrt{ST \log(2T^2U/\delta\gamma)}, \\
        R_3(T) &\leq \frac{280 + 64\sqrt{U}}{3} \Hls\sqrt{ST\log(2T^2U/\delta\gamma)} + \frac{64}{3}\Hls S. 
    \end{align*}
    Following the decompositions in \cref{lemma:regret_decomposition,lemma:regret_decomposition_LS_L0}, the expected dynamic regret is upper bounded via
    \begin{align*}
        \E &\left[R(T) \mid \calE \right] = R_\gamma(T) + R_1(T) + R_2(T) + R_3(T) \\
            &< \left(\frac{2456}{15}+\left(\frac{64}{3} + 384\max\{h,b\}\right)\sqrt{U} + \frac{256 \log_2(1/\gamma)}{7\sqrt{U}} \right)\Hls \sqrt{ST \log\left(\frac{2T^2U}{\delta\gamma}\right)} 
            + \frac{64}{3}\Hls S + T\gamma \max\{h,b\}.
    \end{align*}
    By \cref{lemma:probability_good_event}, we have that ${\Pr(\calE) \geq 1 - \delta}$. Moreover, using the Lipschitz continuity of the cost function (\cref{lemma:lipschitz}), it holds that ${\Exp{R(T) \mid \calE^{c}} \leq TU \max\{h,b\}}$.
    Therefore, by decomposing the expectation over $\calE$ and $\calE^c$ we obtain
    \begin{align*}
        \Exp{R(T)} &= \Exp{R(T) \mid \calE} \Pr(\calE) + \Exp{R(T) \mid \calE^{c}} \Pr(\calE^{c}) \\
        & < \left(1 - \delta\right) \left(\left(\frac{2456}{15}+\left(\frac{64}{3} + 384\max\{h,b\}\right)\sqrt{U} + \frac{256 \log_2(1/\gamma)}{7\sqrt{U}} \right)\Hls \sqrt{ST \log(2T^2U/\delta\gamma)} \right.\\
        & \qquad \qquad \quad \left. + \frac{64}{3}\Hls S + T\gamma \max\{h,b\} \right) + \delta TU\max\{h,b\}.
    \end{align*}    
    Finally, with ${\Hls = 216U\max\{h,b\}}$ and by setting ${\gamma = \mathcal{O}\left(T^{-1/2}\right)}$ as well as ${\delta = \Theta\left(T^{-2}\right)}$ this gives ${\E \left[R(T)\right] = \tilde{\mathcal{O}}\left( U^{3/2} \sqrt{ST} \right)}$.
\end{rproofof}

\subsection{\texorpdfstring{Proof of \Cref{thm:regret_LS_L_positive}}{Regret Analysis Lost-Sales with Positive L}} \label[appendix]{app:regret_LS_Lpositive}

\LSpositiveLRegret*

\cref{thm:regret_LS_L_positive} provides an upper bound on the dynamic regret {\em in expectation}, since it depends on the number of time steps in which orders are suspended to reduce inventory, which \Cref{ass:main_assumption}.\ref{ass:bounded_depletion_steps} ensures is bounded in expectation. 
We analyze regret conditional on the event $\calE$ and state the total regret bound in expectation. 

\begin{proof}
To prove \cref{thm:regret_LS_L_positive}, we start with the regret decomposition specific to the \NSLSICL algorithm. 

\noindent\textbf{Regret Decomposition.}\quad
By \cref{lemma:probability_good_event}, our concentration guarantees hold for all time intervals in which the sum of on-hand and in transit inventory does not exceed the designated base-stock level. Transitioning from a higher level $\tau_v^k$ to a lower level $\tau_v^{k+1}$ requires a number of steps ${\alpha_{v,k}, \dots , \bar{\alpha}_{v,k}-1}$ to reduce inventory and cost observations collected during these intermediate steps do not contribute to the accuracy of the cost estimates.
Convenient for the following regret analysis, we partition $\Rb(T)$ based on the two sets ${\bigcup_{v,k} \{t \mid \alpha_v^k \leq t < \bar{\alpha}_v^k\}}$ and ${[T] \setminus \bigcup_{v,k} \{t \mid \alpha_v^k \leq t < \bar{\alpha}_v^k\}}$, and whether the optimal base-stock level is at most as large as or greater than $\tau_v^k$, writing ${\Rb(T) = \sum_{t \in [T]} \left(\mu_t(\tau_t) - \mu_t(\taugammaopt)\right) = R_{\tau_0}(T) + R_1(T) + R_2(T)}$ with
\begin{align*}
    R_{\tau_0}(T) \coloneq \ \ \ \sum_{\mathclap{t \in \bigcup_{v,k} \{t' \mid \alpha_v^k \leq t' < \bar{\alpha}_v^k\}}} \  \left(\mu_t(\tau_t) - \mu_t (\taugammaopt) \right), \ \
    R_1(T) \coloneq \ \ \ \sum_{\mathclap{\substack{t \in \bigcup_{v,k} \{t' \mid \bar{\alpha}_v^k \leq t' < \alpha_v^{k+1}\}, \\ \taugammaopt \leq \tau_v^k}}} \ \left(\mu_t(\tau_t) - \mu_t (\taugammaopt) \right), \ \ 
    R_2(T) \coloneq \ \ \ \sum_{\mathclap{\substack{t \in \bigcup_{v,k} \{t' \mid \bar{\alpha}_v^k \leq t' < \alpha_v^{k+1}\}, \\ \taugammaopt > \tau_v^k}}} \ \left(\mu_t(\tau_t) - \mu_t (\taugammaopt) \right). 
\end{align*}

Next, we provide the regret analysis for each of the terms $R_{\tau_0}(T), R_1(T)$ and $R_2(T)$.

\noindent\textbf{Bound of $R_{\tau_0}(T)$.}\quad
    By \cref{lemma:probability_good_event}, the concentration bound in \cref{lemma:concentration} does not apply in the inventory depletion time steps. We can, however, bound the expected number of such time steps as well as the cost incurred in each of these time steps.
    First, we bound the expected number of time steps of inventory reduction within an episode~$v$, that is, ${\abs*{\bigcup_{k} \{t \mid \alpha_v^k \leq t < \bar{\alpha}_v^k\}}}$.
    Note that the expected number of steps required to deplete an amount ${\tau_t - \tau_{t+1}}$ of inventory is ${\mathds{1}\{\tau_t \neq \tau_{t+1}\} L + \left(\tau_t - \tau_{t+1}\right)\nu}$. This holds since after at most $L$ time steps of ordering zero units the pipeline of outstanding orders has cleared and the inventory on-hand is at most~$\tau_t$.  By \Cref{ass:main_assumption}.\ref{ass:bounded_depletion_steps}, it takes another $(\tau_t - \tau_{t+1}) \nu$ steps until the inventory position has fallen down to $\tau_{t+1}$. Therefore, we have that 
    \begin{align*}
        \textstyle \E\left[\bigg|\bigcup_{k} \{t \mid \alpha_v^k \leq t < \bar{\alpha}_v^k\}\bigg|\right] &\leq \textstyle \sum_{t \in [t_v, t_{v+1}-1]} \left(\mathds{1}\{\tau_t \neq \tau_{t+1}\} L + \left(\tau_t - \tau_{t+1}\right) \nu \right) \\
        &\leq \textstyle L \sum_{t \in [t_v, t_{v+1}-1]} \mathds{1}\{\tau_t \neq \tau_{t+1}\} + \nu \sum_{t \in [t_v, t_{v+1}-1]} \left(\tau_t - \tau_{t+1}\right).
    \end{align*}
    The number of times the base-stock level is decreased in episode~$v$ is bounded as 
        ${\sum_{t \in [t_v, t_{v+1}-1]} \mathds{1}\{\tau_t \neq \tau_{t+1}\} \leq \abs*{\A_{\gamma}} - 1 \leq \lceil U/\gamma \rceil}$.
    Moreover, we know that the sequence of implemented base-stock levels over the epochs in ${[t_v, t_{v+1}-1]}$ is non-decreasing. Therefore, we have that the total amount by which inventory is reduced accumulated over the steps in episode~$v$ is $\sum_{t \in [t_v, t_{v+1}-1]} \left(\tau_t - \tau_{t+1}\right) \leq U$.    
    This implies 
        ${\E\left[\Big|\bigcup_{k} \{t \mid \alpha_v^k \leq t < \bar{\alpha}_v^k\}\Big|\right] \leq L \lceil U/\gamma \rceil + U\nu}$.
    
    Finally, by Lipschitz continuity of $\mu_t(\cdot)$ (\cref{lemma:lipschitz}), in each step a regret of at most ${U \max \{ h, b \}}$ is incurred. Since the total number of episodes~$\hat{S}$ is at most $S$ (\cref{lemma:false_alarms}), summing over all episodes yields the expected regret upper bound
    \begin{align*}
        R_{\tau_0}(T) = \ \ \ \sum_{\mathclap{t \in \bigcup_{v,k} \{t' \mid \alpha_v^k \leq t' < \bar{\alpha}_v^k\}}} \ \ \left(\mu_t(\tau_t) - \mu_t (\taugammaopt) \right) 
        \leq \sum_{v=1}^{\hat{S}} \left( L \lceil U/\gamma \rceil + U \nu\right) U\max\{ h, b \} \leq \left(L\lceil U/\gamma \rceil + U \nu \right) S U \max\{h, b\}.
    \end{align*}

\noindent\textbf{Bound of $R_1(T)$.}\quad
    Let ${t \in I_{v,i}}$ for some stationary interval ${I_{v,i}=[\beta_{v,i}, \beta_{v, i+1}-1]}$ in episode~$v$. Furthermore, let ${\bar{\alpha}_{v,k} \leq t < \alpha_{v,k+1}}$ for an epoch~$k$ with ${\taugammaopt \leq \tau_v^k}$. To bound the regret in these time steps, we use the fact that the costs of both base-stock level $\taugammaopt$ and $\tau_v^k$ are observed in each time step. 

    We know that base-stock level ${\tau_v^k = \sup \Avk{v}{k}}$ was not evicted before the start of epoch~$k+1$. This implies that $\tau_v^k$ does not satisfy the {\em elimination condition}~\eqref{eq:elimintation_condition_lsl}, that is,
    \vspace{-5pt}
    \begin{equation} \label{eq:R1_first_inequality_lsl}
        \hat{\mu}(\tau_v^k, s, t) - \min_{\tau \in \A_{\gamma}, \tau \leq \tau_v^k} \hat{\mu}(\tau, s, t) \leq 4 b_{s, t} \text{ for all } s \in [\bar{\alpha}_{v,k},t],
    \end{equation}
    or base-stock level $\tau_v^k - \gamma$ violates the {\em separation condition}~\eqref{eq:separation_condition}, and consequently
    \vspace{-5pt}
    \begin{equation} \label{eq:R1_second_inequality_lsl}
        \hat{\mu}(\tau_v^k - \gamma, \bar{\alpha}_{v,k}, t) - \min_{\tau \in \A_{\gamma}, \tau \leq \tau_v^k} \hat{\mu}(\tau, \bar{\alpha}_{v,k}, t) \leq 2b_{\bar{\alpha}_{v,k}, t} + \max \{ h, b\} \gamma.
    \end{equation}
    We consider both cases separately.

    \noindent\textbf{Case: $\tau_v^k$ Does Not Satisfy the Elimination Condition.}\quad
    We start with the case in which we assume that \cref{eq:R1_first_inequality_lsl} holds. Here, base-stock level $\tau_v^k$ was not evicted from $\Avk{v}{k}$ because its empirical mean cost is not sufficiently separated from that of the empirically best level. Then, for ${s = \bar{\alpha}_{v,k}}$ in \cref{eq:R1_first_inequality_lsl} and with $\min_{\tau \in \A_{\gamma}, \tau \leq \tau_v^k} \hat{\mu}(\tau, \bar{\alpha}_{v,k}, t) \leq \hat{\mu}(\taugammaopt, \bar{\alpha}_{v,k}, t)$ since $\taugammaopt \leq \tau_v^k$, we have $\hat{\mu}(\tau_v^k, \bar{\alpha}_{v,k}, t) - \hat{\mu}(\taugammaopt, \bar{\alpha}_{v,k}, t) \leq 4 b_{\bar{\alpha}_{v,k}, t}$.
    Because we consider time steps ${t \in [\bar{\alpha}_{v,k}, \alpha_{v,k} - 1]}$ within stationary interval $I_{v,i}$ in which ${\taugammaopt \leq \tau_v^k}$, following \cref{lemma:probability_good_event} it holds that
        ${\abs*{\hat{\mu}(\tau_v^k, \bar{\alpha}_{v,k}, t) - \mu_t(\tau_v^k)} \leq b_{\bar{\alpha}_{v,k}, t}}$ and ${\abs*{\hat{\mu}(\taugammaopt, \bar{\alpha}_{v,k}, t) - \mu_t(\taugammaopt)} \leq b_{\bar{\alpha}_{v,k}, t}}$.
    As a result, we can bound the regret per time step for all ${\bar{\alpha}_{v,k} \leq t < \alpha_{v,k+1}}$ by
        $\mu_{t}(\tau_v^k) - \mu_t(\taugammaopt) \leq 6b_{\bar{\alpha}_{v,k}, t}$.
    
    To determine the cumulative regret over all time steps in ${\bigcup_{k} \{t \in I_{v,i} \mid \bar{\alpha}_v^k \leq t < \alpha_v^{k+1}\}}$, we consider the set of epochs which start or end in interval~$I_{v,i}$. The number of such epochs is upper bounded by the total number of epochs in episode~$v$, which is at most ${\abs*{\mathcal{A}_{\gamma}} \leq \lceil U/\gamma\rceil + 1}$. The number of time steps in these epochs is no larger than $\abs{I_{v,i}}$. With Jensen's inequality, we get the following regret over the time steps in the interval
    \begin{align*}
        \ \ \sum_{\mathclap{\substack{t \in \bigcup_{k} \{t' \in I_{v,i} \mid \bar{\alpha}_v^k \leq t' < \alpha_v^{k+1}\}, \\ \taugammaopt \leq \tau_v^k}}} \quad \left( \mu_{t}(\tau_v^k) - \mu_{t}(\taugammaopt) \right) & \leq 6\Hls \sqrt{2 \log(2T^2U/\delta\gamma)} \sum_{k : [\bar{\alpha}_{v,k}, \alpha_{v,k+1}-1] \cap I_{v,i} \neq \emptyset} \sum_{t=1}^{\alpha_{v,k+1} - \bar{\alpha}_{v,k}} \sqrt{\frac{1}{t}} \\
        & \textstyle \leq 12\Hls \sqrt{2 \log(2T^2U/\delta\gamma)\abs*{I_{v,i}}\left( \Big\lceil \frac{U}{\gamma} \Big\rceil + 1 \right)}.
    \end{align*}
    Since there are at most $2S-1$ intervals $I_{v,i}$ of length at most $T$, it follows the regret bound
    \begin{equation*}
        \sum_{v=1}^{\hat{S}} \sum_{i=1}^{S_v} \sum_{\substack{t \in \bigcup_{k} \{t' \in I_{v,i} \mid \bar{\alpha}_v^k \leq t' < \alpha_v^{k+1}\}, \\ \taugammaopt \leq \tau_v^k}} \left( \mu_t(\tau_v^k) - \mu_t(\taugammaopt) \right) \leq 24 \Hls \sqrt{ST \log(2T^2U/\delta\gamma)(U / \gamma + 2)}.
    \end{equation*}

    \noindent\textbf{Case: $\tau_v^k$ Does Not Satisfy the Separation Condition.}\quad
    Next, we consider the time steps in which \cref{eq:R1_second_inequality_lsl} is satisfied. Here, the base-stock level $\tau_v^k$ was not eliminated from the active set due to the {\em separation condition}~\eqref{eq:separation_condition}.
    Using $\underset{\tau \in \A_{\gamma}, \tau \leq \tau_v^k}{\min} \hat{\mu}(\tau, \bar{\alpha}_{v,k}, t) \leq \hat{\mu}(\taugammaopt, \bar{\alpha}_{v,k}, t)$, it follows from \cref{eq:R1_second_inequality_lsl} that
    \[
        \hat{\mu}(\tau_v^k - \gamma, \bar{\alpha}_{v,k}, t) - \hat{\mu}(\taugammaopt, \bar{\alpha}_{v,k}, t) \leq 2b_{\bar{\alpha}_{v,k}, t} + \max \{ h, b\} \gamma.
    \]
    By \cref{lemma:probability_good_event} and since both $\tau_v^k - \gamma$ and $\taugammaopt$ are at most as large as $\tau_v^k$, we have that
    \[
        \abs*{\hat{\mu}(\tau_v^k - \gamma, \bar{\alpha}_{v,k}, t) - \mu_t(\tau_v^k - \gamma)} \leq b_{\bar{\alpha}_{v,k}, t} \quad \text{and} \quad \abs*{\hat{\mu}(\taugammaopt, \bar{\alpha}_{v,k}, t) - \mu_t(\taugammaopt)} \leq b_{\bar{\alpha}_{v,k}, t}
    \]
    and we can bound the regret per time step for ${\bar{\alpha}_{v,k} \leq t < \alpha_{v,k+1}}$ by ${\mu_t(\tau_v^k - \gamma) - \mu_t(\taugammaopt) \leq 4b_{\bar{\alpha}_{v,k}, t} + \max \{ h, b\} \gamma}$.
    Using the Lipschitz continuity of $\mu_t(\cdot)$ as per \cref{lemma:lipschitz}, we obtain ${\mu_t(\tau_v^k) - \mu_t(\taugammaopt) \leq \mu_t(\tau_v^k - \gamma) - \mu_t(\taugammaopt) + \max\{h,b\} \gamma \leq 4b_{\bar{\alpha}_{v,k},t} + 2 \max\{h,b\} \gamma}$.
    Applying the same arguments as in the first case yields the following bound on the regret accumulated over the interval
    \vspace{-5pt}
    \begin{align*}
        \qquad \sum_{\mathclap{\substack{t \in \bigcup_{k} \{t' \in I_{v,i} \mid \bar{\alpha}_v^k \leq t' < \alpha_v^{k+1}\}, \\ \taugammaopt \leq \tau_v^k}}} \ \ \left( \mu_{t}(\tau_v^k) - \mu_{t}(\taugammaopt) \right) 
        &\leq 4\Hls \sqrt{2 \log(2T^2U/\delta\gamma)} \quad \sum_{\mathclap{\substack{k \geq 1: \\ [\bar{\alpha}_{v,k}, \alpha_{v,k+1}-1] \cap I_{v,i} \neq \emptyset}}} \quad \sum_{t=1}^{\alpha_{v,k+1} - \bar{\alpha}_{v,k}} \sqrt{\frac{1}{t}} + \qquad \quad \sum_{\mathclap{\substack{t \in \bigcup_{k} \{t' \in I_{v,i} \mid \bar{\alpha}_v^k \leq t' < \alpha_v^{k+1}\}, \\ \taugammaopt \leq \tau_v^k}}} \ \ 2 \max\{h,b\} \gamma \\
        & \leq \textstyle 8\Hls \sqrt{2 \log(2T^2U/\delta\gamma)\abs*{I_{v,i}}\left( \Big\lceil \frac{U}{\gamma} \Big\rceil + 1 \right)} + 2\abs{I_{v,i}}\max\{h,b\}\gamma.
    \end{align*}
    Finally, by Jensen's inequality, the sum over at most ${2S - 1}$ intervals with a total length of $T$ is bounded via
    \vspace{-5pt}
    \begin{equation*}
        \sum_{v=1}^{\hat{S}} \sum_{i=1}^{S_v} \sum_{\substack{t \in \bigcup_{k} \{t' \in I_{v,i} \mid \bar{\alpha}_v^k \leq t' < \alpha_v^{k+1}\}, \\ \taugammaopt \leq \tau_v^k}} \left( \mu_t(\tau_v^k) - \mu_t(\taugammaopt) \right) \leq 16 \Hls \sqrt{ST \log(2T^2U/\delta\gamma)(U / \gamma + 2)} + 2T\gamma \max\{h,b\}.
    \end{equation*}

    \noindent\textbf{Final Bound of $R_1(T)$.}\quad
    The regret component $R_1(T)$ is therefore bounded as
    \[
        R_{1}(T) = \sum_{\substack{t \in \bigcup_{v,k} \{t' \mid \bar{\alpha}_v^k \leq t' < \alpha_v^{k+1}\}, \\ \taugammaopt \leq \tau_v^k}} \left(\mu_t(\tau_t) - \mu_t (\taugammaopt) \right) \leq 24 \Hls \sqrt{ST \log(2T^2U/\delta\gamma)(U / \gamma + 2)} + 2T\gamma \max\{h,b\}.
    \]

\noindent\textbf{Bound of $R_2(T)$.}\quad
    Let ${t \in I_{v,i}}$ for some stationary interval ${I_{v,i} = [\beta_{v,i}, \beta_{v,i+1} - 1]}$ in episode~$v$. Moreover, let ${\bar{\alpha}_{v,k} \leq t < \alpha_{v,k+1}}$ for some epoch ${k \geq 2}$ and ${\tau_v^k < \taugammaopt}$. Note that in the first epoch of each episode it holds that ${\taugammaopt \leq \tau_v^k}$ since ${\tau_v^k = U}$ and hence the regret in these epochs is captured by the term $R_1(T)$.
    In the following, we first show that no new epoch is triggered in such time steps~$t$ and hence interval~$I_{v,i}$ lies entirely within epoch~$k$. We then use this fact to upper bound the length of $I_{v,i}$, after which we bound the regret per-time step given that the base-stock level implemented must always satisfy the {\em separation condition}~\eqref{eq:separation_condition}.

    We begin by showing that the next epoch~$k+1$ is not started in interval $I_{v,i}$. This is because if ${\tau_v^k < \taugammaopt}$, then by convexity we have that ${\mu_t(\tau_v^k) \leq \mu_t(\tau)}$ for all ${\tau \leq \tau_v^k}$ and together with the Lipschitz continuity of the cost function (\cref{lemma:lipschitz}) we have that
    \vspace{-5pt}
    \begin{align*}
        \mu_t(\tau_v^k) \leq \mu_t \Big( \underset{\tau \in \A_{\gamma}, \tau \leq \tau_v^k}{\arg \min}\, \hat{\mu}(\tau, \bar{\alpha}_{v,k}, t) \Big) \quad 
       \text{and} \quad \mu_t(\tau_v^k - \gamma) \leq \mu_t \Big( \underset{\tau \in \A_{\gamma}, \tau \leq \tau_v^k}{\arg \min}\, \hat{\mu}(\tau, \bar{\alpha}_{v,k}, t) \Big) + \max\{h,b\} \gamma.
    \end{align*}
    It follows with \cref{lemma:probability_good_event} that ${\hat{\mu}(\tau_v^k - \gamma, \bar{\alpha}_{v,k}, t) - \underset{\tau \in \A_{\gamma}, \tau \leq \tau_v^k}{\min} \hat{\mu}(\tau, \bar{\alpha}_{v,k}, t) \leq 2b_{\bar{\alpha}_{v,k}, t} + \max\{h,b\} \gamma}$.
    Therefore, the {\em separation condition}~\eqref{eq:separation_condition} is violated by base-stock level ${\tau_v^k - \gamma}$. Thus, $\tau_v^k$ is not eliminated from the active set $\Avk{v}{k}$ and hence sampled throughout the entire interval~$I_{v,i}$, until either the next change that causes the optimal base-stock level to be smaller than $\tau_v^k$ occurs, or until the change at $\beta_{v,i}$ is detected. This means, we have ${[\beta_{v,i}, \beta_{v,i+1}] \subseteq [\alpha_{v,k}, \alpha_{v,k+1} - 1]}$. For notational clarity, we furthermore assume that ${\beta_{v,i} \leq \bar{\alpha}_{v,k}}$, that is, inventory and orders are reduced to the target base-stock level $\tau_v^k$ after step $\beta_{v,i}$. Otherwise, all following arguments hold with $\bar{\alpha}_{v,k}$ replaced by $\beta_{v,i}$.
    
    As a result, the estimators $\hat{\mu}(\tau, \bar{\alpha}_{v,k}, t)$ are well-defined for all $\tau \leq \tau_v^{k}$ and ${\bar{\alpha}_{v,k} < t < \beta_{v,i+1}}$, and our concentration guarantees (\cref{lemma:probability_good_event}) apply since $\tau_v^{k}$ does not change between step $\alpha_{v,k}$ and any ${t < \beta_{v,i+1}}$. The same holds for estimators $\hat{\mu}(\tau, \bar{\alpha}_{v,k-1}, \alpha_{v,k})$ with $\tau \leq \tau_v^{k-1}$ and epoch $k-1$ in the same episode because $\tau_v^{k-1}$ is fixed over the time steps in $[\bar{\alpha}_{v,k-1}, \alpha_{v,k} - 1]$.

    \noindent\textbf{Bound of $\abs{I_{v,i}}$.}\quad
    Next, we determine the maximum number of time steps until the change is detected. 
    Let ${\tau' \coloneq \arg\min_{\tau \in \A_{\gamma}, \tau \leq \tau_v^{k-1}} \hat{\mu}(\tau, \bar{\alpha}_{v,k-1}, \alpha_{v,k})}$ represent the base-stock level with the minimum estimated cost in epoch $k-1$.
    Note that since
        ${\hat{\mu}(\tau', \bar{\alpha}_{v,k-1}, \alpha_{v,k}) - \min_{\tau \in \A_{\gamma}, \tau \leq \tau_v^{k-1}} \hat{\mu}(\tau, \bar{\alpha}_{v,k-1}, \alpha_{v,k}) = 0 < 4 b_{\bar{\alpha}_{v,k-1}, \alpha_{v,k}}}$
    base-stock level $\tau'$ does not satisfy the {\em elimination condition}~\eqref{eq:elimintation_condition_lsl}, and hence it must hold that ${\tau' \leq \tau_v^k}$.
    
    Using ${\tau' \leq \tau_v^k < \taugammaopt}$ and the convexity of the cost function, we get that ${\mu_t(\taugammaopt) \leq \mu_t(\tau_v^k) \leq \mu_t(\tau')}$ and with \cref{lemma:probability_good_event} that
    \vspace{-5pt}
    \begin{equation} \label{eq:R3_lower_bound_difference_new_estimator}
        \hat{\mu}(\tau_v^k, \bar{\alpha}_{v,k}, t) - \hat{\mu}(\tau', \bar{\alpha}_{v,k}, t) \leq \mu_t(\tau_v^k) - \mu_t(\tau') + 2b_{\bar{\alpha}_{v,k}, t} \leq 2b_{\bar{\alpha}_{v,k}, t}.
    \end{equation}
    Furthermore, the change has not been detected at time $t$. Therefore, the {\em change point condition}~\eqref{eq:change_condition_lsl} is not satisfied, neither for $\tau_v^k$ nor for $\tau'$ and it holds that
    \vspace{-5pt}
    \begin{align} \label{eq:R3_change_condition_good_1}
        \abs*{\hat{\mu}(\tau',\bar{\alpha}_{v,k}, t) - \hat{\mu}(\tau',  \bar{\alpha}_{v,k-1}, \alpha_{v,k})} &\leq b_{\bar{\alpha}_{v,k}, t} + b_{ \bar{\alpha}_{v,k-1}, \alpha_{v,k}} \\  \label{eq:R3_change_condition_good_2}
        \text{and} \quad \abs*{\hat{\mu}(\tau_v^k,\bar{\alpha}_{v,k}, t) - \hat{\mu}(\tau_v^k, \bar{\alpha}_{v,k-1}, \alpha_{v,k})} &\leq b_{\bar{\alpha}_{v,k}, t} + b_{\bar{\alpha}_{v,k-1}, \alpha_{v,k}}.
    \end{align}
    From \crefrange{eq:R3_lower_bound_difference_new_estimator}{eq:R3_change_condition_good_2} it follows
    \begin{align*}
        &\hat{\mu}(\tau_v^k, \bar{\alpha}_{v,k-1}, \alpha_{v,k}) - \hat{\mu}(\tau', \bar{\alpha}_{v,k-1}, \alpha_{v,k}) \\ &\leq \abs*{\hat{\mu}(\tau_v^k, \bar{\alpha}_{v,k-1}, \alpha_{v,k}) - \hat{\mu}(\tau_v^k,\bar{\alpha}_{v,k}, t)} + \abs*{\hat{\mu}(\tau_v^k,\bar{\alpha}_{v,k}, t) - \hat{\mu}(\tau',\bar{\alpha}_{v,k}, t)} + \abs*{\hat{\mu}(\tau',\bar{\alpha}_{v,k}, t) - \hat{\mu}(\tau', \bar{\alpha}_{v,k-1}, \alpha_{v,k})} \\
        &\leq 4b_{\bar{\alpha}_{v,k}, t} + 2b_{ \bar{\alpha}_{v,k-1}, \alpha_{v,k}}.
    \end{align*}
    Moreover, we know that the {\em separation condition}~\eqref{eq:separation_condition} is satisfied by $\tau_v^k$ over the interval ${[\bar{\alpha}_{v,k-1}, \alpha_{v,k} - 1]}$ since that base-stock level has not been evicted at the termination of epoch $k-1$, and consequently
    \begin{equation} \label{eq:R3_lower_bound_difference_old_estimator}
        \hat{\mu}(\tau_v^k, \bar{\alpha}_{v,k-1}, \alpha_{v,k}) - \hat{\mu}(\tau', \bar{\alpha}_{v,k-1}, \alpha_{v,k}) > 2b_{ \bar{\alpha}_{v,k-1}, \alpha_{v,k}} + \max \{ h, b\} \gamma.
    \end{equation}
    Comparing upper and lower bound yields ${4b_{\bar{\alpha}_{v,k}, t} + 2b_{ \bar{\alpha}_{v,k-1}, \alpha_{v,k}} > 2b_{ \bar{\alpha}_{v,k-1}, \alpha_{v,k}} + \max \{ h, b\} \gamma}$ and hence
        ${4b_{\bar{\alpha}_{v,k}, t} > \max \{ h, b\} \gamma}$.
   
    As this holds for all ${t \in \bigcup_{k} \{t' \in I_{v,i} \mid \bar{\alpha}_v^k \leq t' < \alpha_v^{k+1}\}}$, in particular the last step ${t = \beta_{v, i+1} - 1}$, with ${b_{\bar{\alpha}_{v,k}, \beta_{v,i+1}} = \Hls \sqrt{\frac{2\log(2T^2U/\delta\gamma)}{\beta_{v,i+1} - \bar{\alpha}_{v,k}}}}$ the maximum number of steps until the change is detected is given by
    \begin{align} \label{eq:R3_upper_bound_length_I}
        \abs*{I_{v,i}} \leq \beta_{v,i+1} - \bar{\alpha}_{v,k} = \frac{2\log(2T^2U/\delta\gamma)}{(b_{\bar{\alpha}_{v,k}, \beta_{v,i+1}}/\Hls)^2} 
        \leq \frac{32 \log(2T^2U/\delta\gamma) \Hls^2}{\max \{ h, b\}^2 \gamma^2}.
    \end{align}
    
    \noindent\textbf{Bound of the Regret per Time Step.}\quad 
    Next, we bound the instantaneous regret $\mu_t(\tau_v^k) - \mu_t(\taugammaopt)$. 
    We begin by writing $\tau_v^k$ in terms of $\tau'$ and $\taugammaopt$ as
    \[
        \tau_v^k = \frac{\taugammaopt - \tau'}{\taugammaopt - \tau'} \tau_v^k = \frac{\taugammaopt \tau' - \tau_v^k \tau' + \tau_v^k \taugammaopt - \taugammaopt \tau'}{\taugammaopt - \tau'} = \frac{\taugammaopt - \tau_v^k}{\taugammaopt - \tau'} \tau'+ \frac{\tau_v^k - \tau'}{\taugammaopt - \tau'} \taugammaopt.
    \]
    Then, since $\tau' \leq \tau_v^k < \taugammaopt$ and by convexity we have that
    \begin{align} \label{eq:R3_bound_regret_per_t}
         \mu_t(\tau_v^k) &\leq \frac{\taugammaopt - \tau_v^k}{\taugammaopt - \tau'} \mu_t(\tau') + \frac{\tau_v^k - \tau'}{\taugammaopt - \tau'} \mu_t(\taugammaopt) \notag \\
          \left((\taugammaopt - \tau_v^k) + (\tau_v^k - \tau')\right) \mu_t(\tau_v^k) &\leq (\taugammaopt - \tau_v^k) \mu_t(\tau') + (\tau_v^k - \tau') \mu_t(\taugammaopt) \notag \\
          (\tau_v^k - \tau')\mu_t(\tau_v^k) - (\tau_v^k - \tau') \mu_t(\taugammaopt) &\leq (\taugammaopt - \tau_v^k) \mu_t(\tau') - (\taugammaopt - \tau_v^k) \mu_t(\tau_v^k) \notag \\
         \mu_t(\tau_v^k) - \mu_t(\taugammaopt) &\leq \frac{\taugammaopt - \tau_v^k}{\tau_v^k - \tau'} \left(\mu_t(\tau') - \mu_t(\tau_v^k)\right).
    \end{align}
    
    What remains is an upper bound on the right-hand side of \cref{eq:R3_bound_regret_per_t}. We trivially upper bound the difference $\taugammaopt - \tau_v^k$ by $U$ and lower bound the denominator $\tau_v^k - \tau'$ using the Lipschitz continuity (\cref{lemma:lipschitz}), the concentration inequality in \cref{lemma:probability_good_event}, and the {\em separation condition} in \cref{eq:R3_lower_bound_difference_old_estimator} by
    \begin{align*}
        \tau_v^k - \tau' &\geq \left(\mu_{\bar{\alpha}_{v,k-1}}(\tau_v^k) - \mu_{\bar{\alpha}_{v,k-1}}(\tau')\right) / \max\{ h, b\} \\
        &\geq \left(\hat{\mu}(\tau_v^k, \bar{\alpha}_{v,k-1}, \alpha_{v,k}) - \hat{\mu}(\tau', \bar{\alpha}_{v,k-1}, \alpha_{v,k}) - 2 b_{\bar{\alpha}_{v,k-1}, \alpha_{v,k}}\right) / \max\{h ,b\} \geq \gamma.
    \end{align*}
     Lastly, we establish a bound for the difference between $\mu_t(\tau')$ and $\mu_t(\tau_v^k)$. By \cref{eq:R3_lower_bound_difference_old_estimator} we have that $\hat{\mu}(\tau', \bar{\alpha}_{v,k-1}, \alpha_{v,k}) + 2b_{\bar{\alpha}_{v,k-1}, \alpha_{v,k}} < \hat{\mu}(\tau_v^k, \bar{\alpha}_{v,k-1}, \alpha_{v,k}) $ and consequently
    \begin{align*}
        \hat{\mu}(\tau', \bar{\alpha}_{v,k}, t) - b_{\bar{\alpha}_{v,k}, t} & \leq \hat{\mu}(\tau', \bar{\alpha}_{v,k-1}, \alpha_{v,k}) + b_{\bar{\alpha}_{v,k-1}, \alpha_{v,k}} \quad \text{by \cref{eq:R3_change_condition_good_1}} \\
        & < \hat{\mu}(\tau_v^k, \bar{\alpha}_{v,k-1}, \alpha_{v,k}) - b_{\bar{\alpha}_{v,k-1}, \alpha_{v,k}}  \quad \text{by \cref{eq:R3_lower_bound_difference_old_estimator}}\\
        & \leq \hat{\mu}(\tau_v^k, \bar{\alpha}_{v,k}, t) + b_{\bar{\alpha}_{v,k}, t}  \quad \text{by \cref{eq:R3_change_condition_good_2}}.
    \end{align*}
    Therefore, with \cref{lemma:probability_good_event} we get ${\mu_t(\tau') - \mu_t(\tau_v^k) \leq \hat{\mu}(\tau', \bar{\alpha}_{v,k}, t) - \hat{\mu}(\tau_v^k, \bar{\alpha}_{v,k}, t) + 2b_{\bar{\alpha}_{v,k}, t} < 4b_{\bar{\alpha}_{v,k}, t}}$.
    Ultimately, substituting this into \cref{eq:R3_bound_regret_per_t} we obtain the following upper bound for the regret per time step
        ${\mu_t(\tau_v^k) - \mu_t(\taugammaopt) < \frac{U}{\gamma} 4b_{\bar{\alpha}_{v,k}, t}}$.
    
\noindent\textbf{Final Bound of $R_2(T)$.}\quad
    Summing over the time steps in interval~$I_{v,i}$, it follows
    \begin{align*}
        \qquad \sum_{\mathclap{\substack{t \in \bigcup_{k} \{t' \in I_{v,i} \mid \bar{\alpha}_v^k \leq t' < \alpha_v^{k+1}\}, \\ \taugammaopt > \tau_v^k}}} \quad \left( \mu_t(\tau_v^k) - \mu_t(\taugammaopt) \right) &< \frac{U}{\gamma} 4\Hls \sqrt{2 \log(2T^2U/\delta\gamma)} \sum_{t = 1}^{\abs*{I_{v,i}}} \sqrt{\frac{1}{t}} 
        \leq \frac{U}{\gamma} 8\Hls \sqrt{2 \log(2T^2U/\delta\gamma)\abs*{I_{v,i}}}.
    \end{align*}
    On the one hand, we can upper bound the regret summed over all intervals using the Cauchy-Schwarz inequality and the fact that there are at most ${2S-1}$ such intervals~$I_{v,i}$ with a total length of at most $T$ by
    \begin{align*}
      \sum_{\substack{t \in \bigcup_{v,k} \{t' \mid \bar{\alpha}_v^k \leq t' < \alpha_v^{k+1}\}, \\ \taugammaopt > \tau_v^k}} \left( \mu_t(\tau_v^k) - \mu_t(\taugammaopt) \right) &< \sum_{v=1}^{\hat{S}} \sum_{i=1}^{S_v} \frac{U}{\gamma} 8\Hls \sqrt{2 \log\left(\frac{2T^2U}{\delta\gamma}\right) \abs*{I_{v,i}}} 
      \leq \frac{U}{\gamma} 16\Hls \sqrt{ST \log\left(\frac{2T^2U}{\delta\gamma}\right)}.
    \end{align*}    
    On the other hand, with the maximum length of interval~$I_{v,i}$ in \cref{eq:R3_upper_bound_length_I} we can upper bound the regret incurred over the interval by
    \vspace{-5pt}
    \begin{align*}
        \sum_{\substack{t \in \bigcup_{k} \{t' \in I_{v,i} \mid \bar{\alpha}_v^k \leq t' < \alpha_v^{k+1}\}, \\ \taugammaopt > \tau_v^k}} \left( \mu_t(\tau_v^k) - \mu_t(\taugammaopt) \right) &< \frac{U}{\gamma} 8\Hls \sqrt{2 \log(2T^2U/\delta\gamma)} \sqrt{\frac{32\log(2T^2U/\delta\gamma) \Hls^2}{\max \{ h, b\}^2 \gamma^2}} \\[-12pt]
        &= \frac{64U\Hls^2 \log(2T^2U/\delta\gamma)}{\max \{ h, b\} \gamma^2}
    \end{align*}
    and the sum over the up to $2S-1$ intervals by
    \vspace{-5pt}
    \begin{align*}
        \sum_{\substack{t \in \bigcup_{v,k} \{t' \mid \bar{\alpha}_v^k \leq t' < \alpha_v^{k+1}\}, \\ \taugammaopt > \tau_v^k}} \left(\mu_t(\tau_t) - \mu_t (\taugammaopt) \right) < \sum_{v=1}^{\hat{S}} \sum_{i=1}^{S_v} \frac{64U\Hls^2 \log\left(\frac{2T^2U}{\delta\gamma}\right)}{\max \{ h, b\} \gamma^2} \leq \frac{128\Hls^2 S U \log\left(\frac{2T^2U}{\delta\gamma}\right)}{\max \{ h, b\} \gamma^{2}}.
    \end{align*}
    As a result, the regret component $R_2(T)$ is bounded via
    \vspace{-5pt}
    \begin{align*}
        R_{2}(T) = \sum_{\substack{t \in \bigcup_{v,k} \{t' \mid \bar{\alpha}_v^k \leq t' < \alpha_v^{k+1}\}, \\ \taugammaopt > \tau_v^k}} \left(\mu_t(\tau_t) - \mu_t (\taugammaopt) \right) &< \min\bigg\{\frac{U}{\gamma} 16\Hls \sqrt{ST {\textstyle\log\left(\frac{2T^2U}{\delta\gamma}\right)}}, \frac{128\Hls^2 S U \log\left(\frac{2T^2U}{\delta\gamma}\right)}{\max \{ h, b\} \gamma^{2}}\bigg\} \\[-12pt]
        &= \frac{16\Hls U}{\gamma}\min\bigg\{\sqrt{ST {\textstyle\log\left(\frac{2T^2U}{\delta\gamma}\right)}}, \frac{8\Hls S \log\left(\frac{2T^2U}{\delta\gamma}\right)}{\max \{ h, b\} \gamma}\bigg\}.
    \end{align*}

\noindent\textbf{Final Regret Bound.}\quad
    From \cref{lemma:regret_decomposition} and the preceding analysis, we conclude that the expected dynamic regret of \NSLSICL, conditioned on the event $\calE$, is upper bounded by
    \begin{align*}
        \E[R(T) \mid \calE] &= R_{\gamma}(T) + R_{\tau_0}(T) + R_1(T) + R_2(T) \\
        &< 24 \Hls \sqrt{ST {\textstyle\log\left(\frac{2T^2U}{\delta\gamma}\right)\left(\frac{U}{\gamma} + 2\right)}} + \frac{16\Hls U}{\gamma}\min\bigg\{\sqrt{ST {\textstyle\log\left(\frac{2T^2U}{\delta\gamma}\right)}}, \frac{8\Hls S \log\left(\frac{2T^2U}{\delta\gamma}\right)}{\max \{ h, b\} \gamma}\bigg\} \\
        & \qquad + \left(3T\gamma + LSU\lceil U/\gamma \rceil + SU^2 \nu \right) \max\{h,b\}.
    \end{align*}
    Together with ${\Pr(\calE) \geq 1 - \delta}$ (\cref{lemma:probability_good_event}) and ${\Exp{R(T) \mid \calE^{c}} \leq TU \max\{h,b\}}$ by the Lipschitz continuity of the cost function (\cref{lemma:lipschitz}), we obtain
    \begin{align*}
        \Exp{R(T)} &= \Exp{R(T) \mid \calE} \Pr(\calE) + \Exp{R(T) \mid \calE^{c}} \Pr(\calE^{c}) \\
        & < \left(1 - \delta\right) \Bigg(24 \Hls \sqrt{ST {\textstyle\log\left(\frac{2T^2U}{\delta\gamma}\right)\left(\frac{U}{\gamma} + 2\right)}} + \frac{16\Hls U}{\gamma}\min\bigg\{\sqrt{ST {\textstyle\log\left(\frac{2T^2U}{\delta\gamma}\right)}}, \frac{8\Hls S \log\left(\frac{2T^2U}{\delta\gamma}\right)}{\max \{ h, b\} \gamma}\bigg\} \\
        & \qquad \qquad \quad + \left(3T\gamma + LSU\lceil U/\gamma \rceil + SU^2 \nu \right) \max\{h,b\} \Bigg) + \delta TU\max\{h,b\}.
    \end{align*}    
    
    Consequently, with ${\Hls = 72(L+3)U \max\{h,b\}}$ and setting  ${\delta = \Theta\left(T^{-2}\right)}$ as well as ${\gamma = \Theta\left(U(L+1)^{2/3}T^{-1/4}\right)}$ or ${\gamma = \Theta\left(U(L+1)^{2/3}T^{-1/3}\right)}$, respectively, the expected regret is bounded via 
    \[
        \E[R(T)] = \tilde{\mathcal{O}}\left(U(L+1)^{2/3}\min\{ S^{1/2}T^{3/4}, ST^{2/3}\} \right).
    \]
    If the value of $S$ is known in advance, we can set the discretization parameter ${\gamma = \Theta\left(U(L+1)^{2/3}S^{1/3}T^{-1/3}\right)}$ and obtain $\E[R(T)] = \tilde{\mathcal{O}}\left(U (L+1)^{2/3} S^{1/3}T^{2/3}\right)$.
\end{proof}
\section{Performance under Alternative Regret Metric} \label[appendix]{app:regret_transformation}

We compare our main regret benchmark with the following alternative benchmark defined using the expected instantaneous costs incurred by the algorithm
\begin{equation} \label{def:regret_expected_cost}
    R_C(T) \coloneq \sum_{t=1}^T \Big( \E_{F_t}[C_t(\tau_t)] - \min_{\tau \in [0,U]} \mu_t(\tau)\Big).
\end{equation}
The comparator is the same as in \Cref{def:regret}, only the cost assigned to the algorithm is $\E_{F_t}[C_t(\tau_t)]$ instead of $\mu_t(\tau_t)$.
Our main result for this benchmark is the following.

\begin{restatable}{theorem}{AlternativeRegretBenchmark}
\label{thm:alternative_regret_benchmark}
The expected dynamic regret (\cref{def:regret_expected_cost}) of \NSBLIC (\cref{algo:backlog}) is upper bounded via
\[
    \E[R_C(T)] = \begin{cases} \tilde{\mathcal{O}}\left( \sqrt{ST}(L+1) \sigma + S \sqrt{T} LU(L+U)\right) & \text{with } \gamma = \Theta(T^{-1/2}) \\
                                \tilde{\mathcal{O}}\left( \sqrt{ST}\left((L+1) \sigma + LU(L+U)\right)\right) & \text{with } \gamma = \Theta(T^{-1/2}S^{1/2}).
                \end{cases}
\]
For \NSLSIC (\cref{algo:lost_sales}) we have 
\vspace{-7pt}
\[
    \E[R_C(T)] = \tilde{\mathcal{O}}\left( U^{3/2}\sqrt{ST}\right) \ \text{ with } \gamma = \Theta(T^{-1/2}),
\] 
and for \NSLSICL (\cref{algo:lost_sales_lead_time})
\[
    \E[R_C(T)] = \begin{cases} \tilde{\mathcal{O}}\left( U(L+1)^{2/3}S^{1/2}T^{3/4} + UL^{1/3}ST^{1/4} \right) & \text{with } \gamma = \Theta(U(L+1)^{2/3}T^{-1/4}) \\
                               \tilde{\mathcal{O}}\left( U(L+1)^{2/3}ST^{2/3} \right) & \text{with } \gamma = \Theta(U(L+1)^{2/3}T^{-1/3}) \\
                               \tilde{\mathcal{O}}\left( U(L+1)^{2/3}S^{1/3}T^{2/3} + UL^{1/3}S^{2/3}T^{1/3} \right) & \text{with } \gamma = \Theta(U(L+1)^{2/3}S^{1/3}T^{-1/3}).
                \end{cases}
\]
\end{restatable}

The proof of \Cref{thm:alternative_regret_benchmark} follows by combining our regret bounds for $R(T)$ in \cref{thm:regret_backlogging,thm:regret_LS_L_zero,thm:regret_LS_L_positive} with the following transformation lemma, which controls the transient state correction between the two benchmarks.

\begin{restatable}[Regret Transformation]{lemma}{RegretTransformationLemma} \label{lemma:regret_transformation}
    The difference between $R_C(T)$ and the dynamic regret $R(T)$ in \cref{def:regret} satisfies
    \begin{align*}
        \E[R_C(T) - R(T)] = \begin{cases}        
            \mathcal{O}\left( SU \left( \frac{L(U+L)}{\gamma} + U + L \right) \right) \text{ under backlogging} \\[6pt]
            \mathcal{O}\left( SU\left(\frac{LU}{\gamma} + U + L\right) \right) \text{ under lost-sales.}
        \end{cases}
    \end{align*}
\end{restatable}

\subsection{\texorpdfstring{Proof of \Cref{lemma:regret_transformation} (Regret Transformation)}{Proof of Regret Transformation Lemma}}

We start off with a decomposition of the expected difference between the two regret metrics.

\begin{lemma}\label{lemma:decomposition_regret_difference}
    We have $\E[R_C(T) - R(T)] = \Delta_{\tau_0} + \Delta_1$ where
    \begin{align*}
        \Delta_{\tau_0} \coloneq \E\Bigl[\sum_{v=1}^{\hat{S}} \sum_{k \geq 2} \sum_{t \in [\alpha_{v,k}, \bar{\alpha}_{v,k} -1]} \left(C_t(\tau_v^k) - \mu_t(\tau_v^k)\right)\Bigr] \ \text{ and } \
        \Delta_1 \coloneq \E\Bigl[\sum_{v=1}^{\hat{S}} \sum_{k \geq 1} \sum_{t \in [\bar{\alpha}_{v,k}, \alpha_{v,k+1} - 1]} \left(C_t(\tau_v^k) - \mu_t(\tau_v^k)\right)\Bigr].
    \end{align*}
\end{lemma}
\begin{proof}
    We split each epoch into a (possibly empty) depletion and a non-depletion period. In epoch $k \geq 2$ of episode $v$, the depletion period is $[\alpha_{v,k},\bar{\alpha}_{v,k}-1]$, during which orders are suspended until the inventory position falls to the new target $\tau_v^k$. On the remaining period $[\bar{\alpha}_{v,k},\alpha_{v,k+1}-1]$, the inventory position is maintained at $\tau_v^k$. Since the comparator $\min_{\tau \in [0,U]} \mu_t(\tau)$ cancels out, we have
    \[
        \E[R_C(T) - R(T)] = \E\Bigl[\sum_{t=1}^T \left(C_t(\tau_t)-\mu_t(\tau_t)\right)\Bigr]
        =
        \Delta_{\tau_0}+\Delta_1,
    \]
    where $\Delta_{\tau_0}$ collects the depletion periods and $\Delta_1$ collects the non-depletion periods.
\end{proof}

We next use the Lipschitz structure of the one-period cost to relate the cost mismatch to the distance between the actual inventory level before demand and a stationary counterpart.
\begin{lemma}\label{lemma:lipschitz_in_s}
    Let $s_t$ denote the inventory level after replenishment in step $t$ under policy $\tau_t$ and demand distribution $F_t$, and let $s_t^{\mathrm{stat}}$ denote the inventory position in the corresponding stationary system. Then, we have that
    \[
        \abs*{\E_{F_t}[C_t(\tau_t)] - \mu_t(\tau_t)} \leq \max\{h, b\} \cdot \E_{F_t}\left[\abs*{s_t - s^{\mathrm{stat}}_t}\right].
    \]
\end{lemma}
\begin{proof}
    We first show that the immediate cost $C_t$ (\cref{eq:cost_function}) is Lipschitz continuous in inventory after replenishment $I_t + Q_{t-L}$ with factor $\max\{h,b\}$.
    Let $C_t(s_t, D_t)$ be the cost incurred when the inventory before demand is $s_t = I_t + Q_{t-L}$ and demand $D_t$ is realized.
    From \cref{eq:cost_in_lipschitz_proof} in the proof of the Lipschitz continuity in the base-stock level (\cref{lemma:lipschitz}) it follows for two inventory levels $s_t$ and $s'_t$ that
    \begin{align*}
        \abs*{C_t(s_t, D_t) - C_t(s'_t, D_t)} &\leq \frac{h+b}{2} \abs*{\abs*{s_t-D_t} - \abs*{s'_t-D_t}} + \frac{\abs*{h-b}}{2} \abs*{s_t - s'_t} \\
        &\leq \frac{h+b}{2} \abs*{s_t - s'_t} + \frac{\abs*{h-b}}{2} \abs*{s_t - s'_t} = \max\{h,b\} \abs*{s_t - s'_t}.    
    \end{align*}
    In particular, the Lipschitz continuity implies
        $\abs*{C_t(s_t,D_t)-C_t(s_t^{\mathrm{stat}},D_t)} \leq \max\{h,b\} \abs*{s_t-s_t^{\mathrm{stat}}}$.
    
    By the ergodicity of the inventory process under policy $\tau_t$, the long-run average cost equals the expected immediate cost under the unique stationary distribution, that is, ${\mu_t(\tau_t)=\E_{F_t}\left[C_t(s_t^{\mathrm{stat}},D_t)\right]}$ (see \cref{eq:stationary_distribution_cost_function_backlogging} for backlogging, \citet{huh2009adaptive} for lost-sales).
    Taking expectations on the above inequality and using the triangle inequality yields the result.
\end{proof}

For the analysis of $\Delta_1$, we partition the non-depletion steps 
\(
    \bigcup_{v,k}[\bar{\alpha}_{v,k}, \alpha_{v,k+1} - 1]
\)
into maximal intervals on which both the demand distribution and the implemented base-stock level are fixed. Equivalently, an interval starts whenever a non-depletion period begins, at some $\bar{\alpha}_{v,k}$, or whenever a demand change occurs within a non-depletion period. We denote these intervals by $K_1,\ldots,K_J$, ordered by their start times, and write $\zeta_j=\min K_j$ to be the start point of interval $K_j$.  By construction, the inventory system evolves over each interval under a fixed demand distribution and base-stock level.  Thus we can write $\Delta_1$ in \cref{lemma:decomposition_regret_difference} as
\begin{equation} \label{eq:Deltaone}
    \Delta_1 = \E\biggl[\sum_{j=1}^J \sum_{t \in K_j} \left(C_t(\tau_{\zeta_j}) - \mu_t(\tau_{\zeta_j})\right)\biggr].
\end{equation}

Next, we bound the number $J$ of intervals $K_j$. There are $S-1$ demand change points in total. Moreover, within each episode the base-stock levels decrease monotonically, so there are at most $\abs{\mathcal{A}_{\gamma}}-1 \leq \lceil U/\gamma\rceil$ epoch changes per episode. Since the number of episodes is at most $S$ by \cref{lemma:false_alarms}, the total number of epoch starts is at most $S\lceil U/\gamma\rceil$. Each interval starts either at a demand change or an epoch start, and thus
\vspace{-5pt}
\begin{equation} \label{eq:J}
    J \leq S\left(\lceil U/\gamma\rceil+1\right)-1 .
\end{equation}
We bound $\Delta_{\tau_0}$ and $\Delta_1$ for the two inventory models separately.

\begin{rproofof}{\Cref{lemma:regret_transformation}}
\noindent\textbf{Bound of $\Delta_1$ under Backlogging.}\quad
    By construction, the system operates on each interval $K_j$ under a fixed demand distribution and base-stock level. By \cref{lemma:stationary_distribution_backlogging}, after at most $L$ time steps following $\zeta_j$, the backlog inventory process reaches the stationary distribution associated with $(F_{\zeta_j},\tau_{\zeta_j})$. Hence only the first $L$ steps of each interval can contribute to the transient mismatch. It follows
    \begin{equation} \label{eq:Deltaone_backlog}
        \Delta_1 = \E\bigg[\sum_{j=1}^J \sum_{t = \zeta_j}^{\min\{\zeta_j+L, \zeta_{j+1}\}-1} \left(C_t(\tau_{\zeta_j}) - \mu_{\zeta_j}(\tau_{\zeta_j})\right) \bigg].
    \end{equation}
    
    Consider a time step ${t \in [\zeta_j, \min\{\zeta_j+L, \zeta_{j+1}\}-1]}$. Let the actual inventory after replenishment be $s_t = I_t + Q_{t-L}$ and that in steady state under $(\tau_{\zeta_j}, F_{\zeta_j})$ be denoted $s^{\mathrm{stat}}_t$. With \cref{lemma:lipschitz_in_s} we have that
    \begin{equation}
    \label{eq:EC_mu}
        \abs*{\E_{F_{\zeta_j}}[C_t(\tau_{\zeta_j})] - \mu_{\zeta_j}(\tau_{\zeta_j})} \leq \max\{h, b\} \cdot \E_{F_{\zeta_j}}\left[\abs*{s_t - s^{\mathrm{stat}}_t}\right].
    \end{equation}
    Since $t\in K_j$, the inventory position is maintained at $\tau_{\zeta_j}$, hence
    \[
        s_t = I_t + Q_{t-L} = I_t + \sum_{i=0}^{L} Q_{t-i} - \sum_{i=0}^{L-1} Q_{t-i} = \tau_{\zeta_j} - \sum_{i=0}^{L-1} Q_{t-i}.
    \]
    Moreover, by \cref{lemma:stationary_distribution_backlogging}, we have ${s^{\mathrm{stat}}_t = \tau_{\zeta_j} - \sum_{i=1}^L D^{(j)}_{t-i}}$ where we denote $\{D^{(j)}_{t'}\}_{t' \in [T]}$ as an i.i.d.\ demand sequence sampled from distribution $F_{\zeta_j}$.
    The absolute difference between the two inventory levels is bounded by the triangle inequality via
    \begin{equation} \label{eq:Q_t}
        \abs*{s_t - s^{\mathrm{stat}}_t} = \abs*{\sum_{i=1}^L D^{(j)}_{t-i} - \sum_{i=0}^{L-1} Q_{t-i}} \leq \sum_{i=1}^L D^{(j)}_{t-i} + \sum_{i=0}^{L-1} Q_{t-i}.
    \end{equation}
    It remains to bound the recent order quantities.  Let $P_t \coloneq I_t + \sum_{i=0}^L Q_{t-i}$ denote the inventory position after ordering. The backlogging dynamics imply
        $P_t = P_{t-L} - \sum_{i=1}^L D_{t-i} + \sum_{i=0}^{L-1} Q_{t-i}$.
    Rearranging gives
    \[
        \sum_{i=0}^{L-1} Q_{t-i} = P_t-P_{t-L} + \sum_{i=1}^L D_{t-i}.
    \]
    Since $P_t\leq U$ and $P_{t-L}\geq 0$, we have that $\sum_{i=0}^{L-1} Q_{t-i} \leq U + \sum_{i=1}^L D_{t-i}$.  Moreover, denoting $\bar{D} \coloneq \sup_{t \in [T]} \E_{F_t}[D]$ (since under \Cref{ass:main_assumption}.\ref{ass:sub-Gaussian_demand} the absolute first moments of the demand distributions $F_t$ are finite constants) we have via \cref{eq:Q_t} that
    \[
        \E_{F_{\zeta_j}}\left[\abs*{s_t - s^{\mathrm{stat}}_t}\right] \leq \sum_{i=1}^L \E_{F_{\zeta_j}}\left[D^{(j)}_{t-i}\right] + U + \sum_{i=1}^L \E_{F_{t-i}}\left[D_{t-i}\right] \leq U + 2L\bar{D}
    \]
    Finally, together with \cref{eq:J,eq:Deltaone_backlog,eq:EC_mu} it follows that
    \[
        \Delta_1 \leq \E\biggl[\sum_{j=1}^J L \max\{h,b\} (U + 2L\bar{D})\biggr] \leq (S\left(\lceil U/\gamma \rceil + 1 \right) - 1) L  \max\{h,b\} (U + 2L\bar{D}).
    \]

\noindent\textbf{Bound of $\Delta_{\tau_0}$ under Backlogging.}\quad
    Fix an episode $v$ and epoch $k\geq 2$. During the depletion period $[\alpha_{v,k},\bar{\alpha}_{v,k}-1]$, orders are suspended until the inventory position falls down to the new target $\tau_v^k$. Let $s_t=I_t+Q_{t-L}$ denote the actual inventory after replenishment, and let $s_t^{\mathrm{stat}}$ denote the corresponding stationary inventory after replenishment under policy $\tau_v^k$ and demand distribution $F_t$. By \cref{lemma:lipschitz_in_s}, we have that
    \begin{equation} \label{eq:Ct_mu_st_ststat}
        \abs*{\E_{F_t}[C_t(\tau_v^k)] - \mu_t(\tau_v^k)} \leq \max\{h, b\} \E_{F_t}\left[\abs*{s_t - s_t^{\mathrm{stat}}}\right].
    \end{equation}

    Expressing the inventory on-hand at time $t$ as
        $I_t = I_{t-L} + \sum_{i=1}^L Q_{t-L-i} - \sum_{i=1}^L D_{t-i}$
    we obtain
    \[
        s_t = I_t + Q_{t-L} = I_{t-L} + \sum_{i=0}^L Q_{t-L-i} - \sum_{i=1}^L D_{t-i}.
    \]
    We observe that the sum $I_{t-L} + \sum_{i=0}^L Q_{t-L-i}$ is precisely the inventory position after ordering at time $t-L$ which is at least $\tau_{t-L}$. It immediately follows that
        $\tau_{t-L}-\sum_{i=1}^L D_{t-i} \leq s_t \leq U$.
    Moreover, by \cref{lemma:stationary_distribution_backlogging}, the steady-state inventory after replenishment under policy $\tau_v^k$ and demand distribution $F_t$ satisfies $s^{\mathrm{stat}}_t = \tau_v^k - \sum_{i=1}^L D^{(t)}_{t-i} \leq \tau_v^k$, where $\{D^{(t)}_{t'}\}_{t'\in[T]}$ are i.i.d.\ draws from $F_t$. 
    Hence, if $s_t\ge s_t^{\mathrm{stat}}$, then
    \[
        s_t-s_t^{\mathrm{stat}} \leq U- \bigg(\tau_v^k - \sum_{i=1}^L D^{(t)}_{t-i}\bigg) \leq U + \sum_{i=1}^L D^{(t)}_{t-i}.
    \]
    Otherwise, if $s_t < s_t^{\mathrm{stat}}$, then
    \[
        s_t^{\mathrm{stat}}-s_t \leq \tau_v^k - \bigg(\tau_{t-L} - \sum_{i=1}^L D_{t-i}\bigg) \leq U + \sum_{i=1}^L D_{t-i}.
    \]
    With $\bar{D} \coloneq \sup_{t \in [T]} \E_{F_t}[D]$, this yields
        $\E_{F_t}\left[\abs*{s_t-s_t^{\mathrm{stat}}}\right] \leq U + L \bar{D}$.
    Substituting this into \cref{eq:Ct_mu_st_ststat} gives
    \[
        \abs*{\E_{F_t}[C_t(\tau_v^k)] - \mu_t(\tau_v^k)} \leq \max\{h, b\} \left(U + L \bar{D}\right).
    \]
    Finally, it takes in expectation at most $L + (\tau_v^{k-1} - \tau_v^k)\nu$ steps to deplete an amount $\tau_v^{k-1} - \tau_v^k$ of inventory at the start of epoch $k$ in episode $v$ (see proof of bound on $R_{\tau_0}(T)$ in \cref{app:regret_LS_Lpositive}).
    The number of epochs per episode is deterministically upper bounded by $\lceil U/\gamma\rceil$ and the total amount by which the base-stock level decreases over the course of an episode by $U$. This yields
    \begin{align*}
        \E\biggl[\sum_{k \geq 2} \sum_{t \in [\alpha_{v,k}, \bar{\alpha}_{v,k} - 1]} \abs*{\E_{F_t}[C_t(\tau_v^k)] - \mu_t(\tau_v^k)}\biggr] &\leq \max\{h, b\} \left(U+L\bar{D} \right) \E\biggl[\sum_{k \geq 2} (L + (\tau_v^{k-1} - \tau_v^k)\nu) \biggr] \\
        &\leq \max\{h, b\} \left(U+L\bar{D} \right) \left(\lceil U/\gamma\rceil L + U\nu \right).
    \end{align*}
    Summing across all up to $S$ episodes (\cref{lemma:false_alarms}), we obtain
    \[
        \Delta_{\tau_0} \leq \E\biggl[\sum_{v=1}^{\hat{S}} \sum_{k \geq 2} \sum_{t \in [\alpha_{v,k}, \bar{\alpha}_{v,k} - 1]} \abs*{\E_{F_t}[C_t(\tau_v^k)] - \mu_t(\tau_v^k)}\biggr] \leq S \max\{h, b\} \left(U+L\bar{D} \right) \left(\lceil U/\gamma\rceil L + U\nu \right).
    \]

\noindent\textbf{Bound of $\Delta_1$ under Lost-Sales.}\quad
    Fix an interval $K_j$ and let $h^{(j)}$ be the bias function for the average cost problem with demand distribution $F_{\zeta_j}$ and base-stock level $\tau_{\zeta_j}$. We temporarily modify notation here, and denote $C_t(s, \tau)$ to be the random cost under demand $D \sim F_t$ with state $s$ and base-stock policy $\tau$.  Note that $C_t(\tau_t) = C_t(s_t, \tau_t)$ under this notation.
    Then, for every state $s$ and $t \in K_j$, we have
    \[
        \mu_{\zeta_j}(\tau_{\zeta_j}) + h^{(j)}(s) = c_{\zeta_j}(s, \tau_{\zeta_j}) + \E_{F_{\zeta_j}}\left[ h^{(j)}(S_{t+1}) \mid S_t = s \right],
    \]
    or equivalently
    \(
        c_{\zeta_j}(s, \tau_{\zeta_j}) - \mu_{\zeta_j}(\tau_{\zeta_j}) = h^{(j)}(s) - \E_{F_{\zeta_j}}\left[ h^{(j)}(S_{t+1}) \mid S_t = s \right].
    \)
    Summing over one stationary interval and rearranging the telescoping sum across $t$ gives
    \begin{align*}
        \sum_{t \in K_j} \left( \E_{F_{\zeta_j}}[C_t(\tau_{\zeta_j})] - \mu_{\zeta_j}(\tau_{\zeta_j}) \right) &= \sum_{t \in K_j} \E_{F_{\zeta_j}}\left[ h^{(j)}(S_t) - h^{(j)}(S_{t+1})\right] 
        = \E_{F_{\zeta_j}}\left[ h^{(j)}(S_{\zeta_j}) - h^{(j)}(S_{\zeta_{j+1}})\right].
    \end{align*}
    The bias function in the lost-sales model has a uniformly bounded span of $36\max\{h,b\}LU$ \citep[Lemma 9]{agrawal2022learning}, that is, $\max_s h^{(j)}(s) - \min_s h^{(j)}(s) \leq 36 \max\{h,b\}LU$.
    Therefore, we have that
    \[
        \sum_{t \in K_j} \left( \E_{F_{\zeta_j}}[C_t(\tau_{\zeta_j})] - \mu_{\zeta_j}(\tau_{\zeta_j}) \right) \leq 36\max\{h,b\}LU.
    \]
    Summing over the $J$ intervals, with \cref{eq:Deltaone,eq:J} we obtain
    \begin{align*}
        \Delta_1 &= \E\biggl[\sum_{j=1}^J \sum_{t \in K_j} \left( C_t(\tau_{\zeta_j}) - \mu_{\zeta_j}(\tau_{\zeta_j})\right) \biggr]
        \leq (S\left(\lceil U/\gamma \rceil + 1 \right) - 1) 36\max\{h,b\}LU.
    \end{align*}

\noindent\textbf{Bound of $\Delta_{\tau_0}$ under Lost-Sales.}\quad
    Consider an epoch $k$ in episode $v$ and a time step $t \in [\alpha_{v,k},\bar{\alpha}_{v,k}-1]$. Under \Cref{ass:main_assumption}.\ref{ass:bounded_action_space}, both the actual and stationary inventory after replenishment in the lost-sales regime lie in $[0,U]$. Thus, with $s_t=I_t+Q_{t-L}$ and $s_t^{\mathrm{stat}}$ denoting the corresponding stationary inventory after replenishment under policy $\tau_v^k$ and demand distribution $F_t$, \Cref{lemma:lipschitz_in_s} gives
    \[
         \abs*{\E_{F_t}[C_t(\tau_v^k)] - \mu_t(\tau_v^k)} \leq \max\{h, b\} \E_{F_t}\left[\abs*{s_t - s^{\mathrm{stat}}_t}\right] \leq \max\{h,b\} U.
    \]
    With the upper bound on the expected number of inventory depletion steps per episode (see proof of the bound on $R_{\tau_0}(T)$ in \cref{app:regret_LS_Lpositive}) and the upper bound on the number of episodes $\hat{S}$ (\cref{lemma:false_alarms}) we obtain
    \[
        \Delta_{\tau_0} \leq \E\biggl[\sum_{v=1}^{\hat{S}} \sum_{k \geq 2} \sum_{t \in [\alpha_{v,k}, \bar{\alpha}_{v,k} -1]} \abs*{\E_{F_t}[C_t(\tau_v^k)] - \mu_t(\tau_v^k)}\biggr] \leq S(L \lceil U/\gamma \rceil + U\nu) \max\{h,b\} U.
    \]

\noindent\textbf{Final Bound.}\quad
    Overall, we have for the backlogging model
    \begin{align*}
        \Delta_1 \leq (S\left(\lceil U/\gamma \rceil + 1 \right) - 1) L  \max\{h,b\} (U + 2L\bar{D}) \ \ \text{ and } \ \ 
        \Delta_{\tau_0} \leq S \max\{h, b\} \left(U+L\bar{D} \right) \left(\lceil U/\gamma\rceil L + U\nu \right),
    \end{align*}
    and for the lost-sales model
    \begin{align*}
        \Delta_1&\leq 36 (S\left(\lceil U/\gamma \rceil + 1 \right) - 1) \max\{h,b\}LU \ \ \text{ and } \ \
        \Delta_{\tau_0} \leq  S\left(L \lceil U/\gamma \rceil + U\nu\right) \max\{h,b\} U.
    \end{align*}
    Combining this with \Cref{lemma:decomposition_regret_difference} gives the final result. 
\end{rproofof}

\subsection{Discussion and Related Literature}
\label{app:related_lit_discussion_avg_cost}

Below, we summarize the implications of this alternative benchmark and relate our guarantees to existing non-stationary MDP results.

\paragraph{Implications of Alternative Benchmark.} The benchmark $R_C(T)$ in \cref{def:regret_expected_cost} is more demanding than the stationary average cost regret $R(T)$ in \cref{def:regret}.  The main difference is that $R_C(T)$ also accounts for transient state effects after policy or demand distribution changes.  Our results show that these transient effects do not change the regret order in the zero lead time settings.  For both backlogging and lost-sales with $L = 0$, the bounds under $R_C(T)$ differ from those under $R(T)$ only by an additional term which is independent of $T$. Hence, the stronger benchmark does not affect the asymptotic regret scaling in this regime, even when the number of distribution changes is unknown.

With positive lead times, the transient effects are more pronounced because the current state depends on past orders. When $S$ is known, however, the algorithms can be tuned so that the rates for $R_C(T)$ are of the same order as $R(T)$. Thus, knowledge of the degree of non-stationarity allows the algorithm to compensate for the transient effects.
When $S$ is unknown, the bounds acquire additional terms that are linear in $S$ but do not worsen the leading dependence on $T$. Specifically, under backlogging, this results in an additional $\mathcal{O}(S\sqrt{T})$ term, while for lost-sales the additional term scales as $\mathcal{O}(ST^{1/4})$. Moreover, we observe that the additional $T$-dependent terms arise from the discretization argument, specifically through the number of epochs per episode in the bound on depletion periods and through the number of intervals $K_j$. To the best of our knowledge, comparable transient regret guarantees are not available for continuous-state, continuous-action infinite horizon MDPs with non-stationarity, hence it remains unclear whether this dependence is intrinsic or an artifact of our analysis.

\paragraph{Relation to Non-Stationary MDP Literature.}
We first note that, to the best of our knowledge, all papers that study non-stationary infinite horizon average cost MDPs consider the transient cost benchmark (\cref{def:regret_expected_cost}).  Below we compare the resulting regret guarantees and assumption on whether or not the degree of non-stationarity is known.

Most of the literature largely focuses on finite state and action spaces, and many sharp guarantees require prior knowledge of the degree of non-stationarity.
For example, \citet{auer2008near} study piecewise stationary finite MDPs, and their restarting-based algorithm obtains a dynamic regret of $\tilde{\mathcal{O}}(S^{1/3}T^{2/3})$. To avoid the information loss caused by periodic restarts, \citet{gajane2018sliding} introduce a sliding-window approach which achieves the same order. Similarly, \citet{ortner2020variational} also consider continuously  varying, finite MDPs and establish $\tilde{\mathcal{O}}(V^{1/3}T^{2/3})$ regret. However, these guarantees strictly depend on knowing the exact number of shifts $S$ or the total variation budget $V$ to calibrate their window sizes or restart frequencies.

Parameter-free approaches remove this requirement, typically at the cost of a weaker regret rate. \citet{wei2021non} obtain optimal dynamic regret guarantees for finite MDPs when structural quantities such as the diameter, $S$, or $V$ are known, while their parameter-free reduction incurs an additional $\tilde{\mathcal{O}}(T^{3/4})$ term.
Likewise, \citet{cheung2023nonstationary} consider finite MDPs and achieve $\tilde{\mathcal{O}}(V^{1/4}T^{3/4})$ regret when the variation budget is known, while their Bandit-over-RL approach achieves the same order without prior knowledge of the variation budget. For single-item lost-sales inventory systems with zero lead time, fixed ordering costs, and known $V$, they further obtain $\tilde{\mathcal{O}}(V^{1/3}T^{2/3})$ under additional assumptions on the demand distribution. Our results instead address inventory control problems with continuous demand and action spaces, positive lead times, and unknown non-stationarity. We obtain near-optimal rates when $L=0$, $\sqrt{T}$-type rates for backlogging with $L>0$ under the stationary average cost benchmark, and comparable expected cost guarantees when the degree of non-stationarity is known.


Overall, the transient regret analysis demonstrates that our algorithms remain competitive even under this substantially stronger benchmark. The results suggest that the main difficulty introduced by transient evaluation is the combination of a continuous policy space, strictly positive lead times and an unknown frequency of change. To the best of our knowledge, this is the first work to establish dynamic regret guarantees in $R_C(T)$ under this general setting, while maintaining the strong performance under $R(T)$ whenever $L=0$ or the degree of non-stationarity is known, and otherwise still retaining the rate with respect to $T$.

\subsection{Simulation Results}
We also empirically evaluate our algorithms under the expected cost benchmark in \cref{def:regret_expected_cost}. The results are shown in \cref{fig:simulations_expected_cost_vs_S,fig:simulations_expected_cost_vs_L} and are directly comparable to the dynamic regret experiments in \cref{sec:results_experiments} (see \cref{fig:simulations_regret_vs_S,fig:simulations_regret_vs_L}). Overall, the empirical behavior under expected cost is qualitatively consistent with the behavior under stationary average cost regret.  We omit replicating the other simulations as such.

As shown in \cref{fig:simulations_expected_cost_vs_S}, increasing the number of demand changes $S$ increases the cumulative expected cost across all models and demand distributions. The dependence on $S$ follows the same broad trends observed for dynamic regret.  The algorithms remain stable as the environment changes more frequently, and the gap between the backlogging and lost-sales regimes is modest in the tested instances.  \Cref{fig:simulations_expected_cost_vs_L} shows the dependence on the lead time $L$. As in the dynamic regret experiments, the cumulative expected cost grows approximately linearly with $L$. This supports the conclusion that the proposed algorithms retain favorable dependence on lead time even under the stronger expected cost benchmark, in contrast to generic finite-state approaches whose state spaces grow rapidly with $L$.
Taken together, these experiments suggest that the transient effects captured by $R_C(T)$ do not materially change the qualitative conclusions from \cref{sec:results_experiments}.

\begin{figure}[h]
    \centering
    \begin{subfigure}{0.9\textwidth}
        \begin{tikzpicture}

\begin{axis}[
    name=plot_normal,
    width=6.5cm,
    height=3.7cm,
    xmajorgrids=true, 
    grid style={dashed, gray!30},
    xmode=log,
    ymode=log,
    ymax=7.2e6,
    ytick={5e6},
    yticklabels={$5 \cdot 10^6$},
    axis x line=bottom,
    axis y line=middle,
    xmax=1000,
    enlargelimits=false,
    xtick={1,5,10,22,100,464},
    xticklabels={
                {$1$},
                {\hspace{-4mm}$5$},
                {\hspace{-0mm}$\log(T)$},
                {\hspace{7mm}$T^{1/3}$},
                {\hspace{3mm}$T^{1/2}$},
                {$T^{2/3}$}
                },
    xlabel={$S$},
    ylabel={$\sum_{t=1}^{T}\E[C_t]$},
    xlabel style={
    at={(axis description cs:1,-0.06)},
    anchor=north
    },
    ylabel style={
    at={(axis description cs:-0.2,0.65)},
    anchor=north
    },
    clip=false
]

\addplot[thick, mark=o, color=nupurple] 
coordinates {
    (1, 464685.5555)
    (5, 872759.6469)
    (10, 925998.2271)
    (22, 959525.9627)
    (100, 1001339.937)
    (464, 1057087.505)
};

\addplot[thick, mark=square*, color=kulblue] 
coordinates {
    (1, 903974.2835)
    (5, 2084257.401)
    (10, 2252243.227)
    (22, 2389729.174)
    (100, 2539277.092)
    (464, 2759834.811)
};

\addplot[thick, mark=triangle*, color=darkorange] 
coordinates {
    (1, 1567575.268)
    (5, 3266646.553)
    (10, 4230753.177)
    (22, 4664757.753)
    (100, 4980443.711)
    (464, 5382121.586)
};
\end{axis}

\hspace{0.5cm}
\begin{axis}[
    name=plot_normal_LS,
    at={(plot_normal.right of south east)},
    xshift=1.25cm,
    width=6.5cm,
    height=3.7cm,
    xmajorgrids=true, 
    grid style={dashed, gray!30},
    xmode=log,
    ymode=log,
    ymax=7.2e6,
    ytick={5e6},
    yticklabels={$5 \cdot 10^6$},
    axis x line=bottom,
    axis y line=middle,
    xmax=1000,
    enlargelimits=false,
    xtick={1,5,10,22,100,464},
    xticklabels={
                {$1$},
                {\hspace{-4mm}$5$},
                {\hspace{-0mm}$\log(T)$},
                {\hspace{7mm}$T^{1/3}$},
                {\hspace{3mm}$T^{1/2}$},
                {$T^{2/3}$}
                },    
    xlabel={$S$},
    ylabel={$\sum_{t=1}^{T}\E[C_t]$},
    xlabel style={
    at={(axis description cs:1,-0.06)},
    anchor=north
    },
    ylabel style={
    at={(axis description cs:-0.2,0.65)},
    anchor=north
    },
    clip=false
]

\addplot[thick, mark=o, color=nupurple] 
coordinates {
    (1, 470285.1073)
    (5, 870403.8514)
    (10, 938986.3421)
    (22, 983827.7092)
    (100, 1027735.493)
    (464, 1069195.598)
};

\addplot[thick, mark=square*, color=kulblue] 
coordinates {
    (1, 869809.8537)
    (5, 1362679.913)
    (10, 1572747.25)
    (22, 2089951.252)
    (100, 2260819.937)
    (464, 2363510.707)
};

\addplot[thick, mark=triangle*, color=darkorange] 
coordinates {
    (1, 1215430.557)
    (5, 1828500.464)
    (10, 2222742.841)
    (22, 2949814.985)
    (100, 3931385.495)
    (464, 4074598.266)
};
\end{axis}

\begin{axis}[
  hide axis,
  name=dummy_axis,
  xmin=0, xmax=1,
  ymin=0, ymax=1,
  legend columns=1,
  legend cell align=left,
  at={(plot_normal_LS.north east)},
  anchor=north east,                 
  xshift=1.5cm,                     
  yshift=-0.0cm,  
  legend style={
    draw=gray!30,
  }
]
\addlegendimage{thick, mark=o, color=nupurple}
\addlegendentry{0}

\addlegendimage{thick, mark=square*, color=kulblue}
\addlegendentry{2}

\addlegendimage{thick, mark=triangle*, color=darkorange}
\addlegendentry{5}
\end{axis}

\end{tikzpicture}
        \caption{Normal.}
        \label{fig:simulations_expected_cost_vs_S_normal}
    \end{subfigure}
    \begin{subfigure}{0.9\textwidth}
        \begin{tikzpicture}

\begin{axis}[
    name=plot_uniform,
    width=6.5cm,
    height=3.7cm,
    xmajorgrids=true, 
    grid style={dashed, gray!30},
    xmode=log,
    ymode=log,
    ymax=7.2e6,
    ytick={5e6},
    yticklabels={$5 \cdot 10^6$},
    axis x line=bottom,
    axis y line=middle,
    xmax=1000,
    enlargelimits=false,
    xtick={1,5,10,22,100,464},
    xticklabels={
                {$1$},
                {\hspace{-4mm}$5$},
                {\hspace{-0mm}$\log(T)$},
                {\hspace{7mm}$T^{1/3}$},
                {\hspace{3mm}$T^{1/2}$},
                {$T^{2/3}$}
                },    
    xlabel={$S$},
    ylabel={$\sum_{t=1}^{T}\E[C_t]$},
    xlabel style={
    at={(axis description cs:1,-0.06)},
    anchor=north
    },
    ylabel style={
    at={(axis description cs:-0.2,0.65)},
    anchor=north
    },
    clip=false
]

\addplot[thick, mark=o, color=nupurple] 
coordinates {
    (1, 133241.832)
    (5, 620993.61)
    (10, 695300.7436)
    (22, 757592.2475)
    (100, 878853.8624)
    (464, 985570.0092)
};

\addplot[thick, mark=square*, color=kulblue] 
coordinates {
    (1, 391437.7705)
    (5, 1011736.636)
    (10, 1432024.992)
    (22, 2125273.05)
    (100, 2493436.454)
    (464, 2825036.987)
};

\addplot[thick, mark=triangle*, color=darkorange] 
coordinates {
    (1, 928444.4544)
    (5, 2020792.1)
    (10, 2913646.151)
    (22, 4653134.907)
    (100, 5300733.744)
    (464, 5804595.701)
};
\end{axis}

\hspace{0.5cm}
\begin{axis}[
    name=plot_uniform_LS,
    at={(plot_uniform.right of south east)},
    xshift=1.25cm,
    width=6.5cm,
    height=3.7cm,
    xmajorgrids=true, 
    grid style={dashed, gray!30},
    xmode=log,
    ymode=log,
    ymax=7.2e6,
    ytick={5e6},
    yticklabels={$5 \cdot 10^6$},
    axis x line=bottom,
    axis y line=middle,
    xmax=1000,
    enlargelimits=false,
    xtick={1,5,10,22,100,464},
    xticklabels={
                {$1$},
                {\hspace{-4mm}$5$},
                {\hspace{-0mm}$\log(T)$},
                {\hspace{7mm}$T^{1/3}$},
                {\hspace{3mm}$T^{1/2}$},
                {$T^{2/3}$}
                },
    xlabel={$S$},
    ylabel={$\sum_{t=1}^{T}\E[C_t]$}, 
    xlabel style={
    at={(axis description cs:1,-0.06)},
    anchor=north
    },
    ylabel style={
    at={(axis description cs:-0.2,0.65)},
    anchor=north
    },
    clip=false
]

\addplot[thick, mark=o, color=nupurple] 
coordinates {
    (1, 145549.8391)
    (5, 619903.1319)
    (10, 729627.9607)
    (22, 813887.8876)
    (100, 981692.4734)
    (464, 1062779.677)
};

\addplot[thick, mark=square*, color=kulblue] 
coordinates {
    (1, 371103.4197)
    (5, 890994.0255)
    (10, 1348154.767)
    (22, 1956823.464)
    (100, 2282070.27)
    (464, 2552032.664)
};

\addplot[thick, mark=triangle*, color=darkorange] 
coordinates {
    (1, 682369.6621)
    (5, 1342325.789)
    (10, 1869484.54)
    (22, 2862547.352)
    (100, 4048856.688)
    (464, 4311104.664)
};
\end{axis}

\begin{axis}[
  hide axis,
  name=dummy_axis,
  xmin=0, xmax=1,
  ymin=0, ymax=1,
  legend columns=1,
  legend cell align=left,
  at={(plot_uniform_LS.north east)},
  anchor=north east,                 
  xshift=1.5cm,                     
  yshift=-0.0cm,  
  legend style={
    draw=gray!30,
  }
]
\addlegendimage{thick, mark=o, color=nupurple}
\addlegendentry{0}

\addlegendimage{thick, mark=square*, color=kulblue}
\addlegendentry{2}

\addlegendimage{thick, mark=triangle*, color=darkorange}
\addlegendentry{5}
\end{axis}

\end{tikzpicture}
        \caption{Uniform.}
        \label{fig:simulations_expected_cost_vs_S_uniform}
    \end{subfigure}
    \caption{Expected cost vs. $S$ of \NSBLIC under backlogging (left), and \NSLSIC or \NSLSICL under lost-sales (right), shown on a log-log scale and for different values of lead time~$L$.  In \cref{fig:simulations_expected_cost_vs_S_normal} demand is drawn from the truncated normal distribution and in \cref{fig:simulations_expected_cost_vs_S_uniform} from the uniform distribution.}
    \label{fig:simulations_expected_cost_vs_S}
\end{figure}

\begin{figure}[ht]
    \centering
    \begin{subfigure}{0.9\textwidth}
        \begin{tikzpicture}

\begin{axis}[
  name=plot_normal,
  axis x line=bottom,
  axis y line=middle,
  width=5.5cm,
  height=3.7cm,
  xmajorgrids=true, 
  grid style={dashed, gray!30},
  xmin=0,
  xmax=0.85,
  ymin=0,
  ymax=5.4e6,
  ytick={5e6},
  yticklabels={$5 \cdot 10^6$},
  scaled y ticks=false,
  enlargelimits=false,
  xtick={0,0.3,0.75},
  xticklabels={$0$,$2$,$5$},
  xlabel={$L$},
  ylabel={$\sum_{t=1}^{T}\E[C_t]$},
  xlabel style={
  at={(axis description cs:1,-0.03)},
  anchor=north
  },
  ylabel style={
  at={(axis description cs:-0.25,0.65)},
  anchor=north
  },
  clip=false
]

\addplot[thick, mark size=2pt, color=black, mark=o, solid, forget plot]
coordinates {(0,464685.5555) (0.3,903974.2835) (0.75,1567575.268)};

\addplot[thick, mark size=2pt, color=red!85!black, mark=square*, solid, forget plot]
coordinates {(0,872759.6469) (0.3,2084257.401) (0.75,3266646.553)};

\addplot[thick, mark size=2pt, color=darkorange, mark=triangle*, solid, forget plot]
coordinates {(0,925998.2271) (0.3,2252243.227) (0.75,4230753.177)};

\addplot[thick, mark size=2pt, color=green!70!black, mark=diamond*, solid, forget plot]
coordinates {(0,959525.9627) (0.3,2389729.174) (0.75,4664757.753)};

\addplot[thick, mark size=2pt, color=kulblue, mark=star, forget plot]
coordinates {(0,1001339.937) (0.3,2539277.092) (0.75,4980443.711)};

\addplot[thick, mark size=2pt, color=nupurple, mark=x, forget plot]
coordinates {(0,1057087.505) (0.3,2759834.811) (0.75,5382121.586)};

\end{axis}

\begin{axis}[
  name=plot_normal_LS,
  at={(plot_normal.right of south east)},
  xshift=1.75cm,
  axis x line=bottom,
  axis y line=middle,
  width=5.5cm,
  height=3.7cm,
  xmajorgrids=true, 
  grid style={dashed, gray!30},
  xmin=0,
  xmax=0.85,
  ymin=0,
  ymax=5.4e6,
  ytick={5e6},
  yticklabels={$5 \cdot 10^6$},
  scaled y ticks=false,
  enlargelimits=false,
  xtick={0,0.3,0.75},
  xticklabels={$0$,$2$,$5$},
  ylabel={$\sum_{t=1}^{T}\E[C_t]$},
  xlabel={$L$},
  xlabel style={
  at={(axis description cs:1,-0.03)},
  anchor=north
  },
  ylabel style={
  at={(axis description cs:-0.25,0.65)},
  anchor=north
  },
  clip=false
]

\addplot[thick, mark size=2pt, color=black, mark=o, solid, forget plot]
coordinates {(0,470285.1073) (0.3,869809.8537) (0.75,1215430.557)};

\addplot[thick, mark size=2pt, color=red!85!black, mark=square*, solid, forget plot]
coordinates {(0,870403.8514) (0.3,1362679.913) (0.75,1828500.464)};

\addplot[thick, mark size=2pt, color=darkorange, mark=triangle*, solid, forget plot]
coordinates {(0,938986.3421) (0.3,1572747.25) (0.75,2222742.841)};

\addplot[thick, mark size=2pt, color=green!70!black, mark=diamond*, solid, forget plot]
coordinates {(0,983827.7092) (0.3,2089951.252) (0.75,2949814.985)};

\addplot[thick, mark size=2pt, color=kulblue, mark=star, forget plot]
coordinates {(0,1027735.493) (0.3,2260819.937) (0.75,3931385.495)};

\addplot[thick, mark size=2pt, color=nupurple, mark=x, forget plot]
coordinates {(0,1069195.598) (0.3,2363510.707) (0.75,4074598.266)};

\end{axis}

\begin{axis}[
  hide axis,
  name=dummy_axis,
  xmin=0, xmax=1,
  ymin=0, ymax=1,
  legend columns=2,
  legend cell align=left,
  at={(plot_normal_LS.north east)}, 
  anchor=north east,                  
  xshift=4.0cm,                    
  yshift=-0.0cm,    
  legend style={
    draw=gray!30,
  }
]

\addlegendimage{thick, mark=o, color=black}
\addlegendentry{1}

\addlegendimage{thick, mark=square*, color=red!85!black}
\addlegendentry{$\mathcal{O}(1)$}

\addlegendimage{thick, mark=triangle*, color=darkorange}
\addlegendentry{$\log(T)$}

\addlegendimage{thick, mark=diamond*, color=green!70!black}
\addlegendentry{$T^{1/3}$}

\addlegendimage{thick, mark=star, color=kulblue}
\addlegendentry{$T^{1/2}$}

\addlegendimage{thick, mark=x, color=nupurple}
\addlegendentry{$T^{2/3}$}
\end{axis}

\end{tikzpicture}
        \caption{Normal.}
        \label{fig:simulations_expected_cost_vs_L_normal}
    \end{subfigure}
    \begin{subfigure}{0.9\textwidth}
        \begin{tikzpicture}

\begin{axis}[
  name=plot_uniform,
  axis x line=bottom,
  axis y line=middle,
  width=5.5cm,
  height=3.7cm,
  xmajorgrids=true, 
  grid style={dashed, gray!30},
  xmin=0,
  xmax=0.85,
  ymin=0,
  ymax=5.4e6,
  ytick={5e6},
  yticklabels={$5 \cdot 10^6$},
  scaled y ticks=false,
  enlargelimits=false,
  xtick={0,0.3,0.75},
  xticklabels={$0$,$2$,$5$},
  xlabel={$L$},
  ylabel={$\sum_{t=1}^{T}\E[C_t]$},
  xlabel style={
  at={(axis description cs:1,-0.03)},
  anchor=north
  },
  ylabel style={
  at={(axis description cs:-0.25,0.65)},
  anchor=north
  },
  clip=false,
]

\addplot[thick, mark size=2pt, color=black, mark=o, solid, forget plot]
coordinates {(0,133241.832) (0.3,391437.7705) (0.75,928444.4544)};

\addplot[thick, mark size=2pt, color=red!85!black, mark=square*, solid, forget plot]
coordinates {(0,620993.61) (0.3,1011736.636) (0.75,2020792.1)};

\addplot[thick, mark size=2pt, color=darkorange, mark=triangle*, solid, forget plot]
coordinates {(0,695300.7436) (0.3,1432024.992) (0.75,2913646.151)};

\addplot[thick, mark size=2pt, color=green!70!black, mark=diamond*, solid, forget plot]
coordinates {(0,757592.2475) (0.3,2125273.05) (0.75,4653134.907)};

\addplot[thick, mark size=2pt, color=kulblue, mark=star, forget plot]
coordinates {(0,878853.8624) (0.3,2493436.454) (0.75,5300733.744)};

\addplot[thick, mark size=2pt, color=nupurple, mark=x, forget plot]
coordinates {(0,985570.0092) (0.3,2825036.987) (0.75,5804595.701)};

\end{axis}

\begin{axis}[
  name=plot_uniform_LS,
  at={(plot_uniform.right of south east)},
  xshift=1.75cm,
  axis x line=bottom,
  axis y line=middle,
  width=5.5cm,
  height=3.7cm,
  xmajorgrids=true, 
  grid style={dashed, gray!30},
  xmin=0,
  xmax=0.85,
  ymin=0,
  ymax=5.4e6,
  ytick={5e6},
  yticklabels={$5 \cdot 10^6$},
  scaled y ticks=false,
  enlargelimits=false,
  xtick={0,0.3,0.75},
  xticklabels={$0$,$2$,$5$},
  xlabel={$L$},
  ylabel={$\sum_{t=1}^{T}\E[C_t]$},
  xlabel style={
  at={(axis description cs:1,-0.03)},
  anchor=north
  },
  ylabel style={
  at={(axis description cs:-0.25,0.65)},
  anchor=north
  },
  clip=false
]

\addplot[thick, mark size=2pt, color=black, mark=o, solid, forget plot]
coordinates {(0,145549.8391) (0.3,371103.4197) (0.75,682369.6621)};

\addplot[thick, mark size=2pt, color=red!85!black, mark=square*, solid, forget plot]
coordinates {(0,619903.1319) (0.3,890994.0255) (0.75,1342325.7892)};

\addplot[thick, mark size=2pt, color=darkorange, mark=triangle*, solid, forget plot]
coordinates {(0,729627.9607) (0.3,1348154.767) (0.75,1869484.54)};

\addplot[thick, mark size=2pt, color=green!70!black, mark=diamond*, solid, forget plot]
coordinates {(0,813887.8876) (0.3,1956823.464) (0.75,2862547.352)};

\addplot[thick, mark size=2pt, color=kulblue, mark=star, forget plot]
coordinates {(0,981692.4734) (0.3,2282070.27) (0.75,4048856.688)};

\addplot[thick, mark size=2pt, color=nupurple, mark=x, forget plot]
coordinates {(0,1062779.677) (0.3,2552032.664) (0.75,4311104.664)};

\end{axis}

\begin{axis}[
  hide axis,
  name=dummy_axis,
  xmin=0, xmax=1,
  ymin=0, ymax=1,
  legend columns=2,
  legend cell align=left,
  at={(plot_uniform_LS.north east)},
  anchor=north east,                  
  xshift=4.0cm,                     
  yshift=-0.0cm,                     
  legend style={
    draw=gray!30,        
  }
]

\addlegendimage{thick, mark=o, color=black}
\addlegendentry{1}

\addlegendimage{thick, mark=square*, color=red!85!black}
\addlegendentry{$\mathcal{O}(1)$}

\addlegendimage{thick, mark=triangle*, color=darkorange}
\addlegendentry{$\log(T)$}

\addlegendimage{thick, mark=diamond*, color=green!70!black}
\addlegendentry{$T^{1/3}$}

\addlegendimage{thick, mark=star, color=kulblue}
\addlegendentry{$T^{1/2}$}

\addlegendimage{thick, mark=x, color=nupurple}
\addlegendentry{$T^{2/3}$}
\end{axis}

\end{tikzpicture}
        \caption{Uniform.}
        \label{fig:simulations_expected_cost_vs_L_uniform}
    \end{subfigure}
    \caption{Expected cost vs. lead time~$L$ of \NSBLIC under backlogging (left), and \NSLSIC or \NSLSICL under lost-sales (right), shown for different values of~$S$. In \cref{fig:simulations_expected_cost_vs_L_normal} demand is drawn from the truncated normal distribution and in \cref{fig:simulations_expected_cost_vs_L_uniform} from the uniform distribution.}
    \label{fig:simulations_expected_cost_vs_L}
\end{figure}

\end{APPENDICES}








\end{document}